\documentclass[12pt]{article}
\usepackage{amsfonts}
\usepackage{graphicx}
\usepackage{amsmath}
\usepackage{hyperref}
\usepackage{color}

\newtheorem{theorem}{Theorem}[section]

\newtheorem{corollary}[theorem]{Corollary}

\newtheorem{definition}[theorem]{Definition}

\newtheorem{lemma}[theorem]{Lemma}

\newtheorem{notation}[theorem]{Notation}

\newtheorem{proposition}[theorem]{Proposition}
\newtheorem{remark}[theorem]{Remark}

\newenvironment{proof}[1][Proof]{\textbf{#1.} }{\ \rule{0.5em}{0.5em}}

\def \L{\Lambda}
\def \<{\langle}
\def \>{\rangle}

\def \R{\mathbb R}

\def \D{{\cal D}}

\def \Hc{{\mathcal H}}

\def \H^0{{\cal H}^0 or}

\def \G{{\cal G}}

\def \P{{\cal P}}

\def \Y{{\cal Y}}

\def \w{\omega}

\def \kf{\frak k}

\def \p{\partial}
\def \beq{\begin{equation}}
\def \eeq{\end{equation}}

\def \n{\nabla}

\def \eref{\eqref}
\def \V{\mathcal V}

\def \ra{\rho_a}
\def \rp{\rho_p}

\def \lrc{\lrcorner\,}
\numberwithin{equation}{section}

\begin{document}

\title{The Yang-Mills heat equation with finite action
\author{Leonard Gross \\
Department of Mathematics\\
Cornell University\\
Ithaca, NY 14853-4201\\
{\tt gross@math.cornell.edu}}
}

\maketitle

\begin{abstract}
The existence and uniqueness of solutions to the Yang-Mills heat equation
is proven over $\R^3$ and over a bounded open convex set in $\R^3$. 
The initial data is taken to lie in the Sobolev space of order one half,
which is the critical Sobolev index for this equation over a three dimensional
manifold. The existence is proven by 
    solving first an augmented, strictly parabolic equation
and then gauge transforming the solution  to a solution of the Yang-Mills heat equation
itself. The gauge functions needed to carry out this procedure lie in the critical gauge group
of Sobolev regularity three halves, which is a complete topological group in a natural metric but is 
not a Hilbert Lie group.  The nature of this group must  
 be understood in order to carry
out the reconstruction procedure. Solutions to the Yang-Mills heat
equation are shown to be strong
solutions modulo these gauge functions. Energy inequalities and Neumann
domination inequalities  are used to establish needed initial behavior properties of solutions
to the augmented equation.

\end{abstract}

\tableofcontents

\section{Introduction} \label{secintro}
The Yang-Mills heat equation is a weakly parabolic, quasi-linear differential equation for a Lie algebra valued 1-form on $\R^n$.
Denote by $\kf$ the Lie algebra of a compact Lie group $K$.
Let
\beq
A(x,t) = \sum_{j=1}^n A_j(x,t) dx^j, \ \ x \in \R^n, \ \ t \ge 0,    \label{I1}
\eeq
 with coefficients $A_j(x,t) \in \kf$.
 The Yang-Mills heat  equation has the form
\beq
\frac{\p}{\p t} A(x, t) = -d^*d A(x,t) +(\text{quadratic terms + cubic terms) in}\  A. \label{I2}
\eeq
The linear terms are missing a portion of the Laplacian, $-\Delta = d^*d + d d^*$,
on 1-forms. For this reason the equation is only weakly parabolic. This paper is concerned
with the question of existence and uniqueness of solutions to the Cauchy problem
for \eref{I2} with fairly rough initial data:
Let $A_0$ be a $\kf$ valued 1-form on $\R^n$. We seek solutions to \eref{I2} such that
\beq
A(0) = A_0.                                                                          \label{I3}
\eeq
There is a standard approach to the problem of existence of solutions to a quasi-linear 
parabolic equation, $\p u/\p t = (Lu)(t) + F(u(t))$, wherein $L$ is an elliptic linear operator and 
$F$ is a possibly non-linear function of the unknown $u|_t$. One converts the differential
equation to the more or less equivalent integral equation 
$u(t) = e^{tL} u_0 + \int_0^t e^{(t-s)L} F(u(s)) ds$ and uses then a contraction
principle in a space of paths $u:[0,T]\rightarrow$ 1-forms on $\R^n$.
 But if $L$ is not elliptic then the basis for all the estimates that one needs
in order  to carry out this procedure breaks down. In the case of \eref{I2}, 
one has $L = -d^*d$ on 1-forms
and $L$ is therefore not elliptic. The failure of this standard method is
 accompanied by the  failure
of the equation itself to smooth out initial data. This is quite visible
 in \eref{I2} in case the compact
group $K$ is just the circle group. In this case all the nonlinear terms are zero. The resulting
equation has time independent solutions of the form $A(t, x) = d\lambda(x)$ for any function 
$\lambda:\R^n \rightarrow \R$. $\lambda$ does not even have to be differentiable
for this to be a solution because $d^2 =0$ in any reasonable generalized sense. 
But if, say, $\lambda \in C_c^2(\R^n)$ then this is a classical solution and clearly
 the 
 evolution does not smooth the initial data $A_0: = d\lambda$.  Even if $K$ is not
 commutative such ``pure gauge'' solutions exist and are similarly propagated by the equation
  in a time independent way. This phenomenon  greatly affects the nature of the
  solutions that one  would         expect if the equation were parabolic.   
        Ignoring this for a moment, one can compute that the critical
  Sobolev index  for our problem   in dimension three is one half in the sense of scaling. 
  That is, the Sobolev $H_a(\R^3)$ norm of a 1-form is invariant under the scaling
   $\R^3\ni x \mapsto c x$   if (and only if) $a = 1/2$, 
   while the equation itself is invariant under  the scaling $x,t\mapsto cx, c^2t$. 
    Since we will be concerned only with
  spatial dimension $n =3$, one can hope, at best,  that the Cauchy problem has long time
  solutions when the initial data $A_0$ is in the Sobolev space $H_{1/2}$. 
  This is indeed what we will   prove. But the fact that some  data
   propagates in a time independent way
  means that one cannot expect that the solution will always be 
  smooth enough,   even at strictly positive time,  to give
  clear meaning to those nonlinear terms which depend on the first
  spatial derivatives of $A(t)$, since $A(t)$ need not  be in $H_1(\R^3)$ for $t >0$.

  To be more precise,
  denote by $B:= dA + A\wedge A$ the curvature (magnetic field) of the connection
   form $A$ over $\R^3$. 
Then the Yang-Mills heat equation is
\beq
\frac{\p}{\p t} A(x,t) = -d_A^* B,                  \label{I4}
\eeq
where $d_A^* = d^* +$ the  interior product by $A$. One can verify easily that this has the form
indicated in \eref{I2}. Recall that a function $g:\R^3 \rightarrow K$ determines the gauge transformation $A\mapsto A^g$, defined by
\beq
 A^g = g^{-1} A g + g^{-1} dg.                  \label{I5}
 \eeq
 This in turn induces
an action on the curvature given 
by $B\mapsto g^{-1} B g$. In particular if $A =0$ then $B= 0$ and the curvature
 of the ``pure gauge'' $g^{-1}dg$ is therefore 
 zero. Thus if $A_0 = g^{-1}dg$ then
  $A(t): = g^{-1}dg$ is clearly a solution to the Yang-Mills heat equation \eref{I4}.
 It can happen, therefore,  that  if the initial data is only in $H_{1/2}$, the solution
 need be no smoother than this for positive time.
 Yet, once one has  computed the curvature in some generalized sense,
  the equation  \eref{I4} may now involve only classical derivatives
   of the curvature (which is zero in this example). 
Thus  the first spatial derivatives of $A(t)$
 need to be defined in some generalized sense in this example while the needed  second derivatives are definable classically. This is the reverse of what one usually
  considers to be a weak solution.  
 The main theorem of this paper will  prove the existence and uniqueness of long time 
 solutions to the equation \eref{I4} for initial data $A_0 \in H_{1/2}$, wherein the notion of 
 solution will allow first spatial derivatives of $A(t)$ to exist only in a generalized sense
 while the resulting ``weak'' curvature of $A(t)$ is actually in $H_1$ for all $t >0$. 
 The theorem  will also  point to the 
 source of this phenomenon  by showing that there is a gauge function $g_0$ such that
 $A(t)^{g_0}$ is itself in $H_1$ for all $t > 0$. 
 Otherwise said, any initial data $A_0 \in H_{1/2}$ gives
 rise to a strong solution up to gauge transformation.

 The question of existence and uniqueness of solutions to the Yang-Mills
  heat equation  is of intrinsic interest, partly because it is a naturally
   occurring quasilinear diffusion equation and partly because of the way that gauge invariance
   intervenes in the very  formulation of the Cauchy problem. But, as in \cite{CG1, CG2},
   this work is ultimately aimed at the  
   construction  of gauge invariant functions of distributional
   initial data by a completion procedure sketched in the introduction to \cite{CG1},
    in an anticipated 
 application to quantum field theory.
 In order  to construct local observables for this application
 we will 
 be interested in solutions over a bounded open subset of $\R^3$, 
 as well as over  all of $\R^3$.  
 The question of
  boundary conditions therefore arises.
  In  \cite{CG1}  we considered Dirichlet, Neumann and Marini boundary conditions.
  The latter consists in setting the normal component of the curvature to zero on the boundary,
  \cite{Ma3, Ma4, Ma7, Ma8}.
  These are  the only boundary conditions  commensurate with the intended
   applications to quantum field theory.
 Solutions satisfying
  Marini boundary conditions will be derived from solutions satisfying
   Neumann boundary conditions in a future work.
   In this paper we will only consider
   Dirichlet and Neumann boundary conditions. 
 The meaning of the space $H_{1/2}$ will henceforth 
  depend on the choice of boundary conditions.

We are also going to derive existence and uniqueness theorems in case
 the initial data $A_0$ is in the
Sobolev space $H_a$ for some $a > 1/2$. This will illuminate the way   in which 
 some results and some techniques break down as $a$ decreases to its critical value, 1/2.
 In particular we will see that the usual contraction mechanism for proving
  existence of solutions to integral equations breaks down
  as $a \downarrow 1/2$ and must be replaced by a different contraction mechanism, special
  to the Yang-Mills heat equation.

The strategy for proving existence of solutions to the initial value problem
 \eref{I4}, \eref{I3} consists of the following components. 

{\it ZDS procedure.} 
 The   Zwanziger-Donaldson-Sadun (ZDS) method of ameliorating the failed 
 ellipticity of $d^*d$ will underly the  approach in this paper, as it did in \cite{CG1}.
     In the ZDS method one deals at first with  a modified version of \eref{I4}, obtained by adding
      a term to the right hand side,
      which makes the resulting equation strictly parabolic.
  The so augmented equation is 
  \begin{align}
  \frac{\p}{ \p t} C(t)  = - d_C^* B_C - d_C d^*C,       \label{I6}
  \end{align}
  wherein $C(t) $ is a $\kf$ valued 1-form with the same initial data $A_0$ as \eref{I3}
  and $B_C(t)$ is the curvature of $C(t)$.
  The desired solution $A(t)$ is then  recovered from $C(t)$ by a time
   dependent gauge transformation,
   \beq
   A(t) = C(t)^{g(t)},                     \label{I7}
   \eeq
   where $g(t,x)$ is determined  from $C(\cdot)$ for each point $x \in \R^3$  
   by
     a simple ordinary differential equation:
   \beq
   \frac{d}{dt} g(t, x) = (d^*C(t, x)) g(t, x), \ \ \   g(0, x) =  \text{ identity element of } K.    \label{I8}
   \eeq
   The difficulty in carrying out the ZDS procedure for the recovery of 
   $A(\cdot)$  from $C(\cdot)$ arises
   from the fact that $d^*C(t)$ has very singular behavior as $t\downarrow 0$. 
   Indeed  $d^*C(0)$ need only be in $H_{-1/2}$.  
   This reflects  itself in a corresponding degree of irregularity of the gauge
    function $g(t, \cdot)$ and its
    spatial differential $ g(t)^{-1} dg(t)$, both of  which enter into the transformation \eref{I7}. 
 To carry out the ZDS procedure it will be necessary to understand first the
  nature of the  group
 of gauge functions in which each $g(t)$ lies.

{\it Gauge groups.} If $A_0 \in H_{1/2}(\R^3)$ then the solution $C(\cdot)$ to \eref{I6} will
 be shown to be a continuous function into $  H_{1/2}(\R^3)$. We wish to construct a solution
 $A(\cdot)$ of \eref{I4} which is also a continuous function into $H_{1/2}(\R^3)$.
  A gauge transformation, $C \mapsto g^{-1} C g + g^{-1} dg$, such as occurs in\eref{I7},
 will take $H_{1/2}$ into itself if $ g \in H_{3/2}(\R^3, K)$.
 The statement that
 $g \in H_{3/2}(\R^3;K)$  needs an interpretation that 
 makes this set of 
  gauge functions into a topological  group in a way that serves 
  the needs of the ZDS procedure.
  It would be reasonable, for example, to
   define such functions to be those of the form 
  $g(x) = \exp(\alpha(x))$, with $\alpha \in H_{3/2}(\R^3;\kf)$. 
   But the $H_{3/2}$ norm of $\alpha(\cdot)$
  just barely fails to control the supremum norm of $\alpha$, with the result that $\exp(\alpha(x))$
  wraps around $K$ in an uncontrolled manner as $x$ runs over $\R^3$, 
   leading to failure  of this set to be a    topological group 
   as well as failure to serve the needs of the ZDS procedure.  
   We will define instead a 
   group $\G_{3/2}$ of gauge functions, which in its natural metric topology 
    is a complete topological group in the pointwise product and which serves the needs
    of the ZDS procedure.
    Similarly, if $A_0 \in H_a$ for some $a \in (1/2, 1]$ we will define a group $\G_{1+a}$
    of gauge functions appropriate for implementing the ZDS procedure in this case.
    For $a > 1/2$ this group is a nice Hilbert manifold, whereas for $a=1/2$
    there appears to be no tangent space at the identity.
      This distinction is one of the many ways that distinguish
    the case $a> 1/2$ from the critical case $a = 1/2$.
     
     As to whether the solution $g(t)$ to \eref{I8} actually lies in $\G_{3/2}$ for each $t$
      depends on the nature of the coefficient $d^* C(t)$, which typically has
       a strong  singularity at $t = 0$, as already noted. 
     Most of this paper, accordingly,  
     is   devoted
     to proving that the  function $t\mapsto g(t)$ is actually a continuous function 
     into $\G_{3/2}$ (or into $\G_{1+a}$ if $A_0 \in H_a$).  
     The proof of this, in turn,    
     relies  
     on obtaining 
     detailed information about the
     singular initial behavior of the solution $C(\cdot)$ of the augmented equation \eref{I6}.

{\it Initial behavior of $C(\cdot)$.}    The nature  of the initial singularity of  
the solution $C(\cdot)$
which is needed to establish the required properties of the conversion 
functions $g(t)$ 
  will   be analyzed  in three steps.
 First, some knowledge of the singular behavior
of $C(t)$ as $t \downarrow 0$ comes immediately  from the fact that the solution
 lies in the path space that will be 
used for proving existence of a solution to the integral equation corresponding to \eref{I6}.
 Second, we will derive  energy estimates,
which use  the fact that the function $C(\cdot)$ not only lies in the path space but is also
a solution  of the augmented equation \eref{I6}. Generally this gives
 only $L^p$ information for $2\le p \le 6$ by
Gaffney-Friedrichs-Sobolev
inequalities. Third, we will derive information from a Neumann domination technique
 that builds on the previous energy estimates.
 This will give $L^p$ information for all $p\le \infty$.

{\it Finite action.} 
     The functional
      \beq
     \int_0^1 t^{-1/2} \|B(t)\|_{L^2(M)}^2 dt           \label{I9}
    \eeq
  is a gauge invariant functional of the initial data $A_0$, 
  wherein $M = \R^3$ or is a bounded open subset of $\R^3$ . 
  It captures    in a gauge invariant way  the 
  $H_{1/2}$ norm of $A_0$, which is itself not a gauge invariant norm. It controls
  many of the estimates needed in this $H_{1/2}$ theory. 
  It has thus important technical usefulness
  for us in this paper and important conceptual significance  for the applications.  We will
  say that a solution to the Yang-Mills heat equation has finite action if
  the functional \eref{I9} is finite. 
        (This terminology is motivated by the fact that, when \eref{I9} is finite, the initial data $A_0$
  has an extension to a time interval  in Minkowski  space which assigns finite value to the magnetic
  contribution to   the Lagrangian.)
  We will prove that if $\|A_0\|_{H_{1/2}}$ is sufficiently small,
  then the solution has finite action. If $\|A_0\|_{H_a} <\infty$ for some $a > 1/2$ then the
  corresponding ``a-action'' is always finite. This is yet another distinction between the
  critical case $a =1/2$ and the cases $a > 1/2$.

  The techniques in this paper rely heavily on the results in \cite{CG1} and \cite{CG2},
   which  deal with the Yang-Mills heat equation for initial data in $H_1$. 
   The Bianchi identity,  $d_A B =0$, encodes much of what is special about
    the form of the Yang-Mills heat equation.
 To take advantage of this it is essential to formulate identities and inequalities 
   in terms of the gauge covariant exterior derivative $d_A$ and its adjoint. All of the
   key inequalities we get with this use are gauge invariant.
  The gauge invariant Gaffney-Friedrichs inequality, 
 which gives information about the gradient of a form in terms of the
  exterior derivative and co-derivative of the form,  
   is   needed here, as in \cite{CG1}, for enabling 
   use of   Sobolev inequalities. It  will continue to be a major tool.

     Once one has propagated the initial data for a short time
     one can apply the results of  \cite {CG1} to establish
       long time existence  of the solution $A(\cdot)$. 
    Concerning uniqueness of solutions, a standard kind of 
    proof, based on a Gronwall type of
   argument, applies to our equation in case $a > 1/2$, but breaks down when $a = 1/2$.
   For $a = 1/2$ we will give a proof of uniqueness which is special to the structures
   at our disposal.
       
         The first proof of existence of solutions of the Yang-Mills heat equation over
 a three dimensional closed manifold  with $H_1$ initial data was given by J. R\aa de  in \cite{Ra}.
 In his proof,   R\aa de took the magnetic field as an independent unknown function, 
 in addition to the gauge potential $A$, and showed in the end that, for the joint solution 
 of the resulting parabolic system,
 the independent magnetic field is indeed the curvature of $A$.    This method goes back
 to Ginibre and Velo \cite{GV1, GV2} in the context of the hyperbolic Yang-Mills
  equations in $2+1$  space time dimensions and to DeTurck \cite{DeT} 
 in the  context of the  parabolic Ricci flow problem. 
 The ZDS procedure used in the present paper produces the gauge function $g_0$
 for which $A(t)^{g_0}$ is a strong solution.
  Any method alternative to the ZDS procedure  would  also be required 
  to produce such a
   gauge function 
   because of the conceptual role that $g_0$ plays when
    the initial data    $A_0$ is in $H_a$ with $a <1$. Remark \ref{remgq2} discusses this further.
    More history of the Yang-Mills heat flow is given in the introduction of \cite{CG1}.
 
       The ZDS procedure, which we used in \cite{CG1, CG2} and in the present paper,
has its origins in a suggestion by D. Zwanziger \cite{Z} in the context of stochastic quantization,
and in the work of Donaldson \cite{Do1} and Sadun \cite{Sa}. 
 Recently, Sung-Jin Oh \cite{Oh1, Oh2} 
has used the Yang-Mills heat equation
to  provide a very novel way to attack the Cauchy problem for the hyperbolic Yang-Mills
equation  in $3+1$ space time dimensions. He combines the hyperbolic and parabolic
equations into one system in order to force a  continuously changing gauge choice,
which is  favorable for the hyperbolic system.  The ZDS method underlies his analysis
of the parabolic  portion of the system. He is concerned with $H_1$ initial data
 for the heat equation since it matches with  the initial data of the hyperbolic equation.
 But the critical Sobolev index for the hyperbolic equation is also 1/2. 
 One could reasonably anticipate that Oh's method might be implementable for 
 $H_{1/2}$  initial data for the hyperbolic equation by combining it with the heat equation results
 which we derive here.

Complementary to the question of long time existence of the Yang-Mills heat flow 
is the question of short time blow up of solutions. 
  In four space dimensions the existence of long time solutions with $H_1$ initial
data has recently been proven  by Waldron in \cite{Wal}. Previously, long time solutions 
  had been proven in four space dimensions 
if the solution were
  restricted  by some strong symmetry conditions or some additional invariant of the equation.
 See e.g.  \cite{Stru2, Stru3, HT1} and \cite{Do1}.
But in \cite{Grot1},  J. Grotowski showed that over $\R^n$,  with $n\ge5$, 
 solutions  with smooth initial data can blow up in a finite time even if the initial
  data and solution are restricted by stringent symmetry conditions. 
  In  \cite{Nai} and  \cite{Wein}  the nature of the singularity formation has been investigated 
  and   \cite{Grot2} makes a comparison of blow up phenomena in Yang-Mills evolution and harmonic map evolution.
 In \cite{Gast1, Gast2} it was  shown how singularity formation is associated
  with non-uniqueness of the flow. Semi-probabilistic methods for determining
   blow up and no blow up are described in   \cite{ABT}  
  and \cite{Pu}. 
 
It is a pleasure to acknowledge useful comments from Nelia Charalambous and Artem Pulemotov.

\section{Statement of results} \label{secstate}

 $M$ will 
 denote  either $\R^3$ or the closure of a bounded, 
 open 
 set in $\R^3$
 with smooth boundary.
      $K$ will denote a compact Lie group contained in the unitary (resp. orthogonal)
  group of some finite dimensional inner product space $\V$. $\<\cdot, \cdot\>$ will denote
  an Ad $K$ invariant inner product on its Lie algebra $\kf$.

  We continue to use the commutator-wedge product notation from \cite{CG1} given 
      by $[(\xi dx^{j_1}\wedge \cdots \wedge dx^{j_p}) \wedge ( \eta  dx^{k_1}\wedge\cdots \wedge dx^{k_q})] = [\xi, \eta] dx^{j_1}\wedge \cdots  dx^{k_q}$ for $\kf$ valued functions $\xi$ and $\eta$.
For a $\kf$ valued connection form $A$ on $M$ the exterior derivative $d$ and
 the gauge invariant exterior derivative $d_A$ are related by 
 $d_A \w = d\w + [A\wedge \w]$,  wherein $\w$ is a  $\kf$ valued p-form.  
 If $u$ is a $\kf$ valued r-form
  and $v$ is a $\kf$ valued p-form with $ r \le p$ then the interior product $[u\lrc v]$ is
   the (p-r)-form  defined  by $\<[u\lrc v], \w\>_{\L^{p-r}\otimes \kf} = \< v, [u\wedge \w]\>_{\L^p\otimes \kf}$ for all (p-r)-forms
   $\w$. The adjoint of $d_A$ in $L^2(M; \L^*\otimes \kf)$ is then given by 
   $d_A^* \w = d^*\w + [A\lrc \w]$  for any $\kf$ valued p-form $\w$. In the following,
   $W_1$ refers to the Sobolev space of order one without boundary conditions.
 
\subsection{Strong solutions} 
 
\begin{definition}\label{defstrsol} {\rm   Let $0 < T \le \infty$.
A {\it strong solution} to the Yang-Mills heat
equation over $[0,T)$  is a continuous function 
\beq
A(\cdot): [0,T) \rightarrow L^2(M;\L^1\otimes \kf) \subset \{\frak k\text{-valued 1-forms on}\ M \}       \label{str1}
\eeq
such that
\begin{align}
a)&\  A(t) \in W_1\ \text{for all}\ \ t\in(0,T)\ \text{and}\ A(\cdot):(0,T)\rightarrow W_1\
                                                                \text{is continuous},      \label{str2}\\
b)& \ B(t):= dA(t) +A(t)\wedge A(t)  \in W_1 \  \text{for each}\ \  t\in (0,T),\  
                                    \label{str3}\\
c)& \  \text{the strong $L^2(M)$ derivative $A'(t) \equiv dA(t)/dt $}\ 
\text{exists on}\ (0,T), \ \   \text{and}\ \  \notag\\ 
&\ \ \ \ \  A'(\cdot):(0, T) \rightarrow L^2(M) \ \ \text{is continuous},    \label{str4}    \\
d)& \  A'(t) = - d_{A(t)}^* B(t)\ \ \text{for each}\ t \in(0, T).     \label{str5}
\end{align}
An {\it almost strong solution} is a function $A(\cdot)$ satisfying all of the preceding conditions except a).
 In this case the spatial exterior derivative $dA(t)$, which enters
into the  definition of the curvature $B$, 
must be interpreted  as a weak derivative.
}
\end{definition}

\begin{remark}\label{weakcurv} {\rm (Weak curvature) If a) does not hold then the weak derivatives \linebreak
$\p A(t)/\p x_j$, 
$j = 1,2,3$ need not be  functions. Yet condition b) requires 
   that  the particular
combination of derivatives that enter into $dA(t)$ be a function.  
    This can happen easily, as one sees in the example  $A = d\lambda$, where $\lambda$  is an arbitrary 
    real valued distribution on $\R^3$. In this case    the distribution $dA$ is the function identically equal to zero.
    Many of the problems that we will need to deal with in this paper arise from the presence of pure gauges,
     which are the non-commutative analogs of this example.
}
\end{remark}

\begin{definition}\label{defbc}{\rm (Boundary conditions) Let
\beq
-\Delta = d^*d + dd^*\ \ \text{on $\kf$ valued 1-forms over $M$}
\eeq
In case $M\ne \R^3$ then, for a  $\kf$ valued 1-form $\w$ on $M$,
 the Neumann and Dirichlet  domains for $\Delta$ are given by
\begin{align}
(N)\ \w_{norm} &=0, \ \ \ \ \ \, \ d\w_{norm}=0   \ \ \ \ \ \ \text{Neumann domain},  \label{ST11N0}\\
(D)\ \ \  \w_{tan} &= 0, \ \ \ \ (d^*\w)|_{\p M} =0\ \ \ \ \ \   \text{Dirichlet domain},     \label{ST11D0}
\end{align}
where $\w_{norm}$ and $\w_{tan}$ are the normal and tangential components of $\w$ on $\p M$.
The boundary conditions (N) and (D) are respectively absolute and relative boundary conditions
 in the sense of  Ray and Singer, \cite{RaS}. See \cite[Remark 2.11]{CG1} for further discussion.
Both versions of $-\Delta$ are non-negative self-adjoint operators on the appropriate domains.
The corresponding Sobolev spaces are
\beq
H_a = \text{Domain of}\ (-\Delta)^{a/2}\ \text{in} \ L^2(M;\L^1\otimes \kf)   \label{ST18}
\eeq
with norm
\beq
     \|\w\|_{H_a} = \| (1- \Delta)^{a/2}\w \|_{L^2(M;\L^1\otimes \kf)}.         \label{ST19}
     \eeq
With this definition one has
     \begin{align}
     \| \w\|_{H_a} \le c_{a,b} \|\w\|_{H_b}\ \ \ \  \text{if} \ \ \ 0 \le a \le b <\infty,       \label{ST20}
     \end{align}
 where $c_{a,b}$ is a constant independent of $M$. In particular, $H_b \subset H_a$ 
 when $ a \le b$.
}
\end{definition}
A theorem stated without specifying Neumann or Dirichlet boundary conditions
applies to both cases when $M \ne \R^3$, as well as to the case $M =\ R^3$.

\begin{remark}\label{mabc}{\rm  (More about boundary conditions) 
Suppose that $M$ is the closure
of a bounded open set in $\R^3$ with smooth boundary. It is well understood
that the usual Neumann Laplacian for real valued functions $u$ on $M$ can be simply defined
as the unique self-adjoint operator $-\Delta_N$ whose quadratic form  is given by 
$Q_N(u) = \int_M  |\n u(x)|^2 dx$. $\D(Q_N)$ consists of all measurable
 functions $u$ for which $Q_N(u)$ is finite.   $\Delta_N$ is unique under the usual 
 assumption that $\D(\Delta_N) \subset \D(Q_N)$.
Similarly the Dirichlet Laplacian is the unique self-adjoint
operator $-\Delta_D$ whose quadratic form  $Q_D$ is again given by this integral but
with domain consisting of those functions $u$ in the domain of $Q_N$ which are zero on 
$\p M$. (More precisely $Q_D$ is the closure of the form $Q_N|C_c^\infty(M^{int})$).
In this paper we are concerned with Laplacians  on 1-forms over $M$. 
 There are two natural senses for a 1-form $\w$ to be zero on the boundary.
Define
\begin{align}
Q_{norm}(\w) &=\int_M \Big(|d\w(x)|_{\L^2}^2 + |\delta \w(x)|_{\L^0}^2 \Big) dx,
          \ \ \w_{norm} = 0\ \text{on}\ \p M \\
Q_{tan}(\w) &=\int_M \Big(|d\w(x)|_{\L^2}^2 + |\delta \w(x)|_{\L^0}^2 \Big) dx, 
        \ \ \ \ \w_{tan} = 0 \ \text{on}\ \p M,
\end{align}
where $d$ denotes the exterior derivative on 1-forms in $C^\infty(M)$ and $\delta$ denotes the coderivative  on 1-forms in $C^\infty(M)$. The domains of both of these quadratic forms
is specified  by imposing a Dirichlet type condition on $\w$. Yet the Laplacian naturally
 associated to each one forces a form $\w$  in its domain to satisfy 
 not only the Dirichlet type condition $\w_{norm} =0$, resp. $\w_{tan}=0$, but also
 a derivative type condition $(d\w)_{norm} =0 $, resp. $(\delta \w)_{tan} =0$. The latter are
 Neumann type conditions. Thus the two Laplacians associated to the two quadratic forms
 are neither Dirichlet Laplacians nor Neumann Laplacians. As noted above, 
 Ray and Singer \cite{RaS} have called
 the associated boundary conditions absolute and relative boundary
  conditions, respectively, because of their
 role in absolute and relative cohomology.
 We are going to continue to follow Conner \cite{Co},
 who refers to them as the Neumann Laplacian and Dirichlet Laplacian, respectively,
 because they reduce to these on zero forms.  We will often deal with
 small fractional  powers of the Laplacian.  
 When $1/2 < a < 3/2$ the domain of   
 $(-\Delta_N)^{a/2}$ is restricted only by the boundary condition $\w_{norm} =0$ and 
  not by a condition on the derivative of $\w$. A similar comment applies to $(-\Delta_D)^{a/2}$
  and $\w_{tan} =0$.
  }
\end{remark}

\subsection{Gauge groups}
\begin{notation}\label{notgg} {\rm (Gauge groups) 
 There are several gauge groups over $M$ 
that  mediate the formulation of the existence theorems. 
A measurable function $g: M\rightarrow K \subset End\ \V$ is a bounded 
function into a linear space and therefore
its weak derivatives over the interior of $M$ are well defined. 
We will  be interested  in such functions $g$
for which  the weak derivatives $ \p_j g$ are in $L^2(M; End\, \V)$, $j = 1,2,3$.  
We will say that $g \in W_1$
in this case.  For a function $g \in W_1$ the functions $x \mapsto g(x)^{-1} \p_j g(x)$ take their
values a.e. in the Lie algebra $\kf \subset End\ \V$.  
Thus the $\kf$ valued 1-form $g^{-1}dg$ lies in $L^2(M; \L^1\otimes \kf)$. 
Postponing  till Section \ref{secgg} a precise discussion of the boundary conditions on $g$ itself,
let us define $\G_{1+a}$ to  consist of those functions $g \in W_1$ such
 that $\|g^{-1}dg\|_{H_a} <\infty$.
 }
 \end{notation}

\begin{theorem} \label{thmgg1} The set $\G_{1+a}$ is a complete topological group
under the pointwise product if $1/2 \le a \le 1$. $($See Section \ref{secgg} for the topology.$)$
\end{theorem}

\begin{remark}\label{remgg} {\rm (Failure of $\exp (H_{3/2})$)
As noted in the introduction, there is a 
 technical disadvantage in attempting to 
use the representation $g(x) = \exp((\alpha(x))$ for elements of the critical gauge group
$\G_{3/2}$.  The inverse of the
exponential map can 
  be  poorly behaved at some points $\exp(\alpha(x))$ because $\exp( \alpha(x))$ can
   wrap around the whole group $K$ even when $\|\alpha\|_{H_{3/2}}$ is small, leading to failure
   of inversion and multiplication to be continuous. 
            In four dimensions the analogous borderline gauge group is $\G_2$, since  the
Sobolev $W_2$ norm also just fails to control the supremum norm.  K. Uhlenbeck has
already pointed out \cite[page 33] {Uh2} that in this four dimensional case 
multiplication and inversion fail to be continuous in $\G_2(\text{Ball  in}\ \R^4)$ 
if one defines the Sobolev behavior of an element of $\G_2$ by means of the exponential 
representation.   In our definition of the topological group $\G_{3/2}$, 
 the set $\{\exp((\alpha(x)): \alpha \in H_{3/2}(M)\}$
  is not likely  even to cover any neighborhood of the identity.
Further discussion of this can be found in Remark \ref{remdiffstr}.
}
\end{remark}

\subsection{Main theorem} \label{secMT}

\begin{definition}\label{finaact}{\rm(Finite a-action)
An almost strong solution $A(\cdot)$ to the Yang-Mills heat equation
has {\it finite a-action} if
\beq
\int_0^\epsilon s^{-a} \|B(s)\|_{L^2(M)}^2 ds  < \infty \ \ \text{for some}\ \ \epsilon > 0 ,  \label{fa1a}
\eeq
where $B(s)$ is the curvature of $A(s)$.
This definition is of interest for $ 1/2 \le a <1$. For $a = 1/2$ it reduces to the definition \eref{I9} 
of finite action given in the introduction.
}
\end{definition}

\begin{definition}\label{convex} 
      {\rm (Convexity of $M$) For some of the results in this paper we will need to 
assume that $M$, if it isn't all of $\R^3$,  is the closure of a bounded, open convex
 subset of $\R^3$ with smooth boundary. This ensures that the second fundamental
  form of $\p M$ is  everywhere non-negative.
 When needed, we will simply refer to such a set as convex.
Discussion of when convexity of $M$ is not needed may be found in Remark \ref{remconvex}.
}
\end{definition}

\begin{theorem} \label{thmeu>} $(a >1/2)$ 
Let $1/2 < a < 1$. Assume that $M$ is all of $\R^3$ or is convex in the sense
 of Definition \ref{convex}. Suppose that  $A_0 \in H_a(M)$.
Then

1$)$  there exists an almost strong  solution $A(t)$ to \eref{str5}
         over $[0,\infty)$ with initial value $A_0$ and with the following properties. 
         
2$)$ There exists  a gauge  function $g_0 \in \G_{1+a}$
such that $A(t)^{g_0}$ is a strong solution.         

3$)$ $A(\cdot)$ and $A(\cdot)^{g_0}$ are continuous functions on $[0, \infty)$ into $H_a$.
 
4$)$   $A(\cdot)$ and $A(\cdot)^{g_0}$  both have finite $a$-action.         
         
5$)$ If $M \ne \R^3$ then the following boundary condition is 
            satisfied by both $A(t)$ and $A(t)^{g_0}$.                  
\begin{align}
   (curvature|_{t})_{norm} &=0 \ \ \text{in case}\ \ (N)\ \ \    \forall t >0     \label{s50} \\
   (curvature|_{t})_{tan}  &=0  \ \ \text{in case}\ \ (D)\ \ \     \forall t >0 .     \label{s51}
   \end{align}
Moreover 
   \begin{align}
 (A(t)^{g_0})_{norm} &=0         \ \ \text{in case}\ \ (N)\ \ \        \forall t >0  \label{s52}\\
 (A(t)^{g_0})_{tan} &= 0             \ \ \text{in case}\ \ (D)\ \ \     \forall t >0 .      \label{s53}
\end{align}

 6$)$ Strong solutions are unique among solutions with finite $a$-action
 under  the boundary condition $($in case $M \ne \R^3$$)$
  \begin{align}
  B(t)_{norm}& =0 \ \ \text{for all}\ \  t>0\ \ \ \text{in case}\ \ (N)       \label{uN} \\
  A(t)_{tan} &=0 \ \ \text{for all}\ \  t>0\ \ \ \text{in case}\ \ (D).           \label{uD}
  \end{align} 
  \end{theorem}

\begin{theorem} \label{thmeu=} $(a =1/2)$. Assume that $M$ is all of $\R^3$ or is 
   convex in the sense of Definition \ref{convex}.
 Suppose that  $A_0 \in H_{1/2}(M)$.
Then

1$)$  there exists an almost strong  solution $A(t)$ to \eref{str5}
         over $[0,\infty)$ with initial value $A_0$. Its curvature  satisfies the boundary conditions
          \eref{s50} resp. \eref{s51} if $M \ne \R^3$. 
         
2$)$ There exists  a gauge  function $g_0$ 
such that $A(t)^{g_0}$ is a strong solution.  $A(t)^{g_0}$  satisfies the boundary conditions
 \eref{s50} resp. \eref{s51} as well as \eref{s52} resp. \eref{s53} when $M \ne \R^3$.
 
 3$)$    If $\|A_0\|_{H_{1/2}}$ is sufficiently small then $A(\cdot)$ and $A(\cdot)^{g_0}$ have finite 
 $(1/2)$ action.  In this case one may choose $g_0$ to lie in $\G_{3/2}$
 
 4$)$  If $\|A_0\|_{H_{1/2}}$ is sufficiently small then $A(\cdot)$ is a continuous function from $[0, \infty)$
 into $H_{1/2}$. If, in addition, $g_0$ is chosen to lie in $\G_{3/2}$ then $A^{g_0}(\cdot):[0,\infty) \rightarrow H_{1/2}$
 is also continuous.

 5$)$ Strong solutions are unique among solutions with finite $(1/2)$-action under
   the boundary condition $($in case $M \ne \R^3$$)$ \eref{uN} resp. \eref{uD}.
 \end{theorem}

 \begin{remark}{\rm (Meaning of boundary conditions) Since the curvature $B(t)$ is
  in $W_1$ for $t >0$,  by the definition
 of a strong or almost strong solution, the restriction $B(t)|\p M$ is well 
 defined a.e. on $\p M$, and consequently
 the boundary conditions \eref{s50} and \eref{s51} are meaningful. However
 $A(t)$ need not be in $W_1$ for an almost strong solution. The boundary conditions 
 \eref{s52} and \eref{s53}  may therefore not be meaningful for $A(t)$ itself but only for the
 gauge transformed solution $A(t)^{g_0}$.
 This relates to Remarks \ref{weakcurv} and \ref{remau}.
 }
 \end{remark}

\begin{remark}\label{remau} {\rm (Uniqueness for almost strong solutions)
  Our proof of uniqueness depends 
on initial behavior bounds for  $\|B(t)\|_\infty$, which we will derive via Neumann domination
techniques in Section \ref{secibN}. These bounds depend in turn on energy bounds
 for initial behavior, which in turn depend on finite a-action. 
For this reason our formulation of uniqueness specifies
uniqueness only among solutions of finite action. It is reasonable to ask whether uniqueness
holds among almost strong solutions of finite action. In the case that $M \ne \R^3$ one
needs of course to impose a boundary condition such as \eref{uN} or \eref{uD}. 
Since $B(t) \in W_1$ for  almost strong solutions,
  the requirement \eref{uN} is meaningful for an almost strong solution
   for Neumann boundary conditions.
             But, since $A(t)$ need not be in $W_1(M)$ for an almost strong solution, 
 one would need to interpret  the boundary condition \eref{uD}  properly to
  address uniqueness for almost strong solutions
 in the case of Dirichlet boundary conditions.   In this case the assertion that $A(t) \in H_a$
 already reflects a boundary condition on $A(t)$. For example, it can be expected, on the basis of Fujiwara's theorem   \cite{Fuj},
 that $A(t)_{tan} =0$ when $a >1/2$ and that 
 this holds in a mean sense when $a =1/2$.

 Aside from the problem  of formulation of uniqueness for almost strong solutions (in Dirichlet case)
 there are some (seemingly) technical issues in justifying the computations that lead to the key
 inequality \eref{u11} needed for the proof of uniqueness. 
We do not have available for almost strong solutions such a good approximation
 mechanism as can be found  in \cite[Lemma 9.1]{CG1}.
      The issues raised by the question of uniqueness
 for almost strong solutions of finite action  relate to other problems, 
 which will be addressed elsewhere.
 We will not consider uniqueness of almost strong solutions in this paper.
  }
 \end{remark}

         \begin{remark}{\rm (Smoothness) The solution $A^{g_0}$ produced in Theorems 
\ref{thmeu>} and \ref{thmeu=} is actually in $C^\infty((0, T]\times M; \L^1\otimes \kf)$
for some time $T <\infty$. Very likely it is in  $C^\infty((0, \infty) \times M; \L^1\otimes \kf)$.
But our proof does not rule out the possibility  that it loses smoothness if one doesn't make occasional gauge transforms, 
 even though it remains in $H_1(M)$ for all $t > 0$. See Theorem \ref{thmreca}. 
 However  it will be shown  in \cite{CG3} that gauge covariant derivatives of all orders exist.
}
\end{remark}

\begin{remark}\label{remgq}{\rm The gauge transformation $g_0$  that converts an almost strong
 solution, $A(\cdot)$, 
to a strong solution  $A(\cdot)^{g_0}$ is not unique: If $g_1$ lies in the gauge group $\G_{2}$
 then it gauge transforms a strong solution to another strong solution and
  consequently $A(\cdot)^{g_0g_1}$
is also a strong solution. It will be shown in \cite{G72} that this is the extent of the non-uniqueness.
Denote by  $\Y_a(M)$ the set of strong solutions over $M$ with initial  data in $H_a(M)$. (Choose either Dirichlet or Neumann boundary conditions when $M \ne \R^3$.)
$\G_{1+a}(M)$ acts continuously on $H_a(M)$ in its natural action 
$A\mapsto g^{-1} A g + g^{-1}dg$. 
Theorem \ref{thmeu>} asserts that each fiber in  the bundle $H_{a}(M) \mapsto  H_a(M)/\G_{1+a}(M)$
contains at least one element of $\Y_a(M)$. Thus we have
\begin{align}
H_a/\G_{1+a} = \Y_a/ \G_2,                       \label{St40}
\end{align}
given the assertion above concerning the extent of the non-uniqueness of $g_0$.
We will see in \cite{G72} that
$\Y_a$ is a complete Riemannian manifold with respect to a $\G_2$ invariant Riemannian metric
associated to the action \eref{fa1a}. An identity similar to \eref{St40} holds also for $a = 1/2$ if one restricts
to small $\|A_0\|_{H_{1/2}}$ in accordance with Theorem \ref{thmeu=}. 
In the sense of \eref{St40}, the non-gauge invariant norm on the linear space 
$H_a(M)$ is captured,  up to gauge transformations, by a gauge invariant
Riemannian metric on the manifold $\Y_a$. 
}
\end{remark}

\begin{remark}\label{remgq2} {\rm 
 If the initial data $A_0$ is in $H_1$ then the role
 of the gauge functions $g(t)$
produced by the ZDS procedure, discussed in the introduction and in the next subsection,
 is an auxiliary one in the sense that it 
 is needed only to produce the solution
$A$ from $C$. But if $A_0$ is only in $H_a$ for some $a <1$ then
 these gauge functions play a more
fundamental role. A solution with initial value $A_0\in H_a$ need not be a strong solution.
But the ZDS procedure produces a gauge function $g_0$ which transforms $A_0$ into another
element of $H_a$, which is the initial value of  a strong solution.
This is 
 a reflection of the identity \eref{St40}. 
 The gauge function $g_0$ therefore plays an indispensable role in the formulation
  of the Cauchy problem.
 There does not appear to be a way to decompose the initial data space
$H_a$ into ``longitudinal plus transverse'' subsets that propagate as
 ``constant solutions", respectively strong solutions, 
 so that the transverse subset might serve as a section 
  over the quotient space  $H_a/\G_{1+a} $.
}
\end{remark}

\begin{remark} \label{remconvex} {\rm (Convexity of $M$)  
  The convexity of $M$ enters only in the proof
 of Neumann domination bounds.
  Convexity of $M$ is therefore not needed in our discussion of gauge groups (Section \ref{secgg})
 or in our proof of existence and uniqueness of solutions to the augmented Yang-Mills heat 
 equation (Section \ref{secmild}). 
 It is reasonable to anticipate that  convexity
   could be replaced by a negative lower bound on the second fundamental form of $\p M$
   in some version of the Neumann domination bounds.
}
\end{remark}

\begin{remark}\label{rembdd} {\rm (Boundedness of $M$) In case $M\ne \R^3$ we have
 assumed that $M$ is bounded,
in addition to being  the closure of an open set with smooth boundary. Together, these assumptions
ensure that standard Sobolev inequalities hold over $M$  as well as  the needed 
operator bounds for the Neumann Laplacian heat semigroup.
  But boundedness of $M$ is not essential for these to hold. 
  Classes of unbounded domains for which these hold have been extensively investigated. 
 For the anticipated applications of this paper,  it suffices to note  that  these unbounded domains
  include   half-spaces and infinite slabs.
If $M$ is unbounded and the standard Sobolev inequalities and Neumann heat operator bounds hold,
then those  of our results which are dependent on the  
Gaffney-Friedrichs inequality  will  also hold when, in addition, 
 the second fundamental form of $\p M$ is bounded below,
and those of our results dependent on Neumann domination will   also hold when, in addition, 
 the second fundamental form of $\p M$
is non-negative.  Of course, if $M = \R^3$ then the Neumann Laplacian is to be replaced
  by the self-adjoint version of the Laplacian over $\R^3$. All of our results hold in this case.
}
\end{remark}

\subsection{The ZDS procedure and the augmented equation}
 The proof of Theorems \ref{thmeu>} and \ref{thmeu=} 
 will be based on the ZDS procedure, already discussed in the introduction,
  and in particular on use of 
 the following modified Yang-Mills heat equation. 
 See \cite{CG1} for more discussion of the ZDS procedure.

\begin{definition}{\rm (Augmented equation.) The {\it augmented Yang-Mills heat equation}
 is
\beq
-\frac{\p}{\p t}C(t) =  d_{C(t)}^* B_C(t) + d_{C(t)} d^* C(t),\ \ C(0) = C_0.               \label{aymh}
\eeq
Here $C(t)$ is a $\kf$ valued 1-form on $M$ for each $t \ge 0$ and $B_C(t)$ is its curvature. 
 Equation \eref{aymh}
 differs from the Yang-Mills heat equation \eref{str5} by the addition of the second term on the right.
The added term makes the equation strictly parabolic. 
If $M \ne \R^3$ the equation goes along with one of the following two kinds of boundary conditions, 
(N) (for Neumann) or (D) (for Dirichlet).
\begin{align}
(N)\ \  C(t)_{norm}=0\ 
    &\text{for}\ t \ge 0, \ \ (B_C(t))_{norm}\  =0 \ \ \text{for}\ t >0  \label{ST11N} \\
(D)\ \    C(t)_{tan} =0\ \ \ 
    &\text{for}\ t \ge 0,\ \ (d^*C(t))|_{\p M} = 0\ \ \text{for}\ t >0. \label{ST11D} 
\end{align}
By a {\it strong solution} to the augmented Yang-Mills heat equation \eref{aymh} 
 over an interval $[0, T]$ we mean   a continuous function 
 $C(\cdot):[0,T]\rightarrow L^2(M; \L^1\otimes \kf)$
 satisfying the four conditions a) - d) of Definition \ref{defstrsol}, with $A$ replaced by $C$,
 $B$ replaced by $B_C$  and \eref{str5} replaced by  \eref{aymh}. We will be concerned only 
 with $T < \infty$ for the augmented equation.
}
\end{definition}

  \begin{theorem}\label{thmaug} $($Solutions to the augmented equation.$)$ Let $1/2 \le a <1$. 
          Suppose that    $C_0 \in H_a(M)$.  Then there exists a real number $T >0$ and a 
          continuous function
          $C: [0,T]\rightarrow H_a(M)$ such that $C(0) = C_0$ and
          
 a$)$ $C(\cdot)$ is a strong solution to the augmented equation  \eref{aymh} over $(0,T]$ 
     satisfying  the respective boundary conditions \eref{ST11N} or \eref{ST11D}, 
     when $M \ne \R^3$.
 
 b$)$ $t^{(1-a)} \|C(t)\|_{H_1}^2 \rightarrow 0$ as $t\downarrow 0$.

\noindent          
The solution is unique under the preceding conditions. 
Moreover $C(\cdot)$ lies in $C^\infty((0, T)\times M; \L^1\otimes \kf)$

          If $ a >1/2$ then the solution has {\bf finite strong a-action} in the  sense that
 \beq
\int_0^T s^{-a} \| C(s)\|_{H_1}^2 ds <  \infty .            
                                    \label{fa1C}
\eeq          
          
          If  $a =1/2$ and $\|C_0\|_{H_{1/2}}$ is sufficiently small then 
          \eref{fa1C} holds with $a =1/2$.   
\end{theorem}

The proof of this theorem will be given in Section \ref{secmild}.

\begin{remark}\label{remlink}{\rm  The link between the augmented ymh equation \eref{aymh}
and the ymh equation \eref{str5} is provided by the ZDS procedure outlined in the introduction. 
As already noted there,  much  of this paper is   concerned with determining 
  the behavior of of the initial singularity of $C(\cdot)$ and its derivatives in order 
   to establish the required differentiability properties of the solution $g(\cdot)$ to
    the Equation \eref{I8}. These, in turn, 
  will give the differentiability  properties of $A(\cdot)$ asserted in Theorems 
  \ref{thmeu>} and \ref{thmeu=}.  
   Thus, most of   
   this paper is devoted to proving the following theorem, which is 
   stated here just for $a = 1/2$ for simplicity.
}
\end{remark}

\begin{theorem}\label{thmrec}  Assume that $M = \R^3$ or is convex in the sense
 of Definition \ref{convex}. Suppose that $A_0 \in H_{1/2}$ and that $C(\cdot)$
is a strong solution to the augmented  equation \eref{aymh} with finite strong
 action over $[0,T]$ and with $C_0 = A_0$.
 Then there exists a continuous function
 \beq
 g:[0, T] \rightarrow \G_{3/2}
 \eeq
such that the gauge transform $A(\cdot)$ defined by 
\beq
A(t) = C(t)^{g(t)} \ \ 0 \le t \le T                               \label{rec215}
\eeq
is an almost strong  solution to the  Yang-Mills heat equation over $(0, T)$,  
whose  curvature satisfies the boundary condition \eref{s50} resp. \eref{s51}. The function
\begin{align}
A(\cdot):[0, T] \rightarrow H_{1/2}
\end{align}
is continuous. In particular,  $A(t)$ converges in $H_{1/2}$ norm to $A_0$ as $t\downarrow 0$. 

 If $0 < \tau < T$ and
$g_0 \equiv g(\tau)^{-1}$ then the function $t\mapsto A(t)^{g_0}$ is a strong solution
 to the  Yang-Mills heat equation  satisfying, if $M \ne \R^3$,  the boundary condition
  \eref{s50} resp \eref{s51} as well as  the boundary condition 
\eref{s52} resp. \eref{s53}.  
$A^{g_0}(\cdot)$ is a continuous function on $[0,T]$ into $H_{1/2}$, and 
in particular $A(t)^{g_0}$ converges in $H_{1/2}$ norm to $A_0^{g_0}$ as $t\downarrow 0$.

Furthermore $A(\cdot)$ and $A^{g_0}(\cdot)$ have finite action:
\beq
\int_0^Ts^{-1/2} \|B(s)\|_2^2 ds < \infty.            \label{rec216}
\eeq
\end{theorem}

This theorem, along with its analog for $a > 1/2$, will be proved in Section \ref{secrec}, 
(Theorem \ref{thmreca}).
In case $A_0 \in H_{1/2}$ but $C(\cdot)$ does not have finite strong action one
needs a weaker version of Theorem \ref{thmrec}, based on infinite (1/2)-action, 
in order to prove items 1) and 2) in Theorem \ref{thmeu=}.
The infinite action version of Theorem \ref{thmrec}  will be stated and proved
 in Section \ref{secrec} also (Theorem \ref{thmrec0}).

\begin{remark}\label{remcontrast} {\rm (Contrast with $H_1$ initial data) 
In case the initial data $A_0$ is in $H_1(M)$ there is no need to invoke the use of gauge transformations
in the the formulation of the existence theorem because  the solution that the ZDS procedure produces
is already a strong solution. Here is a version of the main result from \cite{CG1}, formulated in $\R^3$
rather than over a compact Riemannian manifold.
}
\end{remark}

\begin{theorem}\label{thmH1} \cite{CG1} Let $M$ be either $\R^3$ or 
the closure of a bounded, convex, open subset of $\R^3$ with smooth boundary. 
Suppose that $A_0 \in H_1(M)$. Then there exists a strong solution $A(\cdot)$ to the Yang-Mills
heat equation \eref{str5} over $[0, \infty)$ with initial value $A_0$.
 Moreover $A:[0,\infty) \rightarrow H_1$ is continuous.
 \end{theorem}

 \begin{remark} {\rm 
 1. The conclusion of Theorem \ref{thmH1} should be contrasted with the 
          conclusions of Theorems  \ref{thmeu>} and \ref{thmeu=}. One does not need to gauge transform
          the solution  to obtain a strong solution when $A_0 \in H_1$.
          
 2. When $M\ne \R^3$, boundary conditions  on $A_0$   are 
       implied    in Theorem \ref{thmH1} by  the assumption that $A_0 \in H_1$. These are $(A_0)_{norm} =0$
       in case $(N)$ or $(A_0)_{tan} =0$ in case $(D)$. Compare Remark \ref{mabc}. Moreover the solution
       satisfies, for all $t >0$, the boundary conditions \eref{s50} resp. \eref{s51} on its curvature,
       as well as $A(t)_{norm} =0$ resp. $A(t)_{tan} =0$. No gauge transform need intervene as  in 
       \eref{s52} and \eref{s53}. Uniqueness holds just under the condition \eref{uN} resp. \eref{uD}.   
       
 3. The case $M =\R^3$ is not stated in \cite{CG1}. However if $M =\ R^3$, then all the steps in the
  proof in \cite{CG1} go through without essential change. In fact 
  the proof in this case is considerably simpler 
   because all of the desiderata concerning boundary conditions can be ignored.    Finite volume of $M$
   is never used. 
}
\end{remark}

\section{Solutions for the augmented Yang-Mills heat equation} \label{secmild}

\subsection{The integral equation and path space } \label{secinteq}

Throughout Section \ref{secmild} $M$ will be assumed to be either all of $\R^3$ or
 else the closure of a bounded open set in $\R^3$ with smooth boundary. $M$ need not be convex.
      
      We will convert the augmented Yang-Mills heat equation \eref{aymh} to an
       integral equation 
       and  then show that the integral  equation
        has a unique solution  for a short time.  In Section \ref{secstr} it will be shown
     that the solution 
     is actually a strong solution      to \eref{aymh}.
    
      To carry out the conversion to an integral
     equation one   needs to separate the linear terms from the non-linear terms in \eref{aymh}.     
     Throughout this section we will write $d$ for the exterior derivative
     with the understanding that this represents the maximal or minimal
     version, in agreement with the boundary conditions when $M \ne \R^3$. 
     See \cite{CG1} for a discussion of these domains.       
               It will be convenient to use the exterior and interior commutator products
   of $\kf$ valued forms defined at the beginning of Section \ref{secstate}.

Writing $ B \equiv B_C = dC +(1/2) [C\wedge C]$, one  can compute that
\beq
d_C^* B + d_Cd^*C = (d^*d + d d^*)C - X(C),        \label{ST13}
\eeq
where $X$ is the first order nonlinear differential operator
 on $\frak k$ valued
1-forms  $C$ defined by
\beq
-X(C) = - [C\lrc B] +(1/2) d^*[C\wedge C] + [ C, d^*C], \ \ \ \
       C:M \rightarrow \Lambda^1\otimes \frak k    .                   \label{ST14}
\eeq

The terms in $X(C)$ which are cubic in $C$ involve no derivatives  of $C$
while the terms which are quadratic all involve a factor of one spatial 
derivative of $C$. 
As in \cite{CG1} we will write this symbolically as
\beq
X(C) = C^3 + C \cdot \p C.                                     \label{ST14.1}
\eeq
$X(C)$ contains all the non-linear terms in Eq.\ \eref{aymh}, which can now be rewritten as
\beq
C'(t) = \Delta C(t) + X(C(t)), \ \ C(0) = C_0,     \label{ST15} 
\eeq
wherein $\Delta$ is the self-adjoint Laplacian on $\kf$ valued 1-forms over $\R^3$, or the
Neumann, resp. Dirichlet Laplacian defined in Definition  \ref{defbc}, in case $M \ne \R^3$.

Informally, the equation \eref{ST15} is equivalent to the integral equation
\beq
C(t) = e^{t\Delta} C_0
+ \int_0^t e^{(t-\sigma)\Delta} X(C(\sigma)) d\sigma.   \label{ST25} 
\eeq

A solution to the integral equation \eref{ST25} is sometimes referred to as a mild solution to the 
differential equation \eref{ST15}  \cite[Definition 11.15]{RR2}.
We will show in Section \ref{secstr} that such a mild solution is actually a strong solution.
The existing general theorems showing that a mild solution is  a strong solution
seem inapplicable to our case.

\begin{remark}{\rm Any choice of path space within which one wishes to
 seek a solution to \eref{ST25} with initial data $C_0 \in H_{a}$ should be
  contained in $C([0,T]; H_{a}(M))$
  and should include paths having 
  arbitrary initial  value in $H_{a}(M)$.
 But it must also
 have a strong enough metric to allow control of the non-linear function $X(C)$.
 The following path space
 seems well adapted to this purpose for  our particular non-linearities and initial conditions.
 }
 \end{remark}

 \begin{notation} \label{notpatha}{\rm (Path space.) 
 Suppose that  $0 < a < 1$ and  $ 0 < T < \infty$. 
 Let $C_0 \in H_a\equiv H_a(M; \L^1\otimes \kf)$.  Define 
\begin{align}
\P_{T}^a = \Big\{ C(\cdot)  \in C\Big([0,T]; H_a\Big) &\cap C\Big((0,T]; H_1\Big) : \notag\\
\qquad \qquad \ \ \  & i.\ \ \  C(0) = C_0                                                             \label{ST411a}\\
\qquad \qquad \ \ \ & ii.\ \ t^{1-a} \|C(t)\|_{H_1}^2 \rightarrow 0\ \ \text{as}\ 
                                                                \ t\downarrow 0 \Big\}. \label{ST413a}
\end{align}
Define also
  \beq
  |C|_t = \sup_{0 < s \le t} s^{(1-a)/2} \|C(s)\|_{H_1},    \  0<t \le T . \label{ST415a}
  \eeq
  Then, for $C\in \P_T^a$, we have
  \begin{align}
  \|C(s)\|_{H_1} &\le s^{(a-1)/2} |C|_t \ \ \ \text{for} \ \ 0 < s \le t \le T\ \   \text{and}   \label{ST416a}\\
   |C|_t &\le |C|_T\ \qquad \ \ \ \ \,  \text{for}\ \ 0 <t \le T.           \label{ST418a}
  \end{align}
  Condition {\it ii.}   ensures that
  \beq
  |C|_t \rightarrow 0,\ \text{as}\ \ t\downarrow 0.      \label{ST417a}
  \eeq  
 $\P_{T}^a$ is a complete metric space in the metric
\beq
dist(C_1, C_2) = \sup_{0\le t \le T}\|C_1(t)- C_2(t)\|_{H_{a}}  + |C_1 - C_2|_T .  \label{ST414a}
\eeq
The inequality \eref{ST416a} ensures that, for some Sobolev constant $\kappa_6$, one has
\beq
\| C(s)\|_6 \le s^{(a-1)/2} |C|_t \kappa_6,\ \ \text{for}\ \ \ 0 < s \le t.       \label{ST420a}
\eeq
}
\end{notation}
These spaces will  be useful only for $1/2  \le a <1$.

The next theorem is the mild version of Theorem \ref{thmaug}. 
 It will be proven in the following four sections.

\begin{theorem}\label{thmmild} Let $1/2 \le a <1$ and let $C_0 \in H_{a}(M;\L^1\otimes\kf)$. 

i.$)$  There exists $T>0$
 depending on $C_0$  $($and not just on $\|C_0\|_{H_{a}}$. See Remark \ref{phil5}.$)$  such
  that the integral equation \eref{ST25} has a
  unique solution in $\P_T^a$. 
  
  ii.$)$ If $ 1/2 < a < 1$ then the solution has finite strong a-action in the sense of \eref{fa1C}.
  
  iii.$)$ If $a = 1/2$ and $\|C_0\|_{H_{1/2}}$ is sufficiently small then the solution
   has finite strong action in the sense of \eref{fa1C} with $a = 1/2$.
\end{theorem}

The proof of this theorem requires establishing properties  
of each of the terms on the right side of \eref{ST25}.
Section \ref{secfreeprop} will show that the first term lies in $\P_T^a$. Section \ref{secce}
will establish the needed  contraction estimates for the second term. These will be put 
together in Section \ref{seceum} to prove item {\it i.}), the existence and uniqueness
 portion of the theorem.
Items {\it ii.}) and {\it iii.}), finite action,  will be proven in Section \ref{secfinact}. 
Section \eref{secstr} will show
 that solutions to the integral equation \eref{ST25}  are actually strong solutions.

\subsection{Free propagation lies in the path space $\P_T^a$} \label{secfreeprop}

We will show in this subsection that the first term on the right in \eref{ST25}
lies in $\P_T^a$ and has finite strong a-action. 
All estimates in this subsection will be made
 for initial data $C_0 \in H_a$ with $0 \le a < 1$ since there is no simplification  for $a\ge1/2$
 and  the greater generality will be needed later.

               \begin{lemma}\label{freeesta1} Let  $ 0 \le a <1$  
 and suppose that  $C_0 \in H_a$.  Then, for some real 
 constants $c_a$ and $\gamma_a$
 there holds
\begin{align}
e^{2t} c_a \|C_0\|_{H_a}^2 &\ge t^{1-a} \|e^{t\Delta} C_0\|_{H_1}^2 \rightarrow 0\
                         \text{as}\ \ t\downarrow 0\ \ \ \text{and}            \label{ST449a} \\
\int_0^T t^{-a} \|e^{t\Delta} C_0\|_{H_1}^2 dt 
         &\le e^{2T} \gamma_a^2 \|C_0\|_{H_a}^2.                       \label{ST450a} 
\end{align}
\end{lemma}
          \begin{proof}  
Denote by $E(d\lambda)$ the spectral resolution for the operator $1-\Delta$ and let
$\mu(d\lambda) = (E(d\lambda) D^{a} C_0, D^{a}C_0)$, where $D\: = \sqrt{ 1 - \Delta}$.
 In view of the definition
 \eref{ST19} of the   $H_a$ norm  we may write
\begin{align}
e^{-2t}\|e^{t\Delta} C_0\|_{H_1}^2 & = \|D e^{t(\Delta -1)} C_0\|_2^2\notag\\
&= \|D^{1-a}e^{-tD^2} D^{a}C_0\|_2^2       \notag\\
&= (D^{2(1-a)} e^{-2tD^2} D^{a}C_0,D^{a}C_0)   \notag\\
&=  \int_1^\infty \lambda^{(1-a)} e^{-2t\lambda} \mu(d\lambda).  \notag 
\end{align}
Hence
\begin{align}
e^{-2t}t^{1-a}\|e^{t\Delta} C_0\|_{H_1}^2 
                = \int_1^\infty (t\lambda)^{1-a} 
     e^{-2t\lambda} \mu(d\lambda).        \label{ST456a}
\end{align}
The integrand is uniformly bounded in \{$t >0$ and $\lambda\ge 0$\}
 by $c_a \equiv \sup_{\sigma >0} \sigma^{1-a}e^{-2\sigma}$ and therefore the integral
  is at most  $c_a \|D^a C_0\|_2^2$. Moreover for
  each point $\lambda \in [0,\infty)$ the integrand
 goes to zero as $t\downarrow 0$.  Since $\mu$ is a finite measure the dominated convergence
 theorem implies the remainder of \eref{ST449a}.

       Using now  \eref{ST456a} again, and
      substituting $\tau = t\lambda$, we find
 \begin{align*}
e^{-2T} \int_0^T t^{-a} \|e^{t\Delta}C_0\|_{H_1}^2 dt
 &=e^{-2T} \int_0^T e^{2t}\int_1^\infty (t\lambda)^{-a} \lambda e^{-2t\lambda}\mu(d\lambda)  dt\\
 &\le\int_1^\infty\int_0^T  (t\lambda)^{-a}e^{-2t\lambda} \lambda dt\ \mu(d\lambda)  \\
 &= \int_1^\infty \Big(\int_0^{T\lambda}\tau^{-a} e^{-2\tau} d\tau\Big) \mu(d\lambda) \\
 &\le \gamma_a^2\int_1^\infty \mu(d\lambda) = \gamma_a^2 \| D^{a} C_0\|_2^2,
 \end{align*}  
 where $\gamma_a^2 =\int_0^\infty \tau^{-a} e^{-2\tau} d\tau$. This proves \eref{ST450a}.
\end{proof}

\begin{corollary}\label{corcefree} Let $ 0 \le a <1$ and let $C_0 \in H_{a}$. Then the function
 \beq
 [0, T] \ni t\mapsto C(t) := e^{t\Delta} C_0
 \eeq
 lies in $\P_T^a$ for all $T \in (0, \infty)$.
 \end{corollary}
        \begin{proof} $C(\cdot)$ is  a continuous function on $[0,T]$ into $H_{a}$ because
        $e^{t\Delta}$ is a strongly continuous semigroup in $H_{a}$.  The second assertion 
        in  \eref{ST449a} shows that
   $ t^{(1-a)/2} \| e^{t\Delta} C_0\|_{H_1} \rightarrow 0\ \ \text{as}\ \ t\downarrow 0$,
 which is condition  \eref{ST413a}.
  Since $C(t) \in H_1$ for any $t > 0$, $C(\cdot)$   is also a continuous function
  on $(0, T]$  into $H_1$.
  \end{proof}

\begin{remark}\label{phil2}{\rm (Pointwise behavior vs integral behavior)
Let $f(t) =  \| e^{t\Delta} C_0\|_{H_1}^2$. Observe that \eref{ST449a}  
says that $t^{-a}f(t)  =o(t^{-1})$ while \eref{ST450a}  says that $t^{-a} f(t)$ is integrable over
$(0, T)$. Neither assertion implies the other.
Both hold for this particular function.
Both types of inequalities, pointwise in $t$  and integral, will be needed for solutions 
$C(\cdot)$ to  \eref{aymh}.
Many of the apriori estimates that we will derive
 will show the strong interplay between them. 
 This  interplay was already  a key tool in \cite{CG1}.
 A pointwise inequality in $t$, such as \eref{ST449a} or \eref{ST413a}, provides a mechanism for 
 proving the existence of solutions for $H_{a}$ initial data.
 But  it is an integral  condition, such as \eref{ST450a}  or \eref{fa1C},  or more 
 particularly their gauge invariant version \eref{fa1a}, which has direct
  physical significance and  which we will address in a more gauge invariant
   formulation in a future work, \cite{G72}.
}
\end{remark}

\subsection{Contraction estimates}\label{secce}

           For $C(\cdot)$ in the path space $\P_T^a$ we have at
 our disposal two kinds of size conditions for use in  estimating  the terms in \eref{ST14.1}. 
$\|C(s)\|_{H_{a}}$ is continuous and therefore bounded on $[0,T]$, 
and therefore so also is $\|C(s)\|_{q_a}$, by Sobolev, where $q_a^{-1} = 1/2 - a/3$.
(It may be useful to keep in mind that $q_{1/2} = 3$.)
  In addition, we have an $s$ dependent bound on $\|C(s)\|_{H_1}$ of the form
  $\|C(s)\|_{H_1} \le s^{(a-1)/2} \cdot |C|_T$, from \eref{ST416a}.
  These two bounds will be used in different combinations.
 The following lemma lists several kinds of estimates that will be needed for
the two different types of terms in $X(C)$. The proofs just rely on H\"older inequalities together
with the Sobolev inequality $\| C\|_6 \le \kappa_6 \|C\|_{H_1}$. 
 We use 
 $\|\p C\|_2 \le \|C\|_{H_1}$.
  We also continue to use, as in \cite{CG1}, the constant
 $c\equiv \sup \{ \| ad\ x \|_{\kf\rightarrow \kf}: \|x\|_\kf \le 1\}$, which measures the non-commutativity of $\kf$.

         \begin{lemma} Let $C$ be a $\kf$ valued 1-form on $M$. Then the following inequalities hold.
The H\"older inequality arithmetic needed in the proof is on the same line as the inequality.
The power of $c$ reflects the number of commutators that appear on the left.
\begin{align}
\|C^3\|_{6/5} &\le c^2\kappa_6 \ \ \|C\|_{H_1}\|C\|_3^2    \ \ \ \  5/6 = 1/6 +1/3 + 1/3   \label{ce13}\\
\|C^3\|_{3/2} &\le c^2\kappa_6^2 \ \ \|C\|_{H_1}^2\|C\|_3 \ \ \ \ 2/3 =  1/6 +1/6 +1/3       \label{ce15}\\
\|C^3\|_2\  \ &\le  c^2 \kappa_6^3\ \  \|C\|_{H_1}^3 \ \ \ \ \  \ \ \ \  \ 1/2 = 1/6+1/6+1/6   \label{ce11}\\
\|C\cdot \p C\|_{6/5} &\le c \ \ \ \ \ \ \|C\|_{H_1}\|C\|_3     \ \ \ \     5/6 = 1/2 + 1/3        \label{ce14}\\
\|C\cdot \p C\|_{3/2} &\le  c\kappa_6\ \  \ \|C\|_{H_1}^2\ \ \ \ \ \ \ \ \ \  2/3= 1/6+1/2   \label{ce12}
\end{align}
\end{lemma}
  \begin{proof} The proofs are in the right hand column.
  \end{proof}

        \begin{remark}{\rm  The following elementary inequality is displayed here for frequent reference.
If $L$ is a non-negative self-adjoint operator
on a Hilbert space and $D = L^{1/2}$ then
\beq
\| D^\alpha e^{-tL}\| \le c_\alpha t^{-\alpha/2},\ \ \  t > 0, \ \ \ \alpha \ge0,     \label{hk0}
\eeq
for some constant $c_\alpha$, as follows from the spectral theorem  and the inequality  
$\sup_{\lambda>0} \lambda^{\alpha/2} e^{-t\lambda} =t^{-\alpha/2} \sup_{\sigma >0} \sigma^{\alpha/2}e^{-\sigma}$.
Here $\lambda \ge 0$ is a spectral parameter for $L$. The case of interest for us
 will be  $L = 1 -\Delta$ acting on $L^2(M; \Lambda^1\otimes \kf)$.
}
\end{remark}

             \begin{lemma}\label{lemce2a} Let $0 < a < 1$ and let $C \in \P_T^a$. Then
 \begin{align}
\| C(s)^3 \|_2 &\le s^{-(3/2)(1-a)} |C|_t^3\ (c^2\kappa_6^3) \ \ \text{for}\ 0<s \le t \le T \ \ 
                                                     \text{and}                                                 \label{ST310a}\\
\|C(s)\cdot \p C(s)\|_{3/2} &\le s^{-(1-a)} |C|_t^2\   (c \kappa_6)\ \ \ \ \ \ \ \
                                                    \text{for}\ 0<s \le t \le T,      \qquad                           \label{ST311a}
\end{align}
where $|C|_t$ is defined by \eref{ST415a}.
Further,
\begin{align}
  \|e^{(t-s) \Delta}\{C(s)^3\}\|_{H_1}  
           &\le   (t-s)^{-1/2} s^{3(a-1)/2}\  |C|_t^3\ C_{40a} .       \label{ST312a} \\
  \|e^{(t-s) \Delta}\{C(s)^3\}\|_{H_{a}}  
         &\le  (t-s)^{-a/2}   s^{3(a-1)/2}\  |C|_t^3\ C_{41a}.      \label{ST313a}\\
 \|e^{(t-s) \Delta} \{ C(s)\cdot \p C(s)\}\|_{H_1}  
         &\le  (t-s)^{-3/4}  s^{a-1}\ \ \ \ \ \  |C|_t^2\    C_{42a}.   \label{ST314a}\\
   \|e^{(t-s) \Delta} \{ C(s)\cdot \p C(s)\}\|_{H_{a}}  
         & \le   (t-s)^{-(2a+1)/4}      s^{a-1}\  |C|_t^2\    C_{43a}. \label{ST315a}
\end{align} 
The constants $C_{ja}$ depend only on Sobolev constants, on the constants $c_\alpha$
in \eref{hk0}, on powers of the commutator norm $c$ in $\kf$ and on $a$.
\end{lemma}
\begin{proof} 
     By \eref{ce11} and  \eref{ST416a} we have
     \begin{align*}
   \|C(s)^3\|_2 \le   c^2 \kappa_6^3 \| C(s)\|_{H_1}^3 
\le c^2\kappa_6^3 (s^{(a-1)/2} |C|_t)^3,
\end{align*}
which is \eref{ST310a}. Combining \eref{ce12} and \eref{ST416a}, one finds
\begin{align*}
 \| C(s)\cdot \p C(s)\|_{3/2} 
      \le c \kappa_6 \|C(s)\|_{H_1}^2
      \le   c \kappa_6 (s^{(a-1)/2} |C|_t)^2,
      \end{align*}
      which is \eref{ST311a}. 
      
      The two inequalities \eref{ST312a} and \eref{ST313a} follow directly 
      from  \eref{ST310a} combined with \eref{hk0} with  $\alpha =1$ or $a$, respectively. 
      Here we are ignoring irrelevant factors of $e^T$ needed to justify 
      $ \|e^{r\Delta} f\|_{H_\alpha} = \| D^\alpha e^{r\Delta} f\|_2$ because we are only concerned with small $T$.
            
      For the remaining two inequalities we need to interpose a Sobolev inequality
    before applying \eref{hk0}.
       We have a bound, 
      $\kappa'$ say, on  the norm of $D^{-1/2}: L^{3/2} \rightarrow L^2$
 because  $1/2 = (2/3) - (1/2)/3$. Thus we can write
       \begin{align*}
       \| e^{(t-s) \Delta} \{C(s)\cdot \p C(s)\} \|_{H_1} &=
 \|D^{1/2}e^{(t-s) \Delta} D^{-1/2} \{C(s)\cdot \p C(s)\}\|_{H_1}  \\
 &=\| D^{3/2} e^{(t-s) \Delta} D^{-1/2} \{C(s)\cdot \p C(s)\}\|_2 \\
 &\le \| D^{3/2} e^{(t-s) \Delta}\|_{2\rightarrow 2}  \|D^{-1/2} \{C(s)\cdot \p C(s))\}\|_2\\
         &\le \Big(c_{3/2} (t-s)^{-3/4}\Big)\cdot\  \Big( \kappa's^{a-1} |C|_t^2\   (c \kappa_6)\Big),
 \end{align*}
 wherein we have  used \eref{ST311a} and \eref{hk0}. This proves \eref{ST314a}. 
 The proof of \eref{ST315a} is the same but with $H_1$ replaced by 
    $H_{a}$ and 
 with  $D^{3/2}$ replaced by $D^{1/2 +a} $ in the second and third lines. In this case
 $\alpha$ in \eref{hk0} should be taken to be $1/2 + a$, giving \eref{ST315a}
\end{proof}

      \begin{remark}\label{remcv}{\rm    The following   identity, which arises frequently,
 is listed here for convenience.
Let  $\mu$ and $\nu$ be real numbers with  $ \mu <1$ and  $\nu <1$. Then
\begin{align}
\frac{1}{t} \int_0^t(t-s)^{-\mu} s^{-\nu} ds 
                        &= t^{-\mu -\nu } C_{\mu, \nu}\                            \label{rec519}  
\end{align} 
for some finite constant $C_{\mu,\nu}$.
         For the proof,
         make the change of variables $s = tr$ 
         to convert the integral  to 
$t^{-\mu -\nu}\int_0^1 (1-r)^{-\mu} r^{-\nu} dr$, which has the asserted form.
}
\end{remark}

          \begin{lemma}\label{lemce6a} Let $1/2 \le a <1$ and let $C(\cdot)$ be in $\P_T^a$. 
    Define 
    \begin{align}
w(t) = \int_0^t e^{(t-s)\Delta} X(C(s)) ds.            \label{ST319}
\end{align}
Then
\begin{align}
w:[0,T] \rightarrow H_{a}  \ \ \text{and}\ \ w:(0,T] \rightarrow H_1   \label{ST319c}
\end{align}
are both continuous. Moreover
\begin{align} 
t^{\frac{1-a}{2}} \|w(t) \|_{H_1}\ 
   &\le  \Big(t^{a -(1/2)}|C|_t^3 + t^{\frac{a -(1/2)}{2}}|C|_t^2\Big)  C_{50a}, \  \label{ST320a} \\
\|w(t) \|_{H_{a}}\ 
      &\le \Big( t^{a - (1/2)} |C|_t^3 +t^{\frac{a -(1/2)}{2}}|C|_t^2\Big)C_{51a}  \ \ \ \ \ \ \text{and} \label{ST331a}\\
\| w(t) \|_{q_a} \  
      &\le  \Big( t^{a - (1/2)} |C|_t^3 +t^{\frac{a -(1/2)}{2}}|C|_t^2\Big) C_{52a},  \label{ST330a}
\end{align}
where $q_a^{-1}= (1/2) - (a/3)$.

      The constants $C_{ja}$ depend only  on Sobolev constants,
  the coefficients $c_\alpha$ in \eref{hk0}, on the commutator norm $c$ and on $a$.
\end{lemma}
               \begin{proof}  The sums on the right sides of these three inequalities
correspond to the decomposition $X(C)  = C^3 + C\cdot \p C$ in \eref{ST14.1}. 
We need to carry out 
the derivation of these inequalities separately for the cases $C(s)^3$ and $C(s)\cdot \p C(s)$
because of the slightly different powers of $(t-s)$ and $s$ that occur in
  \eref{ST312a} - \eref{ST315a}.
The following derivation 
is typical of all of them. We have
\begin{align}
\int_0^t \| e^{(t-s)\Delta} C(s)^3 \|_{H_1} ds 
 \le \int_0^t (t-s)^{-1/2} s^{3(a-1)/2} ds\  |C|_t^3\ C_{40a}       \label{ST333a}
\end{align}
by \eref{ST312a}. All the other three estimates needed in \eref{ST320a} and \eref{ST331a}
 have a similar form. They differ only in the powers $(t-s)^{-\mu} s^{-\nu}$
  that occur.
      The identity \eref{rec519} shows that  $\int_0^t(t-s)^{-\mu} s^{-\nu} ds =
t^{1 - \mu - \nu} \cdot\ constant$.
Thus in the case of \eref{ST333a} one sees that 
$1-\mu - \nu =  1 -(1/2) + 3(a-1)/2 = -1 + (3/2)a $. This gives correctly the power for the first term
on the right in \eref{ST320a} upon taking into account the  factor $t^{(1-a)/2}$ on the left side
of \eref{ST320a}. We leave the arithmetic for the remaining three cases to the reader.
By Sobolev, \eref{ST330a} follows from \eref{ST331a}.

         It remains to prove the two assertions about continuity in \eref{ST319c}. 
         Observe first that  \eref{ST331a}
         implies continuity of $w$ into $H_{a}$ at $t =0$ because 
         $w(0) = 0$ and $|C|_t \rightarrow 0$  as $t \downarrow 0$ by \eref{ST417a}. (Notice that
         for $a = 1/2$ we must rely on $|C|_t \rightarrow 0$ whereas for $a > 1/2$ the strictly
         positive powers of $t$ in \eref{ST331a} are enough to ensure
          that $\|w(t)\|_{H_a} \rightarrow 0$.)
         It suffices, therefore, to prove both continuities on an interval $[\epsilon, T]$ with 
         $\epsilon > 0$. Suppose then
   that $0< \epsilon \le r < t \le T$. The identity
   \begin{align}
   w(t) - w(r) = \int_r^t e^{(t-s)\Delta} F(s) ds 
   + \int_0^r\Big(e^{(t-r) \Delta} - I\Big) e^{(r-s) \Delta} F(s) ds,            \label{ST336}
   \end{align}
   wherein $F(s) = X(C(s))$ is easily verified. We need to show that $\|w(t) - w(r)\|_{H_\alpha} \rightarrow 0$ as $t-r \rightarrow 0$   in the interval 
   $[\epsilon, T]$ for $\alpha = 1$ and $\alpha =a$. 
   First consider the term $F(s) = C(s)^3$ in $X(C(s))$.
 We have, by \eref{ST312a} and \eref{ST313a}, 
 \begin{align*}
 \int_r^t  \|e^{(t-s)\Delta} C(s)^3\|_{H_\alpha} ds 
 \le \begin{cases}  & \int_r^t (t-s)^{-1/2} s^{3(a-1)/2} ds\  |C|_T^3\  C_{40}, \ \ \  \alpha = 1\\ 
 &\int_r^t (t-s)^{-a/2} s^{3(a-1)/2} ds\  |C|_T^3\  C_{41},  \ \ \  \alpha = a.                             
                            \end{cases}
 \end{align*}
 Both integrals on the right go to zero as $t-r \rightarrow 0$ if $r$ and $t$ are bounded away
 from zero.     Similarly, by \eref{ST314a} and \eref{ST315a}, 
  \begin{align*}
 \int_r^t  \|e^{(t-s)\Delta} \{C(s)\cdot\p C(s)\}\|_{H_\alpha} ds 
 \le \begin{cases}& \int_r^t (t-s)^{-3/4} s^{a-1} ds\  |C|_T^2\  C_{42}, \  \alpha = 1\\
 &\int_r^t (t-s)^{-(2a +1)/4} s^{a-1} ds\  |C|_T^2\  C_{43},  \  \alpha = a,                             
                            \end{cases}
 \end{align*} 
 which also goes to zero if $r$ and $t$ lie in the interval $[\epsilon, T] $ and $t-r \rightarrow 0$.

 Concerning the second integral in \eref{ST336} observe that,
  although the operator in parentheses
 goes to zero strongly as $ t-r \downarrow 0$, it does not go to zero in norm.
 Let $ 0 < \delta < 1/4$. Then, for any measurable function
  $F:(0, T]\rightarrow L^2(M; \L^1\otimes \kf)$, we have 
 \begin{align}
  \int_0^r\|\Big(&e^{(t-r) \Delta} - I\Big) e^{(r-s) \Delta} F(s)\|_{H_\alpha} ds\notag\\
  &=\int_0^r \|\Big(e^{(t-r) \Delta} - I\Big)D^{-2\delta} 
                  D^{2\delta} e^{(r-s) \Delta} F(s)\|_{H_\alpha} ds                        \notag\\
  &\le \|\Big(e^{(t-r) \Delta} - I\Big)D^{-2\delta}\|_{2\rightarrow 2}
   \int_0^r\|D^{2\delta} e^{(r-s) \Delta} F(s)\|_{H_\alpha}   ds.       \label{ST339}
   \end{align}
   The operator norm in the first factor goes to zero for any 
   $\delta >0$ as $t-r\downarrow 0$. It suffices to prove therefore
   that the integral factor is uniformly bounded for $r \in [\epsilon, T]$. 
   But    \eref{hk0} implies  that, in the presence of the factor $D^{2\delta}$, 
   each of the factors $(t-s)^{-\mu}$ in the  inequalities \eref{ST312a} - \eref{ST315a}
   need only be replaced by $(t-s)^{ -\mu -  \delta}$.    All of these four exponents remain
   greater than $-1$ for $1/2 \le a \le 1$ because $ \delta < 1/4$. Consequently the four
    estimates needed to bound the integral factor in \eref{ST339}, for $\alpha = 1$ or $a $ and
     $F= C^3$ or $C\cdot \p C$, remain bounded on the
    interval $\epsilon\le r \le T$.
   \end{proof}

\subsection{Proof of existence of mild solutions} \label{seceum}

\begin{notation}\label{notb}{\rm  Let $1/2 \le a < 1$  and suppose that $C_0 \in H_a(M)$.  
Let
\beq
\P_{T,b}^a = \{ C(\cdot)\in \P_T^a:  |C|_T \le b\},
\eeq
where   $|C|_t$ is defined as in \eref{ST415a}.
 For any $b >0$ the set $\P_{T,b}^a$ is complete in  the metric \eref{ST414a}
and is non-empty for some $T>0$ since, by Corollary \ref{corcefree},  $ \P_{T_1}^a$
is non-empty for all $T_1 >0$, and if   $C(\cdot) \in \P_{T_1}^a$  then the restriction
 of $C(\cdot)$ to $[0, T]$ will be in $\P_{T, b}^a$ for some $T \in (0 ,T_1]$ by virtue
 of
  \eref{ST417a}. It is this feature of the spaces $\P_T^a$ that
  will allow our method   to work in the critical case $a = 1/2$. 
  See Remark \ref{usual} for further discussion of this.
 }
\end{notation}

\begin{lemma}\label{lemce5} Define
\beq
Z(C)(t) = e^{t\Delta}C_0 +\int_0^t e^{(t-s)\Delta} X(C(s)) ds \  \ \
                      \text{for}\ \ \ C(\cdot) \in \P_T^a.                                        \label{ST341}
\eeq
Then 
\beq
 Z(\P_T^a) \subset \P_T^a.                    \label{ST342}
\eeq
Let
\beq
b_0 = \sup_{0 < t \le T} t^{(1-a)/2} \|e^{t\Delta}C_0 \|_{H_1}.                   \notag
\eeq
If $C\in \P_{T,b}^a$  and $T \le 1$ then
\beq
|Z(C)|_T   \le b_0 + (b^2 + b^3) C_{50a},               \label{ST343}
\eeq
where $C_{50a}$ is defined in \eref{ST320a}.
If   $b>0$ is chosen so small that   
\beq
(b^2 + b^3) C_{50a} \le b/2    \label{ST344}
\eeq
 and $T >0$ is chosen
so small that 
\beq
b_0 \le b/2     \label{ST345}
\eeq
 then $Z$ takes $\P_{T,b}^a$ into itself.
\end{lemma}
               \begin{proof}  By Corollary \ref{corcefree} the first term in \eref{ST341} lies in $\P_T^a$.
 By Lemma \ref{lemce6a} the second term in \eref{ST341} defines a continuous function
 on $[0, T]$ into $H_{a}$ and a continuous function on $(0,T]$ into $H_1$. Further,
 \eref{ST320a}  shows that condition \eref{ST413a} holds for the second term because
 the factors $|C|_t $ go to zero as $t\downarrow 0$ in accordance with
 \eref{ST417a}. This proves \eref{ST342}.

 It is worth noting 
 a distinction between $a = 1/2$ and $a > 1/2$ that will recur often:  
 For $a = 1/2$ the right side of \eref{ST320a} goes to zero as $t\downarrow 0$ only because
 $|C|_t \rightarrow 0$,  which is built into the definition of $\P_T^a$.  For $a > 1/2$ the
  two strictly positive powers of $t$ on the right side of \eref{ST320a}  contribute further to the
   decay as $t \downarrow 0$.

In view of the definition \eref{ST415a}, the inequality \eref{ST343} follows
 from \eref{ST320a} (with $0 \le t \le T \le 1$)  and the definition of $b_0$. 
        Now if $b$ is chosen so small that  \eref{ST344} holds   then
 \eref{ST343} shows  that $|Z(C)|_T \le b_0 +(b/2)$. 
 Further, \eref{ST449a} 
 shows that we can choose
 $T>0$ so small that  \eref{ST345} holds.  Thus for these values of $b$ and $T$ we find
 $|Z(C)|_T \le b$. Therefore $Z$  takes  $\P_{T,b}^a$ into itself.
\end{proof}

          \begin{lemma}\label{lemce7} $Z$ is a contraction on $\P_{T,b}^a$ for  $b$ and $T$
 sufficiently small.
\end{lemma}
              \begin{proof} If $C_j \in \P_{T,b}^a$ for $j = 1,2$ then, for $0 < t \le T \le 1$,  
   \begin{align}
   t^{(1-a)/2} \|& Z(C_1)(t) - Z(C_2)(t)\|_{H_1}     \notag \\
     &\le
     t^{(1-a)/2} \int_0^t\|e^{(t-s) \Delta} \{X(C_1(s)) - X(C_2(s))\}\|_{H_1} ds \notag\\
     &\le  |C_1 -C_2|_T (2b +3b^2)C_{50a}       \label{ST350}
\end{align}   
by polarization of \eref{ST320a} (with $t = T$ in that inequality.) 
Similarly, by polarizing \eref{ST331a} we find, for $0\le t \le T \le 1$,
\begin{align}
 \|& Z(C_1)(t) - Z(C_2)(t)\|_{H_{a}} \le |C_1 - C_2|_T (2b + 3b^2)C_{51a}.     \label{ST351}
 \end{align}
 Choose  $b$ so small that   not only \eref{ST344} holds, but also  
  \beq
 (2b+3b^2) \max (C_{50a}, C_{51a}) \le 1/4.  \label{ST352}
 \eeq
 Since $C_{50a}$ and $C_{51a}$ are independent of $C_0$ and $T$,
  so is the size restriction  on $b$. 
       As we saw in Lemma \ref{lemce5}, for our fixed $C_0 \in H_{a}$,       
      we can choose $T$ so small that \eref{ST345} holds.  
  For such choices of $b$ and $T$,   $Z$ takes $\P_{T,b}^a$ into itself by 
 Lemma  \ref{lemce5} and 
 \begin{align}
 dist(Z(C_1), Z(C_2)) &= \sup_{0\le t \le T} \|Z(C_1)(t) - Z(C_2)(t)\|_{H_a}  \notag\\
 &  + \sup_{0 < t \le T}  t^{(1-a)/2}\|Z(C_1)(t) - Z(C_2)(t)\|_{H_1} \notag\\
 &\le (1/4) |C_1 - C_2|_T + (1/4) |C_1-C_2|_T \notag\\
& \le (1/2) dist(C_1, C_2)
 \end{align}
  by   \eref{ST350} and \eref{ST351}.
 Thus $Z$  is a contraction on $\P_{T,b}^a$. 
 \end{proof}

\begin{remark} {\rm  The fact that  
  $\sup_{0<t\le T} \|C_1(t)  - C_2(t)\|_{H_{a}}$
  does not enter into  the right sides of \eref{ST350} or \eref{ST351}
  suggests that, in some sense, the behavior \eref{ST413a}, of $\|C(s)\|_{H_1}$  near $ s = 0$, controls $\|C(t)\|_{H_a}$.
  We will see  strong forms of this in the papers \cite{G71} and \cite{G72}.
}
\end{remark}

\bigskip
\noindent
\begin{proof}[Proof of Theorem \ref{thmmild}, Part {\it i.})] The proof is an immediate consequence
of Lemma \ref{lemce7}. 
\end{proof}

\begin{remark}\label{phil5} {\rm (Dependence of $T$ on $C_0$) 
The time $T$  that we have produced in Theorem \ref{thmmild} 
depends on $C_0$ itself, and not just on $\|C_0\|_{H_{a}}$, because the strong limit
in \eref{ST449a} cannot be replaced by a limit in operator norm.
Indeed,  \eref{ST449a} asserts  that 
the operator function $t\mapsto (t^{(1-a)/2} e^{t\Delta}:H_a \rightarrow H_1)$ 
goes to zero  strongly as $t\downarrow 0$. 
 One can verify with the help of the spectral theorem  that it does
  not go to zero in operator norm. 
}
\end{remark}

\begin{remark}\label{usual}{\rm   
 Typically, a proof of existence and uniqueness of solutions
 for the integral equation   \eref{ST25} proceeds by establishing estimates 
 for the non-linear operator $Z$ in \eref{ST341} of the form
 \beq
\| Z(C_1(\cdot)) - Z(C_2(\cdot))\|\le const. T^\alpha \| C_1(\cdot) - C_2(\cdot)\|, \label{ST402}
\eeq
for some $\alpha > 0$ and some norm on   a Banach space containing the
freely propagated term $e^{t\Delta}C_0$ in \eref{ST25}.
One need only take $T$ small to conclude that $Z$ is a contraction. 
Typical estimation methods  for establishing \eref{ST402} in some contexts can be
  found, for example, in  Taylor, \cite[page 273]{Tay3}.  
In our context such a contraction proof  works in case $ a > 1/2$. 
Indeed polarization of \eref{ST320a} and \eref{ST331a} shows  that for $0 < t \le T \le 1$ one
can include a factor  $T^{\frac{a -(1/2)}{2}}$ on the right hand sides
 of \eref{ST350} and \eref{ST351}.  Thus \eref{ST402} holds with  $\alpha =(a - (1/2))/2$.
Since $\alpha >0$ when $a > 1/2$ one could proceed
 in this case in the usual way without having to rely on the fact that 
 $t^{(1-a)/2} \|C(t)\|_{H_1}$ is not only 
bounded, but also goes to zero as $t \downarrow 0$, as was assumed in \eref{ST413a}. 
In case $a = 1/2$ one has $\alpha =0$
and the preceding standard technique for proving contraction fails.
The requirement \eref{ST413a} then becomes essential in the choice of the path space $\P_T^a$.
 It is not clear whether this distinction in techniques  for $a > 1/2$ or $a = 1/2$ 
 is an intrinsic feature of criticality or an artifact of our choice of metric space $\P_T^a$.
   We will see a similar dichotomy in dealing with finite action in the next subsection.
}
\end{remark}

\subsection{$C(\cdot)$ has finite action} 
\label{secfinact}

\begin{theorem}\label{thmfa}
 Let $1/2 \le a <1$. Suppose that $C_0 \in H_a$.
Let $C(\cdot)$ be the solution to the integral equation \eref{ST25} produced in
 Theorem \ref{thmmild}, Part i.$)$.

\noindent 
If $a > 1/2$  then 
\beq
\int_0^T s^{-a} \| C(s)\|_{H_1}^2 ds < \infty \ \   \ \ \  \text{for sufficiently small}\ T .  \label{ce100a}
\eeq
If $a = 1/2$  and $ \|C_0\|_{H_{1/2}}$ is sufficiently small then
\beq
\int_0^T s^{-1/2} \| C(s)\|_{H_1}^2 ds < \infty \ \   \ \ \  \text{for sufficiently small}\ T .\label{ce100}
\eeq
\end{theorem}
The proof will be developed in the next two subsections.

\subsubsection{Abstract action estimates } \label{secae}

We need to make an estimate of the integral term in \eref{ST25} similar to the 
estimate \eref{ST450a} for the freely propagated term.
We will be able to avoid using heat kernel estimates in favor of just the spectral theorem and simple Sobolev inequalities with the help of the following theorem. 
      The parameters $\alpha, \mu, b$ in the theorem  will be chosen to
       fit our various needs   in this paper and its sequel. The operator $L$ of interest to us
       will be $1 -\Delta$ on forms.

          \begin{theorem} \label{thmactint} Let $L$ be a non-negative
  self-adjoint operator on a Hilbert space $\Hc$. Suppose that $\alpha, \mu, b$ are real numbers
  such that
  \begin{align}
 &0 \le \alpha \le 1,                                   \label{ce300}\\
 &0 \le \mu \le b < 1,                                                    \label{ce302}\\
 &\delta \equiv 1-\alpha - \mu \ge0.                       \label{ce301}
  \end{align}
  Then there is a constant $C_{\alpha,\mu}$, depending only on $\alpha$ and $\mu$,
   such that $C_{\alpha,0} \le 1$ and such that for $0 < T < \infty$ and for any measurable function $g:(0,T)\rightarrow \Hc$ there holds
  \begin{align}
 \int_0^T t^{-b} \Big\|\int_0^t s^{-\mu} L^\alpha e^{-(t-s)L} g(s) ds\Big\|^2 dt 
 \le T^{2\delta}\int_0^T s^{-b} \| g(s)\|^2 ds \cdot  C_{\alpha,\mu} .           \label{ce303}
 \end{align}  
 In particular, for $\mu =0$, there holds
 \begin{align}
 \int_0^T t^{-b} \Big\|\int_0^t  L^{ \alpha} e^{-(t-s)L} g(s) ds\Big\|^2 dt 
 \le T^{2(1-\alpha)}\int_0^T s^{-b} \| g(s)\|^2 ds  \ \ \text{for}\ \ 0\le b <1.      \label{ce303b}
 \end{align}
          \end{theorem}

  The proof depends on the following lemmas.

\begin{lemma} \label{lemai1} $($Schwarz-like inequality$)$  Suppose that $(0, t) \ni s \mapsto H(s)$ is strongly continuous function into a set  of commuting, bounded, non-negative Hermitian  operators on a Hilbert space $\cal H$. 
If $ g:(0, T)\rightarrow \cal H$
is measurable then
\begin{align}
\Big\| \int_0^t H(s) g(s) ds \Big\|^2 
         \le \Big\| \int_0^t H(s) ds \Big\|\  \int_0^t ( H(s) g(s), g(s)) ds              \label{ai10}
\end{align}
\end{lemma}
\begin{proof} Let $F(s) = H(s)^{1/2}$. Then 
\begin{align*}
\Big\| \int_0^t H(s) g(s) ds \Big\|^2 &= \Big\|\int_0^t  F(s)^2 g(s) ds \Big\|^2 \\
& =\int_0^t ds_1 \int_0^t ds_2 \Big( F(s_1)^2 g(s_1), F(s_2)^2 g(s_2)\Big) \\
&= \int_0^t ds_1 \int_0^t ds_2 \Big ( F(s_2)F(s_1) g(s_1), F(s_1)F(s_2) g(s_2) \Big)\\
&\le (1/2) \int_0^t ds_1 \int_0^t ds_2 \Big( \| F(s_2)F(s_1) g(s_1)\|^2 
                                    + \| F(s_1)F(s_2) g(s_2)\|^2\Big) \\
&= \int_0^t ds_1 \int_0^t ds_2  \| F(s_2)F(s_1) g(s_1)\|^2  \\
&=  \int_0^t ds_1 \int_0^t ds_2 \Big( F(s_2)^2 F(s_1) g(s_1), F(s_1) g(s_1)\Big) \\
&=   \int_0^t ds_1  \Big( \Big\{ \int_0^t ds_2F(s_2)^2\Big\} F(s_1) g(s_1), F(s_1) g(s_1)\Big) \\
&\le \Big\| \int_0^t F(s_2)^2  ds_2 \Big\|  \int_0^t ds_1 \Big(F(s_1) g(s_1),F(s_1) g(s_1)\Big).
\end{align*}
\end{proof}
\begin{remark}{\rm  A simpler and cruder proof would easily give the inequality \eref{ai10}
but with $\| H(s)\|$ under the integral in the first factor. However in the case of interest to us this integral would diverge.
}
\end{remark}

         \begin{lemma} 
\label{lemai2}  
 Under the hypotheses of  Theorem \ref{thmactint}  we have 
   \beq
  \Big\| \int_0^t s^{-\mu} L^\alpha e^{-(t-s)L} ds \Big\| \le t^\delta (1-\mu)^{\alpha -1}. \label{ce304}
  \eeq  
        \end{lemma}
         \begin{proof}
          By the spectral theorem we need only prove that for any $\lambda >0$ there holds
 \beq
    \int_0^t s^{-\mu} \lambda^\alpha e^{-(t-s) \lambda} ds 
                         \le    t^{1 - \alpha - \mu} (1-\mu)^{\alpha -1} .              \label{ce305}
 \eeq
    Make the substitution $r = s\lambda$ and define $\gamma = t\lambda$ to find
 \begin{align*}
  \int_0^t s^{-\mu} \lambda^\alpha e^{-(t-s) \lambda} ds
  &= e^{-\gamma}\int_0^\gamma r^{-\mu} \lambda^{\alpha +\mu -1} e^r dr \\   
  &= t^{1 -\alpha -\mu}\Big\{ \gamma^{\alpha + \mu -1} 
                                     e^{-\gamma}\int_0^\gamma r^{-\mu} e^r dr\Big\}. 
  \end{align*}
  It suffices to prove that the expression in braces is at most
   $(1-\mu)^{\alpha -1}$ for all $\gamma >0$.
  In case $\gamma \le 1-\mu$  the expression in braces is at most
  \begin{align*}
   \gamma^{\alpha + \mu -1} \int_0^\gamma r^{-\mu}  dr 
           = \gamma^{\alpha + \mu -1} \frac{\gamma^{1-\mu}}{ 1 -\mu} = \frac{\gamma^\alpha}{1-\mu}
           \le \frac{(1-\mu)^\alpha }{1-\mu}.
   \end{align*}
   This proves \eref{ce305} when  $t\lambda \le 1-\mu$.   
       In case $\gamma > 1-\mu$ the expression in braces is at most, considering that $\alpha + \mu -1 \le 0$, 
       \begin{align*}
   (1-\mu&)^{\alpha + \mu -1} e^{-\gamma}\int_0^\gamma r^{-\mu} e^r dr \\
   &\le  (1-\mu)^{\alpha + \mu -1} e^{-\gamma}\Big(\int_0^{1-\mu} r^{-\mu}dr\ e^{1-\mu} 
   + (1-\mu)^{-\mu} \int_{1-\mu} ^\gamma e^{r} dr \Big) \\
   &=(1-\mu)^{\alpha + \mu -1} e^{-\gamma}\Big((1-\mu)^{-\mu}\ e^{1-\mu}  
    + (1-\mu)^{-\mu}(e^\gamma - e^{1-\mu} )\Big) \\
    &= (1-\mu)^{\alpha + \mu -1}  e^{-\gamma}(1-\mu)^{-\mu}e^\gamma.
   \end{align*}
  This proves \eref{ce305} when $t\lambda > 1-\mu$.    
    \end{proof}
    
\begin{lemma} \label{ai3}
 Under the hypotheses of Theorem \ref{thmactint} we have
 \begin{align}
 \Big\| \int_0^t  s^{-\mu}L^\alpha &e^{-(t-s)L} g(s) ds\Big\|^2 \notag\\
         &\le  (1-\mu)^{\alpha -1}\, t^\delta  
                \int_0^t\Big( s^{-\mu}L^\alpha e^{-(t-s)L}g(s), g(s)\Big)ds .   \label{ce210} 
         \end{align}
\end{lemma}
  \begin{proof} Let $H(s) =   s^{-\mu}L^\alpha e^{-(t-s)L} , 0 < s < t$. Then \eref{ce304} shows that
  $\|\int_0^t H(s) ds \| \le  t^{\delta} (1-\mu)^{\alpha - 1}$. Insert this bound into \eref{ai10} to find 
  \eref{ce210}.
\end{proof}

  We will need also an estimate of the following integral over $(s, T)$.

        \begin{lemma}\label{ai4}  
 Under the hypotheses of Theorem \ref{thmactint} we have
 \begin{align}
 \Big\| \int_s^T t^{-b}\, t^\delta L^\alpha e^{-(t-s) L} dt\Big\| 
                    \le s^{\mu -b}\, T^{2\delta}\, C(\mu, \delta)     \label{ce320}
 \end{align}
 for some finite constant $C(\mu, \delta)$ with $C(0, \delta) \le1$. 
 In particular, for $\mu =0$, there holds
 \begin{align}
 \Big\| \int_s^T t^{-b} t^{\delta}L^\alpha e^{-(t-s) L} dt\Big\| 
                    \le s^{-b}T^{2\delta}      \ \ \text{for} \ \  0 \le b <1,             \label{ce320b} \\
   \Big\| \int_s^T t^{-b} L^\alpha e^{-(t-s) L} dt\Big\| 
                    \le s^{-b}T^{\delta}      
    \ \ \text{for} \ \  0 \le b <1,                                                \label{ce320c}
 \end{align}
 where $\delta = 1 -\alpha$.
       \end{lemma}
       \begin{proof}
  By the spectral theorem it suffices to show that
  \begin{align}
  \int_s^T t^{-b}\, t^\delta \lambda^\alpha e^{-(t-s)\lambda} dt 
                           \le s^{\mu- b}\, T^{2\delta}\  C(\mu, \delta)                           \label{ce321}
  \end{align}
  for some finite function $C(\cdot, \cdot)$ with $C(0, \delta) \le 1$.
  Observe that $ t^{-b} t^{\delta} = t^{\mu-b} t^\delta t^{-\mu} \le s^{\mu-b} T^\delta t^{-\mu}$
  for $s \le t \le T$ because $\mu -b \le 0$ and $\delta \ge 0$. Hence
  \begin{align}
  \int_s^T t^{-b} t^\delta \lambda^{\alpha} e^{-(t-s) \lambda} dt \le s^{\mu -b} 
  T^\delta \int_s^T t^{-\mu} \lambda^\alpha e^{-(t-s)\lambda} dt.               \label{ce322}
  \end{align}
  To estimate the last integral make the change of variables $ t = s +(\sigma/\lambda)$ in the integral to find
  \begin{align*}
  \int_s^T t^{-\mu} \lambda^\alpha e^{-(t-s)\lambda} dt 
  &= \int_0^{(T-s) \lambda} \Big(s +\frac{\sigma}{\lambda}\Big)^{-\mu} 
   \lambda^{\alpha-1} e^{-\sigma} d\sigma \\
   &\le \int_0^{T\lambda} \Big(\frac{\sigma}{\lambda}\Big)^{-\mu} 
                                               \lambda^{\alpha-1}e^{-\sigma} d\sigma \\
  &=  \int_0^{T\lambda}\sigma^{-\mu} \lambda^{-\delta} e^{-\sigma} d\sigma\\
  &= T^\delta (T\lambda)^{-\delta} \int_0^{T\lambda}  \sigma^{-\mu}e^{-\sigma} d\sigma \\
  &\le T^\delta \sup_{\tau > 0}\tau^{-\delta}\int_0^\tau \sigma^{-\mu} e^{-\sigma} d\sigma.
  \end{align*}
 It remains, therefore, only to show that the function 
 \beq
 C(\mu,\delta) \equiv  \sup_{\tau > 0}\tau^{-\delta}\int_0^\tau \sigma^{-\mu} e^{-\sigma} d\sigma
 \eeq
  is finite for the allowed values of $\mu$ and 
 $\delta$ and is at most one at $\mu =0$.  
 Since $\delta \ge 0$ and $\mu <1$
    we have 
   $$
   \limsup_{\tau\rightarrow \infty} 
   \Big( \tau^{-\delta} \int_0^\tau \sigma^{ - \mu} e^{-\sigma} d\sigma\Big)
    \le (\limsup_{\tau\rightarrow \infty} \tau^{-\delta})
                                           \int_0^\infty\sigma^{ - \mu} e^{-\sigma} d\sigma <\infty .
    $$     
 For small $\tau$ we have 
    \begin{align*}
    \tau^{-\delta} \int_0^\tau \sigma^{- \mu} e^{-\sigma} d\sigma 
    \le   \tau^{-\delta} \int_0^\tau \sigma^{ - \mu}  d\sigma   
    =   \tau^{-\delta} \tau^{1- \mu} /(1 -\mu)
    = \tau^{1-\delta- \mu}/(1-\mu),
 \end{align*}
 which is bounded for small $\tau$ because $1- \delta - \mu = \alpha \ge 0$. Thus 
 $C(\mu, \delta) < \infty$. Finally,  if $\mu = 0$ then  $\delta = 1 -\alpha$ and 
 $C(0, \delta) = \sup_{\tau > 0} \Big\{\tau^\alpha \tau^{-1}\int_0^\tau   e^{-\sigma} d\sigma \Big\}$. For $\tau\ge 1$ the expression in braces is increasing in $\alpha$
  and for  $\alpha =1$ is at most one, while for $\tau<1$ it is decreasing
   in $\alpha$ and for $\alpha =0$  is at most one. Therefore  the expression
    in braces is at most one for all $\tau>0$.         
    \end{proof}

\bigskip
\noindent
     \begin{proof}[Proof of Theorem \ref{thmactint}] By \eref{ce210}  and \eref{ce320} we have
     \begin{align*}
 \int_0^T t^{-b} &\Big\|\int_0^t s^{-\mu} L^\alpha e^{-(t-s)L} g(s) ds\Big\|^2 dt  \\
  & \le  (1-\mu)^{\alpha-1}  \int_0^T t^{-b}  t^{\delta}
                                  \int_0^t \Big(s^{-\mu} L^\alpha e^{-(t-s)L} g(s), g(s)\Big) ds dt\\
  &= (1-\mu)^{\alpha-1}  \int_0^T s^{-\mu} \int_s^T t^{\delta- b} 
                                     \Big(L^\alpha e^{-(t-s)L} g(s), g(s)\Big) dt ds \\
 & = (1-\mu)^{\alpha-1}  \int_0^T s^{-\mu}
                                \Big(\Big\{\int_s^T t^{\delta- b} L^\alpha e^{-(t-s)L}dt\Big\} g(s), g(s)\Big) ds\\
 &\le (1-\mu)^{\alpha-1}  \int_0^T s^{-\mu}
                           \Big\|\int_s^T t^{\delta- b} L^\alpha e^{-(t-s)L}dt\Big\| \|g(s)\|^2 ds \\
        &\le (1-\mu)^{\alpha-1}  \int_0^T s^{-\mu} \{s^{\mu-b} T^{2\delta} C(\mu, \delta)\}  \|g(s)\|^2 ds. 
 \end{align*}
 Thus we may take $C_{\alpha, \mu} = (1-\mu)^{\alpha -1} C(\mu, \delta)$ to arrive at 
 \eref{ce303}. Since $C(\mu, \delta) \le 1$ if $\mu  = 0$, \eref{ce303b} holds.
 \end{proof}

\subsubsection{Proof of finite action} 
 \label{secpfa}

     Action estimates for the freely propagated term in \eref{ST25} have been made
    in Lemma \ref{freeesta1} for all $a \in (0,1)$. In this section action estimates will be made
    for the integral term in \eref{ST25} for $1/2\le a <1$.

                  \begin{lemma}\label{actint2a} 
                  Define
                  \beq
                  D = (1-\Delta)^{1/2}.     \notag
                  \eeq
Let $6/5 \le p \le 2$. Define $\gamma \in [0, 1]$ by the condition
\beq
1/2 = p^{-1} - \gamma/3.                                                       \label{ce27a}
\eeq
 Let
$\kappa_p$ denote the norm of $D^{-\gamma}$ as an operator from
 $L^p(M; \L^1\otimes \kf)$ into $L^2(M;\L^1\otimes \kf)$.
By Sobolev this is finite. Let $0 < b < 1$ and let $f:(0, T) \rightarrow L^p(M;\L^1\otimes \kf)$ 
be a measurable function. Then
\begin{align}
\int_0^T t^{-b} \|\int_0^t e^{(t-s)\Delta} f(s) ds \|_{H_1}^2 dt  
\le T^{1-\gamma} \int_0^T s^{-b} \|f(s)\|_p^2 ds\cdot        (e^{2T}\kappa_p^2).    \label{ce28a}
\end{align}
\end{lemma}
\begin{proof} Let $g(s) = e^{-s}D^{-\gamma} f(s)$. 
    Then $\|g(s)\|_2 \le \kappa_p \| f(s)\|_p$. We are going to apply
     Theorem \ref{thmactint}     with $ L = (1 -\Delta) = D^2$, $\mu =0$ and 
     $2 \alpha = 1 +\gamma$. Then $2 \delta = 2- 2\alpha = 1-\gamma$.   Since $f(s) = e^s D^\gamma g(s)$ we have, using   \eref{ce303b} 
       in the fourth line,
  \begin{align*}
 \int_0^T t^{-b} \|\int_0^t e^{(t-s)\Delta} f(s) ds \|_{H_1}^2 dt 
 &=  \int_0^T t^{-b} \|\int_0^t D e^{-(t-s)(1-\Delta)} e^{(t-s)} f(s) ds \|_2^2dt \\
 &=\int_0^T t^{-b} \|e^t \int_0^t D^{1+\gamma} e^{-(t-s)(1-\Delta)} g(s) ds \|_2^2dt \\
 &=\int_0^T t^{-b}e^{2t} \| \int_0^t L^\alpha e^{-(t-s)L} g(s) ds \|_2^2dt \\
& \le T^{2\delta} \int_0^T s^{-b} \|g(s)\|_2^2ds\  e^{2T}\\
& \le T^{2\delta} \int_0^T s^{-b} \|f(s)\|_p^2ds(e^{2T} \kappa_p^2 ).
 \end{align*}
 \end{proof}
 
           \begin{lemma}\label{actint3a} Let $ 1/2 \le a <1$. Define 
\beq
q_a^{-1}= (1/2) - (a/3),  \ \ \ p_a^{-1} = (7/6) - (2a/3), \ \ \ \    r_a^{-1} =1 - (a/3) .\label{ce240a}  
\eeq
Let $C(\cdot) \in \P_T^a$  and define
\beq
 \beta_a = \sup_{0\le s \le T} \|C(s)\|_{q_a}.                                          \label{ce241a}
\eeq
 Then, for $ 0 < s \le T$, 
\begin{align}
\|C(s)^3\|_{p_a}  &\le c^2 \kappa_6 \beta_a^2 \| C(s)\|_{H_1},\ \ \  \ \ 
                                                          p_a^{-1} = 2q_a^{-1} + (1/6)  \label{ce242a}\\
 \|C(s)\cdot \p C(s)\|_{r_a} &\le c \beta_a \| C(s)\|_{H_1}, \ \ \ \ \ \ \ \ \ 
  r_a^{-1} = q_a^{-1} + (1/2) .  \label{ce243a}
 \end{align}
\end{lemma}
              \begin{proof} Since $\|C(s)\|_6 \le \kappa_6 \| C(s)\|_{H_1}$  and 
$\|\p C(s)\|_2 \le \| C(s)\|_{H_1}$, H\"older's inequality
proves \eref{ce242a} and \eref{ce243a} in accordance with the arithmetic shown in the 
second column.
\end{proof}

\begin{remark}{\rm  It may be clarifying  to contrast the critical case $a =1/2$ with the
case $a > 1/2$. The definitions  \eref{ce240a}  give            
 \begin{align*} 
 &q_a = 3, \qquad\qquad p_a = r_a = 6/5  \qquad\qquad\qquad \text{for}\ \ \ a = 1/2 \\
 3 < &q_a < 6, \ \ \  6/5 < p_a < 2,\ \ \  6/5 < r_a <  3/2 \ \ \ \ \text{for}\ \ 1/2 < a <1.
 \end{align*}
 For the two types of terms that appear on the left side in \eref{ce242a} and \eref{ce243a}
 we are going to apply Lemma \ref{actint2a} with $p = p_a$ and $p = r_a$, respectively.
 With $\gamma$  
 defined by \eref{ce27a} and with the help of some arithmetic,  the reader
 can verify that the corresponding values of $\gamma$ yield
 \begin{align}
 1 - \gamma =
 \begin{cases} &2(a -\frac{1}{2})\ \ \ \, \text{if}\  p = p_a \\
      &  (a -\frac{1}{2}) \ \ \ \ \ \text{if}\  p = r_a .
      \end{cases}                                             \label{ce247}
      \end{align}
       }
 \end{remark}

       \begin{lemma}\label{actint4a}  Let $1/2 \le a <1$, $ 0 < b <1$ and $0 < T\le 1$. 
    Suppose that $C(\cdot) \in \P_T^a$.
    Then there are constants $c_5$ and $ c_6$ independent of $a, b,T$  and 
   $C( \cdot)$ such that
 \begin{align}
 \int_0^T t^{-b} \|\int_0^t e^{(t-s)\Delta} C(s)^3 ds \|_{H_1}^2 dt         
 &\le  c_5 T^{2a - 1} \beta_a^4  \int_0^T s^{-b} \| C(s)\|_{H_1}^2 ds,            \label{ce251a}\\
 \int_0^T t^{-b} \|\int_0^t e^{(t-s)\Delta} C(s)\cdot\p C(s) ds \|_{H_1}^2 dt 
 &\le c_6  T^{a - \frac{1}{2}} \beta_a^2\int_0^T s^{-b} \| C(s)\|_{H_1}^2 ds.    \label{ce252a}
 \end{align}
             \end{lemma} 
             \begin{proof}  We are going to use Lemma \ref{actint2a} for each
     inequality, with the appropriate choice of $p$.
     
     Take $ p = p_a$ in \eref{ce27a} and choose $f(s) = C(s)^3$ in \eref{ce28a}.     
     In view of       \eref{ce242a} and \eref{ce247} we find
      \begin{align*}
   \int_0^T t^{-b} \|\int_0^t &e^{(t-s)\Delta} C(s)^3 ds \|_{H_1}^2 dt  \\
   &\le  T^{2a-1} \int_0^T s^{-b} \| C(s)^3\|_{p_a}^2 ds\   (\kappa_{p_a}e^T)^2      \\
   &\le  T^{2a-1}  \int_0^T s^{-b} \| C(s)\|_{H_1}^2 ds\  
     (c^2\kappa_6 \beta_a^2)^2         (\kappa_{p_a}e^T)^2  ,                                 
 \end{align*} 
 which is \eref{ce251a} upon taking $c_5 = c^4 \kappa_6^2 \kappa_{p_a}^2 e^2$, since $T\le 1$.

 Take  $ p = r_a$ in \eref{ce27a} and choose $f(s) = C(s)\cdot \p C(s)$ in \eref{ce28a}.
 In view of  \eref{ce243a} and \eref{ce247} we find
 \begin{align*}
 \int_0^T t^{-b} \|\int_0^t &e^{(t-s)\Delta} C(s)\cdot\p C(s) ds \|_{H_1}^2 dt       \\
 &\le  T^{a -(1/2)} \int_0^T s^{-b} \| C(s)\cdot \p C(s)\|_{r_a}^2 ds\ 
                                                              \cdot (\kappa_{r_a}e^T)^2                   \\
 &\le T^{a -(1/2)} \int_0^T s^{-b} \| C(s)\|_{H_1}^2 ds\cdot  (c\beta_a)^2 (\kappa_{r_a}e^T)^2,                                       
\end{align*}   
which is \eref{ce252a} upon taking $c_6 = c^2 \kappa_{r_a}^2 e^2$. $c_5$ and $c_6$ can be taken independent of $a$ because $\kappa_p$ is bounded for $6/5 \le p \le 2$.
\end{proof}

       \begin{lemma}\label{actint5a} Let $1/2 \le a <1$, $0 < b < 1$ and $0 <T \le 1$.  
There is a constant $c_7$ independent of $a, b, T$ and $C(\cdot)$ such that
\begin{align}
\int_0^T t^{-b} \|w(t)\|_{H_1}^2 dt 
\le c_7 \Big(T^{a -\frac{1}{2}} \beta_a^2\Big) 
 \Big(T^{a -\frac{1}{2}} \beta_a^2 + 1\Big) \int_0^T s^{-b} \| C(s)\|_{H_1}^2 ds ,         \label{ce248a}
\end{align}
wherein $w(t)$ is defined by \eref{ST319}.
\end{lemma}
 \begin{proof}
 In view of the definition \eref{ST319} we see that $\int_0^T t^{-b} \| w(t)\|_{H_1}^2 dt$
  is at most twice the sum of the left hand sides  of \eref{ce251a} and \eref{ce252a}.
  Take $c_7 = 2 \max (c_5, c_6)$ and add  \eref{ce251a} and \eref{ce252a} 
  to arrive at \eref{ce248a}.
\end{proof}

\bigskip
\noindent
  \begin{proof}[Proof of Theorem \ref{thmfa}] 
  The integral
   equation \eref{ST25} is given by
$C(t) = e^{t\Delta} C_0 + w(t)$  in view of the definition \eref{ST319} of $w$. Hence,  
for any number $b \in (0,1)$, we have
\begin{align}
\Big(\int_0^T& t^{-b} \|C(t)\|_{H_1}^2 dt \Big)^{1/2}     \notag\\
&\le  \Big(\int_0^T t^{-b} \| e^{t\Delta} C_0\|_{H_1}^2 dt \Big)^{1/2} 
+ \Big( \int_0^T t^{-b} \| w(t)\|_{H_1}^2 dt \Big)^{1/2}            \notag\\
&\le  \Big(\int_0^T t^{-b} \| e^{t\Delta} C_0\|_{H_1}^2 dt \Big)^{1/2} 
+ \mu_a \Big(\int_0^T s^{-b} \|C(s)\|_{H_1}^2 dt \Big)^{1/2} ,    \label{ce275a}
\end{align}
where, by \eref{ce248a},  we can take
\beq
\mu_a = \sqrt{c_7 \Big(T^{a -\frac{1}{2}} \beta_a^2\Big) 
 \Big(T^{a -\frac{1}{2}} \beta_a^2 + 1\Big)}.                           \label{ce276a}
 \eeq
The rest of the proof hinges on whether we can take $\mu_a <1$ by choosing $T$ and/or $C_0$
suitably. Here the proof diverges into two cases. Either $a > 1/2$, in which case 
choosing  $T$ small makes $\mu_a$ small because $T$ occurs with a strictly positive power
in $\mu_a$. Or else $a = 1/2$, in which case 
\beq
\mu_{1/2}^2 = c_7 \beta_{1/2}^2 (\beta_{1/2}^2 +1).         \label{ce277a}
\eeq
In this case we will show that choosing both $\|C_0\|_{H_{1/2}}$ small  and $T$ small ensures
that $\mu_{1/2} < 1/2$.

   Let us carry out the details for the case $a > 1/2$ first. We are considering a particular
   solution of the integral equation \eref{ST25}. Clearly $\beta_a$, defined in \eref{ce241a},
   decreases as $T$ decreases and is finite for all $T$ by Sobolev's inequality. Consequently
   $\mu_a$ decreases to zero as $ T \downarrow 0$ because $a -1/2 >0$. 
   Choose $T$ so small that $\mu_a \le 1/2$. We would like to subtract the last term
   on the right of \eref{ce275a} from the left side to obtain a bound on the $b$ action for 
   $b =a$. However we do not know that the right side is finite for $b =a$ since this
    is what we are trying to prove. Let $b <a$. Since $C(\cdot)$ lies in the path space
    $\P_T^a$ we know that $\|C(s)\|_{H_1}^2 = o(s^{a-1})$ as $s\downarrow 0$ 
    by \eref{ST413a}.
        Hence $s^{-b} \|C(s)\|_{H_1}^2 = o(s^{a-b -1})$ as $s \downarrow 0$. 
        Since $a - b - 1 > -1$ it follows that $\int_0^Ts^{-b} \|C(s)\|_{H_1}^2 ds < \infty$.
We can therefore subtract the last term in \eref{ce275a} from the  left side to find, after squaring,
\begin{align}
(1/4)\int_0^T& t^{-b} \|C(t)\|_{H_1}^2 dt   \le \int_0^T t^{-b} \| e^{t\Delta} C_0\|_{H_1}^2 dt 
\  \  \text{for} \  \ 0 < b <a.
\end{align}
As $b \uparrow a$ we have $ t^{-b}\uparrow t^{-a}$ on $(0, 1]$. Therefore  the monotone convergence theorem now shows that
\begin{align}
(1/4)\int_0^T& t^{-a} \|C(t)\|_{H_1}^2 dt   \le \int_0^T t^{-a} \| e^{t\Delta} C_0\|_{H_1}^2 dt 
\end{align}
The right side is finite by \eref{ST450a}. This completes the proof of Theorem \ref{thmfa} in case
$ a > 1/2$.

Suppose now that $a = 1/2$. By assumption, $C(\cdot)$ is a continuous function into
$H_{1/2}$ and therefore, by Sobolev, a continuous function into $L^3$. 
Hence, if $\|C_0\|_{H_{1/2}}$, and therefore $\|C_0\|_3$, are small
 then $\|C(s)\|_3$ remains small for a short time.   Thus we can choose $\|C_0\|_3$ small
  and then choose $T > 0$ so small that $\beta_{1/2} \equiv \sup \|C(s)\|_3$
 is small  enough to ensure, by virtue of  \eref{ce277a},  that $\mu_{1/2} \le 1/2$. 
 The remaining details of the proof are the same as for the case $a >1/2$. 
 This completes the proof of Theorem \ref{thmfa} as well as Parts {\it ii}) and {\it iii})
 of Theorem \ref{thmmild}.
\end{proof}

\subsection{Mild solutions are strong solutions}\label{secstr}

We wish to show that a function $C(\cdot) \in\P_T^a$  is a solution to the integral
 equation \eref{ST25} if and only if it  is a strong solution to the differential
  equation \eref{aymh}. Combined with Theorem \ref{thmfa}, this will complete the proof
  of Theorem \ref{thmaug}.
 
\begin{theorem}\label{mildstr} Let $1/2 \le a < 1$. Suppose that $C(\cdot)$ is a mild solution
 on $[0, T)$ lying in $\P_T^a$.
Then $ C(\cdot)$ is a strong solution on $(0,T)$. Moreover, 
$C(\cdot) \in C^\infty((0, T) \times M; \L^1\otimes \kf)$
and  satisfies, for $0 < t <T$, the Neumann, resp. Dirichlet boundary conditions
\begin{align}
&(N)\  C(t)_{norm}=0,\  (dC(t))_{norm} =0, \  (B_{C(t)})_{norm} =0. \label{ST140}\\
&(D)\  C(t)_{tan} =0, \ \ \ (dC(t))_{tan} =0,\  (B_{C(t)})_{tan} =0,\ 
                   (d^*C(t))_{tan}=0.                                                   \label{ST141}
\end{align}
Two strong solutions lying in $\P_T^a$ are equal.
\end{theorem}
        \begin{proof}
      Let $ 0 < \tau < t < T$. Suppose that $C(\cdot)$ is a solution of the integral 
      equation  \eref{ST25} lying in $\P_T^a$. 
      Define $f(s) = X(C(s))$. The integral equation  \eref{ST25}
      may be rewritten as 
      \begin{align}
      C(t) &= e^{t\Delta}C_0 + \int_0^t e^{(t-s) \Delta} f(s) ds \notag \\
           &  = e^{(t-\tau) \Delta}\Big( e^{\tau \Delta} C_0 
 + \int_0^\tau e ^{(\tau -s) \Delta} f(s) ds \Big) + \int_\tau^ t e^{(t-s)\Delta} f(s) ds      \notag\\
  & =e^{(t-\tau) \Delta} C(\tau) +   \int_\tau^ t e^{(t-s)\Delta} f(s) ds.      \label{ST142}
           \end{align}
  Thus $ C(t)$ is a mild solution over the interval $[\tau, T]$ with initial value
   $C(t)|_{t = \tau} = C(\tau)$. Now $C(\cdot)$ is a continuous function into $H_1$
    over $[\tau, T]$ as well as into $H_{a}$ and therefore lies in the path space
    $\P_{[\tau, T]}^a$, defined as in  Notation \ref{notpatha}  but with $\tau$ as the origin. 
    The corresponding path space $\hat \P_{[\tau, T]}$ used in \cite{CG1}
     is contained in  $ \P_{[\tau, T]}^a$.  Since $C(\tau) \in H_1$, \cite[Theorem 7.3]{CG1}
     assures that there exists a mild solution  $\hat C(\cdot)$ over some interval 
     $[\tau, \tau+\epsilon]$ with initial value $C(\tau)$ and which lies in
      $\hat \P_{[\tau, \tau + \epsilon]}$. 
     We may assume without loss of generality that $\tau + \epsilon <T$. But
   $\hat \P_{[\tau, \tau+\epsilon]} \subset \P_{[\tau, \tau+\epsilon]}^a $ and  mild solutions are unique 
   in  $\P_{[\tau, \tau+\epsilon]}^a$. Hence $\hat C(t) = C(t)$ for $t \in [\tau, \tau +\epsilon]$.
   Now \cite[Theorem 7.3]{CG1} also assures that $\hat C$ is a strong solution
     over $(\tau, \tau +\epsilon)$. Therefore $C(\cdot)$ is a strong solution
      over  $(\tau, \tau +\epsilon)$. Since $\tau$ is arbitrary in $(0,T)$, $C(\cdot)$ is a strong solution over $(0, T)$. The same argument, using again   \cite[Theorem 7.3]{CG1}, 
      shows that $C(\cdot) \in C^\infty((0,T)\times M; \L^1\otimes \kf) $.
             Moreover, by \cite[Corollary 7.10]{CG1}, $\hat C(t)$ 
         satisfies
      the boundary conditions  \eref{ST140}, resp. \eref{ST141}   over $(\tau, \tau + \epsilon)$. Therefore $C(\cdot) $ does also.

      Conversely, suppose that $C(\cdot)$ is a strong solution to the differential equation
      \eref{aymh} and which satisfies the conditions a) and b) of Theorem \ref{thmaug}. 
      That is to say, $C(\cdot) \in \P_T^a$.    If $0 < \tau < T$ then $C(\cdot)$ is a continuous
      function into $H_1$ over $[\tau, T]$. Consequently, as shown
       in \cite[Proof of Theorem 2.13]{CG1}, $C(\cdot)$ is a solution to the integral equation
       \eref{ST142} for $\tau \le t \le T$. Since, for fixed $t >0$, the integral
       equation \eref{ST142} holds
       for  all $\tau \in (0, t)$, we can let $\tau \downarrow 0$ and find that \eref{ST142} holds also
       for $\tau =0$ by observing first, that $ C(\cdot)$ is continuous into $L^2$ (in fact into $H_a$)
       over $[0, T]$, allowing the strong limit in the first term, and second, that the estimates
       on $f(s)$ made in Section \ref{secce} on the basis of the hypotheses  of
        Theorem \ref{thmaug}  allow us to take the limit in the integral term  in \eref{ST142}.
       Hence a strong solution to the differential equation  \eref{aymh}  which lies
       in $\P_T^a$ 
       is a solution to the integral equation \eref{ST25}. Uniqueness for such
        strong solutions now follows.
        
           This completes the existence and uniqueness portion of  Theorem \ref{thmaug}.
    Since a strong solution is also a mild solution we can apply Theorem \ref{thmfa} to 
    deduce the remaining, finite action, assertions of  Theorem \ref{thmaug}.
\end{proof}

\section{Initial behavior of solutions to the augmented equation}  \label{secib}
We are  going to derive energy estimates for the first and second order spatial
derivatives of a solution $C(\cdot)$ to the augmented Yang-Mills heat equation \eref{aymh}
and then use these to derive additional  bounds via the method of Neumann domination.
These estimates will be used in Section \ref{secconvgp} 
to establish the properties of the conversion group which are needed
 to recover the desired solution
$A(\cdot)$ of \eref{str5} from $C(\cdot)$.

In this section we will take $M$ to be either all of $\R^3$ or the closure of a bounded,
 convex, open subset of $\R^3$ with smooth boundary.

The main technique  in the next few subsections will be based on the  
Gaffney-Friedrichs-Sobolev inequality, which                     
  asserts, for our  convex subset $M$ of $\R^3$,  that for any integer
   $p \ge 1$ and  any  $\kf$ valued p-form $\w$ (satisfying appropriate boundary conditions)
there holds
\beq
\|\w\|_6^2  \le \kappa^2\Big\{ \|d_C^* \w\|_2^2 + \| d_C \w\|_2^2 
                                   + \lambda(B_C) \|\w\|_2^2\Big\}            \label{gaf50}
\eeq
for any $\kf$ valued connection form $C$ on $M$ with curvature $B_C$.
Here  
we have written  
\beq
\lambda(B) =1 + \gamma \|B\|_2^4,           \label{gaf70} 
\eeq
where $\gamma \equiv  (27/4)\kappa^6 c^4$ is a constant depending only on a 
 Sobolev constant  $\kappa$ for $M$, which can be $\R^3$,  and the commutator bound  $c$  defined in Section \ref{secce}.
See \cite[Theorem 2.17, Remark 2.18 and  Equ.(4.31)]{CG1} for the derivation of the  
 inequality    \eref{gaf50}.  If $M$ is not convex then the inequality  \eref{gaf50} still holds,
 but with different constants in  \eref{gaf70} provided that the second fundamental
 form of $\p M$ is bounded below.
 
 Gaffney-Friedrichs inequalities \cite{Fr,Mt01,ME56,Mo56} 
 give information about the gradient of a form in terms of
 the exterior derivative and co-derivative of the form. The use of these is essential for us
 because the differential equations are posed in terms of the gauge covariant exterior
 derivative $d_C$ and its co-derivative $d_C^*$, whereas Sobolev inequalities require
 information about gradients.

 For a $\kf$-valued 0-form $\phi$ on $M$ one has the simple  gauge invariant 
 Sobolev inequality 
 \beq
 \|\phi\|_6^2 \le \kappa_6^2 \Big( \| d_C \phi\|_2^2 + \|\phi\|_2^2\Big)       \label{gaf01}
 \eeq
 where $\kappa_6 \le \kappa$. 
 This is valid independently of boundary conditions on $\phi$.
 It is just a consequence of Kato's inequality, $|grad\ |\phi(x)|\ | \le |d_{C(x)} \phi(x)|$.
 If $M = \R^3$ then the summand $\|\phi\|_2^2$ is not needed. 
 See, e.g., \cite[Notation 2.16]{CG1} for further discussion.

\subsection{Identities} \label{secid}

Suppose that $C(\cdot)$ is a solution to \eref{aymh} over some interval.
Define 
\beq
\phi(s) = d^*C(s).        \label{vs20}
\eeq
 We are   going to derive  energy estimates
for $\phi$ and $B_C$ and their gauge covariant derivatives. Similar energy
 estimates have been made for the curvature of $A(\cdot)$ and its covariant derivatives
 in  \cite{CG1} and \cite{CG2}.  
 The augmented equation \eref{aymh} is a little more complicated
than the Yang-Mills heat equation \eref{str5},  which was the basis for the
 energy estimates in \cite{CG1} and \cite{CG2}. This reflects itself in slightly
 more complicated   energy estimates for $\phi$ and $B_C$.

 \begin{lemma}\label{lemids4p}$($Pointwise identities$)$ Suppose that $C(\cdot)$
  is a smooth solution  to \eref{aymh} over $(0, T)$. Then
 \begin{align}
 d\phi(s)/ds &= d_C^* C' - [C\lrc C'],     \label{vs520} \\
&= - d_C^* d_C \phi - [C\lrc C']\ \ \ \ \ \ \ \ \ \    \text{and} \label{vs510}\\
&= \Delta \phi + [C \lrc (d_C^* B_C - d \phi)].                   \label{vs510.1}
\end{align}
Further, 
\begin{align}
dB_C(s)/ds & = \sum_{j=1}^3 (\n_j^C)^2 B_C + B_C\# B_C - [B_C, \phi], \label{vs511}
\end{align}
where $\#$ denotes a pointwise product coming from the Bochner-Weitzenboch formula.
\end{lemma}
  \begin{proof}
  The definition \eref{vs20}  
  gives
   \begin{align*}
 d\phi/ds  =(d/ds) d^*C = d^* C' = d_C^* C' - [C\lrc C'],
\end{align*}  
which proves \eref{vs520}. In view of the differential equation \eref{aymh} we have
$d_C^* C' = - d_C^* (d_C^* B_C + d_C \phi) = -d_C^* d_C\phi$ by Bianchi's identity.
This proves \eref{vs510}. Expand
 $d_C^*d_C \phi = d^*(d\phi + [C,\phi]) + [C\lrc d_C \phi] 
 = d^*d \phi +d^*[C, \phi]  + [C\lrc d_C \phi]$ and use the identity 
 $d^*[C,\phi] = [d^*C, \phi] + [C\lrc d\phi] =  [C\lrc d\phi]$. The last equality 
 follows from $[d^*C, \phi] = [\phi, \phi] = 0$. 
 Thus 
 $d_C^*d_C \phi = d^*d \phi + [C\lrc (d\phi +d_C \phi)]$.
 But $-[C\lrc C'] = [C\lrc (d_C^*B_C + d_C \phi)]$ by \eref{aymh}. Combine the last two
 equalities with \eref{vs510} to find \eref{vs510.1}.

 Over our flat manifold $M$ the   Bochner-Weitzenboch
formula asserts that, for any $\kf$-valued p-form $\w$,
\beq
-(d_C d_C^*  + d_C^* d_C) \w = \sum_{j=1}^3 (\n_j^C)^2 \w + B_C\# \w   \label{vs516}
\eeq
 for some pointwise product $\#$.  Since $d_C B_C =0$ by the Bianchi identity, we have
 \begin{align*}
 (d/ds)B_C(s)  &= d_C C'(s) \\
 & = -d_C\Big( d_C^* B_C + d_C \phi\Big)\\
 & =  -d_C d_C^* B_C - [B_C, \phi]  \\
 &= -(d_C d_C^* + d_C^*d_C)B_C - [B_C, \phi].
 \end{align*}
 Insert \eref{vs516} with $\w = B_C$ to  arrive at \eref{vs511}. 
\end{proof}

        \begin{lemma}\label{lemids4} $($Integral identities$)$ Suppose that $C(\cdot)$
  is a smooth solution  to \eref{aymh} over $(0, T)$. Then
\begin{align}
\frac{d}{ds} \Big(  \|B_C(s)\|_2^2 + \| \phi(s)\|_2^2 \Big)  + 2 \| C'(s)\|_2^2
           &= - 2(C', [C,\phi])     \label{vs522} 
  \end{align}
  and
  \begin{align}
 \frac{d}{ds} \|C'(s)\|_2^2 + 2\Big\{ \| d_{C(s)}&C'(s)\|_2^2 + \| d_{C(s)}^* C'(s)\|_2^2 \Big\}\notag\\
&= -2(B_C, [C'\wedge C']) + 2([C\lrc C'], d_C^* C') .      \label{vs523}
\end{align}
\end{lemma}
            \begin{proof} 
 The identity \eref{vs522} follows from the computation
\begin{align*}
 (1/2) \frac{d}{ds} \Big(  \|B_C(s)\|_2^2 + \| \phi(s)\|_2^2 \Big) 
 &=(d_C C', B_C) + ( \phi', \phi) \\
 &=(C', d_C^*B_C) +(d_C^* C' - [C\lrc C'], \phi)\\
 &= (C', d_C^*B_C) +(C', d_C\phi) -(C', [C,\phi]) \\
 & = - \| C'\|_2^2  -(C', [C,\phi]).
 \end{align*}
 To prove  \eref{vs523} observe that 
    \begin{align*}
   - C'' &= (d/ds)(d_C^* B_C + d_C\phi)\\
   &=\Big\{ d_C^* B_C' + [C'\lrc B_C]\Big\} + \Big\{ d_C\phi' + [C',\phi] \Big\}\\
   &=  \Big\{ d_C^*d_C C' + [C'\lrc B_C]\Big\} 
                     + \Big\{ d_Cd_C^*C' - d_C[C\lrc C'] + [C',\phi] \Big\}.
\end{align*}
Hence
\begin{align*}
(1/2) &(d/ds) \| C'(s)\|_2^2  =(C'', C') \\
&= -\Big( \Big\{ d_C^*d_C C' + [C'\lrc B_C]\Big\}  
              + \Big\{ d_Cd_C^*C' - d_C[C\lrc C'] + [C',\phi] \Big\}, C'\Big) \\
 &= -\| d_C C'\|_2^2 -([C'\lrc B_C], C') - \| d_C^* C'\|_2^2  + ( [C\lrc C'], d_C^* C')
      - ( [C', \phi], C') \\
 &= -\| d_C C'\|_2^2 - \| d_C^* C'\|_2^2 -( B_C, [C'\wedge C']) + ( [C\lrc C'], d_C^* C')
\end{align*}
because $ ( [C', \phi], C') = (\phi, [C'\lrc C']) =0$. This proves \eref{vs523}.
 \end{proof}

\begin{remark}{\rm (Strategy)  Typically, parabolic equations lead to energy decay via identities
 such as \eref{vs522} and \eref{vs523} when the right hand sides are small or easily controllable.
 However the non-linearities  of the augmented equation  \eref{aymh} produce 
 strong terms on the right side
  with uncontrolled sign. We will   balance out some of the strong
   terms on the right against half of the positive terms on the left. For this we will need $L^6$
   bounds on some factors to estimate the non-linear terms on the right. 
   These in turn will be obtained by applying the Gaffney-Friedrichs-Sobolev inequality
  to the next higher derivative.
  }
  \end{remark}

\subsection{Differential inequalities and initial behavior} \label{secdi}

The identities of the preceding subsection give the following differential inequalities
with the help of the Gaffney-Friedrichs-Sobolev inequality. At the end of this subsection
it will be shown how these differential inequalities give information about the initial behavior.

\begin{theorem}\label{diff-ineq} Suppose that $C(\cdot)$ is a strong solution to \eref{aymh} on 
some interval. There are constants $a_j,b_j$ depending only on  Sobolev constants
and the commutator norm $c$ such that
\begin{align}
\frac{d}{ds} \Big( \|B_C(s)\|_2^2 + \| \phi(s)\|_2^2 \Big) + \| C'(s)\|_2^2 
\le  \alpha(s)    \| \phi(s)\|_2^2,\ \ \ \ \ \ Order 1  \label{vs529}
\end{align}
and 
\begin{align}
\frac{d}{ds} \| C'(s)\|_2^2 +  \Big(\|d_C^* C'(s)\|_2^2 + \|d_C C'(s)\|_2^2\Big)  
    \le  \beta(s)  \| C'(s)\|_2^2,\   Order 2,    \label{vs530}
\end{align}
where
\begin{align}
\alpha(s) &= a_1 + a_2  \|C(s)\|_6^4  \qquad  \ \ \ \ \ \ \ \ \ \ \ \ \ \ \text{and}         \label{vs529.2} \\
\beta(s)  &= b_1 + b_2  \| C(s)\|_6^4  + b_3 \|B_C(s)\|_2^4 .     \label{vs530.2}
\end{align}
The constants  are given explicitly in \eref{vs529.3} and \eref{vs530.3}.
\end{theorem}

The proof  depends on the following lemmas, in which it is assumed that $C(\cdot)$ is a strong solution to \eref{aymh}.

\begin{lemma} \label{lemdi1}$($Order 1$)$ There are constants $a_1, a_2$ such that at
 each time $s$ there holds
\begin{align}
2 \| \, [C,\phi]\, \|_2^2 \le  \alpha(s) \| \phi\|_2^2 +(1/2) \| C'\|_2^2           \label{vs528}
\end{align}
with $\alpha(s)$ given by \eref{vs529.2}.
\end{lemma}
             \begin{proof} Since $(d_C^*)^2 B_C = 0$, it follows that $ d_C^*B_C$ and  $d_C \phi$ are mutually orthogonal in $L^2$. Consequently
\beq
\|C'\|_2^2 = \| d_C^*B_C\|_2^2 + \|d_C \phi\|_2^2.        \label{vs561.1}
\eeq
 Hence $\|d_C \phi\|_2^2  \le \| C'(s)\|_2^2$.                         
 Sobolev's inequality \eref{gaf01} then gives, at each time $s$,
 \begin{align}
 (\kappa_6^{-1}\| \phi\|_6)^2    \le \|C'\|_2^2  + \| \phi\|_2^2.  \label{vs540}
 \end{align} 
Therefore 
\begin{align*}
2\| \, [ C, \phi]\, \|_2^2 &\le 2c^2 \|C\|_6^2 \| \phi\|_3^2 \\
&\le 2c^2 \| C\|_6^2 \| \phi\|_2\|\phi\|_6 =\Big(2\kappa_6 c^2 \| C\|_6^2\|\phi\|_2\Big) \Big( \kappa_6^{-1} \|\phi\|_6 \Big)\\
&\le (1/2) (2\kappa_6 c^2 \| C\|_6^2 \| \phi\|_2)^2 + (1/2) (\kappa_6^{-1}\| \phi\|_6)^2\\
&\le \Big( 2 \kappa_6^2 c^4 \| C\|_6^4 \Big) \| \phi\|_2^2 
                + (1/2) (\| C'\|_2^2 + \|\phi\|_2^2),
\end{align*}
which is \eref{vs528} with 
\beq
 a_1 = 1/2, \ \ \ a_2 =2\kappa_6^2 c^4.      \label{vs529.3}
 \eeq
\end{proof}

\begin{lemma}\label{lemdi2} $($Order 2$)$
\begin{align}
2\Big| ( B_C, [C'\wedge C'])\Big| 
              & \le (1/2) \Big\{ \| d_C^* C'\|_2^2 + \| d_CC'\|_2^2\Big\} \notag\\
 &+ (1/2) \Big\{ \lambda(B_C) + (3\kappa^2)^3 c^4 \|B_C\|_2^4 \Big\} \|C'\|_2^2 . \label{vs610}
 \end{align}
\end{lemma}
     \begin{proof} By the interpolation $ \|f\|_4 \le \|f\|_2^{1/4} \|f\|_6^{3/4} $ we have, 
     for any $\eta > 0$,       
      \begin{align}
 2\Big|(B_C, [C'\wedge C'])\Big| &\le 2c \| B_C\|_2 \|C'\|_4^2   \notag \\
 &\le  2c \| B_C\|_2 \|C'\|_2^{1/2} \|C'\|_6^{3/2} 
 =\Big(2\eta c \| B_C\|_2 \| C'\|_2^{1/2}\Big) \Big(\eta^{-1} \|C'\|_6^{3/2}\Big) \notag\\
 &\le (1/4)  \Big(2\eta c \| B_C\|_2 \|C'\|_2^{1/2}\Big)^4 + 
    (3/4) \Big(\eta^{-1}\|C'\|_6^{3/2}\Big)^{4/3} \notag\\
 & =(1/4)  \Big(2\eta c \| B_C\|_2\Big)^4 \|C'\|_2^2 
                 + (3/4)\eta^{-4/3} \|C'\|_6^2                        \notag\\
 & =\Big((1/2) (3 \kappa^2)^3 c^4 \|B_C\|_2^4\Big) \|C'\|_2^2 + \frac{1}{2 \kappa^2} \|C'\|_6^2, \label{vs611}
\end{align}
wherein we have chosen $(2\eta)^4 =  2(3\kappa^2)^3$, which makes
 $(3/4) \eta^{-4/3} =  1/(2\kappa^2)$.
By the Gaffney-Friedrichs-Sobolev inequality \eref{gaf50} with $\w = C'(s)$ we have
\beq
\kappa^{-2}\|C'(s)\|_6^2 \le  \Big\{ \| d_C^* C'\|_2^2 + \| d_CC'\|_2^2 
                + \lambda(B_C) \| C'\|_2^2 \Big\}.                   \label{vs612}
\eeq
Insert \eref{vs612} into the last term in \eref{vs611} to arrive at \eref{vs610}.
\end{proof}

         \begin{lemma}\label{lemdi3} $($Order 2$)$
\begin{align}
2\Big| ( [C\lrc C'], d_C^* C') \Big| &\le (1/2) \Big\{ \| d_C^* C'\|_2^2 
            + \| d_CC'\|_2^2\Big\}  \notag \\
&+\Big\{ \lambda(B_C) /4 + (4\kappa c^2)^2  \|C\|_6^4 \Big\} \| C'\|_2^2.      \label{vs620}
\end{align}
\end{lemma}
        \begin{proof} H\"older's inequality and the interpolation
         $\|f\|_3 \le \|f\|_2^{1/2} \|f\|_6^{1/2}$         give
\begin{align*}
4 \|\, [C\lrc C']\, \|_2^2 &\le 4c^2\|C\|_6^2  \|C'\|_3^2  \notag\\
    & \le 4c^2\|C\|_6^2      \|C'\|_2 \|C'\|_6 
    =\Big( 4\kappa c^2  \|C\|_6^2 \|C'\|_2 \Big)\Big(\kappa^{-1} \|C'\|_6\Big)\\
 &\le  \Big( 4\kappa c^2  \|C\|_6^2 \|C'\|_2 \Big)^2 +\frac{1}{4\kappa^2}\|C'\|_6^2.
 \end{align*}       
 Hence,
 \begin{align}       
 2\Big| ( [C\lrc C'],& d_C^* C') \Big| \le    
  (1/4) \| d_C^* C'\|_2^2 
              + 4 \|\, [C\lrc C']\, \|_2^2                     \notag     \\
        &\le  (1/4) \| d_C^* C'\|_2^2 + \Big( 4\kappa c^2  \|C\|_6^2 \|C'\|_2 \Big)^2 +\frac{1}{4\kappa^2}\|C'\|_6^2. \label{vs621}
\end{align}   
Insert the Gaffney-Friedrichs-Sobolev inequality \eref{vs612} into the last term
 to arrive at \eref{vs620}.    
 \end{proof}

\bigskip
\noindent 
 \begin{proof}[Proof of Theorem \ref{diff-ineq}]
 To prove \eref{vs529} use the estimate in \eref{vs528}     to find
 \begin{align*}
\Big| 2(C' , [C,\phi]) \Big|  &\le (1/2) \| C'\|_2^2 + 2\|\, [C,\phi]\, \|_2^2 \\
&\le \| C'\|_2^2 + \alpha(s) \| \phi\|_2^2. 
\end{align*}
Estimate the right side of the identity \eref{vs522}  by this bound and then cancel
 the term $\|C'\|_2^2$ with part of the left side of \eref{vs522} to arrive at \eref{vs529}.

To prove \eref{vs530} add the inequalities \eref{vs610} and \eref{vs620} to find
\begin{align}
\Big| 2( B_C, [C'\wedge C'])&  + 2( [C\lrc C'], d_C^* C') \Big|    \notag\\
&\le \Big\{ \| d_C^* C'\|_2^2  + \| d_CC'\|_2^2\Big\}
+ \beta(s) \|C'(s)\|_2^2,       \label{vs630}
\end{align}
where
\begin{align}
\beta(s) &=  (1/2) \Big\{ \lambda(B_C) + (3\kappa^2)^3 c^4 \|B_C\|_2^4 \Big\} +
\Big\{ \lambda(B_C) /4 + (4\kappa c^2)^2  \|C\|_6^4 \Big\}      \notag  \\
&=(3/4) + \Big( (\gamma/2)+ 27 \kappa^6 c^4 + (\gamma/4)\Big) \|B_C\|_2^4
                               +  (4\kappa c^2)^2  \|C\|_6^4. \label{vs631}
\end{align}

We can now estimate  the right side of \eref{vs523} using  \eref{vs630}.
We see that  there is partial cancelation of the expression in braces in \eref{vs523}, leaving
  exactly \eref{vs530} with $\beta(s)$ given by \eref{vs631}. From this we can compute the coefficients in \eref{vs530.2} to be
\beq
  b_1 = 3/4, \ \   b_2= (4\kappa c^2)^2 ,\ \   b_3 = (\kappa^6c^4)b_0 \ \ \   \label{vs530.3}
  \eeq 
 with $b_0 =  19\cdot27/16$. 
  \end{proof}

\bigskip
\noindent

 The differential inequalities of 
 Theorem \ref{diff-ineq}  will yield information about the initial behavior
 of $C(t)$ and its derivatives with the help of the next elementary lemma. We will apply it
 several times  in the following sections.

 \begin{lemma}\label{lem50}$($Initial behavior from differential inequalities$)$
  Suppose that $f, g, h$ are nonnegative continuous 
  functions on $(0, t]$ and that $f$ is differentiable. Suppose also that 
 \beq
 (d/ds) f(s) + g(s) \le h(s)     ,\ \ \ 0<s \le t        .                                  \label{vs90}
 \eeq
 Let $-\infty < b <1$ and assume that 
 \begin{align}
 \int_0^t s^{-b} f(s) ds < \infty.                                                          \label{vs90f}
 \end{align}
 Then 
 \beq
 t^{1-b} f(t) +\int_0^t s^{(1-b)}  g(s)ds \le \int_0^t s^{(1-b)}  h(s) ds 
                   +(1-b) \int_0^t s^{-b} f(s)ds.                                                \label{vs91}
 \eeq
 If equality holds in \eref{vs90} then equality holds in \eref{vs91}.
 
 Suppose, instead of \eref{vs90}, that $ f'\le 0$. Then  
\beq
t^{1-b} f(t) +\int_0^t s^{1-b} (-f'(s)) ds = (1-b) \int_0^t s^{-b} f(s) ds     \label{fa7}
\eeq
and, for   $b \in [0,1)$,
\beq
(1-b) \int_0^t f(s)^q ds \le \Big\{(1-b) \int_0^t s^{-b} f(s) ds\Big\}^q\ \ 
                                     \text{if}                     \ q = (1-b)^{-1}.        \label{fa8}
\eeq
In particular,    
\beq
(1/2) \int_0^t f(s)^2 ds \le \Big\{(1/2) \int_0^t s^{-1/2} f(s) ds\Big\}^2.     \label{fa8.5}
\eeq
 \end{lemma}
                  \begin{proof}                  
                  Multiply \eref{vs90} by $s^{-b}$ to find
    \begin{align*}
   (d/ds) s^{-b} f(s)  +b s^{-b-1} f(s) + s^{-b} g(s)   \le s^{-b} h(s) .
   \end{align*}
   For $0 < \sigma \le t$ integrate this inequality from $\sigma$ to $t$ to arrive at
   \begin{align*}
   t^{-b} f(t)     + b\int_\sigma^t s^{-b-1} f(s) ds +\int_\sigma^t s^{-b} g(s) ds 
   \le \int_\sigma^t s^{-b} h(s) ds        + \sigma^{-b} f(\sigma).
   \end{align*}
 Integrate this inequality now    with respect to $\sigma$ over the
  interval $[0, t]$.  Since all  integrands are positive we can reverse  the order of integration
  in the three double integrals.  This just results in multiplying
   each integrand by $s$, giving 
   \beq
   t^{1-b} f(t) + b\int_0^t s^{-b} f(s) ds +\int_0^t s^{1-b} g(s) ds 
   \le  \int_0^t s^{1-b} h(s) ds +\int_0^t \sigma^{-b} f(\sigma) d\sigma.    \notag
   \eeq
   Subtract the second term on the left from the second term on the right to arrive at  \eref{vs91}.
   If equality holds in \eref{vs90} then equality holds in all steps.
   
   To prove \eref{fa7} take $g = -f'$ and  $h =0$ in \eref{vs90}.
    Then  equality holds in \eref{vs90} and
   \eref{vs91}  reduces to \eref{fa7}.

    To prove \eref{fa8} let 
   $\rho(\sigma) =(1-b)\int_0^\sigma s^{-b} f(s) ds$ for $0 \le \sigma \le t$. \eref{fa7} 
   clearly holds when $t$ is replaced by $\sigma$ and therefore 
    $\sigma^{1-b} f(\sigma) \le \rho(\sigma)$ for $0 < \sigma \le t$. 
    Consequently,    since $(1-b)(q-1) = b$ and $q \ge 1$, we have
   \begin{align*}
\int_0^t f(s)^q ds & = \int_0^t \Big( s^{1-b} f(s) \Big)^{q-1} s^{-b} f(s) ds \\
&\le \Big( \sup_{0 < s \le t} s^{1-b} f(s) \Big)^{q-1} \int_0^ts^{-b} f(s) ds \\
&\le \rho(t)^{q-1} \rho(t)/(1-b)
\end{align*}
   by the monotonicity of $ \rho(\cdot)$. This proves \eref{fa8}. 
   Choose $b = 1/2$, and consequently $ q=2$, in \eref{fa8}
   to arrive at \eref{fa8.5}.
 \end{proof}

\begin{remark}{\rm  A seemingly shorter proof of \eref{vs91} can be derived
by simply using the identity
\beq
(d/ds)s^{1-b} f(s) + s^{1-b} g(s)  \le s^{1-b}h(s) +(1-b)s^{-b} f(s),
\eeq
which follows from \eref{vs90}, 
and then integrating it over $(0, t]$. However the integrated term at the lower limit
is $\lim_{s\downarrow 0} s^{1-b} f(s)$, and we have no way of knowing in advance
that this limit is zero, except in some special circumstances.
 For example in the special case $f' \le 0$,  leading to \eref{fa7}, and for $0 < b <1$,
  the monotonicity  of $f$ gives 
 $s^{1-b} f(s) = f(s)(1-b)\int_0^s \sigma^{-b}ds
 \le (1-b)\int_0^s \sigma^{-b} f(\sigma) d\sigma \rightarrow 0$ 
 as $s\downarrow 0$ because the integrand in integrable.
}
\end{remark}

\subsection{Initial behavior, order 0}   \label{secib0}

    In the classical integrating factor method for solving an ordinary differential equation
    such as $ dx/dt = f(t) x(t) + g(t)$, one changes  the ``dependent variable'' to 
    $y(t) \equiv e^{-\int_0^t f(s)ds} x(t)$ and then uses  the equivalent equation 
    $dy/dt =  e^{-\int_0^t f(s)ds} g(t)$ to solve for $y(t)$ as an integral. 
   The inequalities  \eref{vs529} and \eref{vs530} lend themselves
    to just such a use of integrating factors $e^{-\int_0^t \alpha(s) ds}$ and 
    $e^{-\int_0^t \beta(s) ds}$ respectively. However both functions $\alpha$ and $\beta$
    are quite singular near  $s=0$.  In fact from the sole  knowledge that $C(\cdot)$ lies
     in $\P_T^{1/2}$  
    one can only deduce  that each function is no worse than
     $o(s^{-1})$ near $s = 0$. 
        The existence of $\int_0^t \alpha(s) ds$, and therefore  its utility as an integrating factor,
        is thus in question in the critical case, $a = 1/2$. 
        The same is the case with $\beta(s)$.
      However it was shown in Theorem \ref{thmfa} 
    that for small initial data the solution has finite strong action when $a = 1/2$. Here it will be
    shown that if $C(\cdot)$ has finite strong action then   $\alpha$ and $\beta$ 
    are integrable over $(0, t]$. 
              Their use in  the method of integrating
     factors will then give detailed information about the initial behavior of
    the various derivatives of $C(\cdot)$ of interest to us. 
          This differs significantly from the non-critical case $a > 1/2$, where finite 
          strong  $a$-action
       is automatic. If $a > 1/2$ one need not assume that $\|C_0\|_{H_a}$ is small
       in order to use these integrating factors.
       
          However even when $C(\cdot)$ has infinite action 
          many  of the qualitative
  conclusions needed for  the recovery of $A$ from $C$ hold. 
  This will be shown in Section \ref{secrec}.
           
 The main theorems of this and the next two subsections concern the initial behavior of 
 solutions to the augmented  Yang-Mills equation \eref{aymh}. Some
 simple aspects of this behavior are just consequences of the fact that the function 
 $C(\cdot)$ lies
 in the path space $\P_T^a$. It need not be a solution. The following theorem
 (which we have labeled Order 0)  lists some of these properties.

            \begin{notation}\label{notab}{\rm Denote by $\alpha$ and $\beta$ the functions
    defined in \eref{vs529.2} and \eref{vs530.2}. Define  
  \beq
 \alpha_s^t = \int_s^t \alpha(\sigma) d \sigma, \ \     \ \ \ 
             \beta_s^t = \int_s^t \beta(\sigma) d\sigma\  \ \ \ \ \text{for}\ \ 0 \le s \le  t \le T < \infty. \label{vs39.1}
 \eeq 
    }
 \end{notation}

\begin{theorem}\label{thmord0}  $($Order 0$)$
 Let $ 1/2 \le a <1$  and $0 <T <\infty$. 
 Suppose that $C(\cdot)$ lies in the path space $\P_T^a$. 
 Then
\begin{align}
s^{1-a} \Big(\| \phi(s)\|_2^2 + \| B_C(s)\|_2^2\Big)\  \text{and}\  s^{2-2a} \|C(s)\|_6^4\
                \text{are bounded on}\  (0,T]       
                 \label{vs400a}
\end{align}
and go to zero as $s\downarrow 0$. Further,
\begin{align}
&\sup_{0 < s <T} s^{2-2a}\Big(\| B_C(s)\|_2^4 + \| \phi(s)\|_2^4 +\lambda(B_C(s))\Big) 
    < \infty,                                                         \label{vs420a}\\
&\ \alpha_\infty \equiv \sup_{0<s \le T} s^{2-2a} \alpha(s) < \infty, \ \  
 \beta_\infty \equiv \sup_{0< s \le T} s^{2-2a}\beta(s) < \infty.      \label{vs421a}
\end{align}
In particular, if $a = 1/2$, then
\begin{align}
s^{1/4} \| B_C(s)\|_2,\ \  s^{1/4} \| \phi(s)\|_2 \  \text{and}\ \ s\|C(s)\|_6^4\ \ \ 
                \text{are bounded on}\ \ (0,T]        \label{vs400}
\end{align}
and go to zero as $s\downarrow 0$. Further,
\begin{align}
&\sup_{0 < s <T} s\Big(\| B_C(s)\|_2^4 + \| \phi(s)\|_2^4\Big) 
    < \infty, \ \ \   \sup_{0< s <T}s\lambda(B_C(s)) < \infty,\ \ \         \label{vs420}\\
&\ \alpha_\infty \equiv \sup_{0<s \le T} s \alpha(s) < \infty,  \ \ \ \ \ 
          \text{and}\ \ \ \ \ \ \ \  \beta_\infty \equiv \sup_{0< s \le T} s\beta(s) < \infty.  \label{vs421}
\end{align}
If   $C(\cdot) \in \P_T^a$ and in addition has finite strong a-action then
\begin{align}
\int_0^T \|C(s)\|_{H_1}^4 ds &< \infty,           \label{vs410}\\
\int_0^T s^{-a} \|B_C(s)\|_2^2 ds < \infty\ \   
&\text{and} \  \int_0^T s^{-a} \|\phi(s)\|_2^2 ds  <\infty.                          \label{vs411a}  
\end{align}
Further, 
\begin{align}
      \ \int_0^T\Big( \|B_C(s)\|_2^4 &+ \| \phi(s)\|_2^4 + \| C(s)\|_6^4 
                                 + \lambda(B_C(s))\Big) ds < \infty,\ \    \label{vs422} \\
     & \alpha_0^T < \infty,\ \ \ \ \text{and} \ \ \ \   \beta_0^T < \infty.      \label{vs423}  
  \end{align}  
  In particular   \eref{vs410}- \eref{vs423} hold if $C(\cdot) \in \P_T^a$ with $1/2 < a <1$,
   and also hold for $a = 1/2$  if  $C(\cdot) \in \P_T^{1/2}$     
  and   $\|C_0\|_{H_{1/2}}$ is sufficiently small.
\end{theorem}
       \begin{proof}     
    The curvature $B_C$ is given by the usual expression $B_C = dC + C\wedge C$ for any connection form $C$. We can bound the  $L^2$ norm of $B_C$ as follows.
 \begin{align}
 \|B_C\|_2 &\le \|dC\|_2 + \| C\wedge C\|_2  \notag\\
 &\le \|C\|_{H_1} + c \|C\|_3\|C\|_6      \notag \\
 &\le   \|C\|_{H_1}  + c\kappa_6\|C\|_3 \|C\|_{H_1} .   \label{ST249}
  \end{align}
 Thus
 \begin{align}
\|B_C(s)\|_2 &\le  \|C(s)\|_{H_1} (1 + c\kappa_6 \|C(s)\|_3) \notag\\   
&\le c_3 \|C(s)\|_{H_1},             \label{vs44}
\end{align}
where
\beq
c_3 =1+ c\kappa_6  \sup_{0<s \le T} \|C(s)\|_3 .    \label{vs45}
\eeq
The constant $c_3$ is finite because $C(\cdot)$ lies in $\P_T^a \subset \P_T^{1/2}$
 and is therefore a  continuous function  on $[0,T]$ into 
 $H_{1/2}(M)$
 and therefore into $L^3(M)$.
 Now $s^{1-a} \|C(s)\|_{H_1}^2$  is bounded  on $(0, T)$ because  $C(\cdot) \in \P_T^{a}$.
   Hence $s^{1-a}\|B_C(s)\|_2^2$ is bounded on $(0,T]$. 
            Since 
 \beq
 \|\phi(s)\|_2 \le \|C(s)\|_{H_1}\ \ \text{and}\ \ \|C(s)\|_6 \le \kappa_6 \|C(s)\|_{H_1} \label{vs441}
 \eeq
both of these functions are also bounded after multiplying by $s^{(1-a)/2}$.  
This completes the proof of  \eref{vs400a}.      
 The inequalities \eref{vs420a} and \eref{vs421a}  now follow from \eref{vs400a} 
 in view of the definitions \eref{gaf70}, \eref{vs529.2} and \eref{vs530.2}.
 Put $a = 1/2$ to derive the special case \eref{vs400} - \eref{vs421}.

         Assume now that $C(\cdot)$ has finite strong a-action.
 Since $s^{1-a} \|C(s)\|_{H_1}^2$ is bounded on $(0, T]$ 
 and $a \ge 1/2$    we have
 \begin{align*}
 \int_0^T \|C(s)\|_{H_1}^4 ds & = \int_0^T s^{2a-1}\Big(s^{1-a} \|C(s)\|_{H_1}^2\Big)
                         \Big(s^{-a} \|C(s)\|_{H_1}^2\Big) ds \\
 & \le T^{2a-1} \sup_{0< s <T} \Big(s^{1-a} \|C(s)\|_{H_1}^2\Big)
                            \int_0^T \Big(s^{-a} \|C(s)\|_{H_1}^2\Big) ds \\
                            &< \infty.
\end{align*}
This proves \eref{vs410}.  By \eref{vs44} we have 
\beq
\int_0^T s^{-a}\|B_C(s)\|_2^2ds \le 
c_3^2 \int_0^T s^{-a} \|C(s)\|_{H_1}^2 ds,        \label{vs442}
\eeq
 which is finite by the assumption of finite strong a-action. 
 Since $\|\phi(s)\|_2 \le \|C(s)\|_{H_1}$ the second integral in \eref{vs411a}   
  is also finite.
 This proves \eref{vs411a}.
 
  The first three terms in the integral in \eref{vs422} are dominated by a constant times
    $\| C(s)\|_{H_1}^4$, by \eref{vs44} and  \eref{vs441}.
Hence the integral of these terms is finite. 
So is the integral of the last term, by the definition \eref{gaf70}. 
    The inequalities in \eref{vs423} now follow from the definitions 
      \eref{vs529.2}, \eref{vs530.2} and \eref{vs39.1}.  
 \end{proof}

\subsection{Initial behavior, order 1}\label{secib1}

\begin{theorem}\label{thmord1a} $($First order energy estimate$)$ 
 Let $1/2 \le a <1$.    Suppose that $C(\cdot)$ is  a strong solution to the augmented   equation
    \eref{aymh} lying in $\P_T^{a}$  and having finite
     strong a-action, i.e., \eref{fa1C} holds. Then
\begin{align}
t^{1-a}\Big(&\| B_C(t)\|_2^2 + \|\phi(t)\|_2^2\Big)  
      + \int_0^t s^{1-a}e^{\alpha_s^t}\|C'(s)\|_2^2 ds  \ \ \  \ \ \   (\text{Order 1}) \notag\\
&\le   (1-a)\int_0^t s^{-a}e^{\alpha_s^t}\Big(\| B_C(s)\|_2^2
                           +  \| \phi(s)\|_2^2\Big) ds < \infty. \ \ \    \label{vs38a}
\end{align}
The following weighted $L^6$ bound holds.
     \beq
\int_0^T s^{1-a}\Big( \|B_C(s)\|_6^2 + \|\phi(s)\|_6^2 \Big) ds         \label{vs550a}
                            < \infty. \ \ \ \ \ \ \ \ \ \qquad \qquad   (\text{Order 1})  
\eeq  
Furthermore,
       \beq
\int_0^Ts^{1-a}\Big( \|d_C^* B_C(s)\|_2^2 + \| d_C\phi(s)\|_2^2 
                                    + \| d\phi(s)\|_2^2 \Big) ds < \infty.       \  \   (\text{Order 1}) \label{vs38b}
\eeq
\end{theorem}
\begin{proof}  
By hypothesis $C(\cdot)$ lies in $\P_T^{a}$ and has finite strong a-action.
Therefore $\alpha_0^T < \infty$ by \eref{vs423} of  Theorem \ref{thmord0}.
 We will use this for the following bounds.

   Let $\zeta(t) = \alpha_0^t$ and define 
 $u(s) = \| B_C(s)\|_2^2 + \|\phi(s)\|_2^2$. The inequality \eref{vs529} shows that
$u'(s) + \|C'(s)\|_2^2 \le \zeta'(s) u(s)$.
Hence 
\beq
\frac{d}{ds}\Big(e^{-\zeta(s)} u(s)\Big) +  e^{-\zeta(s)} \| C'(s)\|_2^2 \le 0.  \label{vs558}
\eeq
Let $f(s) = e^{-\zeta(s)} u(s)$, 
$g(s) =e^{-\zeta(s)} \| C'(s)\|_2^2$ and $h(s) =0$.  We will apply Lemma \ref{lem50}
with these choices of $f,g$ and $h$ and with the choice $b =a$.  
Then \eref{vs91} asserts that
\begin{align*}
t^{1-a} e^{-\zeta(t)}u(t) + \int_0^t s^{1-a} e^{-\zeta(s)}  \|C'(s)\|_2^2 ds 
\le (1-a) \int_0^t  s^{-a} e^{-\zeta(s)}  u(s) ds  .
\end{align*}
Multiply by $e^{\zeta(t)}$ to arrive at  \eref{vs38a}. 
Since $\alpha_0^T < \infty$ by  \eref{vs423}, the inequalities in \eref{vs411a} 
 show that  the right hand side of \eref{vs38a} is finite.

To prove the  weighted $L^6$ bound \eref{vs550a} 
  we can use   Sobolev's inequalitiy \eref{gaf01} and the  Gaffney-Friedrichs-Sobolev
    inequality \eref{gaf50},  which give, respectively, (since $\kappa_6^2 \le \kappa^2$)
\begin{align}
\|\phi(s)\|_6^2 
          &\le \kappa^2 ( \| d_C \phi(s)\|_2^2 + \|\phi(s)\|_2^2)                       \label{vs560} \\
 \|B_C(s)\|_6^2 
          & \le \kappa^2 (\|d_C^* B_C(s)\|_2^2 + \lambda(B_C(s)) \| B_C(s)\|_2^2) \label{vs561}
 \end{align}
 since $d_CB_C =0$ by Bianchi's identity.
Adding \eref{vs560} and \eref{vs561}, and using the identity \eref{vs561.1}, we find  
  \begin{align}
 \kappa^{-2} \Big(  \|B_C(s)\|_6^2 &+ \|\phi(s)\|_6^2 \Big)              \notag\\
& \le \| C'(s)\|_2^2 + \| \phi(s) \|_2^2  
 +\lambda(B_C(s)) \| B_C(s)\|_2^2.                     \label{vs562}
 \end{align}
Multiply \eref{vs562} by $s^{1-a}$ to find
 \begin{align}
 \kappa^{-2}s^{1-a}\Big(&  \|B_C(s)\|_6^2 + \|\phi(s)\|_6^2 \Big) \notag\\
& \le  s^{1-a} \| C'(s)\|_2^2 + s^{1-a} \| \phi(s) \|_2^2  
 +\lambda(B_C(s)) \{s^{1-a}\| B_C(s)\|_2^2\}.                               \label{vs38e}
 \end{align}
 The first term on the right, $s^{1-a} \| C'(s)\|_2^2$, is integrable over $(0, T]$ by \eref{vs38a}. 
 The second term is bounded by \eref{vs400a} and therefore integrable. The third term
 is an integrable function times a bounded function by  \eref{vs422} 
 and \eref{vs400a}, respectively.  This proves \eref{vs550a}.
   
      Concerning the inequality \eref{vs38b} observe first that the  integrability
  of the first two terms   in \eref{vs38b} follows from the 
  orthogonality relation \eref{vs561.1} along with the relation
   $\int_0^T s^{1-a}\| C'(s)\|_2^2ds < \infty$,
  implied by \eref{vs38a}. Since  $\|d\phi(s)\|_2 \le \|d_C\phi(s)\|_2 + \| \, [C(s), \phi(s)]\, \|_2$
  the integrability of the third term would  follow from the integrability of 
  $s^{1-a} \| \, [C(s), \phi(s)]\, \|_2^2$. But
   \begin{align}
  \int_0^Ts^{1-a} \|\, [ C(s), \phi(s)]\,\|_2^2 ds 
  &\le c^2 \int_0^T s^{1-a} \|C(s)\|_6^2 \|\phi(s)\|_3^2 ds        \notag\\
  &\le c^2 \Big(\sup_{0 < s \le T} s^{1-a} \|C(s)\|_6^2\Big) \int_0^T\|\phi(s)\|_3^2ds \notag\\
  &\le (c\kappa_6)^2 |C|_T^2  \int_0^T\|\phi(s)\|_3^2ds.   \label{vs38f}
 \end{align} 
 Furthermore, using $ 0 =(a-(1/2)) -(a/2) +(1-a)/2$ along with the interpolation
  $\|f\|_3^2 \le \|f\|_2\|f\|_6$ we find 
 \begin{align}
 \int_0^T \|&\phi(s)\|_3^2 ds 
 \le \int_0^T s^{a-(1/2)}\Big(s^{-a/2}\|\phi(s)\|_2\Big)\Big(s^{(1-a)/2} \|\phi(s)\|_6\Big) ds \notag\\
 &\le T^{a-(1/2)}\Big(\int_0^T s^{-a} \|\phi(s)\|_2^2 ds \Big)^{1/2} 
                     \Big(\int_0^T s^{1-a} \|\phi(s)\|_6^2 ds \Big)^{1/2}          \notag \\
  &\le   T^{a-(1/2)}\Big(\int_0^T s^{-a} \|C(s)\|_{H_1}^2 ds \Big)^{1/2} 
                     \Big(\int_0^T s^{1-a} \|\phi(s)\|_6^2 ds \Big)^{1/2} .      \label{vs38g}
 \end{align}
 The first integral is finite because $C(\cdot)$ has finite action. The second integral is finite 
 by \eref{vs550a}. 
  The same argument shows that
  \begin{align}
 \int_0^T \|B_C(s)\|_3^2 ds < \infty  .         \label{vs38j}
 \end{align} 
   Combining \eref{vs38f} and \eref{vs38g} it follows that
 \beq
  \int_0^Ts^{1-a} \|\, [ C(s), \phi(s)]\,\|_2^2 ds <\infty.     \label{vs38m}
  \eeq
  This  proves the integrability of the last term in \eref{vs38b} and completes the
   proof of Theorem \ref{thmord1a}.
\end{proof}

\subsection{Initial behavior, order 2} 
  \label{secib2a}

 \begin{theorem}\label{thmord2a} $($Order 2$)$  Let $1/2 \le a <1$.
     Suppose that $C(\cdot)$ is  a strong solution to the augmented   equation
    \eref{aymh} lying in $\P_T^{a}$  and having finite
     strong a-action, i.e., \eref{fa1C} holds. Then
\begin{align}
t^{2-a} \| C'(t)\|_2^2 
   &+ \int_0^t s^{2-a}  e^{\beta_s^t} \Big( \| d_C^* C'(s) \|_2^2 
                                 + \| d_C C'(s)\|_2^2\big) ds              \ \ \ \  (\text{Order 2})       \notag\\
&\le (2-a)(1-a)e^{\beta_0^t} \int_0^t s^{-a}  
                  \Big( \| B_C(s) \|_2^2 + \|\phi(s)\|_2^2\Big) ds.                                  \label{vs640a}
\end{align}
The following $L^6$ bounds also hold.
\begin{align}
   \sup_{0< t <T} t^{2-a}\Big( \|B_C(t)\|_6^2 + \|\phi(t)\|_6^2 \Big) 
                            &< \infty.  \ \    (\text{Order 2})                             \label{vs641pa}\\
 \int_0^T s^{2-a}\Big( \|C'(s)\|_6^2 + \| d_C^* B_C(s)\|_6^2 + \| d_C\phi(s)\|_6^2 \Big) ds 
                            &< \infty.    \ \    (\text{Order 2})                            \label{vs641a} \\
 \int_0^T s^{2-a} \Big( \|d \phi(s)\|_6^2 +\|\, [C(s), \phi(s)]\, \|_6^2\Big)ds 
                           &< \infty. \ \   (\text{Order 2})                                \label{vs641b}
\end{align}
\end{theorem}

In preparation for the proof of the $L^6$ bounds we will need the following lemma.

\begin{lemma} For a solution to the augmented equation \eref{aymh} the following
 three $(a$-independent$)$ inequalities hold.
 \begin{align}
   \kappa^{-2} \| d_{C(s)}^* B_C(s)\|_6^2 &\le \| d_C d_C^* B_C \|_2^2 + \lambda(B_C) \| d_C^* B_C \|_2^2.           \label{vs646}\\
   \kappa^{-2} \| d_{C(s)} \phi(s)\|_6^2 &\le \| d_C^* d_C \phi\|_2^2 + \| \, [B_C, \phi]\, \|_2^2 + \lambda(B_C) \| d_C \phi\|_2^2.      \label{vs647} \\
\kappa^{-2}\Big(\|C'(s)\|_6^2 + \| &d_{C(s)}^* B_C(s)\|_6^2
                                         +\| d_{C(s)} \phi(s)\|_6^2\Big)                     \label{vs649}\\
&\le  3  \|  d_C C'\|_2^2  + 2\| d_C^* C'\|_2^2  + 3\|\, [B_C, \phi]\, \|_2^2
 +2\lambda(B_C) \| C'\|_2^2 .            \notag
 \end{align}
\end{lemma}
\begin{proof}
 Use the    GFS inequality  \eref{gaf50}  twice  and the Bianchi identity twice, once for $d_C^* d_C^* B_C=0$ and once
   in $d_C^2 \phi = [B_C, \phi]$, to find \eref{vs646} and \eref{vs647}.
 In view of the identities  
    \begin{align*}
  -d_C^* C'(s) &=  d_C^* d_C\phi(s) \ \ \ \ \ \ \ \ \ \ \text{and}\  \\
     - d_C C'(s) &=  d_C d_C^* B_C(s) + [B_C(s), \phi(s)],
  \end{align*}
the first term on the right of each line in \eref{vs646} and \eref{vs647}
 can be expressed in terms of $C'$ and $[B_C, \phi] $.  We may add them and use \eref{vs561.1}
 to find
 \begin{align*}
\kappa^{-2}\Big( &\| d_{C(s)}^* B_C(s)\|_6^2+  \| d_{C(s)} \phi(s)\|_6^2\Big) \\
&\le \|  d_C C' +[ B_C,  \phi]\, \|_2^2   + \| d_C^* C'\|_2^2 + \|\, [B_C, \phi]\, \|_2^2  
                                   +  \lambda(B_C) \| C'\|_2^2 \\
 &\le  2\|d_C C'\|_2^2  +  \| d_C^* C'\|_2^2 + 3 \|\, [B_C, \phi]\, \|_2^2  +  \lambda(B_C) \| C'\|_2^2 .
\end{align*}  
        To this we may add the $\|C'(s)\|_6$ bound  \eref{vs612}  to arrive at \eref{vs649}.
\end{proof}

 \bigskip
\noindent
\begin{proof}[Proof of Theorem \ref{thmord2a}] 
From \eref{vs423}  in Theorem \ref{thmord0} we know that $\beta_0^T < \infty$.                      
                     Let $\zeta(s)= \beta_0^s$. Since $\zeta'(s) = \beta(s)$, 
  the inequality \eref{vs530} implies that 
  \begin{align}
\frac{d}{ds} \Big( e^{-\zeta(s)} \|C'(s)\|_2^2\Big) 
+e^{-\zeta(s)} \Big(\|d_C^* C'(s)\|_2^2 + \|d_C C'(s)\|_2^2\Big) \le 0. \label{vs605}
\end{align}                     
        In Lemma \ref{lem50} choose  $b = a-1$, 
         $h(s) =0$,
 \begin{align*}
         f(s) =e^{-\zeta(s)} \|C'(s)\|_2^2,\ \ \ \ 
    g(s) = e^{-\zeta(s)} \Big(\|d_C^* C'(s)\|_2^2 + \|d_C C'(s)\|_2^2\Big)
    \end{align*} 
    and use \eref{vs605}  and \eref{vs91} to find
    \begin{align}
t^{2-a}e^{-\zeta(t)} \| C'(t)\|_2^2  
    &+ \int_0^t s^{2-a} e^{-\zeta(s)} \Big(\|d_C^* C'(s)\|_2^2 + \|d_C C'(s)\|_2^2\Big)ds\notag\\
 &\le  (2-a) \int_0^t s^{1-a} e^{-\zeta(s)} \|C'(s)\|_2^2 ds.
  \end{align}  
  Since $\zeta(t) - \zeta(s) = \beta_s^t$, multiplication  by $e^{\zeta(t)}$ gives
    \begin{align}
t^{2-a} \| C'(t)\|_2^2  
    &+ \int_0^t s^{2-a} e^{\beta_s^t} \Big(\|d_C^* C'(s)\|_2^2 + \|d_C C'(s)\|_2^2\Big)ds\notag\\
 &\le (2-a) \int_0^t s^{1-a} e^{\beta_s^t} \|C'(s)\|_2^2 ds.    \label{vs607}
  \end{align}  
 From the  coefficients  in   \eref{vs529.3} and  \eref{vs530.3} one sees that 
  $\beta(\sigma) - \alpha(\sigma) \ge 0$, and therefore 
  $\beta_s^t  -\alpha_s^t \le  \beta_0^t - \alpha_0^t$. Hence, using \eref{vs38a}, one finds
 \begin{align*}
 (2-a) \int_0^t s^{1-a} e^{\beta_s^t} &\|C'(s)\|_2^2 ds  
 \le (2-a)e^{\beta_0^t - \alpha_0^t}\int_0^t s^{1-a} e^{\alpha_s^t} \| C'(s)\|_2^2 ds \\
 &\le (2-a)(1-a) e^{\beta_0^t - \alpha_0^t}\int_0^t s^{-a}e^{\alpha_s^t}\Big( \| B_C(s) \|_2^2 + \|\phi(s)\|_2^2 \Big) ds \\
 &\le  (2-a)(1-a) e^{\beta_0^t}\int_0^t s^{-a}\Big( \| B_C(s) \|_2^2 + \|\phi(s)\|_2^2 \Big) ds. 
 \end{align*}
 This proves \eref{vs640a}.

 To prove   \eref{vs641pa} multiply \eref{vs38e} 
  by $s$ to find
   \begin{align}
 \kappa^{-2}s^{2-a}\Big(&  \|B_C(s)\|_6^2 + \|\phi(s)\|_6^2 \Big) \label{vs643}\\ 
& \le s^{2-a} \| C'(s)\|_2^2 +s^{2-a} \| \phi(s) \|_2^2  
 +\{s\lambda(B_C(s))\}\{s^{1-a}\| B_C(s)\|_2^2\}.     \notag
 \end{align}
 The first term on the right is bounded by \eref{vs640a}. 
 The second  term is bounded by \eref{vs400a}. The third term is bounded
 by \eref{vs420a} and \eref{vs400a} since $2 -2a \le 1$. This proves \eref{vs641pa}.

 To prove \eref{vs641a} we can start with the a-independent inequality \eref{vs649}
and multiply it by $s^{2-a}$.  The resulting first two terms on the right 
are integrable by \eref{vs640a}.
The third term is  
$
 s^{2-a} \| \, [ B_C(s), \phi(s)]\, \|_2^2\le   \Big(c^2 s^{2-a} \| B_C(s)\|_6^2\Big) \| \phi(s) \|_3^2,
 $
 which is the product of a bounded function, by \eref{vs641pa},  times an integrable function,
 by \eref{vs38g}.   The last term is 
$\Big(s \lambda(B_C(s))\Big) \Big( s^{1-a}\|C'(s)\|^2\Big)$, which is a bounded function,
by \eref{vs420a}, times an integrable function, by \eref{vs38a}.    This proves \eref{vs641a}.

 Concerning the proof of \eref{vs641b}, it suffices to prove the integrability of either
 one of the two terms because we already know that
 $\int_0^T s^{2-a} \| d\phi(s) +[ C(s), \phi(s)]\, \|_6^2 ds < \infty$ by \eref{vs641a}.
 We will prove the integrability of the second term and do this by invoking  
 the GFS inequality  \eref{gaf50}. Thus, taking $\w = [C(s), \phi(s)]$ in \eref{gaf50},
 it suffices to show that
 \begin{align}
\int_0^T s^{2-a}\Big( \| d_C^* [ C(s), \phi(s)]\, \|_2^2 
         +  \| d_C [ C(s), \phi(s)]\, \|_2^2\Big) ds < \infty                              \label{vs57}
\end{align}
and
\beq
\int_0^Ts^{2-a} \|B_C(s)\|_2^4 \|\, [C(s), \phi(s)]\,\|_2^2 ds < \infty.        \label{vs58}
\eeq
To this end, observe first the identities
  \begin{align*}
  d_C^* [C, \phi] &= [C\lrc d_C\phi], \\
  d_C[C,\phi] &= - [C\wedge d_C\phi] + [B_C +(1/2)[C\wedge C], \phi],
 \end{align*}
 which follow from $[d_C^* C, \phi] = [\phi,\phi] =0$ and $d_C C = B_C + (1/2) [C\wedge C]$.
 Now   $\| [C\lrc d_C\phi]\, \|_2 \le c \| C\|_6 \| d_C \phi\|_3$ with the same bound
 for  $\| C\wedge d_C \phi\|_2$. Therefore
\begin{align}
 \int_0^T s^{2-a} &\Big( \| [C\lrc d_C\phi],\|_2^2 + \| C\wedge d_C \phi\|_2^2 \Big) ds       \notag\\
 &\le 2c^2 \Big(\sup_{0 \le s \le T}s^{1-a} \|C(s)\|_6^2\Big) \int_0^T s\|d_C \phi(s)\|_3^2 ds  \notag \\
 &\le (c\kappa_6)^2|C|_T^2 \int_0^T s\|d_C \phi(s)\|_3^2 ds.   \label{vs38h}
 \end{align}
Using now $ 1 =(a-(1/2))+  (1-a)/2 +(2-a)/2$ along with the interpolation
  $\|f\|_3^2 \le \|f\|_2\|f\|_6$ we find 
 \begin{align}
 \int_0^Ts \|&d_C\phi(s)\|_3^2 ds 
 \le \int_0^T s^{a-(1/2)}\Big(s^{(1-a)/2}\|d_C\phi(s)\|_2\Big)
                                         \Big(s^{(2-a)/2} \|d_C\phi(s)\|_6\Big) ds \notag\\
 &\le T^{a-(1/2)}\Big(\int_0^T s^{1-a} \|d_C\phi(s)\|_2^2 ds \Big)^{1/2} 
                     \Big(\int_0^T s^{2-a} \|d_C\phi(s)\|_6^2 ds \Big)^{1/2}          \notag \\
  &<\infty      \label{vs38k}
 \end{align}
by \eref{vs38b} and   \eref{vs641a}.  Hence
\begin{align*}
 \int_0^T s^{2-a} \Big( \| [C\lrc d_C\phi],\|_2^2 + \| C\wedge d_C \phi\|_2^2 \Big) ds
 <\infty.
 \end{align*}
Further, 
\begin{align*}
\int_0^T s^{2-a} &\|\,  [B_C +(1/2)[C\wedge C], \phi]\, \|_2^2 ds \\
& \le \sup_{0\le s \le T} \{s\| B_C +(1/2)[C\wedge C]\, \|_3^2\}\int_0^T s^{1-a} \|\phi(s)\|_6^2 ds.
\end{align*}
The integral in the last line is finite by \eref{vs550a}. The supremum in that line is also
 finite as follows   from the inequalities 
\begin{align}
s\|B_C(s)\|_3^2 
 &\le s^{a - (1/2)} \Big( s^{\frac{1-a}{2}} \| B_C(s)\|_2\Big) \Big( s^\frac{2-a}{2} \|B_C(s)\|_6\Big)
         \  \text{and}  \label{vs46} \\
s \| \, [ C(s)\wedge C(s) ] \|_3^2 &\le c^2 s^{2a -1} \Big( s^{\frac{1-a}{2} }\|C(s)\|_6\Big)^4 
\le c^2 s^{2a-1} (\kappa_6 |C|_T)^4,           \label{vs47}
\end{align}
since $|C|_T <\infty$ by the definition of $\P_T^a$ while the two expressions in parentheses
in line \eref{vs46} are bounded on $(0, T]$ by respectively  \eref{vs400a} and  \eref{vs641pa}
for each $a \in [1/2, 1)$. This proves \eref{vs57}.

It remains only to prove \eref{vs58}.
But    
 $$
 s^{2-a}\|B_C(s)\|_2^4 \|\, [C(s), \phi(s)]\,\|_2^2 = s^{2a-1}\Big(s^{1-a} \|B_C(s)\|_2^2\Big)^2
 \Big( s^{1-a}
  \|\, [C(s), \phi(s)]\,\|_2^2\Big).
  $$
  The first factor in parentheses is bounded by \eref{vs400a} while the last factor in parentheses
  is integrable by \eref{vs38m}. This completes the proof of \eref{vs641b}.
 \end{proof}

\subsection{The case of infinite action} \label{secia}
If a solution to  the augmented variational equation \eref{aymh} does not have 
finite a-action then our proofs of the estimates
given in Theorems \ref{thmord1a} and \ref{thmord2a} do not hold. 
 We are concerned in this section
only with the case $a = 1/2$ because finite a-action always holds for $ a > 1/2$.
 We are going to replace the first and second order initial behavior 
 bounds of Theorems \ref{thmord1a} and \ref{thmord2a}  by slightly weaker  bounds.
  All of these are outgrowths of the inequality
 \begin{align}
 \int_0^T s^{-1/2 + \delta} \|C(s)\|_{H_1}^2 ds < \infty,              \label{vs500}
 \end{align}
 which holds for all paths $C(\cdot) \in \P_T^{1/2}$ and $\delta >0$. 
 This follows from \eref{ST413a} with $a = 1/2$, which shows that
 $\|C(s)\|_{H_1}^2 = o(s^{-1/2})$ as $ s \downarrow 0$.
      The next theorem gives weaker information about the nature of the initial singularity
      than we obtained under the assumption of finite action, 
      but is adequate for implementing a weaker version of the ZDS procedure: 
      We will only be able to prove that  $g(t) \in \G_{1,q}$ for $q < 3$ rather
       than in the smaller group $\G_{3/2}$.

 \begin{theorem}\label{thmia1} $($Order 1$)$
 Suppose that $C_0 \in H_{1/2}$ and that $C(\cdot)$ is
  a strong solution
 of \eref{aymh} lying in $\P_T^{1/2}$. Define $\alpha_\infty$ and $\beta_\infty$ as in \eref{vs421}.
  Then, for any $\delta >0$, there holds
  \begin{align}
t^{(1/2) + \delta}\Big(&\| B_C(t)\|_2^2 + \|\phi(t)\|_2^2\Big)  
      + \int_0^t s^{(1/2) + \delta}\|C'(s)\|_2^2 ds  \notag\\
&\le   (\alpha_\infty +(1/2) + \delta) \int_0^t s^{\delta - (1/2)}
              \Big(\| B_C(s)\|_2^2 +\| \phi(s)\|_2^2\Big) ds        \notag    \\    
& < \infty.\ \ \  
\qquad\qquad \qquad\qquad \qquad\qquad\ \ \   \ \ \  (\text{Order 1})  \label{vs38d}
\end{align}
Moreover
\begin{align}
\int_0^T s^{\delta +(1/2)}\Big( \|B_C(s)\|_6^2 + \|\phi(s)\|_6^2 \Big) ds &< \infty
\qquad\qquad\  
                                                   \text{and}                               \label{vs550ep} \\
 \int_0^T s^{\delta + (1/2)} \Big( \| d_C^* B_C(s)\|_2^2 
            + \| d_C\phi(s)\|_2^2 &+ \| d\phi(s)\|_2^2\Big) ds < \infty.    \label{vs551ep}
\end{align}
\end{theorem}
         \begin{proof} 
 We  start with the differential inequality \eref{vs529}.
  We cannot put the 
    function $\alpha(\cdot)$ into an integrating factor, as we did under the assumption
     of finite action, because $\alpha$  is  not integrable near zero. Instead we will apply the machinery of Lemma \ref{lem50} directly to \eref{vs529}.   
     
      Let $d = \delta +(1/2)$. 
      In \eref{vs90} take $f(s) = \| B_C(s)\|_2^2 + \|\phi(s) \|_2^2$, $g(s) = \|C'(s)\|_2^2$
and $h(s) = \alpha(s) \|\phi(s)\|_2^2$.  Then \eref{vs529} asserts that \eref{vs90} holds.
Take $b =1 -d$ in \eref{vs91} to find
\begin{align*}
t^d f(t) + \int_0^t s^d \| C'(s)\|_2^2 ds \le \int_0^t s^{d} \alpha(s) \| \phi(s)\|_2^2 ds +
d \int_0^t s^{d-1} f(s) ds.
\end{align*}
But $s^d \alpha(s) \le s^{d-1}\alpha_\infty$ and $\|\phi(s)\|_2^2 \le f(s)$.
Hence the first inequality in \eref{vs38d} holds.
 By \eref{vs400} one has  
 $s^{\delta -(1/2)}\Big( \| B_C(s)\|_2^2 +  \|\phi(s)\|_2^2\Big) 
= O(s^{\delta -1})$ as $s \downarrow 0$, which is integrable. 
Since $\alpha_\infty < \infty$, by \eref{vs421},
 the inequality   \eref{vs38d} holds.

    The proof of the weighted $L^6$ bound \eref{vs550ep}  imitates the proof
     of the corresponding bound  for the finite action case: 
     We may    replace \eref{vs38e} (with $a = 1/2$) by    
   \begin{align}
 \kappa^{-2}&s^{d}\Big(  \|B_C(s)\|_6^2 + \|\phi(s)\|_6^2 \Big) \notag\\
& \le  s^{d} \| C'(s)\|_2^2 + s^{d} \| \phi(s) \|_2^2  
 +s^{\delta} \lambda(B_C(s)) \Big(s^{1/2}\| B_C(s)\|_2^2\Big)      \label{vs553ep}
 \end{align}
 with  $d = 1/2 + \delta$. From \eref{vs38d} we see that
   $\int_0^Ts^{d} \| C'(s)\|_2^2ds < \infty$. The second term  in \eref{vs553ep} is bounded,
   by \eref{vs400}. Since
   \beq
   \int_0^T s^{\delta} \lambda(B_C(s)) ds 
             =\int_0^T \Big(s^\delta + O(s^{-1 + \delta}) \Big) ds < \infty,
   \eeq
   the last term in \eref{vs553ep} is integrable. This proves \eref{vs550ep}.
   
   Concerning the inequality  \eref{vs551ep} observe that the integrability of the first two
   terms follows directly from \eref{vs38d}  
   because $\|C'(s)\|_2^2 = \| d_C B_C(s)\|_2^2 + \|d_C \phi(s)\|_2^2$. Just as in the proof
   of \eref{vs38b}, it suffices to prove that \linebreak
   $\int_0^T s^d \|\, [ C(s), \phi(s)]\, \|_2^2 ds < \infty$. But
   \begin{align*}
   \int_0^T s^d &\|\, [ C(s), \phi(s)]\, \|_2^2 ds 
   \le c^2|C|_T^2 \int_0^T s^\delta \|\phi(s)\|_2 \|\phi(s)\|_6 ds \\
   &\le c^2|C|_T^2 \Big(\int_0^T s^{\delta - (1/2)} \|\phi(s)\|_2^2 ds \Big)^{1/2}
   \Big(\int_0^T s^{\delta + (1/2)} \|\phi(s)\|_6^2 ds \Big)^{1/2}\\
   & < \infty
      \end{align*}
 by \eref{vs550ep}.  
     \end{proof}
     
\begin{theorem}\label{thmia2} $($Order 2$)$ Suppose again that $C_0 \in H_{1/2}$ and 
   that $C(\cdot)$ is   a strong solution
 of \eref{aymh} lying in $\P_T^{1/2}$. Define $\alpha_\infty$ and $\beta_\infty$ as in \eref{vs421}.
  Then, for any $\delta >0$, there holds
\begin{align}
t^{(3/2) + \delta}& \|C'(t) \|_2^2 
+ \int_0^t s^{(3/2) + \delta} \Big(\|d_C^* C'(s)\|_2^2 + \|d_C C'(s)\|_2^2\Big)ds \notag\\
&\le (\beta_\infty + d)(\alpha_\infty + (1/2) + \delta) 
                        \int_0^ts^{\delta -(1/2)}\Big( \|B_C(s)\|_2^2 + \| \phi(s)\|_2^2 \Big)ds \notag\\
& < \infty, \qquad \qquad\qquad \qquad \qquad\qquad   \qquad           (\text{Order 2}) \label{vs640d}
\end{align}
where $d = (3/2) + \delta$. 
Moreover 
\begin{align}
\sup_{0< t <T} t^{(3/2)+\delta}\Big( \|B_C(t)\|_6^2 
                       + \|\phi(t)\|_6^2 \Big) &< \infty,\ \                                       \label{vs641pep} \\
\int_0^T s^{(3/2)+\delta}\Big( \|C'(s)\|_6^2 + \| d_C^* B_C(s)\|_6^2 
                 + \| d_C\phi(s)\|_6^2 \Big) ds  &< \infty  \ \ \   \text{and}          \label{vs641ep}\\
 \int_0^T  s^{(3/2)+\delta}\Big(\|d\phi(s)\|_6^2 
                                     + \|\, [C(s), \phi(s)]\, \|_6^2 \Big) ds &< \infty.       \label{vs642ep}
\end{align} 
 \end{theorem}
         \begin{proof}          
         The proof follows the pattern of that for the finite action
  case, Theorem    \ref{thmord2a},  with modifications   similar to those used in the proof of 
  Theorem \ref{thmia1}. We will omit the details.
  \end{proof}

\subsection{High $L^p$ bounds  via Neumann domination}  \label{secibN}
Our energy estimates were  able to produce bounds on $L^p$ norms of various
 functions on $M$ 
for $ 2 \le p \le 6$. In this section we will derive $L^p$ bounds  of some functions 
for $6 \le p \le \infty$.

\begin{theorem} \label{thmnd1} $(p =\infty)$ Let $1/2 \le a <1$ and $0 < T <\infty$. 
Assume that either $M= \R^3$ or that $M$ is convex in the sense of Definition \ref{convex}.
          Suppose that $C(\cdot)$  is a  strong solution to \eref{aymh} lying in $\P_T^a$
    and having finite a-action. Then
\begin{align}
\int_0^T t^{(3/2) -a} \Big(\|B_C(t)\|_\infty^2 + \| \phi(t)\|_\infty^2\Big) dt < \infty.  \label{nd10}
\end{align}
In particular,
\beq
\int_0^T t \Big(\|B_C(t)\|_\infty^2 + \| \phi(t)\|_\infty^2\Big)
 dt < \infty\ \ \ \ \ \ \ \text{if} \ \ a = 1/2                                       \label{nd11}
\eeq
and
\beq
\int_0^T \Big(\|B_C(t)\|_\infty  + \|\phi(t)\|_\infty \Big)dt < \infty \ \ \ \ 
                                                               \text{if} \ \ a > 1/2.         \label{nd12}
\eeq
Furthermore,
\begin{align}
\|B_C(t)\|_\infty  = t^{-1 + \frac{a- (1/2)}{2}}\ o(1) \ \   
                                  \text{as}\ \ t\downarrow 0\ \ \ \text{if}\ \ 1/2 \le a <1.    \label{nd13}
\end{align}
In particular,
\begin{align}
\|B_C(t)\|_\infty = o(t^{-1})\ \ \text{as} \ \ t\downarrow 0\ \   \text{if}\ \ a = 1/2.    \label{nd15}
\end{align}
\end{theorem}

\begin{corollary}\label{cornd2} 
  If, in Theorem \ref{thmnd1}, we drop the hypothesis that
$C(\cdot)$ has finite a-action then we still have
\beq
\sup_{\epsilon\le t \le T} \|B_C(t)\|_\infty  < \infty        \label{nd15.1}
\eeq
for any $\epsilon >0$.
\end{corollary}

\begin{theorem}\label{thmndp} $($$p < \infty$$)$ Under the same hypotheses as in 
 Theorem \ref{thmnd1},    if $ 6 \le p <\infty$ then
\begin{align}
\int_0^T t^{(3/2) -a -(3/p)} \Big(\|B_C(t)\|_p^2 + \| \phi(t)\|_p^2\Big) dt < \infty .    \label{nd10p}
\end{align}
In particular,
\beq
\int_0^T t^{1 - (3/p)} \Big(\|B_C(t)\|_p^2 + \| \phi(t)\|_p^2\Big)  dt < \infty\ \ \ \ \ \text{if} \ \ a = 1/2      \label{nd11p}
\eeq
and
\begin{align}
\int_0^T \big( \|B_C(t)\|_p + \|\phi(t)\|_p\Big) dt < \infty\ \ \   \text{for}\ \ 1/2 \le a <1.  \label{nd12p}
\end{align}
Furthermore, for $1/2 \le a <1$, one has 
\begin{align}
\|B_C(t)\|_p +\| \phi(t)\|_p = t^{-1 +(3/2p) + \frac{a- (1/2)}{2}}\ o(1)) \ \ 
                  \text{as}\ \ t\downarrow 0.       \label{nd13p}
\end{align}     
In particular,
\begin{align}
\|B_C(t)\|_p = o(t^{-1 +(3/2p)})\ \ \text{as} \ \ t\downarrow 0\ \   \text{if}\ \ a = 1/2. \label{nd15p}
\end{align}
\end{theorem}
Theorems \ref{thmnd1} and \ref{thmndp}  and Corollary \ref{cornd2} will be 
proven in the remainder of this subsection.

\begin{remark}\label{stratnd} {\rm (Strategy) The method of proof of these theorems  depends  on showing 
that $|\phi(t,x)|$ and $|B_C(t,x)|$ satisfy partial differential inequalities of the form
$\p f/\p t \le \Delta_N f +$ non-linear terms, where   $\Delta_N$ is  the 
Neumann Laplacian  on real valued functions over $M$ 
(or simply the Laplacian on real valued functions if $M = \R^3$.
  Unfortunately, $\phi(0,x)$ and $B_C(0, x)$ are, typically,  distributions lying
 in $H_{-1/2}(M)$ and therefore use of their absolute values at time zero seems infeasible.
 Instead, we are going to represent $\phi(\sigma, x)$ as the solution to \eref{vs510}
 over an interval $[s, t]$ with initial data   $\phi(s,x)$ and with $s >0$.
 We can then apply the Neumann domination techniques developed in
  \cite{CG2} over the interval  $[s,t]$ to derive $s$-dependent inequalities  for $|\phi(t,x)|$,
  which  can then be   averaged with respect to $s$ over the interval $(0, t)$. 
    
     The $L^p$  bounds, for large $p$,  that we will  derive from these Neumann
 domination   inequalities   will rely on the energy estimates made in the
  previous subsections for low $p$.
}
\end{remark}

\subsubsection{Neumann domination with averaging.}

 \begin{proposition} \label{propnda} $($Neumann domination with averaging$)$ 
 Assume that $M =\R^3$ or is the closure of a bounded open set in $\R^3$ with smooth boundary.
  Suppose that
  $A:(0, T] \rightarrow C^1(M; \L^1\otimes \kf)$ is a time dependent  1- form on $M$ which is continuous in the time variable.
 Let $\w: (0,T) \rightarrow C^2( M;\L^p \otimes \kf)$ be a time dependent, $\kf$ valued, p-form on $M$ which is continuously differentiable in the time variable and satisfies the equation
 \begin{align}
 \w'(s,x)  =\sum_{j=1}^N(\n_j^{A(s)} )^2 \w(s,x) + h(s, x),  \ \ \ 0 < s <T,   \label{nda5}
 \end{align}
 where $h \in C((0, T]\times M; \L^p \otimes \kf)$.  If $M \ne \R^3$ then assume also that 
 \begin{align}
 \n_{n} |\w(s,x)|^2 \le 0\ \ \ \text{for}\ \ \ 0< s <T,  \ \ \  x\in \p M.             \label{nda6}
 \end{align}
 Then 
 \begin{align}
 |\w(t,x)| &\le t^{-1} \int_0^t e^{(t-s)\Delta_N} |\w(s, \cdot)| ds\ (x)    \notag\\
 &\qquad\qquad +  t^{-1}\int_0^t  e^{(t-s)\Delta_N} s|h(s, \cdot)| ds\ (x)  .        \label{nda7}
 \end{align} 
 If $M = \R^3$ then the Neumann Laplacian in \eref{nda7} should be replaced by the self-adjoint
 version $\Delta$ over $\R^3$.
 \end{proposition}
        \begin{proof} The proof of  \cite[Proposition 2.7]{CG2}  shows        
  that  under the hypotheses of this Proposition    there holds
    \begin{align}
 |\w(t,x)| &\le e^{(t-s)\Delta_N} |\w(s, \cdot)| \ (x)    \notag\\
 &\qquad\qquad 
 + \int_s^t  e^{(t-\sigma)\Delta_N} |h(\sigma, \cdot)| d\sigma\ (x) , \ \   0 < s < t < T .      \label{nda8}
 \end{align}   
 One need only take the origin  in \cite[Proposition 2.7]{CG2} to be $s$ in our present setting. 
 The statement of  \cite[Proposition 2.7]{CG2} includes the assumption that $M$ is convex,
 which is used only to show that  our hypothesis \eref{nda6} holds in the cases of interest.
We will prove separately, in Lemma \ref{normderiv}, that \eref{nda6} 
holds for our circumstances.  
The statement of  \cite[Proposition 2.7]{CG2}  further hypothesizes that  $M$ is compact.
But when $M=\R^3$ the proof given there applies even more easily because  one need
not be concerned with boundary conditions. 
Instead one can allow $|\w| + |grad\ \w| \in L^2(\R^3)$ or even mild growth (e.g. polynomial)
 of these function as  $x\to \infty$. 
 These conditions  will be satisfied for the functions $\w = B$ or $\w = \phi$ of interest to us.

 The left side of \eref{nda8} is independent of $s$. We may therefore average 
 \eref{nda8} over the interval $(0, t)$ to find
 \begin{align}
 |\w(t,x)| &\le t^{-1} \int_0^t e^{(t-s)\Delta_N} |\w(s, \cdot)| ds\ (x)    \notag\\
 &\qquad\qquad 
  +  t^{-1}\int_0^t \int_s^t e^{(t-\sigma)\Delta_N} |h(\sigma, \cdot)| d\sigma ds \ (x).   \label{nda9}
 \end{align} 
 Since $e^{(t-\sigma)\Delta_N}$ is a positivity preserving operator we can reverse the
 $\sigma$ and $s$ integrals in the last line to find 
 $t^{-1}\int_0^t e^{(t-\sigma)\Delta_N}\sigma |h(\sigma, \cdot)| d\sigma$. 
 This proves   \eref{nda7}. 
 \end{proof}

\subsubsection{Pointwise bounds} 
The proofs of Theorems \ref{thmnd1}  and \ref{thmndp} 
 depend on the following representation inequality.

            \begin{theorem}\label{ndom2} $($Pointwise bounds$)$
        Assume that $M$ is as in the statement of Theorem \ref{thmnd1}.
            Let $C(\cdot)$ be a smooth solution over $(0,T)$ to the
 augmented equation \eref{aymh}
satisfying either Neumann or Dirichlet boundary conditions,
 \eref{ST11N}, resp. \eref{ST11D} in case $M\ne \R^3$.   Then,
for $0 < t <T$,  the following pointwise bounds hold.
\begin{align}
|B_C(t,& x)| \le \frac{1}{t} \int_0^t e^{(t-s)\Delta_N} | B_C(s, \cdot)| ds\ (x)    \notag\\
&+ \frac{1}{t} \int_0^t e^{(t-s)\Delta_N}\  
  s \ \Big| B_C(s)\# B_C(s) - [ B_C(s), \phi(s)] \Big|   ds  \ (x)  \label{nd20} 
\end{align}
and
\begin{align}
|\phi(t,x)| \le \frac{1}{t} \int_0^t e^{(t-s)\Delta_N}&|\phi(s, \cdot)| ds\  (x)  \notag\\
 +\frac{1}{t} &\int_0^t  e^{(t-s)\Delta_N}
  \  s\ \Big|\, [C(s) \lrc C'(s)]\, \Big|ds \ (x).      \label{nd21}  
\end{align}
In case $M = \R^3$ the Neumann Laplacian $\Delta_N$ should be replaced by the self-adjoint
Laplacian $\Delta$  here and in the following.  Here $\#$ denotes a pointwise product as in \eref{vs511}.
\end{theorem}

The proof of Theorem \ref{ndom2} depends on the following lemma.

\begin{lemma}\label{normderiv}$($Normal derivatives$)$  Assume that $M$ is as in
 the statement of Theorem \ref{thmnd1} but $M\ne \R^3$. 
 Let $C(\cdot)$ be a smooth solution to \eref{aymh} over $(0, T)$ satisfying either
 Neumann  boundary conditions  \eref{ST11N} or Dirichlet boundary 
 conditions \eref{ST11D}. Then
 \begin{align}
 \n_n |B_C(t)|^2 &\le 0,\ \ \ 0 < t < T\ \ \ \text{and}    \label{nda50}\\
 \n_n |\phi(t)|^2 &= 0 , \ \ \ 0 < t < T .                         \label{nda51}
 \end{align}
\end{lemma}
\begin{proof} For fixed $t \in (0, T)$ let $\w = B_C(t)$. 
In the case of Neumann boundary conditions \eref{ST11N} we have 
$\w_{norm} =B_C(t)_{norm} =0$
by \eref{ST11N} and $(d_C\w)_{norm}=0$ by the Bianchi identity.
 Therefore we may apply
\cite[Corollary 2.4]{CG2} to find \eref{nda50} in the case of Neumann
 boundary conditions.
 
 In the case of Dirichlet boundary conditions we have $C(t)_{tan}=0$ by \eref{ST11D}
 and therefore  $dC(t)_{tan}=0$ by \cite[Equ. (3.19)]{CG1}. 
 Since also $(C(t)\wedge C(t))_{tan} =0$
it follows  that   $B_C(t)_{tan}=0$.      In order to apply   \cite[Corollary 2.4]{CG2}
we need only show that $(d_C^*B_C(t))_{tan}=0$. But the differential equation
\eref{aymh} shows that $(d_C^*B_C(t))_{tan}= - C'(t)_{tan} - (d_C \phi(t))_{tan}$. The
first term is zero by differentiation of   $C(t)_{tan}=0$.   The second term is zero
by virtue of    \cite[Equ. (3.19)]{CG1}, since $\phi(t)_{tan}=0$ by
 the assumption   \eref{ST11D}. Thus \eref{nda50} holds for both Neumann and Dirichlet boundary conditions.
 
 The proof of the identity \eref{nda51} for the zero form $\phi(t)$ does not require convexity of $M$, unlike the proof of \eref{nda50}:
    At any boundary point $x$, the normal derivative of $|\phi(t,x)|^2$ is given by
    \begin{align*}
    \n_n | \phi(t, x)|^2 &= \< (d_C\phi(t,x))_{norm} , \phi(t, x)\>_\kf 
                    +\< \phi(t, x), (d_C\phi(t,x))_{norm}\>_{\kf}. 
   \end{align*}
   This is zero in the case of Dirichlet boundary conditions since
   $\phi(t) = d^*C(t) =0$ on $\p M$   by \eref{ST11D}. In the case of Neumann boundary
    conditions we need to use 
   the differential equation \eref{aymh}, which shows that 
   $ (d_C\phi(t))_{norm} = - C'(t)_{norm} - (d_C^* B_C(t))_{norm}$.
   The first term is zero by virtue of \eref{ST11N}. The second term is zero  by 
   \cite[Equ. (3.20)]{CG1}, since $B(t)_{norm}=0$. 
\end{proof}

\bigskip
\noindent
         \begin{proof}[Proof of Theorem \ref{ndom2}] Take $\w(s) = B_C(s)$ in 
         Proposition \ref{propnda}. The boundary condition \eref{nda6} is satisfied
          in this case by Lemma \ref{normderiv}. The identity \eref{vs511} shows that
          \eref{nda5} holds with $h(s) =  B_C(s)\# B_C(s) - [ B_C(s), \phi(s)] $. 
          The role of the connection form $A$ in Proposition \ref{propnda} is
           played here by $C$. \eref{nd20} now follows from \eref{nda7}.
           
           For the proof of \eref{nd21} 
      take $\w(t) = \phi(t)$ in Proposition \ref{propnda}.
The required boundary condition \eref{nda6} is satisfied, in accordance with Lemma 
\ref{normderiv}. The identity \eref{vs510} shows that \eref{nda5} holds
 with $h(s) = -[C(s)\lrc C'(s)]$. \eref{nd21}  now follows from \eref{nda7}.
\end{proof}

\subsubsection{A convolution inequality and energy bounds}

  \begin{lemma}\label{lemu1}$($A convolution inequality$)$ 
          Let $0 \le c <1$. Suppose that  $\alpha$ and $\beta$ are
          non-negative functions  on $(0, T]$ such that
\begin{align}
\alpha(t) \le (1/t)&\int_0^t (t-s)^{-c}  \beta(s) ds\ \ \ \text{for}\ \ \ 0 < t \le T .  \label{u31}
\end{align}
Then for any real number $b < 2c+1$ there holds
\begin{align}
\int_0^T t^b \alpha(t)^2 dt \le \gamma \int_0^T s^{b -2c} \beta(s)^2 ds   \label{u32}
\end{align}
for some constant $\gamma$ depending only on $b$ and $c$.
\end{lemma}
          \begin{proof} Choose $r \in [0, 1)$ such that $b < 2c +r$. 
          By \eref{u31}, the Schwarz inequality and \eref{rec519} we have
      \begin{align*}
t^b\alpha(t)^2 &\le t^{b-2} \Big(\int_0^t(t -s)^{-c} s^{-r}ds\Big) 
\Big(\int_0^t (t-s)^{-c}  s^r\beta(s)^2 ds \Big) \\
&=   C_{c,r}\Big( t^{b-2} t^{1-c-r} \Big) \Big(\int_0^t (t-s)^{-c}  s^r\beta(s)^2 ds \Big).
\end{align*}     
Therefore, substituting $t = s/u$ in the second line below, we find  
\begin{align*}
\int_0^T t^b \alpha(t)^2 dt 
        &\le C_{c,r} \int_0^T   t^{b-1 -c-r }\Big(\int_0^t (t-s)^{-c}   s^r\beta(s)^2 ds \Big) dt\\
 &= C_{c,r}\int_0^T \Big(\int_s^T t^{b-1-c-r} (t-s)^{-c}dt\Big)  s^r\beta(s)^2 ds \\
&= C_{c, r}\int_0^T  s^{b-2c-r }\Big(\int_{s/T}^1  u^{2c + r -b -1} (1-u)^{-c} du\Big)    s^r\beta(s)^2 ds \\
&\le C_{c, r}\int_0^T  s^{b-2c-r }\Big(\int_{0}^1  u^{2c + r -b -1} (1-u)^{-c} du\Big)    s^r\beta(s)^2 ds \\
& =C_{c, r}C_{c,1 -(2c+r -b))} \int_0^T s^{b-2c-r }\    s^r\beta(s)^2 ds\\
&=  \gamma \int_0^T s^{b-2c } \beta(s)^2 ds,
  \end{align*}  
  wherein we have used \eref{rec519} in the fourth line with $\mu = c$ and $\nu =    1 -(2c+r -b) <1$.
  The coefficient $\gamma$ depends on the choice of $r \in [0,1)$. For definiteness we can choose 
  $r =0$ in case $b - 2c <0$ and we can choose $r$ midway between $b-2c$ and $1$ if $b - 2c >0$.
  In either case we have $b-2c < r$, as required by this proof.  
 \end{proof}

\begin{lemma} \label{lemenbds} $($Energy bounds$)$ Assume that $M=\R^3$ or is the closure
of a bounded, convex,  open set in $\R^3$ with smooth boundary. 
For $ 1/2 \le a <1$ there holds
\begin{align}
\int_0^T s^{2-a}\Big( \|B_C(s)\# B_C(s)\|_2^2 + \| \, [ B_C(s), \phi(s)]\, \|_2^2 
\Big) ds < \infty.  \label{nd80}\\
s\Big( \|B_C(s)\# B_C(s)\|_2 + \| \, [B_C(s), \phi(s)]\, \|_2\Big) 
       = o(s^{a - (3/4)})\ \          \text{as}\ \ s\downarrow 0.                                   \label{nd81}\\
\int_0^T s^{2-a}\|\, [ C(s)\lrc C'(s)]\, \|_2^2 ds < \infty.                                               \label{nd82}\\
s \|\, [ C(s)\lrc C'(s)]\, \|_{3/2} = o(s^{a - (1/2)})   \ \ \text{as}\ \ s\downarrow 0.     \label{nd83}
\end{align} 
\end{lemma} 
  \begin{proof}  By H\"older we find 
$  \|B(s)\# B(s)\|_2^2  \le c^2 \|B(s)\|_4^4 
  \le c^2 \|B(s)\|_2 \|B(s)\|_6^3 $.
  Hence
  \begin{align}
 & \int_0^T s^{2-a}  \|B(s)\# B(s)\|_2^2 ds  \label{nd84} \\
  &\le c^2\int_0^T  s^{a-(1/2)}\Big(s^{(1-a)/2} \|B(s)\|_2\Big) 
                 \Big(s^{(2-a)/2} \|B(s)\|_6\Big) \Big(s^{1-a} \|B(s)\|_6^2 \Big) ds.        \notag
  \end{align}
  The first factor in parenthesis is bounded by \eref{vs400a}.  The second factor in
   parenthesis is bounded by virtue of the energy estimate 
   \eref{vs641pa}. 
   The third factor is integrable by 
   \eref{vs550a}. Hence the integral is finite.
   Notice that if $a = 1/2$ then there is no room to spare in these estimates, whereas
   if $ a > 1/2$ there is an extra factor of $s$ with strictly positive exponent.
   
   Concerning the second term in \eref{nd80} we have 
   \begin{align*}
   \|\, [ B(s)&, \phi(s)]\, \|_2^2 \le \|B(s)\|_4^2 \|\phi(s)\|_4^2  \\
  & \le\Big(\|B(s)\|_2 \|\phi(s)\|_2\Big)^{1/2}  \Big(\|B(s)\|_6 \|\phi(s)\|_6\Big)^{1/2} 
                        \Big(\|B(s)\|_6 \|\phi(s)\|_6\Big).
   \end{align*}
   Distribute the available factor $s^{2-a}$ among the three factors,
   assigning $s^{(1-a)/2}$ to the first
   square root, $ s^{(2-a)/2}$ to the second square root, and $s^{1-a}$ to the last parenthesis,
   leaving a factor $s^{a- (1/2)}$ as before.  Again we find two  bounded  products times an integrable product, as before. This completes the proof of \eref{nd80}.
   
        The proof of \eref{nd81} follows from the same kind of estimates. For the
        first term in \eref{nd81}, it suffices to show that
          $s^{3/2 - 2a} s^2 \|B_C(s) \# B_C(s)\|_2^2 =o(1)$.
        But, as in the first line of this proof, we have
        \begin{align*}
 s^{3/2 - 2a} s^2 \|B_C(s) \# B_C(s)\|_2^2 \le   c^2 s^{3/2 - 2a} s^2\|B_C(s)\|_2 \|B_C(s)\|_6^3\\
 = c^2\Big( s^{(1-a)/2} \|B_C(s)\|_2\Big) \Big( s^{(2-a)/2}  \|B_C(s)\|_6\Big)^3
 \end{align*}
 and all factors are bounded, while the first is $o(1)$. As in the proof of \eref{nd80} a
  polarization-like   argument applies  to the second term in \eref{nd81} also.
  
  For a proof of \eref{nd82} observe that 
  \begin{align*}
  s^{2-a}&\|\, [C(s)\lrc C'(s)]\, \|_2^2 \le c^2 s^{2-a}\|C(s)\|_6^2 \|C'(s)\|_3^2 \\
  &\le   c^2 s^{a- (1/2)} \Big(s^{1-a}\|C(s)\|_6^2\Big) \Big(s^{(1-a)/2} \|C'(s)\|_2\Big)
  \Big(s^{(2-a)/2} \|C'(s)\|_6\Big).
  \end{align*}
  Over the interval $(0, T]$ the first factor is at most $T^{a-(1/2)}$, the second factor
  is bounded by  \eref{ST420a}  and the last two factors are square integrable by \eref{vs38a} and \eref{vs641a}.
  This proves \eref{nd82}.

     To prove \eref{nd83} observe that  \eref{ST420a} and \eref{vs640a} yield 
    $s\|\, [C(s)\lrc C'(s)]\, \|_{3/2} \le sc \|C(s)\|_6 \|C'(s)\|_2 
    \le s s^{(a-1)/2}\kappa_6|C|_t s^{(a-2)/2} = O(s^{a-(1/2)}) |C|_t$ for $ 0 < s \le t$.
  \end{proof}  

\subsubsection{Proof of high $L^p$ bounds} 

\begin{proof}[Proof of Theorem \ref{thmnd1}]  Combine the two terms in \eref{nd20} to find
\begin{align}
|B_C(t, x)| \le \frac{1}{t} \int_0^t e^{(t-s)\Delta_N} \beta(s)ds\ (x),    \label{nd98}
\end{align}
where
\begin{align}
\beta(s,x) = |B_C(s, x)| +  s\Big( |B_C(s,x) \# B_C(s,x)| + |\, [B_C(s,x), \phi(s,x)]\, |\Big).  
                                          \label{nd99}
\end{align}
 Since 
$\|e^{(t-s)\Delta_N}\|_{2\rightarrow \infty} \le c_1(t-s)^{-3/4}$ we find
\begin{align}
\|B_C(t)\|_\infty \le \frac{c_1}{t} \int_0^t (t-s)^{-3/4} \|\beta(s)\|_2 ds.   \label{nd100}
\end{align}
Take $c= 3/4$ in Lemma \ref{lemu1}  
and take $b = 3/2 -a$. Then  $b-2c = -a < 0$. We can therefore apply Lemma  \ref{lemu1}.
\eref{u32} yields 
\begin{align}
&\int_0^T t^{(3/2) - a} \|B_C(t)\|_\infty^2dt  
        \le \gamma c_1^2 \int_0^T s^{-a} \|\beta(s)\|_2^2  ds,  
                              \label{nd102} \\
&\le  c_2  \int_0^T\Big\{s^{-a} \|B_C(s)\|_2^2  +  s^{2-a}  \Big( \|B_C(s)\# B_C(s)\|_2  
                                  +  \|\, [ B_C(s), \phi(s)]\,\|_2\Big)^2\Big\} ds       \notag 
 \end{align}
 with $c_2 = 2\gamma c_1^2$. 
 The integral of the first term is finite because $C$ has finite a-action. The integral 
 of the second term is finite by \eref{nd80}. This proves  half of \eref{nd10}.
 
Starting with \eref{nd21},  the same argument as above shows that
\begin{align}
\|\phi(t)\|_\infty \le \frac{c_1}{t} \int_0^t (t-s)^{-3/4} \|\hat \beta(s)\|_2 ds,   \label{nd100f}
\end{align}
where
\begin{align}
\hat \beta(s,x) = |\phi(s, x)| + s |\, [C(s,x)\lrc C'(s,x)]\, |.             \label{nd99f}
\end{align}
Lemma \ref{lemu1} applies with the same values of $c$ and $b$ and shows that
\begin{align}
\int_0^T t^{(3/2) - a}  \|\phi(t)\|_\infty^2 dt 
\le \gamma c_1^2 \int_0^T s^{-a} \|\hat \beta(s)   \|_2^2 ds.         \label{nd50}
\end{align}
The right hand side of \eref{nd50} is finite since $\int_0^T s^{-a} \|\phi(s)\|_2^2 ds < \infty$
by finite a-action, while 
 $\int_0^T s^{2-a} \|\, [C(s)\lrc C'(s)]\, \|_2^2 ds <\infty $ is proved in the energy
 estimate \eref{nd82}. 
 This completes the proof of  \eref{nd10}. 
  Set $a = 1/2$ in \eref{nd10} to derive \eref{nd11}.  
  
   For the proof of  \eref{nd12}  let $f(t) = \Big(\|B_C(t)\|_\infty  + \|\phi(t)\|_\infty \Big)$.
     Then, by the Schwarz inequality,
    $\Big(\int_0^T f(t) dt\Big)^2  \le \int_0^T t^{a -(3/2)} dt \int_0^T t^{(3/2) -a} f(t)^2 dt < \infty$
    because $a > 1/2$ and \eref{nd10} holds. This proves \eref{nd12}.

  The bound \eref{nd13} is a pointwise (in $t$) bound rather than an integral bound
   and has a slightly different proof.    
  Return to \eref{nd100} and insert the following pointwise (in $s$) bounds on $\|\beta(s)\|_2$,
  which are a little different for the two types on terms in \eref{nd99}.
  We have $\|B_C(s)\|_2 = o(s^{(a-1)/2})$ by \eref{vs400a}.
   On the other hand \eref{nd81}
   shows that   
   $s\Big( \|B_C(s)\# B_C(s)\|_2 + \| \, [B_C(s), \phi(s)]\, \|_2\Big)  = o(s^{a - (3/4)})$.  
   Thus \eref{nd100} shows that 
  \begin{align*}
   \|B_C(t)\|_\infty &=t^{-1}\int_0^t (t-s)^{-3/4}\Big(  s^{(a-1)/2} +  s^{a - (3/4)}\Big)ds\  o(1) \\
   &= \Big( t^{-1 +\frac{a - (1/2)}{2}}  +  t^{a - (3/2)}             \Big)o(1).
 \end{align*} 
 Since $t^{a -(3/2)} = t^{-1} t^{\frac{a - (1/2)}{2}} O(1)$ the assertion \eref{nd13} follows.
Put $a = 1/2$ in \eref{nd13} to find \eref{nd15}.
 \end{proof}

\bigskip
\noindent
\begin{proof}[Proof of Corollary \ref{cornd2}]
Let $0 < \delta < T$. Over the interval $[\delta, T]$ the function $C(\cdot)$ is a strong solution
lying in $\P_{[\delta, T]}^a$ (with obvious meaning for this notation). 
Since $\|C(t)\|_{H_1}$ is bounded on this interval we have 
$\int_\delta^T (s - \delta)^{-a} \|C(s)\|_{H_1}^2 ds < \infty$ for any $a < 1$. That is, $C(\cdot)$
has finite strong a-action over the interval $[\delta, T]$. We can apply Theorem \ref{thmnd1}
and conclude from \eref{nd13} that $(t-\delta) \|B_C(t)\|_\infty$ is bounded  over $(\delta, T]$.
In particular, if we choose $\delta = \epsilon/2$ and restrict $t$ to $[\epsilon, T]$ we find
that $(\epsilon/2)  \|B_C(t)\|_\infty$  is bounded over this interval. This proves
Corollary \ref{cornd2}.
\end{proof}

\bigskip
\noindent
\begin{proof}[Proof of Theorem \ref{thmndp}]
In view of the heat kernel bound
 $\|e^{(t-s)\Delta_N}\|_{2\rightarrow p} \le c_1(t-s)^{-3/4 +(3/2p)}$, the inequality \eref{nd98}
 shows that
 \begin{align}
 \|B_C(t)\|_p \le \frac{c_1}{t} \int_0^t (t-s)^{-(3/4) +(3/2p)} \|\beta(s)\|_2 ds.   \label{nd100p}
 \end{align}
 In Lemma \ref{lemu1} choose  $c= 3/4 -(3/2p)$    and   $b = 3/2 -a - (3/p)$. We have again
 $b -2c = -a$, which is strictly negative.  Lemma \ref{lemu1} now shows that
 \begin{align}
 \int_0^T   t^{(3/2) - a -(3/p)} \|B_C(t)\|_p^2 dt 
       \le \gamma c_1^2 \int_0^T s^{-a} \|\beta(s)\|_2^2 ds.   \label{nd102p}
 \end{align}
The right side is the same as that of \eref{nd102}, which we have already proven to be finite.
 This proves half of \eref{nd10p}.  Similarly,  with these new values of  $c$ and $b$, the
 inequality \eref{nd100f} changes to 
 \begin{align}
\|\phi(t)\|_p\le \frac{c_1}{t} \int_0^t (t-s)^{-(3/4)+(3/2p)} \|\hat \beta(s)\|_2 ds,   \label{nd100p}
\end{align} 
and therefore, by Lemma \ref{lemu1},
\begin{align}
 \int_0^T   t^{(3/2) - a -(3/p)} \|\phi(t)\|_p^2 dt 
       \le \gamma c_1^2 \int_0^T s^{-a} \|\hat \beta(s)\|_2^2 ds.   \label{nd50p}
 \end{align}
 The right side has already been shown to be finite in the discussion after \eref{nd50}.
This completes the proof of \eref{nd10p}.

 Put $a =1/2$ in \eref{nd10p} to find \eref{nd11p}. 
      Since $(3/2) - a - (3/p) < 1$ for all $a \in [1/2, 1)$ the inequality  \eref{nd12p} follows
       from the Schwarz inequality and \eref{nd10p} just as in the proof of \eref{nd12}.

          For the proof of the pointwise bounds  \eref{nd13p} a slight deviation from
     these choices of $c$ will be needed.
 We can choose again $c = (3/4) - (3/2p)$  to find
   \begin{align*}
   \|B_C(t)\|_p &=t^{-1}\int_0^t (t-s)^{-(3/4) +(3/p)}\Big(  s^{(a-1)/2} +  s^{a - (3/4)}\Big)ds\  o(1) \\
   &= \Big( t^{-1 +(3/p) +\frac{a - (1/2)}{2}}  +  t^{a - (3/2) +(3/p)}   \Big)o(1) \\
   &= \Big(t^{-1 +(3/p) +\frac{a - (1/2)}{2}} \Big) o(1)\ \ \text{as}\ \ t\downarrow 0.
 \end{align*} 
Here we have used again $t^{a - (3/2)} = O(t^{-1 +\frac{a - (1/2)}{2}})$. 
This proves half of \eref{nd13p}

For the corresponding bound on $\|\phi(t)\|_p$ we must go back to the Neumann pointwise 
bound \eref{nd21}, which we may write as
\begin{align}
|\phi(t,x)| \le \frac{1}{t} \int_0^t e^{(t-s)\Delta_N} \hat \beta(s)ds\ (x),     \label{nd60p}
\end{align}
with $\hat \beta$ as given in \eref{nd99f}.  For the first term in $\hat \beta$ we have
  $\|\phi(s)\|_2 = o(s^{(a-1)/2})$  because $C \in \P_T^a$. Therefore,
choosing $c  = (3/4) - (3/2p)$ again, we find
\begin{align}
\Big\| \frac{1}{t} \int_0^t e^{(t-s)\Delta_N} |\phi(s)|\ ds\ (\cdot)\Big\|_p
&\le   \frac{c_1}{t} \int_0^t (t-s)^{-(3/4) +(3/p)} \|\phi(s)\|_2 ds \notag \\
&=  \frac{1}{t} \int_0^t  (t-s)^{-(3/4) +(3/p)}s^{(a-1)/2} ds\ o(1) \notag \\
&= t^{-1 +(3/p) +\frac{a - (1/2)}{2}}\  o(1).  \notag
\end{align}
The second term in $\hat \beta$ must be estimated differently because we have only the 
$L^{3/2}$ bound  \eref{nd83}.
We must use the heat operator bound  
$\|e^{(t-s)\Delta}\|_{3/2 \rightarrow p} \le c_1 (t-s)^{ -1 +(3/p)}$. 
(We were not able to use this
  in case $p = \infty$ because the kernel $(t-s)^{-1}$ is not integrable.)
  Thus, in view of \eref{nd83},  we have 
  \begin{align*}
  \Big\| \frac{1}{t} &\int_0^t e^{(t-s)\Delta_N} |s [C(s)\lrc C'(s)]\, |\ ds\ (\cdot)\Big\|_p \\
  &\le \frac{c_1}{t}\int_0^t (t-s)^{-1 +(3/p)}  \|s [C(s)\lrc C'(s)]\,\|_{3/2} ds \\
  &\le  \frac{c_1}{t}\int_0^t (t-s)^{-1 +(3/p)} s^{a -(1/2)} ds\ o(1)\\
  &=  t^{a-(3/2) + (3/p)} \ o(1) ,
  \end{align*}
  which is also   $t^{-1 +(3/p) + \frac{ a-(1/2)}{2}} o(1)$.
This completes the proof of \eref{nd13p}  and of Theorem \ref{thmndp}.
\end{proof}

\begin{remark}\label{rem4th}
{\rm (Missing $\|\phi(t)\|_\infty$) Among the initial behaviors that have been described in Theorems  
\ref{thmnd1} and \ref{thmndp} the behavior 
$\|\phi(t)\|_\infty = o(t^{-1 + \frac{a- (1/2)}{2}})$ is noticeably missing.
It is the $\phi$ analog of \eref{nd13}.
 Our proof for $\phi$ is not
symmetrical to our proof for $B_C(t)$ because the energy bound  \eref{nd83},
with $3/2$ replaced by some $p > 3/2$,  would require third order energy 
estimates for $C$,     
 which are not in this paper.  We are forced thereby to use the
 index $3/2$ in \eref{nd83}.
But the heat operator bound $\|e^{t\Delta_N}\|_{3/2\rightarrow \infty} = O(t^{-1})$  is
not integrable and therefore cannot be used in the argument that produced \eref{nd13}.
 It is very likely that
third order energy estimates would succeed in proving this $\|\phi(t)\|_\infty$ bound.
But it is not needed in this paper. 
}
\end{remark}

\section{Gauge groups} \label{secgg}

\subsection{Notation and statements}  \label{secggstate} 

\begin{notation}\label{notgg1}{\rm    (Gauge Groups) 
  In this section 
  we will take $M$ to be either all of $\R^3$ or the closure of a bounded
   open set in $\R^3$ with smooth boundary.  We will not require $M$ to be convex. 
    Denote by $\Delta$ the self adjoint version of the Laplacian on $\kf$ valued
     1-forms on $\R^3$ 
 in case $M =\ R^3$, or  the  Dirichlet or Neumann   Laplacian  on $L^2(M; \L^1\otimes\kf)$
 in case $M \ne \R^3$.
    The Dirichlet and Neumann domains were  defined in Definition \ref{defbc}. 
   See \cite{CG1} for further discussion of these domains.
   For a measurable function $ g: M\rightarrow K\subset End\ \V$ the weak 
   derivatives $\p_j g(x)$ are well  defined, $End\ \V$ valued distributions on $M^{int}$
    (or on $\R^3$ if $M = \R^3$). 
          We will say that $g \in W_1(M; K)$ if 
  $\|g- I_\V\|_2 < \infty$ and the derivatives
  $\p_j g \in L^2(M; End\ \V)$.    
  If $g \in W_1$ and $M \ne \R^3$ then the restiction $g|\p M$ is well defined a.e. with
 respect to surface  measure by a Sobolev trace theorem.

  Write  $g^{-1}dg$ for the 1-form $ \sum_{j=1}^3\Big(g(x)^{-1}\p_j g(x)\Big) dx^j$.
     The coefficients  $g(x)^{-1}\p_j g(x)$  lie in $\kf \subset End\, \V$ for a.e. $x \in M$.
      Thus $g^{-1} dg$ is an a.e. defined  $\kf$
     valued 1-form on $M$.       Let  $D = (1 - \Delta)^{1/2}$. 
     We may  apply powers of the operator  $D$ to the $\kf$ valued 1-form $g^{-1}dg$
    and will write       $g^{-1}dg \in H_a$ if 
   $g^{-1}dg \in \D (D^a)$. 
   Define
   \beq
   \| g^{-1} dg\|_{H_a}= \| D^a (g^{-1}dg) \|_2,  \ \ \ \ 0 \le a \le 1.   \label{gp0}
   \eeq
  This norm has already been defined for general $\kf$ valued 1-forms in Definition \ref{defbc}. 
The Sobolev space $H_a = H_a(M;\L^1\otimes \kf)$ encodes Neumann or Dirichlet boundary
conditions in accordance with Definition \ref{defbc} when $M\ne \R^3$ and $1/2 \le a \le 1$.

In addition to the sets of gauge functions $g \in W_1(M;K)$ for which $\|g^{-1}dg \|_{H_a} < \infty$
we will need to use sets of gauge functions  $g \in W_1(M;K)$  for which
 $\| g^{-1}dg\|_{L^p(M;\L^1\otimes \kf)} < \infty$. Our proofs will make important use
  of these as preliminary target spaces for the gauge functions arising in the ZDS procedure.
   We want to consider the following six kinds of sets of gauge functions. Some of these sets 
    will be shown to be groups under pointwise multiplication.

\bigskip
\noindent   
For $0 \le a \le 1$ we let
 \begin{align}
  \G_{1+a}&(\R^3) 
            =\Big\{g \in W_1(\R^3;K):  g^{-1}dg \in H_a(\R^3; \L^1\otimes \kf)\Big\}  \label{gp3a} \\
 &\qquad\qquad \text{and, if} \ M \ne \R^3 \notag\\
 \G_{1+a}^N
    & =  \Big\{g \in W_1(M;K):  g^{-1}dg \in H_a(M; \L^1\otimes \kf)\Big\},\ \ \qquad \qquad \label{gp3aN}\\
  \G_{1+a}^D& =  \Big\{g \in W_1(M;K):  g^{-1}dg \in H_a(M; \L^1\otimes \kf),  \label{gp3aD}
                            \ g =I_\V\  \text{on}\ \p M\Big\}.
 \end{align}
 For $ 2 \le p \le \infty$ we let
 \begin{align}
 \G_{1,p}&(\R^3)
               =\Big\{g \in W_1(\R^3;K):  g^{-1}dg \in L^p(\R^3; \L^1\otimes \kf)\Big\}, \label{gp3p} \\
             &\qquad\qquad \text{and, if} \ M \ne \R^3                              \notag\\
 \G_{1,p}^N& =  \Big\{g \in W_1(M;K):  g^{-1}dg \in L^p(M; \L^1\otimes \kf)\Big\},\ \ \label{gp3pN} \\
  \G_{1,p}^D& =  \Big\{g \in W_1(M;K):  g^{-1}dg \in L^p(M; \L^1\otimes \kf), 
                            \ g =I_\V\  \text{on}\ \p M\Big\}.   \label{gp3pD}
 \end{align}

In case $p =\infty$ we require also that $g^{-1}dg$ be continuous.
Henceforth $\G_{1+a}$ will refer to any of the three sets \eref{gp3a} - \eref{gp3aD}
and $\G_{1,p}$ will refer to  any of the three sets \eref{gp3p} - \eref{gp3pD}.
For functions $g$ and $h$ in $W_1(M; K)$ define 
\begin{align}
\ra (g,h) &=  \| g^{-1} dg - h^{-1} dh \|_{H_a} + \|g- h\|_2, \ \ 0 \le a \le 1\ \label{gp2am}
\end{align} 
and
\begin{align}
\rp (g, h) &= \| g^{-1} dg - h^{-1} dh \|_p + \|g -h\|_2  , \  \ \ 
                                                               2 \le p \le \infty.\                         \label{gp2pm}
\end{align}
   $\ra$ and $\rp$ are clearly metrics on the sets $\G_{1+a}$ and $\G_{1,p}$ respectively.
}
\end{notation}      
    We will prove that the sets $\G_{1,p}$ and $\G_{1+a}$
   are complete topological groups under pointwise multiplication in their
    respective metrics,  $\rp$  or $\ra$,   for $ 2 \le p \le  \infty$ and $ 1/2 \le a \le 1$.
    We will ignore the case $p =\infty$ in all of the following statements because the
    proofs  in this case are elementary.

\begin{theorem}\label{thmgg3p}
 $\G_{1,p}$ is a complete topological group  under pointwise multiplication in the
  metric $\rp$, when $2 \le p < \infty$.  
\end{theorem}

\begin{theorem}\label{thmgg3a}
 $\G_{1+a} $ is a complete topological group  under pointwise multiplication in the
  metric $\ra$,  when   $1/2\le a \le 1$.
\end{theorem}

\begin{theorem}\label{thmgg2p} Let $ 0 \le b \le 1$ and let $3 \le p \le \infty$. 
If $g \in \G_{1,p}$ then the adjoint action
\beq
u\mapsto (Ad\, g) u = g u g^{-1}, \ \ \ \ \ u \in H_b
\eeq
 is a bounded operator on $H_b$. 
 The representation 
 \beq
 \G_{1,p}\ni g \mapsto (Ad\, g :H_b\rightarrow H_b)
 \eeq
 is strongly continuous if $p =3$. It is norm continuous  if $p >3$ and $M$ has finite volume. 
\end{theorem}

 \begin{corollary}\label{corgg3a} Let  $0 \le b \le 1$. The representation
 \begin{align}
 \G_{1+a}\ni g \mapsto (Ad\, g: H_b\rightarrow H_b)
 \end{align}
 is strongly continuous if $ a = 1/2$ and norm continuous if $1/2 < a \le 1$.
\end{corollary}

The proofs will be given in the next four subsections. 
  
 \begin{remark} {\rm 
 The changeover from norm continuity to strong continuity  in Corollary \ref{corgg3a}
   as $a \downarrow 1/2$ is typical of the 
 contrasts that we have seen before  between $ a >1/2$ and $a = 1/2$. 
 By Sobolev, $\G_{1+a} \subset \G_{1,p}$ if $1/p = 1/2 - a/3$.  Thus $p=3$
 corresponds to $a = 1/2$ in the sense of these containments. Theorem \ref{thmgg2p}
 also shows this loss of norm continuity as $p\downarrow 3$.
 }
 \end{remark}

           \begin{remark}\label{rmkbdycons1} 
   {\rm (More about  boundary conditions for $\G_{1+a}$)  
If $g^{-1} dg \in H_a$ and  $a > 1/2$ then $g^{-1}dg|\p M$  
   is well defined almost everywhere on $\p M$ by well known Sobolev restriction theorems.
    In this case  the boundary condition 
   $(g^{-1}dg)_{tan} =0$ in the Dirichlet case \eref{gp3aD} is consistent with  the 
   condition $g|_{M} = I_\V$ in the definition \eref{gp3aD}.   
    In  the Neumann case \eref{gp3aN} one has $(g^{-1}dg)_{norm}=0$  if $a > 1/2$ and this is the only boundary condition forced on elements of $\G_{1+a}^N$ by \eref{gp3aN}  when $a > 1/2$.

    But in the critical case, $a = 1/2$, the restriction $g^{-1}dg|\p M$ is ill defined.
    Nevertheless the boundary conditions  $(g^{-1}dg)_{tan}=0$,  resp. $(g^{-1}dg)_{norm}=0$
     hold in a mean sense by Fujiwara's theorem \cite{Fuj}.   
     Thus an element in the
     space $\G_{3/2}$ (Neumann) satisfies Neumann boundary conditions in a mean sense
     because of the requirement that $g^{-1}dg \in H_{1/2}$(Neumann),
     while an element of $\G_{3/2}$(Dirichlet) satisfies both $g|\p M = I_\V$ 
     pointwise almost everywhere, and also $(g^{-1}dg)_{tan} =0$ in a mean sense. 
    These functional analytic meanings of the boundary
           conditions will not be needed in this paper, but will be needed in \cite{G73} (Localization), 
           where they will be discussed further.         
           In case $a < 1/2$ the spaces $H_a$ do not force any boundary conditions on $g^{-1}dg$.
}    
\end{remark}  
 
           \begin{remark}\label{rmkbdycons2} {\rm  (Boundary conditions for $\G_{1,p}$) 
The $L^p$ norm imposes no boundary conditions on $g^{-1}dg$.
Thus if $g \in \G_{1,p}^D(M)$ then  the definition \eref{gp3pD} imposes only  the
 boundary condition $g = I_\V$ on $\p M$,
 while if $g \in \G_{1,p}^N$ then no condition need be satisfied 
 at the boundary by $g$ or $dg$.
 }
 \end{remark}

\subsection{Multiplier bounds for $Ad\, g$}   \label{secmb}

The proof of Theorem \ref{thmgg3p} requires little more than  use of H\"older inequalities.
But the proof of Theorem \ref{thmgg3a} requires use of multiplier bounds on Sobolev spaces.
In three dimensions the Sobolev $H_{3/2}$ norm of a function just fails to control its
supremum norm, with the result that multiplication by such a function
 is not a bounded operator on Sobolev spaces. However we are interested
  in multiplication by the $End\, \V$ valued function $Ad\, g(x)$, which is a bounded
   function because $g(x)$ lies in the compact group $K$.  Consequently we are able to derive better multiplier bounds for these functions than one would expect in the critical case.

\begin{proposition}\label{propgp1}
                     $($Multiplier bounds for $Ad\, g$$)$  
Suppose that $g \in H_1(M;K)$ and that $g^{-1} dg \in L^3(M)$.  Let $b \in [0,1]$. 
Then, for any form $u \in L^2(M;\L^1\otimes \kf)$, there holds
\begin{align}
\|(Ad\, g) u\|_{H_b} 
          &\le  \Big( 1 + c_1\| g^{-1} dg\|_3\Big)\| u\|_{H_b}\ \ \ \ \ \ \text{and} \label{gp3} \\
\|(Ad\, g - 1) u\|_{H_b} 
&\le  \Big( \|Ad\, g -1\|_\infty + c_1\| g^{-1} dg\|_3\Big)\| u\|_{H_b}        \label{gp5}
\end{align}
for a constant $c_1$ depending only on the commutator bound $c$ and 
a Sobolev constant.
 Let $ 0 < \delta_1 < 3/2$. Define $p_1 = 3/\delta_1$. Then 
 \begin{align}
   \| (Ad\, g - 1)  u \|_{H_b} 
   \le \Big(\kappa_{\delta_1} \| Ad\, g - 1\|_{p_1}  
                              + c_1 \|g^{-1} dg\|_3\Big) \|u\|_{H_{b+\delta_1}} \ \label{gp8b}
\end{align}  
for some Sobolev constant $\kappa_{\delta_1}$.
\end{proposition}
 The proof depends on the following standard interpolation lemma. 
                 \begin{lemma} \label{cpxinterp}$($Complex interpolation$)$ Suppose that $S$ 
     is a bounded operator on a complex Hilbert space $H$.
Let $D$ be a self-adjoint operator on $H$ such that $ D \ge 1$ and $DSD^{-1}$ 
  is bounded. Then, for $0 \le b \le 1$, the operator 
 $ D^b S D^{-b}$ is bounded by $\max (\|S\|, \|DSD^{-1}\|)$.
          \end{lemma}
         \begin{proof} 
Let $u \in H$ and let $v$ be in
 the spectral subspace of $D$ for the interval $[1, \lambda]$ with $\lambda < \infty$.
   Then, for $z = x+iy$
 in the strip $ 0 \le x \le 1$, the function
 $f(z) \equiv (SD^{-z} u, D^{\overline z} v)$ is bounded and continuous and analytic in the interior.
  On the left hand edge of the strip we have $|f(0+iy)| \le \|S\| \|u\| \|v\| $
 because $D^{iy}$ is unitary.  On the right  hand edge of the strip we have
 \begin{align}
 |f(1 + iy)| = | ( SD^{-1} D^{-iy} u , D D^{-iy} v)|            \notag
  \le \|DSD^{-1}\| \| u\| \|v\| .
  \end{align}
  By the three lines theorem 
  $f(b + iy)$ is bounded by the maximum of the right
  sides of the last two inequalities. In particular, at $y= 0$, we have  
  $|(S D^{-b} u, D^b v)| \le \gamma \| u\| \|v\| $,
  where $\gamma$ is the maximum  of $\|S\|$ and $ \| DS D^{-1}\|$. 
  Since  this   holds for all  $u \in H$ and for all $v$ in a core for $D$, it follows that 
  $\| D^b SD^{-b}\| \le \gamma$.  
\end{proof}

            \begin{lemma}\label{lemgp2} Let   $D = (1 -\Delta)^{1/2}$ as in Section \ref{secggstate}
and let $u \in L^2(M; \L^1\otimes \kf)$.
If   $R = Ad\, g$ or  $R = Ad\, g - 1$ and $\delta\ge0$ then
 \begin{align}
 \|R D^{-\delta} u\|_{H_1} \le (m + c_1 \|g^{-1}dg\|_3) \| u\|_{H_1} ,          \label{gp9.0}
 \end{align}
 where $c_1 =2^{1/2} c\kappa_6$  and 
 \beq
 m = \max\{\| RD^{-\delta}| L^2(M; \L^j\otimes\kf)\| : j =0,1,2\}.            \label{gp9.00}
 \eeq 
              \end{lemma}
              \begin{proof} We will need to carry out the computations in terms of $d$ and $d^*$
              rather than in terms of the partial derivatives $\p_j$ because only the former
              respect the boundary conditions adequately. Let $h = g^{-1}dg$.
  Observe first the identities
 \begin{align}
 d\{(Ad\,g) u\} &= Ad\, g \Big(du + [h\wedge u] \Big) \ \ \           \text{and}                 \label{gp9.1}\\
 d^* \{(Ad\,g) u\} &= Ad\, g \Big(d^*u + [h\lrc u] \Big),                                      \label{gp9.2}
 \end{align}
 which may be derived as follows: For an element $\alpha \in \kf$ we have, at each
  point $x \in M^{int}$, 
  \begin{align*}
  \p_ j \{g(x)\alpha g(x)^{-1}\} &= g(x) [ g(x)^{-1}\p_jg(x), \alpha] g(x)^{-1}\\
  &= (Ad\, g(x))(ad\, h_j(x)) \alpha,
  \end{align*}
  where $h_j(x) = g(x)^{-1}\p_j g(x)$. Hence
  \begin{align*}
  d\{(Ad\, g)u\} =(Ad\, g) du +\sum_{j=1}^3 dx^j \wedge (Ad\, g) (ad\, h_j)  u,
  \end{align*}
  which is \eref{gp9.1}.  The derivation of \eref{gp9.2} is similar, using 
  $-d^*\{(Ad\, g)u\} = \sum_{j=1}^3 \p_j\{ (Ad\, g )u_j\}$, where $u = \sum_{j=1}^3 u_jdx^j$.
  
  Thus if either $R = (Ad\, g)$ or $R = Ad\, g -1$ we have
  \begin{align}
  d (Ru) &= Rdu + (Ad\, g)[ h\wedge u] \ \ \ \                       \label{gp9.3}\\
  d^*(Ru) &= R d^*u  +  (Ad\, g)[ h\lrc u]                             \label{gp9.4}
  \end{align} 
  for elements $u \in L^2(M;\L^1\otimes \kf)$ which are in the domain of $d$ and of $d^*$.
            Therefore 
      \begin{align*}
 &\|DRu\|_2^2 = \|dRu\|_2^2 + \|d^* Ru\|_2^2 + \|Ru\|_2^2 \\
 &\qquad \ \ \ \ \   = \|Rdu +  (Ad\, g) [ h\wedge u]\, \|_2^2 
                                                +\| R d^*u +(Ad\, g) [ h\lrc u]\,\|_2^2 + \|Ru\|_2^2 \\
 &\le \Big(\|Rdu\|_2 + \|\, [h\wedge u]\, \|_2\Big)^2 
         +\Big(\| R d^*u\|_2^2 +\|\, [h\lrc u]\, \|_2\Big)^2+ \|Ru\|_2^2\\
 & = \Big(\|Rdu\|_2^2 +\| R d^*u\|_2^2  + \|Ru\|_2^2\Big)  \\
  & \ \ \ \ \ \ + 2\|Rdu\|_2 \|\, [h\wedge u]\, \|_2 + 2\|Rd^*u\|_2 \|\, [h\lrc u]\, \|_2
           +  \|\, [h\wedge u]\, \|_2^2
        +  \|\, [h\lrc u]\, \|_2^2 \\
        &\le  \Big(\|Rdu\|_2^2 +\| R d^*u\|_2^2  + \|Ru\|_2^2\Big)  
        + 2\Big( \|Rdu\|_2^2 +\| R d^*u\|_2^2\Big)^{1/2}  \mu + \mu^2,
 \end{align*}
 where $\mu^2 = \|\, [h\wedge u]\, \|_2^2 + \|\, [ h\lrc u]\, \|_2^2$. Now
 $\mu^2 \le 2(c \|h\|_3 \|u\|_6)^2\le 2(c\kappa_6 \|h\|_3 \|Du\|_2)^2$. 
 So $\mu \le \nu \|Du\|_2$ where $\nu = 2^{1/2} c\kappa_6\|h\|_3= c_1\|h\|_3$. 
 Hence
 \begin{align}
\|DRu\|_2^2  &\le   \Big(\|Rdu\|_2^2 +\| R d^*u\|_2^2  + \|Ru\|_2^2\Big)   \notag\\
&+ 2\Big( \|Rdu\|_2^2 +\| R d^*u\|_2^2\Big)^{1/2}\nu  \|Du\|_2 + \nu^2 \|Du\|_2^2  . \label{gp9.6}
\end{align}
 Now  replace $u$ by $D^{-\delta} u$ and observe that $d$ and $d^*$ commute with $D^2$ and therefore also with $D^{-\delta}$.  Thus, with $m$ defined by \eref{gp9.00}, we have
 $\|Rd D^{-\delta} u\|_2 = \| R D^{-\delta} du\|_2 \le m \|du\|_2$ with similar inequalities
  when $d$ is replaced by $d^*$ or by $I$.  One should note that $D$ is acting in these inequalities on $0$, $1$ or $2$-forms. Then  we find 
 \begin{align}
 \|DRD^{-\delta} u\|_2^2   
        \le m^2 \| Du\|_2^2 + 2m \|Du\|_2 \nu \|D^{-\delta} Du\|_2 
         + \nu^2 \|D^{-\delta} Du\|_2^2        .                                               \label{gp9.7}
 \end{align}
Since $D \ge I$ on 1-forms we can drop the factor $D^{-\delta}$ in the last
 two terms of \eref{gp9.7} and then take the square root to find 
 $ \|DRD^{-\delta} u\|_2 \le (m + \nu) \|Du\|_2$, which is exactly \eref{gp9.0}. 
\end{proof}

\bigskip
\noindent
\begin{proof}[Proof of Proposition \ref{propgp1}]
Choose  $H$ in Lemma \ref{cpxinterp}  to be the complexification
of $L^2(M; \L^1\otimes \kf)$.
To prove  \eref{gp3} and \eref{gp5} take $\delta =0$ in  \eref{gp9.0}. First choose
$R= Ad\, g$ and $S = R$ in Lemma \ref{cpxinterp}.  Then 
$m \equiv \|R\|_{2\rightarrow2} = \|Ad\, g\|_{2\rightarrow 2} =1$. Thus 
\beq
\max (\|R\|_{2\rightarrow2}, \|DRD^{-1}\|_{2\rightarrow 2})
          = 1 + c_1 \| g^{-1}dg\|_3.
 \eeq
\eref{gp3} now follows from Lemma \ref{cpxinterp}.  

Now choose $R = Ad\, g - 1$ and again $S=R$. Since
$m \equiv \|R\|_{2\rightarrow2}  = \| Ad\, g -1\|_\infty$, the inequality \eref{gp5} now follows
the same way.

For the proof of \eref{gp8b} take $\delta = \delta_1$ and $R = Ad\, g -1$ in \eref{gp9.0}. 
We apply Lemma \ref{cpxinterp} this time to the operator $S = R D^{-\delta_1}$. 
In view of \eref{gp9.0}, we need only verify that
$\|R D^{-\delta_1}\|_{2\rightarrow 2} \le \kappa_{\delta_1} \|Ad\, g - 1\|_{p_1}$.
For this, observe that $q^{-1} + p_1^{-1} = 1/2$ implies that $q^{-1} = (1/2) -(\delta_1/3)$ and therefore, by H\"older and Sobolev, 
  \begin{align}
  \| RD^{-\delta_1} v\|_2 &= \|(Ad\, g - 1)D^{-\delta_1} v\|_2              \notag\\
  &\le  \|Ad\, g - 1\|_{p_1} \| D^{-\delta_1} v\|_q                  \notag\\
  & \le   \|Ad\, g - 1\|_{p_1} \kappa_{\delta_1} \|v\|_2          \label{gp16}
  \end{align}
  for all $v \in L^2(M; \L^j\otimes \kf)$, $j = 0 ,1,2$.
This completes the proof of Proposition \ref{propgp1}.
\end{proof}

\subsection{$\G_{1,p}$ and $\G_{1+a}$  are groups}   \label{secpagps}

            \begin{lemma}\label{lemgps} \   
  
  a$)$ For $2 \le p \le \infty$,  $\G_{1,p}$ is a group under pointwise multiplication.
  $\rho_p$ is right invariant and right translation
  is continuous.

  b$)$ For $1/2\le a \le 1$, $\G_{1+a}$ is a group under pointwise multiplication. Right translation is continuous.
  \end{lemma}
       The proof depends on the following properties of the metrics $\rho_p$ and $\rho_a$.

 \begin{lemma} \label{lemrpa} 
 $($Properties of $\rp$ and $\ra$.$)$
 Suppose that  $2\le p \le \infty$. If $g$ and $h$ are in $\G_{1,p}$ then
   \begin{align}
 \rho_p(gh,e) &\le \rho_p(g, e) + \rho_p(h,e)  \label{gp220}\\
 \rho_p(g^{-1}, e) &= \rho_p(g,e)     \label{gp221} \\
 \rho_p(gk, hk) &= \rho_p(g,h)\ \ \ \    \forall k \in \G_{1,p}    \label{gp222p} \\
 \rho_p(hkh^{-1}, e) &\le  \rho_p(k,e) + \|(Ad\, k - 1) h^{-1}dh\|_p  .    \label{gp223pp}
\end{align} 
 Suppose that $1/2\le a \le 1$.   
  If $g$ and $h$ are in $\G_{1+a}$  then 
  \begin{align}
\rho_a(gh, e)&\le  \rho_a(g, e) + \rho_a(h, e) + c_2 \rho_a(g, e) \rho_a(h, e)      \label{gp30a} \\
 \rho_a(g^{-1}, e) &\le \rho_a(g, e) + c_2 \rho_a(g, e)^2                             \label{gp31a} \\
 \rho_a(gk, hk) &\le (1 + c_2\rho_a(k, e)) \rho_a(g,h)\ \ \ \forall k \in \G_{1+a}     \label{gp28a}\\
  \rho_a(hgh^{-1}, e))                              
 &\le \Big(1 + c_1 \rho_3(h, e)\Big) \Big(\rho_a(g,e) + \| (Ad\, g^{-1} -1)(h^{-1}dh)\|_{H_{a}}\Big)
                   \label{gp33aa}
  \end{align} 
  with constants $c_1$ and $c_2$ depending only on Sobolev constants
   and the commutator bound $c$.
  Note:
  \eref{gp33aa} holds for all $a \in [0, 1]$ in this form. 
  \end{lemma}
                 \begin{proof}   The proof of each assertion relies on one of the following identities.                
 \begin{align}
  (hg)^{-1} d(hg)               &= g^{-1}dg + (Ad\, g^{-1}) ( h^{-1}dh)      \label{gp15}\\
  (hg^{-1})^{-1} d(hg^{-1}) &= (Ad\, g)( h^{-1} dh - g^{-1}dg)  \label{gp16}\\
(g^{-1})^{-1} d(g^{-1})       &   = -(Ad\, g)(g^{-1}dg)     \label{gp18} \\
 (hgh^{-1})^{-1} d(hgh^{-1})  
          &=(Ad\, h)\Big( (Ad\, g^{-1} -1)(h^{-1}dh) + g^{-1}dg\Big).  \label{gp19} 
\end{align} 
All of these follow from straightforward computations. We will derive only the last one. 
\begin{align}
 (hgh^{-1})^{-1} d(hgh^{-1})  
 &=hg^{-1} h^{-1}\Big((dh) g h^{-1} + h(dg) h^{-1} -hg(h^{-1} dh) h^{-1}\Big)\notag \\ 
 &= h\Big( g^{-1}(h^{-1}dh)g +  g^{-1}dg  - h^{-1} dh\Big)h^{-1}   \notag \\
          &=(Ad\, h)\Big( (Ad\, g^{-1} -1)(h^{-1}dh) + g^{-1}dg\Big).    \notag
\end{align} 
This proves \eref{gp19}.

For the derivation of \eref{gp220} to \eref{gp223pp}   one need only note that      
   $Ad\, g$ preserves all $L^p$ norms.   We see then that 
     \eref{gp220} follows from the identity \eref{gp15} together with the inequalities
     $ \|gh - I_\V\|_2 =\|(g-I_\V)h + h - I_\V\|_2 \le \| g - I_\V\|_2 + \| h - I_\V\|_2$.     
     \eref{gp221} follows from
   \eref{gp18} along with $\| g^{-1} - I_\V\|_2 = \|I_\V - g\|_2$.   
   \eref{gp223pp} follows from \eref{gp19}.    
   For the proof of \eref{gp222p}  we can compute that 
   \begin{align*}
   \rho_p(gk, hk) &= \| (Ad\, k^{-1})(g^{-1} dg - h^{-1}dh)\|_p +\|gk - hk\|_2\\
    &=  \|g^{-1} dg - h^{-1}dh\|_p + \| g-h\|_2= \rho_p(g,h).
   \end{align*}
   
The derivation of the simple properties \eref{gp220} to \eref{gp223pp} relied on the fact that
multiplication by $Ad\, g$ preserves $L^p$ norms.         
   By contrast,  multiplication by 
   $Ad\, g$ does not preserve the $H_a$ norms. 
   The  inequalities for $\rho_a$ will depend for their proofs on the multiplier bounds
   of Proposition \ref{propgp1}.    
                
            Apply \eref{gp3} with $b =a$ and the appropriate choice of 
            $u$ to the identities \eref{gp15}, \eref{gp18}, \eref{gp19}   to find           
 \begin{align}
&  \rho_a(hg, e)=\| (hg)^{-1} d(hg)\|_{H_a}   + \|hg - I_V\|_2               \label{gp15a}\\
 &\ \ \ \ \le \| g^{-1}dg\|_{H_a} + (1+c_1 \| g^{-1}dg\|_3) \| h^{-1}dh\|_{H_a} 
                                                                                       +  \| g - I_\V\|_2 + \| h - I_\V\|_2\notag\\
&  \rho_a(g^{-1}, e)=\|(g^{-1})^{-1} d(g^{-1})\|_{H_a}    
                                                                                        +\| g^{-1} - I_\V\|_2      \label{gp18a} \\
&\ \ \ \ \le (1+c_1 \| g^{-1}dg\|_3) \|g^{-1}dg\|_{H_a} +\| g - I_\V\|_2              \notag \\
&   \rho_a(gk, hk) = \| (Ad\, k^{-1})(g^{-1} dg - h^{-1}dh)\|_{H_a}  +\|gk - hk\|_2   \label{gp17a}\\
&\ \ \ \  \le (1 + c_1\|k^{-1}dk\|_3) \|g^{-1} dg - h^{-1}dh\|_{H_a} + \|g-h\|_2 \notag\\ 
&   \rho_a(hgh^{-1}, e) 
                        = \|(hgh^{-1})^{-1} d(hgh^{-1})\|_{H_a}  +\|hgh^{-1} - I_\V\|_2     \label{gp19a}\\
&\ \ \ \ \le\Big(1+c_1 \| h^{-1}dh\|_3\Big) \| (Ad\, g^{-1} -1)(h^{-1}dh) + g^{-1}dg\|_{H_{a}} 
                                                             +\|g-I_\V\|_2 \notag\\ 
 & \ \ \ \ \le\Big(1+c_1 \| h^{-1}dh\|_3\Big)
 \Big( \| (Ad\, g^{-1} -1)(h^{-1}dh)\|_{H_a}  +\|g^{-1}dg\|_{H_{a}}\Big) +\|g-I_\V\|_2 .   \notag
 \end{align} 
 Each of these inequalities holds for all $a \in [0,1]$. However we now wish to estimate
 several of the $L^3$ norms that occur in these inequalities by an $H_a$ norm. By Sobolev
 we have $\|u\|_3 \le \kappa_3 \|u\|_{H_{1/2}} \le \kappa_3 \|u\|_{H_a}$ if $a \ge 1/2$.
 Thus we may dominate the factors 
 $(1 + c_1 \|g^{-1}dg\|_3)$ by $(1 + c_2 \|g^{-1}dg\|_{H_a}) \le 1 + c_2\rho_a(g, e)$
 in \eref{gp15a} and \eref{gp18a} and dominate the factor $1 + c_1\|k^{-1} dk\|_3$
 by $1 + c_2 \rho_a(k, e)$ in \eref{gp17a}. This completes the proof of 
 \eref{gp30a} through \eref{gp28a}.  The inequality \eref{gp33aa} follows from \eref{gp19a}
 if one takes into account that $\rho_a(g,e)= \|g^{-1} dg\|_{H_a} +\|g-I_\V\|_2$.
\end{proof}

\bigskip
\noindent
\begin{proof}[Proof of Lemma \ref{lemgps}]  
$\G_{1,p}$ is closed under multiplication and inversion by \eref{gp220} and \eref{gp221}.
The identity  \eref{gp222p} shows that $\rp$ is a right invariant metric. The right invariance
   of $\rp$ ensures that right multiplication is continuous in  this metric. We will use 
   \eref{gp223pp} later to show that left multiplication is also continuous.

$\G_{1+a}$ is closed under pointwise multiplication and inversion by 
\eref{gp30a} and \eref{gp31a}. Although $\rho_a$ is not right invariant, the topology
 induced by the metric $\rho_a$ is invariant under right multiplication, as follows
  immediately from \eref{gp28a}. Hence right multiplication is continuous in $\G_{1+a}$.
\end{proof}

      To show that $\G_{1,p}$ and $\G_{1+a}$ are topological
  groups it still needs to be shown that left multiplication and inversion are  continuous.
   These will be proven in the   next sections.

\subsection{$\G_{1,p}$ is a topological group} \label{topgp-p}

 \begin{theorem} \label{thmG13}\ 
 For  $2 \le p < \infty$ 
 multiplication and inversion are continuous in $\G_{1,p}$.
 $\G_{1,p}$ is a topological group.    
  For any index $q \in [2,\infty)$ the map 
 \beq
\G_{1,p} \ni  g \mapsto  (Ad\, g:L^q \rightarrow L^q)      \label{gp205}
 \eeq
 is a strongly continuous representation of $\G_{1,p}$ into isometries of $L^q(M;\L^1\otimes \kf)$.
 Moreover if $ p > 3$ then the representation is norm continuous.
\end{theorem}

The proof depends on the following lemma. 
 
\begin{lemma}\label{strcontp}$($Strong and norm continuity on $L^q$.$)$ 
           Let $2 \le q < \infty$. 
 
 a$)$ Let $u \in L^q(M; \L^1\otimes \kf)$ and 
 let $\epsilon >0$. Then there exists $\delta >0$,  depending on $\epsilon$ and $u$, such that,
 for any function $g:M\rightarrow K$, one has
\begin{align}
\|(Ad\, g) u - u\|_q < \epsilon\ \ \ \ \text{whenever}\ \ \  \|g  -I_\V\|_2 < \delta.       \label{gp210}
\end{align}

b$)$ Let $p > 3$. Given $\epsilon >0$ there exists $\delta >0$ such that
\begin{align}
\| Ad\, g - 1\|_{L^q\rightarrow L^q} < \epsilon\ \ \ \ 
                         \text{whenever}\ \ \ \|g^{-1} dg\|_p + \|g  -I_\V\|_2  < \delta.      \label{gp211}
\end{align}
\end{lemma}
\begin{proof}  
If, for some $x \in M$, $|u(x)|_{\L^1\otimes\kf} \le \lambda$ then 
$|(Ad\, g(x)- I_\V) v(x)|_{\L^1\otimes \kf} \le 2 \lambda |g(x) - I_\V|_{End\ \V}$.
Here, $|\cdot |_{\L^1\otimes \kf}$ refers to the Euclidean norm on $\L^1\otimes \kf$.
 Thus
 \begin{align*}
 \| (Ad\, g - 1) u \|_q 
  & \le \|(Ad\, g -1) \chi_{|u| > \lambda} u\|_q + 2\lambda \| g- I_\V\|_q\\
   &\le 2 \|\chi_{|u| > \lambda} u\|_q +2\lambda \cdot 2^{1-(2/q)} \| g- I_\V\|_2^{2/q}. 
      \end{align*} 
   Hence, given $u \in L^q$ and $\epsilon >0$, choose $\lambda$ so large that 
   the first term is at most $\epsilon/2$ and then choose $\delta > 0$ so small that the second
   term is also less than $\epsilon/2$ when $\|g - I_\V\|_2 < \delta$. This proves \eref{gp210}.
   Notice that the restraint on $\delta$  depends on $\lambda$, which depends on $u$ 
   and not just on $\|u\|_q$.
   
    For the proof of \eref{gp211} assume that $p > 3$ and  observe first that 
 $\|d (g - I_V)\|_p = \|dg\|_p = \|g^{-1}dg\|_p$. Since $(1/p) - (1/3) < 0$, Sobolev's inequality
 shows that $\|g^{-1}dg\|_p + \| g - I_\V\|_2$ controls $ \|g - I_\V\|_\infty$ and therefore also
 $\| Ad\, g - 1\|_\infty$, which is the norm of $Ad\, g - 1$ as an operator on
  $L^q(M; \L^1\otimes \kf)$. 
 
 (Notice that adding  $\|g^{-1}dg\|_3$  to the norm in \eref{gp210} will not help to dominate
  $\|g - I_\V\|_\infty$  because   $(1/3) - (1/3) = 0$.)
 \end{proof}

 \bigskip
\noindent
               \begin{proof}[Proof of Theorem \ref{thmG13}] 
   Since the metric $\rho_p$ is right invariant, right translation is a 
   homeomorphism of $\G_{1,p}$, and therefore a neighborhood of a point
   $g_0$ can represented in the form $U(g_0) = \{ \alpha g_0: \rho_p(\alpha, e) < \delta\}$.
    If also $h_0 \in \G_{1,p}$ and we take $V = \{\beta h_0: \rho_p(\beta , e) < \delta\}$
     as a neighborhood of $h_0$ then a product
 of points in these neighborhoods may be written 
 $gh = (\alpha g_0) (\beta h_0) = \gamma g_0h_0$ where
 $\gamma = \alpha (g_0\beta g_0^{-1})$. 
 By \eref{gp223pp} we have
 \beq
 \rho_p(g_0 \beta g_0^{-1}, e) 
            \le \rho_p(\beta, e) + \| (Ad\, \beta -1) g_0^{-1}dg_0\|_p.              \label{gp215pp}
 \eeq
  By Lemma \ref{strcontp}, the entire
 right hand side of \eref{gp215pp} can be made small by choosing $\delta$ small. 
 Thus, in view of \eref{gp220}, given $\epsilon >0$  there exists 
 $\delta >0$ such that $\rp(\gamma, e) < \epsilon$. Multiplication is therefore jointly continuous.
 In particular left translations are homeomorphisms of $\G_{1,p}$. Since a right translate
 of a basic neighborhood $N$ of the identity  by an element $g \in \G_{1,p}$ is carried by inversion
 into a left translate by $g^{-1}$ of the inverse $N^{-1}$, which is itself open by \eref{gp221},
 it follows that inversion is continuous. Thus multiplication and inversion are continuous and so
 $\G_{1,p}$ is a topological group.
 
  For the proof of strong continuity of the representation $\G_{1,p} \ni g \mapsto Ad\, g$
  on $L^q(M; \L^1\otimes \kf)$   it suffices to show
   that for fixed $u \in L^q(M; \L^1\otimes \kf)$
    the map $\G_{1,p} \ni g \mapsto (Ad\, g)u$ is continuous into $L^q$  at $ g = I_\V$.
    But $\|g- I_\V\|_2 \le \rho_p(g, I_V)$ by the definition  \eref{gp2pm}. The strong continuity
    now follows from Part a) of Lemma \ref{strcontp}.    
    If $p>3$ then Part b) of Lemma \ref{strcontp}  shows that 
    the map $g\mapsto Ad\, g$ is actually norm continuous on $L^q$.
   \end{proof}

\subsection{$\G_{1+a}$  is a topological group if $a \ge 1/2$} \label{topgp-a}

\begin{theorem} \label{thmG3/2}$($$\G_{1+a}$ is a topological group$)$ \  
            If  $1/2 \le a \le1$ then multiplication and inversion are
 continuous in $G_{1+a}$. In particular $\G_{1+a}$ is  a topological group.
\end{theorem}

The critical case $a = 1/2$ will be the most delicate case in this theorem.
The proof depends on the following strong continuity lemma.

\begin{lemma}\label{strcontb}$($Strong continuity of  $\G_{1,3}$ on $H_b$.$)$          
         For $0 \le b \le 1$ the map 
         \begin{align}
    \G_{1,p} \ni g \mapsto (Ad\, g : H_b\rightarrow H_b)
    \end{align}
    is a strongly continuous representation of $\G_{1,p}$ into bounded operators on $H_b$ if
    $p = 3$. If $p >3$ and $M$ has finite volume then the representation is norm continuous.
\end{lemma}
      \begin{proof}
We already know from \eref{gp3} that $Ad\, g$ is bounded on $H_b$. 
For the proof of strong continuity suppose first that $p = 3$. Let $u \in H_b$ and $\epsilon >0$ be given. We need to show that there exists
$\delta >0$, depending on $\epsilon$ and $u$ such that
\beq
\|(Ad\, g- 1)u\|_{H_b} < \epsilon\ \ \ 
                 \text{whenever} \ \ \ \|g^{-1} dg\|_3 + \|g - I_\V\|_2  < \delta.      \label{gp50}
\eeq
Choose $\delta_1 \in (0, 3/2)$ and let $p_1 = 3/\delta_1$. 
Pick $ \lambda < \infty$
   such that $\| \chi_{[\lambda, \infty)} (D) u\|_{H_b} < \epsilon/6$.   
   Let $v =\chi_{[0, \lambda)} (D) u$ and $w = \chi_{[\lambda, \infty)}(D) u$. Then 
   $u = v + w$ is an orthogonal decomposition of $u$ in $H_b$. 
   Moreover $\|w\|_{H_b} < \epsilon/6$ and 
   $\|v\|_{H_{b+\delta_1}} \le \lambda^{\delta_1} \| v\|_{H_{b}}
                           \le \lambda^{\delta_1} \|u\|_{H_b}$ 
       by    the spectral theorem.  
   In view of \eref{gp8b} and \eref{gp5}, we have
          \begin{align*}
&\|(Ad\, g - 1) u\|_{H_b} \le \|(Ad\, g - 1)v\|_{H_b} + \| (Ad\, g - 1) w\|_{H_b} \\
&\le \Big(\kappa_{\delta_1} \| Ad\, g - 1\|_{p_1}  + c_1 \|g^{-1} dg\|_3\Big) \|v\|_{H_{b+\delta_1}}  \\
 &\ \ \ \ \qquad \qquad \qquad    + \Big( \|Ad\, g - 1\|_\infty + c_1 \|g^{-1}dg\|_3\Big)
                                  \|w\|_{H_b} \\
 &\le \Big(\kappa_{\delta_1} \| Ad\, g - 1\|_{p_1}  + c_1 \|g^{-1} dg\|_3\Big)
                                   \lambda^{\delta_1}  \|v\|_{H_b}
 +\Big( 2 + c_1 \|g^{-1}dg\|_3\Big) \epsilon/6.   
\end{align*}    
 Since $|g(x) - I_\V |_{op} \le 2$ we have the pointwise operator bound \linebreak
$|g(x) - I_\V|_{op}^{p_1} \le 2^{p_1-2} | g(x) - I_\V |_{op}^2$ for $2\le p_1 < \infty$.
Therefore $\|g - I_\V\|_{p_1} \le 2^{1-(2/p_1)} \|g - I_\V\|_2^{2/p_1}$.
Further, since $g$ is unitary (or  orthogonal),  \linebreak
$\| Ad\, g - I_{End\, \V}\|_{p_1} \le 2 \|g - I_\V\|_{p_1}$.  
Hence
 \begin{align}
 \|(Ad\, g - 1) u\|_{H_b} 
 &\le \Big(2^{2- (2/p_1)} \kappa_{\delta_1} \| g - I_\V\|_2^{2/p_1}  
                      + c_1 \|g^{-1}dg\|_3 \Big) \lambda^{\delta_1} \|u\|_{H_b}            \notag\\ 
 &+ \Big( 2 + c_1 \|g^{-1}dg\|_3\Big) \epsilon/6.        \label{gp55}
 \end{align}
 Thus if $\delta$ is chosen small enough in \eref{gp50}, and in particular $c_1 \delta < 1$, then
 the last term in \eref{gp55} will be at most $\epsilon/2$ while the first term on the 
 right side can also be made less than $\epsilon/2$. This proves strong continuity
 at the identity element of $\G_{1,3}$.
 
 In case $p >3$  we can simply use the crude estimate \eref{gp5} in place
  of the deeper estimate\eref{gp8b}  because $\|g- I_\V\|_2$, together with $\| dg\|_p$, which equals
  $ \|g^{-1} dg\|_p$,  control   $\| g-I_\V\|_\infty$ and therefore also $\| Ad\, g - 1\|_\infty$. 
 Thus, given $\epsilon >0$, the entire coefficient of $\|u\|_{H_b}$ on the right 
 side of \eref{gp5} will be less than $\epsilon$ if $\rho_p(g, e) < \delta$ for small enough $\delta$. Here we are using the
 finite volume of $M$ only to dominate the $L^3$ norm of $g^{-1}dg$ by the $L^p$ norm.
This proves norm continuity of the representation  $g \mapsto Ad\, g $ on $H_b$.
\end{proof}

\bigskip
\noindent
\begin{proof}[Proof of Corollary \ref{corgg3a}]
If $1/p = 1/2 - a/3$ then $\G_{1+a}$ embeds continuously
 and homomorphically into $\G_{1,p}$ for $a \in [1/2, 1]$. By Lemma \ref{strcontb} 
  the representation  $\G_{1+a} \ni g \mapsto Ad\, g$ on $H_b$ is therefore strongly
   continuous for $a = 1/2$ and
 norm continuous for $ a \in (1/2, 1]$.
 It is not necessary to specify that $M$ have finite volume to prove  the norm continuity 
 because $H_{a} \subset H_{1/2}$ if $a > 1/2$ and therefore $\G_{1+a} \subset \G_{3/2}$,
 while $\G_{3/2}$ 
controls $\|g^{-1}dg\|_3$, whose control was the only reason for requiring $M$ to have finite
volume in the last two lines of the previous proof. 
\end{proof}

\bigskip
\noindent
               \begin{proof}[Proof of Theorem \ref{thmG3/2}] 
               We need to prove that multiplication on the left is continuous. Surprisingly,
this soft sounding assertion seems to require use of the complex interpolation methods that
underlie  the multiplier bounds of Section \ref{secmb},   at least in the critical case $a = 1/2$.
Otherwise the proof is similar to that for $\G_{1,p}$.

   Since right translation is a homeomorphism of $\G_{1+a}$, a neighborhood of a point
   $g_0$ can represented in the form $U(g_0) = \{ \alpha g_0: \rho_a(\alpha, e) < \delta\}$. 
   If also $h_0 \in \G_{1+a}$ and $V = \{\beta h_0: \rho_a(\beta, e) < \delta\}$ is a 
   neighborhood of $h_0$ then a product
 of points in these neighborhoods may be written 
 $gh = (\alpha g_0) (\beta  h_0) = \gamma g_0h_0$ where
 $\gamma  = \alpha (g_0\beta g_0^{-1})$.  We need to show that $\gamma$ is close to the 
 identity if $\alpha$ and $\beta$ are.  Recall that $g_0$ and $h_0$ are fixed. 
 By \eref{gp33aa} we have
\begin{align}
 \ra(g_0 \beta g_0^{-1}, e) &\le( 1 + c_1 \|g_0^{-1} dg_0\|_3)   
    \Big(\rho_a(\beta, e) +\| (Ad\, \beta^{-1} -1) g_0^{-1}dg_0\|_{H_a} \Big).  \label{gp216}
 \end{align}
  When $\rho_a(\beta,e)$ is small,
 so is $\rho_a(\beta^{-1}, e)$, by \eref{gp31a} and therefore, by Lemma \ref{strcontb}, the entire
 right hand side of \eref{gp216} can be made small by choosing $\delta$ small.  
 Thus, in view of \eref{gp30a}, given $\epsilon >0$  there exists 
 $\delta >0$ such that $\ra(\gamma, e) < \epsilon$. Multiplication is therefore jointly continuous.
 In particular left translations are homeomorphisms of $\G_{1+a}$. Since a right translate
 of a neighborhood $N$ of the identity  by an element $g \in \G_{1+a}$ is carried by inversion
 into a left translate by $g^{-1}$ of the inverse $N^{-1}$, which is itself open by \eref{gp31a},
 it follows that inversion is continuous.
 \end{proof}

\subsection{Completeness}   \label{seccomp}

\begin{lemma}\label{lemcomp}$($Completeness$)$ $\G_{1+a}$ is complete for $ 1/2 \le a \le 1$.
\end{lemma}
\begin{proof} Suppose that $\{g_n\}_{n=0}^\infty $ is a Cauchy sequence in $\G_{1+a}$ with 
$ 1/2 \le a \le 1$.
Choose  a subsequence, denoted again the same way, such that
\beq
\rho_a(g_n , g_{n-1}) \le  1/2^n,\ \ \ n = 1,2,\dots             \label{gp70}
\eeq
Let
\beq 
u_n = g_n^{-1} dg_n - g_{n-1}^{-1} dg_{n-1}, \ \ \ n = 1,2, \dots    \label{gp71}
\eeq
Then 
\begin{align}
\|u_n\|_{H_a} \le 1/2^n     .                \label{gp77}
\end{align}
From \eref{gp71} we see that
\beq
g_{n}^{-1} dg_n = \sum_{k=1}^n u_k + g_0^{-1} dg_0.               \label{gp78}
\eeq
So  the sequence in \eref{gp78} is  convergent in $H_a$. There exists a unique
element $v \in H_a$ such that
\beq
v = \lim_{n\to \infty} \sum_{k=1}^n u_k + g_0^{-1} dg_0\ \ \ \ \ 
            \text{(convergence in $H_a$ norm)}.                                    \label{gp82}
\eeq
To construct the desired limit $g \in \G_{1+a}$ observe first that
\begin{align}
 \|dg_n - dg_k\|_2 &= \|g_n^{-1}(dg_n -dg_k)\|_2 \notag\\
 &=\|g_n^{-1} dg_n - g_k^{-1} dg_k +(g_k^{-1} - g_n^{-1}) dg_k\|_2\notag\\
 &\le \|g_n^{-1} dg_n - g_k^{-1} dg_k \|_2 
                                         + \|(I_\V - g_n^{-1}g_k ) (g_k^{-1} dg_k)\|_2 . \label{gp83}
 \end{align} 
 The first term on the right side of \eref{gp83}  goes to zero as $n, k \to \infty$
 because $\sum_1^\infty \|u_k\|_{H_a} < \infty$.  
 Since $g_k^{-1} dg _k$ converges in $H_a$ it also converges in $L^2(M)$. And, since $\| I_\V - g_n^{-1} g_k\|_2 = \| g_n - g_k \|_2 \rightarrow 0$, the factor $I_\V - g_n^{-1}g_k$
 converges to zero  in measure and boundedly. 
 Therefore the second term in line \eref{gp83} also goes to zero.
             Thus $\| dg_n - dg_k\|_2 \rightarrow 0$ and
$\|g_n - g_k\|_2 \rightarrow 0$. Hence there exists a function $g \in H_1(M; End\, \V)$ 
such that $\| g_n - g\|_{H_1} \rightarrow 0$.  Partly reversing the argument in \eref{gp83} we find
\begin{align*}
\| g_n^{-1} dg_n - g^{-1}dg\|_2  &= \| dg_n - g_n g^{-1} dg\|_2 \\
&\le \| dg_n - dg\|_2 + \|(I_\V -  g_n g^{-1} ) dg\|_2 .
\end{align*}
The first term on the right goes to zero, as we have just seen. Now $\|dg\|_2 < \infty$
while  $(I_\V -  g_n g^{-1} )$ converges to zero in $L^2$ and therefore in measure.
Since this factor is bounded by 2 we can now apply the dominated convergence
theorem to conclude that the second term goes to zero also.

    We already know that $g_n^{-1} dg_n$ converges to $v$ in the $H_a$ sense. Since
  it also converges to $g^{-1}dg$ in $L^2$ it follows that $v = g^{-1}dg$. In particular 
  $g \in \G_{1+a}$ and $\rho_a(g_n, g) \rightarrow 0$.
This completes the proof of Lemma \ref{lemcomp} and Theorem \ref{thmgg3a}.
\end{proof}

\begin{lemma}\label{lempcomp}  For $2\le p <\infty$     
 the groups $\G_{1,p}$ are complete.   
\end{lemma}
\begin{proof}   The proof is similar to the proof  for the groups $\G_{1+a}$:
Dropping to a subsequence such that $\rho_p(g_n, g_{n-1})\le 2^{-n}$  we see that
$\{g_n^{-1}dg_n\}$ is a Cauchy sequence  in $L^p$. The inequality
$\|dg_n - dg_k\|_p \le \|g_n^{-1} dg_n - g_k^{-1} dg_k \|_p 
                                         + \|(I_\V - g_n^{-1}g_k ) (g_k^{-1} dg_k)\|_p$
now shows that $\|dg_n - dg _k\|_p \rightarrow 0$ by the same argument
 used after \eref{gp83}, which uses our hypothesis in this lemma 
 that  $\|g_n - g_k\|_2 \rightarrow 0$.
 Therefore the sequence $\{g_n\}$ is itself a Cauchy sequence in the Sobolev space
\beq
\Big\{f:M\rightarrow End\ \V   \Big |\ \  
          \int_M \sum_{j=1}^3|\p_j f(x)|^p dx +\int_M |f(x)|^2dx < \infty\Big\}
\eeq
and so converges to  some function $g$ which takes its values in $K$ (because there is a subsequence which converges a.e.)  and such that $dg \in L^p(M; \L^1\otimes End\ \V)$
Now the inequality 
$\| g_n^{-1} dg_n - g^{-1}dg\|_p  
\le \| dg_n - dg\|_p + \|(I_\V -  g_n g^{-1} ) dg\|_p$ 
shows that $g_n^{-1} dg_n$ converges in $L^p$ to $g^{-1}dg$. 
Thus $g \in \G_{1,p}$ and $\rho_p(g_n,  g) \rightarrow 0$.
This completes the proof of Lemma \ref{lempcomp} and of Theorem \ref{thmgg3p}.
 \end{proof}

\begin{remark}{\rm We are only interested in the groups $\G_{1,p}$ for $p\ge 2$.
But the for $1 < p < 2$ Theorem \ref{thmgg3p} also holds if one uses the right invariant metric 
defined by  $\rho_p(g, e) = \| g^{-1} dg\|_p + \| g - I_\V\|_p$. 
 Proofs are similar. 
 }
\end{remark}

 \begin{remark}\label{remdiffstr} {\rm  (Differentiable structures,  
  Hilbert and Banach Lie groups) $\G_{1+a}$ is a
  Hilbert Lie group if $ a > 1/2$ and $\G_{1,p}$
   is a Banach Lie group if $ p > 3$. Both assertions follow from the fact that the metric
   on the gauge group controls $\|g - I_V\|_\infty$:   
    Let 
    \begin{align*}
    \mathcal{L}_a&\equiv \{ \alpha:M\rightarrow \kf |\  \ \|\alpha\|_{H_{1+a}} < \infty\}\ \  \text{and} \\ 
   \mathcal{L}_p&\equiv \{ \alpha:M\rightarrow \kf|\  \  \|\alpha\|_{H_{1,p}} < \infty\},
   \end{align*}
    where       we have written $\|\alpha\|_{H_{1,p}} = \|D \alpha\|_p$.
If $a > 1/2$ then  $\|\alpha\|_\infty \le const. \|\alpha\|_{H_{1+a}}$. Consequently $\mathcal{L}_a$
is closed under the pointwise Lie bracket operation and is a Hilbert Lie algebra.
    Similarly if $p > 3$ then $ \|\alpha \|_\infty \le const. \|\alpha\|_{H_{1,p}} $ 
     and so $\mathcal{L}_p$ is a Banach Lie algebra.
  The exponential map
    $\alpha \mapsto (g:x\mapsto \exp{\alpha(x)})$ maps a neighborhood of zero in the
     Lie algebra     $\mathcal{L}_a$, resp. $\mathcal{L}_p$, onto a neighborhood of 
    $I_\V$ in $\G_{1+a}$, respectively $\G_{1,p}$, which follows 
    from the fact that in the metric on the gauge group
    there is a neighborhood of $I_\V$ contained in $\{g: \|g - I_\V\|_{\infty} < \epsilon\}$, 
    as may be seen from the proof of Part b) of Lemma \ref{strcontp} and the continuity
     of the injection $\G_{1+a} \rightarrow \G_{1, p}$ for $p^{-1} = 2^{-1} -(a/3)$. 
     Thus 
    for small $\epsilon$ one can use the known surjectivity of the exponential map in $K$ to
    prove the existence of a function $\alpha$ such that $g(x) = \exp\alpha(x)$ for all $x \in M$.
    If $\phi$ is the inverse of the exponential map on a small neighborhood of the identity in $K$
    then the formula $\alpha(x) = \phi(g(x))$ transfers regularity of $g$ to the same Sobolev
    regularity of $\alpha$. Thus the tangent space at the identity of 
    $\G_{1+a}$, resp. $\G_{1,p}$ can be identified with  $\mathcal{L}_a$, resp. $\mathcal{L}_{p}$.
    We leave to the reader to verify that the  
    topologies on these two classes of groups, given respectively by the
     metrics \eref{gp2am} and \eref{gp2pm},
    agree with those induced by  the norms on the Lie algebras.
      
      This construction of a differentiable structure            
  breaks down in case $a = 1/2$ or $p=3$. It seems
                highly unlikely that in these critical cases there is a useful differentiable structure
                on $\G_{3/2}$ or on $\G_{1, 3}$. 
    The fact  that $\G_{3/2}$ and $\G_{1,3}$ are actually  topological groups
                  (i.e. products and inversion
                  are continuous) is thanks to our avoidance
   of the exponential map in Definitions \eref{gp3a} - \eref{gp2pm}.
          The Hilbert space $\mathcal{L}_{3/2}$ and Banach space $\mathcal{L}_{1,3}$
           are not closed under Lie bracket. 
           Nevertheless $\exp \mathcal{L}_{3/2} \subset \G_{3/2}$ and 
          $\exp \mathcal{L}_{1,3} \subset \G_{1,3}$. It will be shown elsewhere 
          that $\exp \mathcal{L}_{1,3}$
          does not cover any neighborhood of the identity in $\G_{1,3}$ if $K = SU(2)$.
          This strongly suggests that $\exp\mathcal {L}_{3/2}$ also does not cover any
           neighborhood           of the identity in $\G_{3/2}$.
      }
           \end{remark}
          
\begin{remark}\label{remgghistory} {\rm  (History)  Some of these gauge groups have been used in various contexts before. 
     In \cite{Seg126} I. E. Segal discussed possible choices for the phase space of a classical
     Yang-Mills field and chose the group that we have denoted by $\G_2$ in his definition of
      configuration space: Configuration space $=\{\text{some connection forms over $\R^3$}\}$ modulo $\G_2$.
     This use is similar  to our intended use \cite{G72}. Segal already pointed out in that paper  
     that  a Sobolev gauge group  
     is a  Hilbert or Banach manifold when the Sobolev norm controls the sup norm.

            As noted in Remark \ref{remgg},  K. Uhlenbeck pointed out in \cite{Uh2} that  
   in four dimensions the critical gauge group is $\G_2$ and multiplication fails to be continuous
   if one defines the topology by the exponential map. But in \cite{Uh3} she
   introduced  a different topology that made it into  a useful topological group. It's not clear how that
    topology is related to the topology given by $g^{-1}dg \in H_1$, which would be the four dimensional
    analog of the topology used in our groups $\G_{1+a}$.

     In \cite{Fre} Daniel Freed made use of some  one dimensional analogs of our groups $\G_{1+a}$
     for non-critical $a$. His interest 
     was the loop group $Map(S^1;K)$  with
     some Sobolev regularity imposed. The critical Sobolev index is  $1/2$ in one dimension instead of $3/2$.
     He was able to avoid attaching a  meaning to $H_{1/2}(S^1; K)$ as a group even though the
     linear space  $H_{1/2}(S^1;Lie\ K)$ was a central object of study in his work.

        The group that we have denoted by $\G_{1,2}$ has been used by 
        G. Dell'Antonio and D. Zwanziger, \cite{DZ2}, to give a very pretty proof 
        that every gauge orbit intersects $\{A: \int_{M} |A(x)|^2 dx < \infty\}$ at a point which
        minimizes this $L^2$ norm. $M$ can be a $d$ dimensional manifold. Their result
        illuminates the Gribov ambiguity.
}
\end{remark}

\section{The conversion group} \label{secconvgp} 

In this section we will take $M$ to be either all of $\R^3$ or the closure of a bounded,
 convex, open subset of $\R^3$ with smooth boundary.

\subsection{The ZDS procedure}       

\begin{definition}{\rm (Definition of $g_\epsilon$.) Suppose that $C(\cdot)$ is a smooth
 solution to the augmented Yang-Mills heat equation \eref{aymh} over $(0,T)$. 
 Let $\epsilon \in (0,T)$ and define
$g_\epsilon(t)$ to be the solution to the ODE, for each (suppressed) $x \in M$, 
\beq
\frac{dg_\epsilon(t)}{dt} g_\epsilon(t)^{-1} = d^*C(t), \ t \in (0, T),\ 
                                                   \ g_\epsilon(\epsilon) = I_\V.       \label{rec100}
\eeq
Then  $g_\epsilon \in C^\infty((0, T] \times M; K)$ because $d^*C(t,x)$ is smooth
 on $(0,T]\times M$.
}
\end{definition}

The ZDS procedure for recovering a solution to the Yang-Mills heat equation \eref{str5} from a solution to the augmented equation \eref{aymh} is outlined in the Introduction. Informally,
the function $g(t)$ defined in \eref{I8} is the function $g_\epsilon(t)$ for $\epsilon =0$.
As $\epsilon \downarrow 0$, however, the functions $g_\epsilon(\cdot)$ lose smoothness
in both space and time. This results from the strong singularity  of $d^*C(t)$ at $t =0$.
In case $a=1/2$ one has, typically, $d^*C(0) \in H_{-1/2}(M)$. In the next theorem we will
show that as $\epsilon \downarrow 0$ the functions $g_\epsilon(\cdot)$ converge uniformly
over $(0, T]$  as functions into the gauge group $\G_{1+a}$. Typically, a gauge 
function in  $ \G_{1+a}$
is  continuous on $M$ if $ 1/2 < a <1$ but not smooth. If $a = 1/2$ it need not even be 
continuous.

\begin{theorem} \label{thmgconv}$($The conversion group$)$
  Let $ 1/2 \le a < 1$ and $ 0 < T < \infty$. 
  Assume that either $M= \R^3$ or is the closure of  a bounded, convex,  open set in $\R^3$
  with smooth boundary.
     Let  $A_0 \in H_a$. 
Suppose that $C(\cdot)$ is the  strong solution
  of the augmented equation \eref{aymh}  over $[0,T]$ constructed in Theorem \ref{thmaug} 
  with initial value   $ C_0 = A_0$.  Assume that $C(\cdot)$ has finite strong a-action. $($This is automatic for $a > 1/2$.$)$
Define $g_\epsilon$ as in \eref{rec100}. 
Then $g_\epsilon(t) \in \G_{1+a}$ for each $t \in (0, T]$. Further, 

a$)$  $g_\epsilon: (0, T]\rightarrow \G_{1+a}$ is continuous.

 b$)$ There is a unique continuous function 
\beq
g:[0, T] \rightarrow \G_{1+a}                                                    \label{rec221a}
\eeq
such that 
\begin{align}
\lim_{\epsilon \downarrow 0} \sup_{0< t \le T} \ra(g_\epsilon(t), g(t)) = 0.   \label{rec221b}
\end{align}

c$)$ 
\beq
g(0,x) = I_\V\ \ \ \ \forall x \in M.                                                    \label{g20}
\eeq

d$)$ The function $h(t) \equiv g(t)^{-1} dg(t)$ is continuous on $[0, T]$ into $H_a$ and 
\beq
h(0) =0.                                           \label{h21}
\eeq

e$)$  For any time $\tau \in (0, T)$,  the function
\beq
k(t) \equiv      g(t) g(\tau)^{-1}                                                     \label{rec221c}
\eeq
is in $C^\infty( (0, T] \times M; K)$. Moreover $\lim_{t\downarrow 0} k(t) = g(\tau)^{-1}$,
with convergence in the sense of the metric group $\G_{1+a}$.
\end{theorem}

\begin{remark}{\rm (Strategy) The proof of the theorem will proceed in three steps. 
It will first be proven that the functions $g_\epsilon$ converge in 
a relatively weak sense, namely as functions into $L^p(M; End\, \V)$ for all $p < \infty$.
 This will then be used
to show that they are bounded as functions into the metric group $\G_{1,q}$ 
for $q$ appropriately related to $a$. This in turn will
 then be used to prove the strong sense of convergence asserted in Theorem \ref{thmgconv},
 namely as functions into the metric group $\G_{1+a}$.
}
\end{remark}

\begin{remark}\label{nosmooth}{\rm (Smoothness)
The singularity  in $d^* C(t)$
as $ t \downarrow 0$ reflects itself in a lack of smoothness  of $g(t, \cdot)$  for each $t >0$, 
not just ``near'' $t=0$. 
    We will see in Section \ref{secrec} how this then reflects itself in a lack of 
    smoothness of $A(t, \cdot)$ for each $t >0$. The function $A(t, \cdot)$ 
    need not even be in $H_1(M)$ for each $t >0$.
On the other hand  Part e) of the theorem  
 shows that the singularity disappears
from ratios. This lies behind the assertion in Theorems  \ref{thmeu>} and \ref{thmeu=}  
 that the solution is gauge equivalent to a strong solution, which is in fact $C^\infty$ for some time.
 See Theorem \ref{thmreca} for a precise statement.
}
\end{remark}

\subsection{$g$ estimates}      \label{secgests} 

We will  prove in this section that the functions $g_\epsilon$ converge
 as $\epsilon\downarrow 0$, but  in a much weaker sense than that asserted
  in Theorem \ref{thmgconv}.

\begin{lemma} \label{lemgconv} Let $2 \le p <\infty$. 
Under the hypotheses of Theorem \ref{thmgconv} the functions
$(0, T] \ni t \mapsto g_\epsilon(t)$ are continuous functions into $L^p(M; End\, \V)$. There is a continuous function $g:[0, T] \rightarrow L^p(M; End\, \V)$ to which the
 functions $g_\epsilon$ converge, uniformly over $(0, T]$.  That is,
\beq
\sup_{0 <t \le T}\| g_\epsilon(t) - g(t)\|_p \rightarrow 0\ \ \ 
                                              \text{as}\ \ \ \epsilon \downarrow 0.     \label{rec248}
\eeq
 Moreover $g(0, x) = I_\V$ for all $x \in M$. For each $t \in [0, T]$, 
 $g(t,x)$ lies in $K$  
 for almost all  $x \in M$. 
       If $a > 1/2$ then \eref{rec248} holds also for $p = \infty$. In this case  $g(\cdot, \cdot)$
       is continuous on $[0, T] \times M$ into $K$.
\end{lemma}

\begin{remark}{\rm In the critical case $a = 1/2$ it seems doubtful that  the function
 $x\mapsto g(t, x)$ need be 
 continuous on $M$  for any fixed $t > 0$. The strong sense of convergence asserted
  in Theorem \ref{thmgconv}, Part b) does not assure that $ g(t, \cdot)$ is continuous
   for fixed $ t >0$ because the metric on $\G_{3/2}$ does not control  the supremum
    norm on differences
  $g_\epsilon(t, \cdot) - g_\delta(t, \cdot)$ in $End\, \V$. Thus in the critical case there 
  may   be a bundle change for each $t > 0$.
  }
  \end{remark}

\bigskip
\noindent  
          \begin{proof}[Proof of Lemma \ref{lemgconv}]
Let  $\phi(t) = d^* C(t)$ as in \eref{vs20}.   
 The differential equation \eref{rec100} implies,
for $0 < \delta \le \epsilon$ and for each point $x \in M$,  that 
$g_\delta(t) = g_\epsilon(t) g_\delta(\epsilon)$ for all $t \in (0,T]$. Hence
$g_\delta(t) - g_\epsilon(t) = g_\epsilon(t)( g_\delta(\epsilon) - I_\V)$. 
    Moreover, by \eref{rec100} one has $g_\delta' (t,x) = \phi(t,x)g_{\delta}(t, x)$ 
    and therefore 
    $\|g_\delta(\epsilon, x) - I_\V\|_{op} = \|\int_\delta^\epsilon \phi(t,x) g_\delta(t, x)dt\|_{op}
    \le \int_\delta^\epsilon \| \phi(t,x)\|_{op} dt$. Hence, for each point $x \in M$ we have
\begin{align*}
\|g_\delta(t, x) - g_\epsilon(t, x)\|_{op}  =  \|g_\delta(\epsilon, x) - I_{\V}\|_{op} 
\le \int_\delta^\epsilon \| \phi(s, x) \|_{op} ds,\ \ 0 <t\le T.
\end{align*}
 Therefore, for $2 \le p < \infty$, one has 
\begin{align}
\| g_\delta(t) - g_\epsilon(t)\|_p =  \|g_\delta(\epsilon) - I_{\V}\|_p
\le  \int_\delta^\epsilon \| \phi(s)\|_p ds   \ \ \ \text{for}\ \   0 < t \le T.          \label{rec250}
\end{align}
It follows from the integrability $\int_0^T \| \phi(s)\|_p ds < \infty$,
 proven in  \eref{nd12p} (for $6 \le p <\infty$), and implied by  \eref{vs400a} (for $p=2$),
 and which holds therefore for all $p \in [2, \infty)$ by interpolation, 
 that
\begin{align}
\sup_{0 < \delta \le \epsilon}\  \sup_{0< t <T} 
\| g_\delta(t) - g_\epsilon(t)\|_p \rightarrow 0,\ \ \text{as}\ \ \
                                  \epsilon \downarrow 0.                     \label{rec101}
\end{align}
Hence the functions $g_\epsilon$ converge uniformly on $(0, T]$ to a continuous 
 function $g:(0,T]\rightarrow L^p$.
Fix $\epsilon$ in \eref{rec250} and let
$\delta\downarrow 0$ to find 
\beq
\|g(\epsilon) - I_{\V}\|_p
\le  \int_0^\epsilon \| \phi(s)\|_p ds,       \label{rec102}
\eeq
which goes to zero as $\epsilon \downarrow 0$.
Therefore,  by defining  $g(0, x) = I_{\V}$ for all $x \in M$, one obtains a continuous
 function  $g:[0, T) \rightarrow L^p(M;End\ \V)$ for $2 \le p < \infty$.
 
      In case $a > 1/2$ the inequality \eref{nd12}  
      shows that we can simply replace the $L^p$ norm
      in \eref{rec250} - \eref{rec102} by the $L^\infty$ norm. In this case the convergence
       in \eref{rec248} is uniform in both space and time. $g(\cdot,\cdot)$ is therefore continuous. 
    
   It will be useful to record here the observation that
   \beq
   \sup_{0 < \delta \le t} \| g_\delta(t) - I_\V\|_p \to 0 \ \ 
            \text{as}\ \  t\downarrow 0,\ \ \ 2 \le p < \infty,                    \label{rec103}
   \eeq
   which follows from \eref{rec250} with $\epsilon = t$, namely, 
   $ \| g_\delta(t) - I_\V\|_p  \le \int_0^t \| \phi(s)\|_p ds$.
\end{proof}

\subsection{The vertical projection} \label{secvert}

\begin{notation}{\rm In the Hilbert space $L^2(M; \L^1\otimes \kf)$ the subspace
\beq
\Hc \equiv  \{ \w \in L^2(M; \L^1\otimes \kf): d^*\w = 0\}                        \label{vert4}
\eeq
is a closed subspace of $L^2(M; \L^1\otimes \kf)$ because $d^*$ is a closed operator.
If $M\ne \R^3$ then  $d^*$ refers to the maximal operator in the case of Dirichlet
boundary conditions and to the minimal operator in the case of Neumann boundary conditions.
$\Hc$  is the horizontal subspace for the  Coulomb connection at the connection form zero.
Denote by $\Hc^\perp$ its orthogonal complement and by $P^\perp$ the orthogonal 
projection onto $\Hc^\perp$. }
\end{notation}

The next lemma concerns  the  well known projection onto
 the exact 1- forms in the Hodge decomposition. We are going to carry out some of the details
because of possible technical problems associated to Neumann boundary conditions.

        \begin{lemma}\label{lemvert}  
The restriction of $P^\perp$ to $H_1(M;\L^1\otimes \kf)$ is a bounded operator
 from $H_1$ into $H_1$.  Moreover
 \begin{align}
 dP^\perp \w &=0 \ \ \ \ \ \ \ \ \forall\  \w \in L^2(M;\L^1\otimes \kf).     \label{vert2}\\
 d^*P^\perp \w &= d^*\w\ \ \ \ \ \forall\  \w \in \D(d^*) .                \label{vert3}
 \end{align} 
\end{lemma}
         \begin{proof} If $\w \in L^2$ and $u \in \D(d^*)$ then $d^*u \in\D(d^*)$ by 
\cite[Proposition 3.5]{CG1} and $d^* d^* u =0$. Therefore $d^*u \in \Hc$. Hence, for any 1-form
$\w \in L^2(M; L^1\otimes \kf)$ we have $(P^\perp \w, d^*u) =0$ for all $u \in \D(d^*)$. Since $d$ and $d^*$ are closed operators it follows that $P^\perp\w \in \D(d)$ and $dP^\perp \w = 0$ for all
$\w \in L^2$. This proves \eref{vert2}. Now $\w - P^\perp \w \in \Hc \subset \D(d^*)$. So if $\w \in \D(d^*)$ then
$P^\perp\w \in \D(d^*)$ and  $d^*\w - d^*P^\perp \w = 0$.  This proves \eref{vert3}.
From the Gaffney-Friedrichs inequality, \cite[Equ. (2.22)]{CG1}, 
    we then have, for $\w \in H_1$,
    \begin{align*}
    \|P^\perp \w \|_{H_1}^2 
    &\le 2 \Big( \|d (P^\perp \w)\|_2^2 + \|d^*(P^\perp \w)\|_2^2 + \|P^\perp \w\|_2^2\Big)\\
    &=2 \Big(\|d^*\w\|_2^2 +  \|P^\perp \w\|_2^2 \Big) \le 2 \Big(\|d^*\w\|_2^2 +  \| \w\|_2^2 \Big) \\
    &\le 2 \|\w\|_{H_1}^2.
    \end{align*}
Thus $P^\perp:H_1\rightarrow H_1$ is bounded and \eref{vert2} and \eref{vert3} hold.
\end{proof}

\begin{remark}
 {\rm It is well known that $P^\perp$ is given by $P^\perp\w = d (d^*d)^{-1} d^*\w$
for $\w \in H_1$ under various circumstances. This is the case here also
 when $M$ is bounded and
either Neumann or Dirichlet boundary conditions are used. But we will not need this expression.
}
\end{remark}

\begin{lemma}\label{lemvert2} For any $a \in [0, 1]$ the operator 
$P^\perp: H_a \rightarrow H_a$ is bounded.
In particular, if $C(\cdot):[0,T] \rightarrow H_a$ is continuous then the function
\beq
\hat C(t) \equiv P^\perp C(t), \ \ \ 0\le t \le T   \label{vert15}
\eeq
is also continuous into $H_a$
\end{lemma}
         \begin{proof} Writing $D = (1 -\Delta)^{1/2}$ as before, we have
   \beq
\| D^a P^\perp \w\|_2 \le c_a \|D^a\w\|_2,\ \ \ a=0, 1,   \label{vert17}
\eeq
with $c_0 = 1$ 
because $P^\perp$ is a projection on $L^2$, and $c_1 \le 2$ by  Lemma \ref{lemvert}.
 By complex interpolation \eref{vert17}  holds  for all $a \in [0,1]$ with $c _ a \le 2$. 
  The assertion concerning $\hat C$ now follows.
 \end{proof}

\subsection{Integral representation of $g_\epsilon^{-1} dg_\epsilon$}    \label{sechrep}

To show that the functions $g(t, \cdot)$ constructed in Lemma \ref{lemgconv} lie 
in the gauge group $\G_{1+a}$ for each $t$ we will need  information
 about the spatial derivatives of $g$.
The next proposition gives a representation of the spatial derivatives from which we will
derive quantitative bounds  
 in   Sections \ref{sechests} and \ref{secconvh}. The simple representation of $g^{-1}dg$ which
was used in \cite{CG1} is inadequate for use in the estimates we will need in this paper. 
Instead we will use the representation in the next Proposition.

\begin{proposition}\label{prophrep} $($Representation of $g_\epsilon^{-1} dg_\epsilon$$)$
Suppose that $C(\cdot)$ is the  strong solution
  of the augmented equation \eref{aymh}  over $[0,T]$ constructed in Theorem \ref{thmaug} 
  with initial value   $ C_0 = A_0$.
Define $g_\epsilon$ as in \eref{rec100} and define 
\beq
h_\epsilon(t) = g_\epsilon(t)^{-1} d g_\epsilon(t),\ \ \  0 < t \le T           \label{rec224}. 
\eeq
Define $\hat C(t)$ as in \eref{vert15} and let
\begin{align}
 a_\epsilon(t,x) &= Ad(g_\epsilon(t,x)^{-1})\ \ 
                          \text{for}\ 0< t \le T \ \text{and}\ \ x \in M. 
                                     \label{h28e}
\end{align}
Then      
     \begin{align}
  h_\epsilon(t) = \Big(\hat C(\epsilon)  -a_\epsilon(t) \hat C(t) \Big) 
         + \int_\epsilon^t a_\epsilon(s) \chi(s) ds, \ \ \ 0 < t \le T,  \label{h32}
         \end{align}
         where         
     \beq
     \chi(s) =    [\hat C(s), \phi(s) ]  - P^\perp \Big(d_C^* B_C + [C, \phi]\Big).  \label{h33a}
     \eeq      
\end{proposition} 

The proof depends on the following lemma.

         \begin{lemma}\label{lemh3} $($An identity$)$      
      Suppose that $C(\cdot)$ is a $C^\infty$
  solution to the   augmented Yang-Mills heat equation  \eref{aymh} over some interval $(a, b)$.
 Let $g(t,x)$ be a smooth solution  to the ODE
 \beq
 g'(t,x) g(t, x)^{-1} = d^*C(t,x)       \label{h27}
 \eeq
 for each $x \in M$ and $ t \in (a,b)$. 
  Define
 \beq
 a(t,x) = Ad(g(t,x)^{-1})\ \     \text{for}\ t \in (a,b)\ \ \text{and}\ \ x \in M. \label{h28}
 \eeq
 Let $\hat C(t) = P^\perp C(t)$ as in \eref{vert15} and $\phi(t) = d^*C(t)$ as in   \eref{vs20}.
 Then, for each $($suppressed$)$ $s\in (a,b)$, we have
 \begin{align}
 d\phi = -\hat C' -P^\perp\Big(d_C^*B_C + [C, \phi]\Big). \label{h28b}
 \end{align}  
 where $\hat C'(s) =(d/ds) \hat C(s)$. Further, 
  \begin{align}
 (d/ds)( g(s)^{-1} dg(s)) =&  -(d/ds) \Big\{a(s) \hat C(s)\Big\}  \label{h30}           \\
 & + a(s)\Big\{ [\hat C(s), \phi(s) ] -P^\perp \Big(d_C^*B_C + [C, \phi]\Big)\Big\}. \notag
\end{align}
  \end{lemma}
            \begin{proof} The augmented heat equation asserts
             that $ - C' = d_C^*B_C + d\phi  + [C, \phi]$.
    Thus  $d\phi = - \{C'  + d_C^*B_C + [C,\phi]\}$. Since $d\phi$ is vertical and 
    $P^\perp C'(s) = (d/ds) P^\perp C(s)$,
    we can apply the vertical projection $P^\perp$ to this equation to find
    $ d\phi = -\hat C' -P^\perp(d_CB_C^* + [C, \phi])$, which is \eref{h28b}.

    For the derivation of \eref{h30} we need to use the following identity, which is valid for any
     continuous  $\kf$ valued 1-form $\w$ on $M$. 
\begin{align}
 a'(s) \w = a(s) [ \w, \phi].     \label{h36}
 \end{align}
This follows from the definitions \eref{h27} and \eref{h28}  and the computation, 
 at each (suppressed) point $x \in M$,
 $a'(s)\w = (d/ds)\Big(g(s)^{-1}\w g(s)\Big)  
  =g^{-1} \w (g' g^{-1})  g  - g^{-1} (g' g^{-1}) \w g
  =g^{-1} [\w, \phi] g$.

From \eref{h36} and the product rule we have 
 $ (d/ds) (a \hat C) = a\hat C' + a' \hat C = a\hat C' + a[\hat C, \phi]$, so that
 \beq
 a\hat C' =  (d/ds) (a \hat C) - a[\hat C, \phi].                          \label{h34}
 \eeq
 We will make use of the identity   
 \beq
  (g(s)^{-1} dg(s))' = a(s) d\phi(s)                     \label{h35}
  \eeq
  proved in   \cite[Eq (8.18)]{CG1}. 
 From \eref{h28b},  \eref{h34}  and \eref{h35} it follows  that 
 \begin{align*}
 (g(s)^{-1} dg(s))'                    
 & =a(s)\Big\{ -\hat C' -P^\perp\Big(d_C^*B_C + [C, \phi]\Big)\Big\} \\
 & = -(d/ds)(a\hat C) + a[\hat C, \phi] - aP^\perp\Big(d_C^*B_C + [C, \phi]\Big),
 \end{align*}
 which is \eref{h30}.
\end{proof}

\bigskip
\noindent
\begin{proof}[Proof of Proposition \ref{prophrep}]
     If we take $g$ in Lemma \ref{lemh3} to be $g_\epsilon$, defined in \eref{rec100} then
      $h_\epsilon(t)$, defined in \eref{rec224}, satisfies  $h_\epsilon(\epsilon) =0$  
       in view of the initial  condition in \eref{rec100}.
   Moreover   $a_\epsilon(\epsilon)(x) = $ the identity operator on $\kf$ for all $x \in M$. 
   The identity \eref{h30} shows   that 
   \beq
   (d/ds) h_\epsilon(s) =  -(d/ds) \Big\{a_\epsilon(s) \hat C(s) \Big\}         
 + a_\epsilon(s)\chi(s).        \label{h30e} 
 \eeq
     We may integrate  \eref{h30e} to find
  \begin{align*}
  h_\epsilon(t) &= \int_\epsilon^t (d/ds) h_\epsilon(s) ds   \\
     &=\int_\epsilon^t\Big( -(d/ds) \{a_\epsilon(s) \hat C(s)\} +a_\epsilon(s) \chi(s) \Big)ds\\
     &= - a_\epsilon(s) \hat C(s)\Big|_\epsilon^t + \int_\epsilon^t a_\epsilon(s) \chi(s) ds.
     \end{align*}
     This proves \eref{h32}. 
\end{proof}

\subsection{Estimates for $g_\epsilon^{-1} dg_\epsilon$}
\label{sechests} 

We need to make estimates of the integrand in the representation \eref{h32}.
 Our estimates will be described in the following two theorems, which differ in the nature of
 their techniques of proof. The first depends entirely on the initial behavior estimates
 made in Section \ref{secib}. The second depends on the multiplier bounds of Section \ref{secgg}.
 
 \begin{theorem} \label{thmhest1} 
 Let $ 1/2 \le a <1$ and $0 < T <\infty$.    Assume that $M$, $A_0$ and $C(\cdot)$ are as stated
 in Theorem \ref{thmgconv}. 
Define $\chi(s)$ as in  \eref{h33a}. Let
\beq
1/q_a = 1/2 - a/3.                         \label{he9}
\eeq
Then
\begin{align}
\int_0^T \|\chi(s)\|_{H_a} ds& < \infty \ \ \ \ \text{and}     \label{he10}\\
\int_0^T \|\chi(s)\|_{q} ds &< \infty \ \ \  \ \text{for}\ \ \ 3 \le q \le q_a.            \label{he11}
\end{align}
\end{theorem}

\begin{theorem} \label{thmhest2} Under the same hypotheses as in Theorem \ref{thmhest1} there holds
\begin{align}
\int_0^T \|a_\epsilon(s) \chi(s)\|_{q} ds &\le c_{21}\ \  
                                        \forall\ \epsilon \in (0, T)\ \text{and}\  \ 3 \le q \le q_a,   \label{he11.1}\\
\sup_{\epsilon > 0, t>0} \|h_{\epsilon}(t)\|_{q} &< \infty\qquad \qquad\  \ \ \
                                         \ \ \text{for}\ \ 3 \le q \le q_a,                                   \label{he12}\\
\sup_{\{\epsilon: 0< \epsilon \le t\} } \|h_\epsilon(t)\|_{q}& \rightarrow 0 \ \ 
             \text{as} \ \  t \downarrow 0\ \ \  \ \ \  \,  \text{for}\ \ 3 \le q \le q_a,       \label{he13}\\
\sup_{0 <\delta \le T} \int_0^T \| a_\delta(s) \chi(s) \|_{H_a}ds  &<\infty,       \label{he14} \\
\sup_{\{\epsilon: 0< \epsilon \le t\} } \|h_\epsilon(t)\|_{H_a}& \rightarrow 0 \ \ 
                                     \text{as} \ \  t \downarrow 0                                          \label{he15}
\end{align}
for some finite constant $c_{21}$ depending only on $C(\cdot)$.
\end{theorem}
The order in \eref{he12} - \eref{he15} reflects the order in which the proof proceeds.

\begin{corollary}\label{corhest3} 
Suppose that $M$ is as in the statement of Theorem \ref{thmgconv}.
Let $C(\cdot)$ be a strong solution to the augmented Yang-Mills heat equation \eref{aymh} over $[0,T]$ for some $T < \infty$.  Define $\chi(s)$ as in \eref{h33a}. Then
\begin{align}
\int_{\epsilon_0}^T \| \chi(s)\|_3ds &< \infty\ \ \text{for any}\ \ \epsilon_0 >0 \ \ \ 
                                                                                                             \text{and}\label{he316}\\
\int_{\epsilon_0}^T \| \chi(s)\|_{H_1} ds &< \infty\ \ \text{for any}\ \ \epsilon_0 >0.    \label{he16}
\end{align}
\end{corollary}

\begin{remark}\label{strat4}{\rm (Strategy)  
A proof of the main inequality \eref{he10}
 will require a bound on  $\|D^a \w\|_2$ when $\w = \chi(s)$ and $ 1/2 \le a <1$.
 This fractional derivative cannot be computed directly. Instead we will compute first order
 derivatives, $d\w$ and $d^*\w$ and make estimates of their $L^p$ norms for  ``small''  $p$,
  i.e. $ p < 2$. Then we will implement the heuristic 
  $\|D^a \w\|_2 = \|D^{a-1} D\w\|_2 \le \|D^{a-1}\|_{L^p \rightarrow L^2} \|D\w\|_p$, wherein
  the last inequality is a Sobolev inequality.
}
\end{remark}

  \begin{lemma}\label{lemhe3k} $($Riesz avoidance$)$
  Let $0 \le b \le 1$. Define $p$ in the interval $[6/5, 2]$ by 
 \beq
 \frac{1}{2} = \frac{1}{p} - \frac{(1-b)}{3}  .     \label{h513}
 \eeq  
 If $\w$ is a 1-form in $L^p$ with $d\w \in L^p$ and $d^*\w \in L^p$  then $ \w \in H_b$.
  There is a Sobolev constant   $\kappa_{p,2}$ such that
 \beq
 \|\w\|_{H_b} \le \kappa_{p,2} \Big(\|d\w\|_p + \| d^* \w\|_p\Big)  
            + min(\kappa_{p,2} \|\w\|_p, \|\w\|_2) .                    \label{h514}
 \eeq 
 In particular, if $\w = P^\perp \mu$ for some 1-form $\mu$ then
 \beq
 \|\w\|_{H_b} \le \kappa_{p,2} \|d^*\mu\|_p +\|P^\perp \mu\|_2  .       \label{h517}
 \eeq
 \end{lemma}
         \begin{proof} Let
  \beq
  D_j =  (d^*d + dd^* +1)^{1/2}\ \  \text{on j-forms},\ \ j = 0,1,2.
  \eeq
 For any 1-form $\w$ there holds
\begin{align}
\|\w\|_2^2 = \| dD_1^{-1} \w\|_2^2 + \|d^* D_1^{-1} \w\|_2^2 +\|D_1^{-1}\w\|_2^2,      \label{h510}
\end{align}
as follows from the computation
\begin{align*}
(D_1^{-1} d^*d &D_1^{-1} \w, \w) + (D_1^{-1}dd^* D_1^{-1}\w, \w)  +(D_1^{-2}\w, \w) \\
&= (D_1^{-1}D_1^2 D_1^{-1}, \w, \w) = \| \w\|_2^2.
\end{align*}
       Hence
\begin{align}
\| \w \|_{H_b}^2 &=\|D_1^b \w\|_2^2     \notag\\
 &= \|dD_1^{-1} D_1^b \w \|_2^2 +\|d^*D_1^{-1}D_1^b  \w \|_2^2 
                      +\|D_1^{-1} D_1^b \w\|_2^2                                          \notag\\
&=\|D_2^{b-1} d\w\|_2^2 + \|D_0^{b-1} d^* \w\|_2^2  +\|D_1^{b-1} \w\|_2^2\notag\\
&\le \|D_2^{b-1}\|_{p\rightarrow 2}^2  \|d \w \|_p^2 
                  + \| D_0^{b-1}\|_{p \rightarrow 2}^2 \| d^*\w\|_p^2   + \|D_1^{b-1} \w\|_2^2. \notag
 \end{align}
 Let $\kappa_{p,2} =  max(  \|D_j^{a-1}\|_{p\rightarrow 2}^2, j=0,1,2 )$. 
 Since $D_1^{b-1}$ is both a contraction on $L^2$ and bounded from $L^p$ to $L^2$ we have
 $\| D_1^{b-1} \w\|_2 \le  min(\kappa_{p,2} \|\w\|_p, \|\w\|_2)$. 
 Therefore
 \begin{align}
 \| \w \|_{H_b}^2 \le  \kappa_{p,2}^2 (\|d \w\|_p^2 + \|d^* \w\|_p^2 )      \notag
 +  min(\kappa_{p,2} \|\w\|_p, \|\w\|_2)^2,
 \end{align}
 which implies \eref{h514}. 
        In case $\w = P^\perp \mu$ we have $d\w = 0$ and $d^*\w = d^*\mu$ by 
        Lemma \ref{lemvert}. The inequality
    \eref{h514} therefore implies  \eref{h517} in this case. 
\end{proof}

\subsubsection{Proof of Theorem \ref{thmhest1}}

\begin{lemma}  \label{lemhe10} $($Low $p$$)$
Let $3/2 \le p \le 3$ and $(1/r) = (1/p) - (1/6)$. Then 
\begin{align}
 \| d^*\chi(s)\|_ p + \|d\chi(s)\|_p 
            \le c_5\|C(s)\|_{H_1}\Big( \|d_C^* B_C(s)\|_r + \|d\phi(s)\|_r\Big)         \label{he50}
 \end{align}
 for a constant $c_5$ depending only on a Sobolev constant and the commutator bound $c$.
\end{lemma}
       \begin{proof}
       The first of the following three identities
\begin{align}
d[\hat C(s), \phi(s)] &= - [\hat C(s) \wedge d\phi(s)],    \label{h115}\\
d^*[\hat C(s), \phi(s)] & = [\hat C(s) \lrc d\phi(s)],        \label{h116} \\
    d^*\Big(d_C^*B_C + [C, \phi]\Big) &= -\Big[C\lrc\Big(d_C^* B_C - d\phi\Big)\Big],   \label{h30c}      
\end{align}
follows from the product rule: 
$d[\hat C(s), \phi(s)] = [d\hat C(s), \phi(s)] - [\hat C(s) \wedge d\phi(s)] 
                      =- [\hat C(s) \wedge d\phi(s)]$
 because $d\hat C(s) = 0$ by \eref{vert2}.        
 Since $[d^*\hat C, \phi] = [d^* C, \phi] = [\phi, \phi] =0$, the second identity 
 follows from the product rule  $ d^* [\hat C, \phi] = [d^* \hat C, \phi] + [C\lrc d\phi]$. 
In the third identity the second term is $d^* [C, \phi] = [C\lrc d\phi]$, while the Bianchi identity
gives  $d^*d_C^*B_C = d_C^*d_C^*B_C - [C\lrc d_C^*B_C] =- [C\lrc d_C^*B_C]$.
      
    Since $d(P^\perp\w) =0$ for any 1-form $\w \in L^2$, these three identities show that 
    \begin{align}
  \|d\chi(s)\|_p &=\|d[\hat C(s), \phi(s)] - d P^\perp\Big(d_C^*B_C + [C, \phi]\Big)\|_p\notag\\
  &=\|d[\hat C(s), \phi(s)] \|_p\notag\\
    & = \|\, [\hat C(s) \wedge d\phi(s)]\, \|_p \notag\\
    & \le c \|\hat C(s)\|_6 \| d\phi(s)\|_r,                           \label{he55}
    \end{align}
    and also
    \begin{align}
  \|d^*\chi(s)\|_p &= \| d^*[\hat C(s), \phi(s)] - d^*\Big(d_C^*B_C + [C, \phi]\Big)\|_p     \notag\\
  &=\|\, [\hat C(s) \lrc d\phi(s)] +\Big[C\lrc\Big(d_C^* B_C - d\phi\Big)\Big] \|_p                \notag\\
 & \le c\|\hat C(s)\|_6 \|d\phi(s)\|_r + c\|C(s)\|_6 \|\Big(d_C^* B_C - d\phi\Big)\|_r. \label{he56}
  \end{align}
  Since $P^\perp$ is a bounded operator from $H_1$ to $H_1$ (with bound at most $\sqrt{2}$)
  we have $\|\hat C(s)\|_6 \le \kappa_6 \| \hat C(s)\|_{H_1} \le 2 \kappa_6 \| C(s)\|_{H_1}$
  and also $\| C(s)\|_6 \le \kappa_6 \| C(s)\|_{H_1}$. Insert these bounds into \eref{he55} and \eref{he56}
  to find  $\|d\chi(s)\|_p \le 2c\kappa_6\|C(s)\|_{H_1} \|d\phi(s)\|_r$ and
  $\|d^*\chi(s)\|_p \le 2c\kappa_6\|C(s)\|_{H_1} \|d\phi(s)\|_r 
  +c\kappa_6 \| C(s)\|_{H_1}\Big(\|d_C^* B_C(s)\|_r + \|d\phi(s)\|_r\Big) $.
  Add to arrive at \eref{he50} with $c_5 = 5c\kappa_6$.
\end{proof}

        \begin{lemma}  \label{lemhe11} Let $3/2 \le p \le 2$. Define $(1/r) = (1/p) - (1/6)$
 and define $b$ by  \eref{h513}. 
Then 
\begin{align}
\|\chi(s) \|_{H_b} 
 &\le \kappa_{p,2}  c_5\|C(s)\|_{H_1}\Big( \|d_C^* B_C(s)\|_r + \|d\phi(s)\|_r\Big)   \label{he57}\\
&+ \| d_C^*B_C(s) \|_2 + 3c\kappa_6 \|C(s)\|_{H_1} \| \phi(s)\|_3       .                   \label{he58}
\end{align}
\end{lemma}
         \begin{proof}  Choose $\w = \chi(s)$ in \eref{h514}  
Then  \eref{h514} 
 and \eref{he50} show  that  
\begin{align*}
\|\chi(s)\|_{H_b}& \le \kappa_{p,2} ( \|d \chi(s) \|_p + \| d^* \chi(s) \|_p ) + \| \chi(s)\|_2 \\
&\le \kappa_{p,2} c_5 \|C(s) \|_{H_1}  \Big( \|d_C^* B_C(s)\|_r + \|d\phi(s)\|_r\Big)   + \| \chi(s)\|_2.
\end{align*}
 But 
 \begin{align*}
 \|\chi(s)\|_2 &= \|\, [\hat C(s), \phi(s)] - P^\perp(d_C^* B_C(s) + [C(s), \phi(s)])\, \|_2 \\ 
 &\le \|\, [\hat C(s), \phi(s)]\, \|_2 + \| d_C^* B_C(s) + [C(s), \phi(s)] \|_2 \\
 &\le c\Big(\|\hat C(s)\|_6 + \| C(s)\|_6\Big) \|\phi(s)\|_3 +\| d_C^* B_C(s)\|_2 \\
 &\le 3\kappa_6 \| C(s)\|_{H_1} \| \phi(s)\|_3 + \| d_C^* B_C(s)\|_2,
\end{align*}
wherein  the two $L^6$ norms have been estimated in the last line just 
as in the proof of Lemma \ref{lemhe10}.
\end{proof}

           \begin{lemma}\label{lemhe12} Let  $1/2 \le a \le 1$. Define 
 \beq
 r_a^{-1} = (2/3) - (a/3).                                       \label{he19r}
 \eeq
  Then,
 for any 1-form $u(s)$ over $M$ there holds
  \begin{align}
  \int_0^Ts^a \|u(s)\|_{r_a}^2ds \le T^{a-(1/2)} \Big(\int_0^T s^{1-a} \|u(s)\|_2^2ds\Big)^\alpha 
  \Big(\int_0^T s^{2-a} \|u(s)\|_6^2ds \Big)^\beta     \label{he19a}
  \end{align} 
  where  $\alpha = (3/2) - a$ and $\beta = a - (1/2)$.
 \end{lemma}
           \begin{proof}  
 The standard interpolation   formula 
  \beq
  \| \psi\|_r \le \|\psi\|_2^\alpha \|\psi\|_6^\beta,  \label{h114}
  \eeq
 is valid for $ 2 \le r \le 6$,  $\alpha = (3/r) - (1/2)$ and  $\beta = (3/2) - (3/r)$.
 For  $ 2 \le r \le 6$    both $\alpha$ and $\beta$ are non-negative and  $\alpha + \beta= 1$. 
Take $\psi = u(s)$ and observe that
 \begin{align}
  s^a \|u(s)\|_r^2 \le s^{a -\gamma} \Big(s^{1-a} \|u(s)\|_2^2\Big)^\alpha 
  \Big( s^{2-a} \|u(s)\|_6^2 \Big)^\beta,         \label{he19}          
  \end{align}
  where $\gamma =(1-a) \alpha + (2-a) \beta  =1-a + \beta = (5/2) - a -(3/r)$.
  In case $r = r_a$ we therefore have $a -\gamma  = a -(1/2)$. Since $a - (1/2) \ge 0$
 the first factor in \eref{he19} has a non-negative exponent. Integrate both sides of \eref{he19}
 over $(0, T]$, taking the maximum of the first factor out, and use H\"older's inequality 
 to arrive at \eref{he19a}. 
 \end{proof}

\bigskip
\noindent 
 \begin{proof} [Proof of Theorem \ref{thmhest1}]
 Choosing $b = a$ in Lemma \ref{lemhe11}, we need only show that
 each of the four terms on  the 
 right hand side of \eref{he57} + \eref{he58}  is integrable over $(0, T]$. 
 Since $1/2 \le a <1$ we have $3/2 \le p < 2$.
 The value of $r$ determined in Lemma \ref{lemhe11} for the value $b=a$
 is given by $(1/r) = (1/p) - (1/6) = (1/2) +(1-a)/3 -(1/6) = (1/r_a)$. Thus we can apply
 Lemma \ref{lemhe12}.

 For the integral of the   first term  in \eref{he57} we find
 \begin{align}
 &\int_0^T\|C(s)\|_{H_1}  \|d_C^* B_C(s)\|_{r_a}  ds      \label{he20} \\
 &\le \Big(\int_0^Ts^{-a} \|C(s)\|_{H_1}^2 ds \Big)^{1/2}
 \Big(\int_0^T s^a \|d_C^* B_C(s)\|_{r_a} ^2 ds\Big)^{1/2}    \notag\\
 &\le \Big(\int_0^Ts^{-a} \|C(s)\|_{H_1}^2 ds \Big)^{1/2} \cdot     \notag \\
&\  \Big\{ T^{a-(1/2)} \Big(\int_0^T s^{1-a} \|d_C^* B_C(s)\|_2^2ds\Big)^\alpha 
  \Big(\int_0^T s^{2-a} \|d_C^* B_C(s)\|_6^2ds \Big)^\beta\Big\}^{1/2}  . \notag
 \end{align}
 All  three integrals are finite, the first by the assumption of finite strong $a$-action,
 the second by the inequality \eref{vs38b}, and the third by the
  inequality \eref{vs641a}. 
  
  The integral of the second term in \eref{he57} can be bounded the same way: One need 
  only replace $\|d_C^*B_C(s)\|_{r_a}$ by $\|d\phi(s)\|_{r_a}$ in the inequalities \eref{he20}. 
   The final step in the integrability
  argument holds again,  in virtue of  the inequalities \eref{vs38b}   and \eref{vs641b}.
  
     Concerning the first term in line \eref{he58}  we have, by \eref{vs561.1} 
     and\eref{vs640a},
      $ \|d_C^*B_C(s)\|_2 \le \|C'(s)\|_2 = o(s^{-1 + (a/2)})$, which is integrable
     over $(0, T]$. The second  term in line \eref{he58} is integrable by virtue of the inequalities
     \linebreak
  $\Big(\int_0^T \|C(s\|_{H_1} \|\phi(s)\|_3 ds\Big)^2 \le  \int_0^T s^{-a} \|C(s)\|_{H_1}^2 ds
  \int_0^T s^a \| \phi(s)\|_3^2 ds$, which is finite in view of  \eref{vs38g}, since $a \ge 1/2$. 
    This proves \eref{he10}. 
  
  Since $\|\chi(s)\|_b$ is dominated by $\|\chi(s)\|_a$ when $b \le a$ it follows that  \linebreak
  $\int_0^T \|\chi(s)\|_{H_b} ds <\infty $ for $1/2 \le b \le a$. Thus if $3 \le q \le q_a$ and 
  $q^{-1} = 2^{-1} - (b/3)$ then Sobolev gives \eref{he11}.
  This completes the proof of Theorem \ref{thmhest1}.
 \end{proof}

\subsubsection{Proof of Theorem \ref{thmhest2}}

The following three lemmas  prove Theorem \ref{thmhest2}.

\begin{lemma}\label{lemhe6} 
 There is a constant $c_{19} < \infty$, 
independent of $\epsilon $ and $t$,  such that
\begin{align}
 \|h_\epsilon(t)\|_{q} \le c_{19}\ \ \ \text{for}\ \ \ 0 <  t \le T,\ 0<\epsilon <T
                           \ \ \text{and}\ \   3 \le q \le q_a     .                                 \label{h83}
\end{align}
Furthermore
\beq
 \sup_{\{\epsilon: \epsilon \le t\}} \|h_{\epsilon}(t)\|_{q} \rightarrow 0
                            \ \ \ \text{as}\ \ t \downarrow 0 \ \ \ \text{for}\ \ \ 3 \le q \le q_a .\label{h86}
 \eeq
 \end{lemma}
         \begin{proof}
          Since the operators $a_\epsilon(s)$ are isometries in all $L^p$ spaces, 
   the representation \eref{h32} shows that 
\begin{align}
\|h_\epsilon(t)\|_q 
     &\le \|\hat C(\epsilon)\|_q + \|\hat C(t)\|_q   +\Big| \int_\epsilon^t \|  \chi(s)\|_q ds\Big|  \notag\\
&\le 2\sup_{0 <s \le T} \|\hat C(s)\|_q + \int_0^T \|  \chi(s)\|_q ds\ \ \ \ 
                                                                        \text{for}\ \ 0 < t \le  T  .       \label{h84}
\end{align}
Lemma \ref{lemvert2} shows that $\hat C(\cdot)$ is continuous on $[0,T]$ into $H_a$ and therefore
into $H_b$ for all $b \in [1/2, a]$ and therefore into $L^q$ for all $q \in [3,q_a]$. 
Hence, in view of \eref{he11},  the right side of \eref{h84} is finite. This proves \eref{h83}.

      It might be useful to note that we have obtained a bound on $\|h_\epsilon(t)\|_q$ for all 
$q \in [3, q_a]$ by  using $H_a \subset H_b$ if $a \ge b$ but not by using $L^{q_a} \subset L^q$
if $q_a > q$. The latter would require $M$ to be of finite volume, which we do not have when $M =\R^3$.
 The former holds because of  the definition \eref{ST19}. We will need a bound on 
  $\|h_\epsilon(t)\|_3$ in order to apply the multiplier bounds of Section \ref{secmb}.

   Concerning the assertion \eref{h86}, observe that  the $L^q$ norm of the 
   integral term in the representation
   \eref{h32} is at most $\int_0^t \|\chi(s)\|_{q} ds$ 
    for all $ \epsilon \in (0, t]$, which goes to zero
   as $t\downarrow0$ in view of \eref{he11}. 
            In regard to the integrated terms in \eref{h32}, observe that for any norm we have
        \begin{align}
 \|\hat C(\epsilon) - a_\epsilon(t) \hat C(t)\|
 &\le \| \hat C(\epsilon) - \hat C(0)\| + \|(1 - a_\epsilon(t)) \hat C(0) \|             \notag\\
 &+ \| a_\epsilon(t)( \hat C(0) -  \hat C(t))\|       .                                       \label{h87}
 \end{align}
 If the norm is the $L^{q}$     
 norm then, since $a_\epsilon(t)$ is isometric in this norm, the
  first and third  terms on the right add to 
  $\| \hat C(\epsilon) - \hat C(0)\|_{q} +\| \hat C(0) -  \hat C(t)\|_{q}$, 
  which goes to zero as $0< \epsilon \le t \downarrow 0$ 
   because $\hat C(\cdot)$ is continuous into $H_a$,  therefore
   into $H_b$, and therefore into $L^{q}$. 
   The second term on the right side can (and must) be treated as a strong limit
   in the sense of Lemma \ref{strcontp}. Since $\hat C(0)$ is fixed i.e. is independent
    of $\epsilon$ and $t$, Lemma \ref{strcontp} shows that this term will be small when 
    $\|g_\epsilon(t) - I_\V\|_2$ is small. The latter is assured by \eref{rec103} with $p =2$.
    This completes the proof of Lemma \ref{lemhe6}.
\end{proof}

\begin{lemma} \label{lemhe7}  
\eref{he14} and \eref{he15} hold.
\end{lemma}
\begin{proof} In view of \eref{gp3} and \eref{h83} with $q = 3$ we have
\begin{align}
 \| a_\delta(s) \chi(s) \|_{H_a} \le  c_{20} \| \chi(s)\|_{H_a},\ \ 0 < \delta <T,            \label{h87.1}
 \end{align}
 where $c_{20} = 1 + c_1 c_{19}$. Therefore
\begin{align}
            \|\int_0^T a_\delta(s) \chi(s) ds \|_{H_a} 
    \le c_{20}  \int_0^T \| \chi(s)\|_{H_a} ds. 
\end{align} 
 Since the right hand side is finite, by \eref{he10},  and 
independent of $\delta$ the inequality \eref{he14} follows. Concerning the assertion 
\eref{he15} observe that, as in the case of the $L^q$ norms, 
the integral in the representation \eref{h32}
for $h_\epsilon(t)$ is at most $\int_0^t\|a_\epsilon(s) \chi(s)\|_{H_a} ds$, which is dominated by
$c_{20}\int_0^t \| \chi(s)\|_{H_a} ds$ and which goes to zero as $t\downarrow 0$. To address
the integrated terms in \eref{h32} take the norm in \eref{h87} to be the $H_a$ norm. 
This time we need to base our estimates on the inequality  \eref{gp3}, which gives
    \begin{align}
 \|\hat C(\epsilon) - a_\epsilon(t) \hat C(t)\|_{H_a} 
 &\le \| \hat C(\epsilon) - \hat C(0)\|_{H_a} + \|(1 - a_\epsilon(t)) \hat C(0) \|_{H_a}             \notag\\
 &+ (1 +c_1 \| g_\epsilon^{-1} d g_\epsilon(t)\|_3)\|  \hat C(0) -  \hat C(t)\|_{H_a}.  \label{h88}
 \end{align}
 The factor   $\| g_\epsilon^{-1} d g_\epsilon(t)\|_3$ is bounded in $\epsilon$ and $t$ by 
 \eref{he12} (which is why \eref{he12} had to be proven first.) Since $\hat C(\cdot)$ is continuous
 into $H_a$ the first and third terms go to zero as $t\downarrow 0$, uniformly for $0< \epsilon \le t$.
As in the case of the $q$ norms, the second term on the right side can (and must)
 be treated as a strong limit, but this time in the sense of Lemma  \ref{strcontb},
 which requires that $\|g_\epsilon(t) - I_\V\|_2$ go to zero, as assured by \eref{rec103},
 and also that $\|g_\epsilon(t)^{-1}dg_\epsilon(t)\|_3$ go to zero, 
 which is assured by \eref{he13}.  
 \end{proof}

The inequality  \eref{he11.1} follows immediately from \eref{he11}. All the other 
assertions of Theorem \ref{thmhest2} have been proven in the lemmas.

\begin{remark}\label{remchi} {\rm 
 One should contrast \eref{he10} with \eref{he16}.
When $A_0 \in H_a$ with $a <1$  the allowed singularity in $\|\chi(s)\|_{H_1}$ as $s \downarrow 0$ will be too strong to ensure  that $\int_0^T \|\chi(s)\|_{H_1} ds < \infty$. 
 But \eref{he16}
avoids the singularity at zero. Actually, third order initial behavior estimates, which are not in this paper, would show that $\|\chi(s)\|_{H_1}$ is 
 bounded on $[\epsilon_0, T]$. But we only
need the integrability asserted in Corollary \ref{corhest3}.
}
\end{remark}

\bigskip
\noindent
\begin{proof}[Proof of Corollary \ref{corhest3}] 
No assumption on the nature of the initial singularity of $C(\cdot)$ has been made
 in the statement of the corollary. In particular we are not assuming finite strong a-action.
 However the conclusion of the corollary concerns the behavior of $C(\cdot)$ only on the interval
 $[\epsilon_0, T]$. Since $C(\cdot)$ is a continuous function on $[\epsilon_0/2, T]$
 into $H_1(M)$ it has finite strong a-action on the interval $[\epsilon_0/2, T]$ 
 for any $a \in [1/2, 1)$.  By shifting the origin over to $\epsilon_0/2$ we can, 
 without loss of generality, assume that $C(\cdot)$ has finite strong a-action over $[0, T]$ for any 
 $a\in [1/2, 1)$ that we choose. We will make this assumption and leave $a$ unspecified for easy
 comparison with formulas that we have already developed. 
 By doubling $\epsilon_0$ we can continue to write the distance
 from the origin as  $\epsilon_0$ rather than $\epsilon_0/2$. 
    It suffices to prove, therefore, that \eref{he316} and \eref{he16}  hold under the assumption
    that $C(\cdot)$ has finite strong a-action over $[0,T]$.

 \eref{he316} now follows from \eref{he11} because $q_a \ge 3$ for all $a \in [1/2, 1)$.

For the proof of \eref{he16} choose , in Lemma \ref{lemhe11},  $p =2$, $b=1$, and
therefore $r =3$. Then \eref{he57} and \eref{he58} assert that
\begin{align}
\|\chi(s) \|_{H_1} 
 &\le  c_5\|C(s)\|_{H_1}\Big( \|d_C^* B_C(s)\|_3 + \|d\phi(s)\|_3\Big)   \label{he57.1}\\
&+ \| d_C^*B_C(s) \|_2 + 3c\kappa_6 \|C(s)\|_{H_1} \| \phi(s)\|_3       .                   \label{he58.1}
\end{align}
The integral of line \eref{he58.1} over $[\epsilon_0, T]$ (in fact over $(0, T]$) has already been shown to be finite in the proof of Theorem \ref{thmhest1}. The first term in line \eref{he57.1}
can be estimated as in \eref{he20} thus:
\begin{align}
\int_{\epsilon_0}^T \|C(s)\|_{H_1} &\|d_C^* B_C(s)\|_3 ds       \notag\\
&\le \Big(\int_{\epsilon_0}^Ts^{-a} \|C(s)\|_{H_1}^2ds \Big)^{1/2}
             \Big(\int_{\epsilon_0}^T s^a \|d_C^* B_C(s)\|_3^2 ds \Big)^{1/2} .   \notag
\end{align}
The first factor is finite by finite strong a-action. 
It may be illuminating to note that the second integral would not necessarily be finite
 if it were extended down to $s =0$ because, for any $a \in [1/2,1)$, 
  the power $s^a$  would not be high enough
  to match with the use of the $L^3$ norm.
(The distinction between strong and almost strong solutions can be traced back to this point.)
But in our case, using 
$s^a = s^{2a -(3/2)} s^{(1-a)/2} s^{(2-a)/2}$ and the 
 interpolation $\|f\|_3^2 \le \|f\|_2 \|f\|_6$ we find
 \begin{align*}
 \int_{\epsilon_0}^T s^a &\|d_C^* B_C(s)\|_3^2 ds 
 \le \Big( \max_{\epsilon_0 \le s \le T} s^{2a -(3/2)}\Big) \times  \\    
 &\Big(\int_{\epsilon_0}^T s^{1-a}\|d_C^* B_C(s)\|_2^2 ds\Big)^{1/2}
 \Big(\int_{\epsilon_0}^T s^{2-a}\|d_C^* B_C(s)\|_6^2 ds\Big)^{1/2},
 \end{align*}
 which is finite by \eref{vs38b} and \eref{vs641a}.
The second term in line \eref{he57.1} can be estimated similarly.
\end{proof}

\subsection{Convergence of $g_\epsilon^{-1} dg_\epsilon$} 
\label{secconvh}

    \begin{lemma}\label{lemh7a} 
    Let $1/2 \le a < 1$. Under the hypotheses of Theorem \ref{thmgconv},  there is a  
     continuous function $h: [0, T]\rightarrow L^{q_a}$
such that  $h(0) = 0$ and such that,  for each number $t_1 \in (0, T]$, there holds
\begin{align}
 \sup_{t_1 \le t \le T} \| h(t) - h_{\epsilon}(t)\|_{q_a} \rightarrow 0\ \
               \text{as}\ \  \epsilon\downarrow 0 .     \label{h80a}
 \end{align}
\end{lemma}
           \begin{proof}   The representation 
      \eref{h32} for $h_\epsilon$ gives, for each (suppressed) $x \in M$,
\begin{align}
&h_\delta(t) - h_\epsilon(t) = \Big(\hat C(\delta)  - \hat C(\epsilon)\Big)
  - \Big(a_\delta(t) - a_\epsilon(t)\Big) \hat C(t)             \label{h81} \\
 &+ \int_\delta^\epsilon a_\delta(s) \chi(s) ds 
 + \int_\epsilon^t (a_\delta(\epsilon) - 1) a_\epsilon(s) \chi(s) ds \
 \ \ \text{for}\ \ \ 0 < t \le T . \label{h82}
 \end{align}
 Therefore,  for $0< \delta \le \epsilon \le t_1 \le t \le T$, we have
 \begin{align}
 &\|h_\delta(t) - h_\epsilon(t)\|_{q_a} \le \|\hat C(\delta)  - \hat C(\epsilon)\|_{q_a}
  + \|(a_\delta(\epsilon) -I_\V) \{ a_\epsilon(t) \hat C(t)\}\|_{q_a}              \label{h84d} \\
 &\qquad \qquad + \int_\delta^\epsilon \|\chi(s)\|_{q_a} ds 
 + \int_\epsilon^t \| \,|a_\delta(\epsilon) - 1|_{End V} \ | \chi(s)|_{\kf}\|_{q_a} ds .  \label{h85}
\end{align}
We need to show that each of these four terms go to zero uniformly for $ t \in [t_1, T]$ 
as  $0 < \delta \le \epsilon \downarrow 0$.

       Term \# 1  
     goes to zero, uniformly for $t  \in (0, T]$, as 
$0 < \delta \le \epsilon \downarrow 0$ because $\hat C(\cdot)$ is continuous into $H_a$ and therefore into $L^{q_a}$. 

Term \# 2  
   can be dominated for  $ 0 < t_1 \le t \le T$ as follows.
  $ \| a_\epsilon (t) \hat C(t)\|_6 \le \sup_{t_1 \le t \le T} \| \hat C(t)\|_6 < \infty $
  because $\hat C$ is continuous into $H_1$ on $[t_1, T]$, hence into $L^6$ on this interval.
  Since $a < 1$ we have $q_a < 6$. But  $\| a_\delta(\epsilon) -1 \|_p \rightarrow 0$ 
  for all $ p <\infty$ as a consequence of \eref{rec250}. 
  Therefore Term \# 2 converges to zero uniformly over $[t_1, T]$.

Term \# 3  
goes to zero, uniformly for $t \in (0, T]$,  in view of  \eref{he11}. 

Term \# 4  
goes to zero, uniformly for $t \in (0, T]$, 
 because the integrand is dominated by the integrable function $2 \|\chi(s)\|_{q_a}$
 and goes to zero for each $s$ because 
 $\Big(|(a_\delta(\epsilon, x) - 1)|_{End, \V} |\chi(s)|_\kf\Big)^{q_a} 
 \le |a_\delta(\epsilon, x) - 1|_{End\, \V}^{q_a} |\chi(s,x)|_\kf^{q_a}$,
 which goes to zero in measure  and is dominated by $2^{q_a} |\chi(s, x)|^{q_a}$.
 
         Hence there exists a function $h: (0,T] \rightarrow L^{q_a}$ to which the family $h_\epsilon(\cdot)$ converges for each $t \in (0, T]$ and uniformly on each interval $[t_1, T]$. Thus 
      $h$ is continuous from $(0, T]$ into $L^{q_a}$. Moreover 
      $\|h(t) \|_{q_a} 
  \le \sup_{0 <\epsilon \le t} \|h_\epsilon(t)\|_{q_a} \rightarrow 0$ as $t\downarrow 0$ by \eref{he13}. Thus we may define $h(0) =0$ to fulfill all the requirements of the lemma.
\end{proof}

  \begin{lemma} \label{lemh9} 
  Let $1/2 \le a < 1$.
   $h$ is a continuous function on $[0,T]$ into $H_a$. 
   Moreover, for any number $t_1 \in (0, T]$,   
    there holds 
      \begin{align}
      \sup_{t_1 \le t \le T} \| h(t) - h_\epsilon(t)\|_{H_a} \rightarrow 0\ \ 
                    \text{as}\ \ \epsilon\downarrow 0                                   \label{h190}
\end{align}
\end{lemma}
              \begin{proof} As in the proof of Lemma \ref{lemh7a} we will show that the functions
$h_\delta: [t_1, T] \rightarrow H_a$ form a uniformly Cauchy sequence. From the representation \eref{h81}
we have, for $0 <\delta \le \epsilon \le t_1 \le t  \le T$, 
\begin{align}
&\|h_\delta(t) - h_\epsilon(t)\|_{H_a} \le \|\hat C(\delta)  - \hat C(\epsilon)\|_{H_a}
  + \|\Big(a_\delta(t) - a_\epsilon(t)\Big) \hat C(t)\|_{H_a}         \notag\\
 &\qquad\qquad+ \int_\delta^\epsilon \|a_\delta(s) \chi(s)\|_{H_a} ds 
 + \int_\epsilon^t \|(a_\delta(\epsilon) - 1) a_\epsilon(s) \chi(s)\|_{H_a} ds.  \label{h85a}
 \end{align}
 We will show that each of the four terms on the right hand side  
 go to zero   as $ 0 < \delta \le \epsilon \downarrow 0$.
    
          Term \# 1 goes to zero because $\hat C(\cdot)$ is a continuous function on $[0, T]$
          into $H_a$.
          
          Term \# 2 can be dominated as follows. 
     Choose $\delta_1 > 0$ such that $a + \delta_1 \le 1$.  Let $p_1 = 3/\delta_1$.     
     Then, by \eref{gp8b} with $b =a$, we have
   \begin{align}
&\|\Big(a_\delta(t) - a_\epsilon(t)\Big) \hat C(t)\|_{H_a} 
= \| (a_\delta(\epsilon) - 1) a_\epsilon(t) \hat C(t)\|_{H_a}         \notag\\
&\le \Big( \kappa_{\delta_1} \|g_{\delta}(\epsilon) -I_\V\|_{p_1} 
      + c_1 \| g_\delta(\epsilon)^{-1} d g_\delta(\epsilon)\|_3\Big)
                         \| a_\epsilon(t) \hat C(t)\|_{H_{a + \delta_1}} .           \label{h85b}
\end{align}  
The two terms  in parentheses go to zero by \eref{rec250} (with $p = p_1$) and
 by    \eref{he13}(with $q =3$).
 Now $C(\cdot): [t_1,T] \rightarrow H_1$ 
 is  continuous
because $C(\cdot)$  lies in the path space $\P_T^a$. See Notation \ref{notpatha}.  
By   Lemma  \ref{lemvert2}  $\hat C(t)$ is also continuous into $H_1$ and therefore
 also continuous into
 $H_{a + \delta_1}$. Hence 
 \begin{align}
\sup_{t_1 \le t \le T} \| \hat C(t)\|_{H_{a + \delta_1}} < \infty   \label{h182}
\end{align}
if $a +\delta_1 \le 1$.  The inequality 
\beq
\|a_{\epsilon}(t) \hat C(t)\|_{H_{a + \delta_1}} \le 
   (1 + c_1 \| h_\epsilon(t)\|_3) \| \hat C(t)\|_{H_{a+\delta_1}}    \label{h182.1}
   \eeq
   follows from \eref{gp3}.    In view of \eref{he12} with $q =3$, the last factor
    in \eref{h85b} is   therefore bounded over $[t_1, T]$. Hence Term \# 2 goes to zero
    uniformly for $t \in [t_1,T]$ as $ 0 < \delta \le \epsilon \downarrow 0$.

     Term \# 3 is dominated by $c_{20} \int_\delta^\epsilon \|\chi(s)\|_{H_a} ds$ 
    by \eref{h87.1}.   
    In view of \eref{he10} this term also goes to zero, and uniformly for $t \in (0, T]$.

Term \# 4 can be estimated  as follows. 
  Suppose that $\alpha >0$ is small.
   Choose $\epsilon_0 \in (0, T]$  so small that
   \begin{align}
   2c_{20} \int_0^{\epsilon_0}  \| \chi(s)\|_{H_a} ds < \alpha.
   \end{align}
   \eref{he10}     assures that such an $\epsilon_0$ exists. 
    Then, for $0< \delta \le \epsilon \le \epsilon_0$ and for all $t \in [\epsilon, T]$,
    we find, with the help of \eref{h87.1},   
   \begin{align}
   &\| \int_\epsilon^t (a_\delta(s) - a_\epsilon(s)) \chi(s) ds \|_{H_a}  \notag\\
   & \le \int_\epsilon^{\epsilon_0} \| (a_\delta(s) - a_\epsilon(s)) \chi(s) \|_{H_a} ds
   + \int_{\epsilon_0}^T \|  (a_\delta(s) - a_\epsilon(s)) \chi(s) \|_{H_a}ds\notag \\
   &\le 2c_{20} \int_{\epsilon}^{\epsilon_0}\|  \chi(s)\|_{H_a} 
   +  \int_{\epsilon_0}^T \|  (a_\delta(s) - a_\epsilon(s)) \chi(s) \|_{H_a} ds \notag\\
   &< \alpha +  \int_{\epsilon_0}^T \| (a_\delta(s) - a_\epsilon(s)) \chi(s) \|_{H_a} ds \label{h195}
   \end{align}
   for all $t \in [\epsilon, T]$ and for all $0< \delta \le \epsilon \le \epsilon_0$.
   It remains to show that the integral in line \eref{h195} is small for small $\epsilon$.
   
    Choose again $\delta_1 >0$ such that $a + \delta_1 \le 1$ and let  $p_1 = 3/\delta_1$.
    We can apply      \eref{gp8b} with $b = a$ and $g = g_\delta(\epsilon)$ to find
           \begin{align}
  &\int_{\epsilon_0}^T \|  (a_\delta(s) - a_\epsilon(s)) \chi(s) \|_{H_a} ds
  =  \int_{\epsilon_0}^T \| (a_\delta(\epsilon) - 1) a_\epsilon(s) \chi(s) \|_{H_a} ds      \notag\\
 &\le  \Big(\kappa_{\delta_1} \| Ad\, g_\delta(\epsilon) - 1\|_{p_1}  + c_1 \| h_\delta(\epsilon)\|_3\Big)  
 \int_{\epsilon_0}^T\|a_{\delta}(s) \chi(s)\|_{H_{a + \delta_1}} ds                                  \notag\\
 &\le  \Big(\kappa_{\delta_1} \| Ad\, g_\delta(\epsilon) - 1\|_{p_1}  
               + c_1 \|h_\delta(\epsilon)\|_3\Big)
 c_{20}  \int_{\epsilon_0}^T\| \chi(s)\|_{H_{a +\delta_1}} ds   .                          \label{h198}
  \end{align}
  In the last line we have used  \eref{h87.1}.
  The integral in the last line is finite by \eref{h181} with $b = a + \delta_1$.  
  For any $p_1 < \infty$
  the expression in large parentheses in \eref{h198} goes to zero as 
  $0<  \delta \le \epsilon \downarrow 0$ by \eref{rec250} and \eref{he13}.

  Thus Term \#4 goes to zero uniformly for $t$ in any interval $[t_1, T]$ when $t_1 >0$.  
   
  This concludes the proof that the functions $h_\epsilon(\cdot)$ are uniformly Cauchy
  in $H_a$ norm over each interval $[t_1, T]$.     
  The family of functions therefore converges
  to a continuous function $h(\cdot)$ into $H_a$ over $(0, T]$, and since the $H_a$ norm
   dominates the $L^{q_a}$ norm the function $h$ is the same as the one in Lemma \ref{lemh7a},
   over $(0, T]$.
 Since $\|h(t)\|_{H_a} \le \sup_{\epsilon \le t} \|h_\epsilon(t)\|_{H_a}$,
  which goes to zero by \eref{he15}
 as $t \downarrow 0$, it follows that $\|h(t)\|_{H_a} \rightarrow 0$ as $t \downarrow 0$. 
 Since $h(0) =0$ by Lemma \ref{lemh7a}, this concludes the proof of  Lemma \ref{lemh9}. 
    \end{proof}

\subsection{Smooth ratios}

\begin{lemma}\label{smrat} $($Smooth ratio$)$ Let $0 < \tau < T$.  
Define a function $k:(0, T]\times M \rightarrow K$ by
\beq
 g(t) = k(t) g(\tau)\ \ \text{on}\ \ (0, T]\times M,        \label{rec405}
 \eeq
 where $g(t)$ is the gauge function constructed in Lemma \ref{lemgconv}.
 Then $k \in C^\infty ((0, T) \times M; K)$ and, for each (suppressed) point $x \in M$, 
 is the solution to 
 \beq
k'(t) k(t)^{-1} = d^*C(t),   
\ \ \  0 < t < T, \ \ \ \ \ \ \ \ k(\tau) = I_\V.                       \label{rec400}
\eeq  
$k(\cdot)$ satisfies the boundary conditions
 \begin{align}
 (k(t)^{-1} d k(t))_{norm} &= 0\ \ \text{for}\ \ 0 < t < T\ \ \text{in case}\ \ \ (N)          \label{rec400N}\\
 (k(t)^{-1} d k(t))_{tan} &= 0 \  \ \text{for}\ \  0 < t < T\ \ \text{in case}\ \ \ (D).        \label{rec400D}
\end{align} 
 \end{lemma}
      \begin{proof} For each point $x \in M$ let $u(t,x)$ be the unique solution to  \linebreak
       $u'(t,x) u(t,x)^{-1} = d^*C(t)$ 
       on $(0, T)$ for which   $u(\tau, x) = I_\V$. Then $u(t,x)$
       lies in $K$ for all $t \in (0, T]$ and all $x \in M$.
               If $\epsilon < \tau$ then
  $g_\epsilon(t) = u(t) g_\epsilon(\tau)$ for $\epsilon \le  t < T$ because both sides satisfy
  the ODE in   \eref{rec100} and agree at $t = \tau$. 
  For fixed $t >0$ , $g_\epsilon(t)$ and $g_\epsilon (\tau)$
 converge to $g(t)$ and $g(\tau)$, respectively, in $L^p(M; End\ \V)$ by
  Lemma \ref{lemgconv}, as $\epsilon \downarrow 0$. Hence $g(t) = u(t) g(\tau)$
 for $ 0 < t < T$. Therefore $k =u $. Since $ C(\cdot) \in C^{\infty}( (0, T)\times M)$ so is
  $u$ and hence $k$.    
         
     The boundary conditions \eref{rec400N} and \eref{rec400D} follow from
      the boundary conditions
     \eref{ST11N},  respectively \eref{ST11D} for $C(\cdot)$ by the same argument
      given  in \cite[Lemma 8.7]{CG1}.
      \end{proof}

               \begin{lemma}\label{smrat2}   Suppose that $M$ is as in the statement of Theorem \ref{thmgconv}. 
Let $C(\cdot)$ be a strong solution to the augmented Yang-Mills heat equation \eref{aymh} over $[0,T]$ for some $T < \infty$.  Let $\tau >0$ and let $k(\cdot)$ be the solution to
  the initial value  problem \eref{rec400}.
                     Then for  $0 < \epsilon_0 \le \tau$ there holds
\begin{align}
\sup_{\epsilon_0 \le t \le T} \|k(t)^{-1}dk(t)\|_3 &< \infty,      \label{rec435}\\
\sup_{\epsilon_0 \le t \le T} \|k(t)^{-1}dk(t)\|_{H_1}& < \infty.  \label{rec436}
\end{align}
           \end{lemma}
            \begin{proof} Since we are only concerned with the behavior of $k(t)$ 
    for $t \ge \epsilon_0$ we can assume without loss of generality, by the argument
    in the proof of Corollary \ref{corhest3},  that $C(\cdot) \in \P_T^a$
    for any $a \in [1/2, 1)$.
    In Lemma  \ref{lemh3} choose for the function $g$ the function $k$ defined
   in \eref{rec400}. Since $dk(\tau) =0$, we learn from \eref{h30} that
 \begin{align}
 k(t)^{-1} d k(t) 
     =   \Big(a(s) \hat C(s) \Big)\Big|_t^\tau + \int_\tau^t a(s) \chi(s) ds,\ \ \  t >0,  \label{rec437}
 \end{align}
 where $a(s) = Ad\ k(s)^{-1}$ and $\chi(s)$ is again given by \eref{h33a}. 
  Then
 \begin{align*}
 \|k(t)^{-1} d k(t)\|_3 
    & \le \Big\|  \Big(a(s) \hat C(s) \Big)\Big|_t^\tau \Big\|_3 \pm \int_\tau^t \|\chi(s)\|_3 ds \\
 &\le \|\hat C(\tau)\|_3 + \| 
 \hat C(t)\|_3 + \int_{\epsilon_0}^T  \|\chi(s)\|_3 ds
 \end{align*}
 The integral is finite by  \eref{he316}. 
 $\hat C(\cdot)$ is a continuous function into $H_a$
 by Lemma \ref{lemvert2} and therefore into $H_{1/2}$ and therefore into $L^3(M)$. Hence
 $\| \hat C(t) \|_3$ is bounded on $[\epsilon_0, T]$. This proves \eref{rec435}.
 
    To prove \eref{rec436} we will use the representation \eref{rec437} again.
    In view of \eref{gp3} we have, for $\epsilon_0 \le t \le T$, 
      \begin{align*}
  \| k(t)^{-1} dk(t)\|_{H_1} \le \Big( 1 + c_{1} \gamma_{3}\Big)
        \Big(\|\hat C(\tau)\|_{H_1} +  \|\hat C(t)\|_{H_1} + \int_{\epsilon_0}^T \|\chi(s)\|_{H_1} ds\Big)
        \end{align*}
        where $\gamma_3$ denotes the left side of \eref{rec435}. $C(\cdot)$, 
         and therefore $\hat C(\cdot)$, are  continuous functions on $[\epsilon_0, T]$ 
          into $H_1$ because $C(\cdot)$ lies in $\P_T^a$. 
     The lemma now follows from \eref{he16}.  
 \end{proof}

\subsection{Proof of Theorem \ref{thmgconv}}
          Most of the steps in the proof of Theorem \ref{thmgconv} have been carried out
  in the preceding subsections. In Section \ref{secgests} we showed that the functions
  $g_\epsilon : (0,T] \times M \rightarrow K \subset End\, \V$ converge as functions of $t$ into
  $L^p(M; End\, \V)$ and in fact uniformly for  $t \in (0, T]$.  To prove convergence in the sense
  of the metric groups $\G_{1+a}$ one must  show that the logarithmic derivatives
  $h_\epsilon(t) \equiv g_\epsilon(t)^{-1} d g_\epsilon(t)$ converge in $H_a$.  
  It was first shown that the functions $h_\epsilon(\cdot)$ converge in $L^{q_a}$,
   in Lemma \ref{lemh7a},   and then shown, in Lemma \ref{lemh9}, 
    that they also converge in $H_a$,
   with both convergences uniform for $t \in [t_1, T]$ for each $t_1 \in (0, T]$.
  Therefore the functions $g_\epsilon(t)$ converge uniformly over $[t_1, T]$ to $g(t)$
  in the sense of the metric group $\G_{1+a}$. $h(\cdot)$ and $g(\cdot)$ are 
  therefore continuous on $(0, T]$ into $H_a$ and  $\G_{1+a}$ respectively.
  
  The limit function $h(t)$ converges to zero in $H_a$ as $t\downarrow 0$ by virtue 
  of \eref{he15}.
  The limit function $g(t)$ therefore converges to the identity operator
    on $L^2(M; \V)$ in the sense of  the $G_{1+a}$ topology as $t \downarrow 0$.
     $h$ and $g$ are therefore continuous  into $H_a$ and $\G_{1+a}$, respectively, over $[0, T]$.
    This proves Theorem \ref{thmgconv}, Parts  b), c)  and d).
    
    The smoothness of ratios asserted in Part e)
  of Theorem \ref{thmgconv} is proved in Lemma \ref{smrat} because the function $g(t)$
  constructed in Lemma \ref{lemgconv} is the function $g$ defined in Theorem \ref{thmgconv}
  as a limit in the group $\G_{1+a}$.  Since $k(t) = g(t) g(\tau)^{-1}$ and $g(t) \rightarrow I_\V$ in the sense of the metric group $\G_{1+a}$ it follows that $k(t) \rightarrow g(\tau)^{-1}$ in
   this sense also. 
   This concludes the  proof of Theorem \ref{thmgconv}.

\section{Recovery of $A$ from $C$}       \label{secrec}

In this section $M$ will be assumed to be all of $\R^3$ or the closure of a bounded, convex,
 open subset of $\R^3$ with smooth boundary.

\begin{theorem}\label{thmreca} $($Recovery of $A$ from $C$$)$
 Let $1/2 \le a <1$. Let $M=\R^3$ or be the closure of a bounded, convex,  open set in $\R^3$
 with smooth boundary.
  Suppose that $A_0 \in H_{a}(M)$
 and that $C(\cdot)$
is a strong solution to the augmented  equation \eref{aymh} with $C(0) = A_0$ and
 having  finite strong  $a$-action over $[0,T]$.
 Then there exists a continuous function
 \beq
 g:[0, T] \rightarrow \G_{1+a}                                 \label{rec214a}
 \eeq
such that $g(0) = I_\V$ and such that the gauge transform $A(\cdot)$, defined by 
\beq
A(t) = C(t)^{g(t)},  \ \ 0 \le t \le T,                               \label{rec215a}
\eeq
is an almost strong  solution to the  Yang-Mills heat equation over $(0, T]$, whose
 curvature satisfies the boundary condition \eref{s50} resp. \eref{s51}. 
The map
\begin{align}
A(\cdot): [0, T]\rightarrow H_a
\end{align}
is continuous. In particular   
 $A(t)$ converges in $H_{a}$ norm  to $A_0$ as $t\downarrow 0$.

 If  $0 < \tau < T$ and
$g_0 \equiv g(\tau)^{-1}$ then the function $t\mapsto A(t)^{g_0}$ is a strong solution to the 
Yang-Mills heat equation satisfying  the boundary condition \eref{s50} resp \eref{s51} 
as  well as  the boundary condition \eref{s52} resp. \eref{s53}.   
$A^{g_0}(\cdot)$ lies in $C^\infty((0, T)\times M;\L^1\otimes \kf)$.
The map 
\beq
A^{g_0}(\cdot):[0,T] \rightarrow H_a
\eeq
is continuous.  In particular $A(t)^{g_0}$ converges in $H_{a}$ norm 
to $A_0^{g_0}$ as $t\downarrow 0$.

 $A(\cdot)$ and $A^{g_0}(\cdot)$ have finite $a$-action:
\beq
\int_0^Ts^{-a} \|B(s)\|_2^2 ds < \infty .           \label{rec216a}
\eeq
In case $a = 1/2$ and $\|A_0\|_{H_{1/2}}$ is sufficiently small then $C(\cdot)$ has
 finite strong $(1/2)$-action and all the  preceding conclusions hold.
\end{theorem}

If $a > 1/2$ then the solution $C(\cdot)$ to the augmented Yang-Mills equation automatically
has finite a-action, as was proven in Theorem \ref{thmfa}. If $a=1/2$ and $C(\cdot)$ does not 
have finite $(1/2)$-action then there is  a weaker version of Theorem \ref{thmreca} that holds.

\begin{theorem}\label{thmrec0} $($Recovery in case of infinite action$)$ 
Let $M=\R^3$ or be the closure of a bounded, convex, open set in $\R^3$
 with smooth boundary.
Suppose that $A_0 \in H_{1/2}$ and that $C(\cdot)$ is a strong solution to the augmented equation \eref{aymh} with not necessarily finite strong $1/2$-action.
Then there exists a continuous function
\beq
 g:[0, T] \rightarrow \G_{1,2}                                 \label{rec2140}
 \eeq
such that $g(0) = I_\V$ and such that the function $A(t)$, defined by \eref{rec215a},  
is an almost strong  
solution to the  Yang-Mills heat equation over $(0, T]$. 
Its  curvature satisfies the boundary condition \eref{s50} resp. \eref{s51}. If $g_0 = g(\tau)^{-1}$
 as in Theorem \ref{thmrec} then  $A(t)^{g_0}$ is a strong solution to the Yang-Mills heat equation
  satisfying   \eref{s50} resp \eref{s51} as well as \eref{s52} resp. \eref{s53}. 
$A(t)$ converges to $A_0$ in $L^2(M)$ and $A(t)^{g_0}$ converges to $A_0^{g_0}$ in $L^2(M)$.
\end{theorem}

Actually, the function $g(\cdot)$ on $[0, T]$ that we will construct in the 
proof of Theorem \ref{thmrec0}     will be  
 a continuous function into the gauge group $\G_{1,q}$ for any $ q \in [2, 3)$. 
 See Remark \ref{rem0} for this marginal improvement.
 
\begin{remark}{\rm Theorems \ref{thmreca}  and \ref{thmrec0}  prove and extend 
 Theorem \ref{thmrec},  to all $a \in [1/2, 1)$ and to infinite action. They will be used
 in the next subsection to prove the existence portions of the two main theorems,
 Theorems \ref{thmeu>} and \ref{thmeu=}.
  The uniqueness assertions  of these theorems will be proven in  
  Section \ref{secu} after establishing apriori initial behavior  properties
   of  solutions $A(\cdot)$ to \eref{str5}.
} 
 \end{remark}

\subsection{Construction of $A$} \label{secconstruct}

In this section we will prove Theorem \ref{thmreca} 
 and its special case Theorem \ref{thmrec}. We will also prove Theorem \ref{thmrec0}.

\bigskip
\noindent
\begin{proof}[Proof of Theorem \ref{thmreca}] Denote by $g(t)$ the function constructed from $C(\cdot)$ in Theorem  \ref{thmgconv}.
In view of \eref{rec405} we may write the function $A(\cdot)$ defined 
      in \eref{rec215}    as
 \beq
 A(t) = C(t)^{g(t)} = (C(t)^{k(t)})^{g(\tau)} .              \label{rec401}    
 \eeq
 Let
\beq
\hat A(t) = C(t)^{k(t)},\ \ \ 0 < t < T.                      \label{rec402}
\eeq
Then
\beq
A(t) = \hat A(t)^{g(\tau)}. \ \ \ \ \qquad\qquad   \label{rec403}
\eeq
 $\hat A$ is a smooth function on $(0, T)\times M$ because $C$ and $k$ are smooth.
      Moreover \eref{rec400} and \eref{aymh} imply that $\hat A$ is a (smooth) solution to the
 Yang-Mills heat equation, \eref{str5} over $(0,T)$.  See \cite[Lemma 8.6]{CG1} for a proof. 
 This is the ZDS mechanism for constructing  a solution of \eref{str5} from a 
 solution of \eref{aymh}. 
        In accordance with \cite[Lemma 8.6]{CG1}, the curvature and time
  derivative of $\hat A$ and $A$ can be expressed in terms  of $C(\cdot)$ as 
\begin{align}
\hat B(t) &= k(t)^{-1} B_C(t) k(t),\ \ \ \hat A'(t) =k(t)^{-1}\Big(d_C^* B_C(t)\Big) k(t)    \label{rec420} \\
 B(t) &= g(t)^{-1} B_C(t) g(t)= g(\tau)^{-1} \hat B(t) g(\tau)          \label{rec422}\\
 A'(t) &=g(t)^{-1}\Big(d_C^* B_C(t)\Big) g(t)  =g(\tau)^{-1} \hat A'(t) g(\tau)      \label{rec424}
\end{align}
 Since $A(\cdot)$ is the gauge transform of $\hat A(\cdot)$ by a fixed gauge function $g(\tau)$,
 it is also a solution to the Yang-Mills heat equation, at least informally. We need to show that
 $\hat A(\cdot)$ is actually a strong solution and that $A(\cdot)$ is actually 
 an almost strong solution.
 
      By Corollary \ref{cornd2}  $\|B_C(t)\|_\infty$ is bounded on $[\epsilon, T]$ 
  for any $\epsilon >0$.  Secondly, $B_C(t) \in H_1(M)$ because $C(\cdot)$ is, by assumption,
  a strong solution to the augmented Yang-Mills heat equation, \eref{aymh}.
 Thirdly, $k(t)^{-1} d k(t) \in L^2(M)$ and $g(t)^{-1} dg(t) \in L^2(M)$. By the product rule, it follows
 from these three facts and the representations \eref{rec420} and \eref{rec422} 
 that both $\hat B(t)$ and $B(t)$ are in $W_1(M)$ for each $t > 0$. 
 Boundary  conditions will be discussed below. 
 Now $d_C^* B_C(t) \in L^2(M)$ by  \eref{vs561.1} and 
 \eref{vs640a}.  Therefore \eref{rec420} shows that $\hat A'(t) \in L^2(M)$ for $t >0$.
  Either of the two representations in \eref{rec424} shows 
  that $A'(t) \in L^2(M)$ for $t > 0$ also. Since $g(\cdot)$ and $k(\cdot)$ are both continuous into
  $\G_{1+a} \subset \G_{1,2}$ it is routine to show that $\hat A'$ and $A'$ are both continuous into
  $L^2(M)$. Therefore $\hat A(\cdot)$ and $A(\cdot)$ are both almost strong solutions to \eref{str5}.

 There is a distinction now between $k(t)$ and $g(t)$.  In accordance with \eref{rec221a} we know
 that $g(t)^{-1} dg(t)$ lies in $H_a(M)$, but since $a < 1$ we cannot conclude that $A(t)$, 
 which is  $g^{-1} C g + g^{-1} dg$, lies in $W_1(M)$. 
 That is,  \eref{str2} may fail and $A(\cdot)$  may therefore not be a strong solution.
 On the other hand Lemma \ref{smrat2} shows that $k(t)^{-1}dk(t) \in H_1(M)$ for all $t >0$. 
 Thus to show that $\hat A(t) \in H_1(M)$ it remains  only to show that
  $k(t)^{-1} C(t) k(t) \in H_1(M)$ for each $t >0$. 
  But,  in view of \eref{gp3}, this follows from the fact that $C(t)\in H_1(M)$ and 
  $k(t)^{-1} dk(t) \in L^3(M)$, which has been shown in \eref{rec435}. 
  Therefore $\hat A(\cdot)$ is a strong solution.

 The boundary conditions \eref{s50} - \eref{s53} for $\hat A$ and its curvature  $\hat B$
 follow from \eref{rec400N} and \eref{ST11N},
 respectively  \eref{rec400D} and \eref{ST11D}, by the same argument as
  in \cite[Corollary 8.8]{CG1}. Since $B(t) = g(\tau)^{-1} \hat B(t) g(\tau)$ the boundary
   conditions \eref{s50}, respectively \eref{s51}, hold for $B(t)$ also.  But it is well to note at this
   point that in the important case when $a =1/2$ we do not know that $g(t)^{-1} d g(t)$ satisfies any 
   particular boundary conditions and may not even have well defined boundary values
   because we know only that it lies in $H_{1/2}$. (See Theorem \ref{thmgconv}, Part d).)
   We therefore cannot assert an analog of \eref{s52} or \eref{s53} for $A$    itself. 
   By \eref{rec403} we see that $A(t)^{g(\tau)^{-1}} = \hat A(t)$, from which it follows that
   $A^{g_0}(\cdot)$ lies in $C^\infty((0,T)\times M; \L^1\otimes \kf)$, as asserted in the theorem.
   
           Concerning the continuity  of the map
\begin{align}
A(\cdot): [0, T] \rightarrow H_a,
\end{align}
observe  that  $A(t) = g(t)^{-1} C(t) g(t) + g(t)^{-1} dg(t)$, wherein 
the second term is a continuous function on $[0, T]$ into $H_a$ by virtue of
 Theorem \ref{thmgconv}, Part d). The first term lies in $H_a$ for every $ t \in [0, T]$
 by virtue of the inequality (see \eref{gp3})
 $\| g(t)^{-1} C(t) g(t)\|_{H_a} \le (1 + c_1\|g(t)^{-1} dg(t)\|_3) \|C(t)\|_{H_a}$.
 For the continuity of the first term at a point $s \in [0,T]$ we have
 \begin{align}
& \|(Ad\ g(t)^{-1})C(t) - (Ad\ g(s)^{-1}) C(s) \|_{H_a}                  
  \le \| (Ad\ g(t)^{-1})(C(t) - C(s))\|_{H_a}                           \notag \\
 &\qquad\qquad \qquad + \|\{(Ad\ g(t)^{-1}) -(Ad\ g(s)^{-1})\} C(s) \|_{H_a} \notag \\
 &\le (1 + c_1\|g(t)^{-1} dg(t)\|_3) \|C(t) - C(s)\|_{H_a}  \label{rec409}\\
 &\qquad\qquad \qquad  +  \|\{(Ad\ g(t)^{-1}) -(Ad\ g(s)^{-1})\} C(s) \|_{H_a}. \label{rec410}
 \end{align}
 The first factor in line \eref{rec409} is bounded because $t\mapsto g(t)^{-1} dg(t)$ is a continuous
 function into $H_a$ and therefore into $H_{1/2}$ and therefore into $L^3(M)$. 
 Hence, since
 $C(\cdot)$ is a continuous function into $H_a$, Line \eref{rec409} goes to zero as $t\rightarrow s$.
 In Line \eref{rec410} $s$ is fixed and we can therefore use the strong continuity 
 of the representation $\G_{1+a} \ni g\mapsto Ad\ g | H_a$, as in Corollary \ref{corgg3a},  
 to conclude that  Line \eref{rec410} also goes to zero as $t\rightarrow s$. Thus $A(\cdot)$ is 
 a continuous function on $[0, T]$ into $H_a$ and in particular $A(t)$ converges to $A_0$
 in $H_a$ norm as $t \downarrow 0$ (and not just in $L^2$).
               Finally, in view of \eref{gp3}, we have
$\|A(t)^{g_0} - A(s)^{g_0}\|_{H_{a}} = \| (Ad\, g_0^{-1})(A(t) - A(s))\|_{H_{a}} \le ( 1 + c_1\| g_0^{-1} dg_0\|_3) \|A(t) - A(s)\|_{H_{a}} \rightarrow 0$ as $t\rightarrow s$.
 Herein we have used the fact that 
$g_0 = g(\tau)^{-1} \in \G_{1+a} \subset \G_{3/2} \subset \G_{1,3}$. Thus $A(\cdot)^{g_0}$
is also a continuous function on $[0, T]$ into $H_a$ and in particular converges to its
initial value $A_0^{g_0}$ in $H_a$ norm.

  That $A(\cdot)$ and $A(\cdot)^{g_0}$ have finite $a$-action when $C(\cdot)$ has finite 
  strong $a$-action follows  from  \eref{vs411a}      
  since gauge invariance shows that  
  $$\int_0^T s^{-a} \| B(s)\|_2^2 ds
  = \int_0^T s^{-a} \| B_C(s)\|_2^2 ds <\infty.
  $$

In case $a = 1/2$ and $\|A_0\|_{H_{1/2}}$  is sufficiently small then Theorem  \ref{thmfa} 
shows that
the solution $C(\cdot)$ to the augmented equation with initial data $A_0$ has finite strong action.
Therefore all of the preceding assertions in Theorem \ref{thmreca} hold.
    This completes the proof of Theorem   \ref{thmreca} and  its special case Theorem 
    \ref{thmrec}.    
     \end{proof}

\begin{remark}{\rm  
It was pointed out in the introduction that if  $A_0:= u^{-1} du$ 
 is a pure gauge in $H_{1/2}(M)$  then the solution to the Yang-Mills
  heat equation is given by $A(t) := A_0$, which will never be in $H_1(M)$ if $A_0 \notin H_1(M)$.
  This is a simple example of an almost strong solution which is not a strong solution.
}
\end{remark}
 
\bigskip
\noindent
\begin{proof}[Proof of Theorem \ref{thmrec0}]
The proof relies on the weaker estimates for infinite action proved in Section \ref{secia}.
    We are going to use the simple expression \eref{h35} for $(d/ds) (g^{-1} dg)$ rather than
    the more complicated expression \eref{h30} because the latter  does not offer
     an advantage now. Thus we have
     \begin{align}
     h_\epsilon(t) = \int_\epsilon^t a(s) d\phi(s) ds
     \end{align}
     and therefore, for $1 \le q \le \infty$ we have
  \begin{align}
    \| h_\epsilon(t)\|_q  &\le \Big|  \int_\epsilon^t \|a(s) d\phi(s)\|_q ds\Big| \notag\\
    &=  \Big|  \int_\epsilon^t \| d\phi(s)\|_q ds\Big|.   \label{rec510}
     \end{align}   
 The case of immediate interest for us is $q = 2$. 
 For $\delta >0$ we have  
 \beq
 \int_0^T \| d\phi(s)\|_2 ds  \le \Big(\int_0^T s^{-(1/2) - \delta} ds)^{1/2}\Big)
                    \Big(\int_0^Ts^{(1/2) + \delta} \|d\phi(s)\|_2^2 ds \Big)^{1/2}  < \infty       \notag
 \eeq
 by \eref{vs551ep} if $0 < \delta < 1/2$.
  Therefore $\| h_\epsilon(t)\|_2$ remains bounded on $(0, T]$
 as $\epsilon \downarrow 0$. The standard machinery for differences, already used in 
 Section \ref{secconvh}, now shows that the functions $h_\epsilon(t)$ converge uniformly on 
 $(0, T]$,  as functions into $L^2(M; \L^1\otimes \kf)$, to a continuous
  function $h$ on $(0, T]$ into $L^2(M; \L^1\otimes \kf)$
 with limit $\lim_{t\downarrow 0} h(t) =0$. Defining  $h(0) = 0 $ extends $h$ to
  a continuous function
  on $[0, T]$ into $L^2(M; \L^1\otimes \kf)$. The same  arguments used in the
   proof of Theorem \ref{thmgconv} now show that there is a 
   continuous function $g:[0,T] \rightarrow \G_{1,2}$ such that $g(0, x) = I_\V$,
 and to which    the functions $g_\epsilon(t,x)$ converge, 
   uniformly over $(0, T]$, as functions into the metric group $\G_{1,2}$.

We need to show now that the gauge transform $A := C^g$  is an almost strong
 solution of the Yang-Mills heat equation \eref{str5} and that $A^{g_0}$ is a strong solution.
 As in the case of finite strong a-action we have $\|B_C(t)\|_\infty <\infty$ for each $t >0$ by 
 Corollary \ref{cornd2} and $B_C(t) \in H_1(M)$.
 The proof that $g(t)^{-1}B_C(t) g(t) \in W_1(M)$ is therefore the same as for the case of
 finite strong a-action because that proof made use only of these two properties of $B_C(t)$ 
 and the fact that $g(t) \in \G_{1,2}$. The same argument applies to $\hat B(t)$ in view of 
 \eref{rec420}. Each lies in $H_1(M)$ by the same argument as in the strong a-action case.
 As in the case of strong a-action, $A'(t)$ and $\hat A'(t)$ both lie in $L^2(M)$. 
 
 Just as in the case of finite strong a-action, $A(t)$ can fail to lie in $W_1(M)$, 
 whereas $\hat A(t)$ does lie  in $H_1(M)$, the latter  by virtue
          of Lemma \ref{smrat2} (for size) 
  again and     \cite[Corollary 8.8]{CG1} (for boundary conditions). 
 
 Since $g(\cdot)$ and $k(\cdot)$ are both continuous functions on $[0,T]$ into $\G_{1,2}$ it follows
 that $A(\cdot)$ and $\hat A(\cdot)$ are both continuous functions into $L^2(M)$ and
  therefore satisfy the continuity requirement \eref{str1}.

 Finally, $A(t)^{g_0} \in C^\infty((0, T) \times M)$ because it is 
 equal to $\hat A(t)$ by virtue of \eref{rec403}.
    \end{proof}

 \begin{remark}\label{rem0} {\rm (More on infinite action) 
 In Theorem \ref{thmrec0} we showed 
  that even if the solution $C(\cdot)$ does not have finite action a weaker version of the ZDS
  procedure holds. Failure to have finite action can only happen when $a =1/2$.
  If $A_0 \in H_{1/2}$ and does not have finite (1/2)-action the conversion
   function $g(\cdot)$ was only shown to be  continuous on $[0, T]$ into the
  rather large gauge group $\G_{1,2}$ rather than into the natural gauge group $\G_{3/2}$. 
  But the second order initial behavior bounds for infinite action stated in Theorem
  \ref{thmia2} can be used to show that  $g(\cdot)$ is actually continuous
  on  $[0,T]$ into the smaller  the gauge group $\G_{1,q}$ for any $ q \in [2, 3)$.  
 This would imply   that $A(t)$ converges to $A_0$ in $L^q(M)$ and 
 that $A(t)^{g_0}$ converges to $A_0^{g_0}$ in  $L^q(M)$ as $t\downarrow 0$. 
 We will omit here  the details of this marginal improvement because the critical value $q =3$ is 
 still not achieved in the infinite action case.
 }
 \end{remark}

\bigskip
\noindent
\begin{proof}[Proof of Theorems \ref{thmeu>} and \ref{thmeu=}, Existence]
Suppose that $A_0 \in H_a$.  If $a >1/2$ then Theorem \ref{thmaug} ensures that there exists
a strong solution $C(\cdot)$  to the augmented Yang-Mills heat
 equation \eref{aymh} on some interval
$[0,T]$ with initial value $A_0$ and satisfying all the hypotheses of Theorem \ref{thmreca},
which in turn assures the existence of a solution $A(t)$ to \eref{str5} 
and a gauge function $g_0$
satisfying all the conditions required in Theorem \ref{thmeu>} over the interval $[0,T]$. Since
$A(t)^{g_0} $ is a strong solution,  it lies in $H_1(M)$ for $t >0$. Therefore, by \cite{CG1},
it can be extended uniquely to a strong solution over $[0,\infty)$. One can now gauge transform
back via $g_0^{-1}$ to find an almost strong solution over all of $[0,\infty)$ which agrees with
$A(t)$ for $0 \le t \le T$. In this way we have extended the original almost strong solution
over $[0, T]$ to an almost strong solution over $[0,\infty)$. 
This proves items 1) to 5) of Theorem \ref{thmeu>}.

 If $a = 1/2$ then  Theorem \ref{thmaug} ensures that there exists
a strong solution $C(\cdot)$ to the augmented Yang-Mills heat 
equation \eref{aymh} on some interval
$[0,T]$ with initial value $A_0$ and satisfying all the hypotheses of Theorem \ref{thmrec0},
which in turn ensures that there exists a solution $A(\cdot)$ of \eref{str5} and a
 gauge function $g_0$ satisfying the
requirements 1) and 2) of Theorem \ref{thmeu=} after extending the solution to all
 of $[0,\infty)$ by the method described above.
    If, moreover, $\|A_0\|_{H_{1/2}}$ 
is sufficiently small then  Theorem \ref{thmaug} shows that the solution $C(\cdot)$ will have
finite strong (1/2)-action. Theorem \ref{thmreca} now ensures that 
conditions 3) and 4)  of Theorem \ref{thmeu=} also hold.

 This concludes the proof of the existence portions of these two theorems. The uniqueness will
 be proven in Section \ref{secu}.
\end{proof}

\subsection{Initial behavior of $A$}      
  \label{secibA}

\subsubsection{Initial behavior by energy bounds} 

\begin{notation} {\rm Let $1/2 \le a <1$. 
 For a strong solution, $A(\cdot)$,  to the Yang-Mills  heat equation over $(0,\infty)$ 
let
\beq
\rho_a(t) = (1-a) \int_0^t s^{-a} \|B(s)\|_2^2 ds.             \label{fa3}
\eeq
In accordance with Definition \ref{finaact},  a strong solution $A(\cdot)$ has 
 finite $a$-action in case $\rho_a(t) < \infty$  for some   (hence all) $t > 0$.
}
\end{notation}
     $\rho_a(t)$ is a gauge invariant function of the initial data $A_0$. All of the estimates
 in this section will be fully gauge invariant. They will depend only on finiteness of $\rho_a(t)$.
 Finite a-action, as defined by \eref{fa3}, is a slightly weaker notion than finite strong a-action,
 which we have used for $C(\cdot)$, and which is not gauge invariant.
 
We are going to derive initial behavior estimates  of orders one, two and three  for  a solution $A(\cdot)$ and then apply our Neumann domination techniques from 
Section \ref{secibN}  to derive  initial behavior bounds of $\|B(t)\|_\infty$ needed to prove uniqueness of solutions.

 \begin{proposition}\label{p.Aord1} $($Order 1$)$
     If $A(\cdot)$ is  a strong solution with  finite  $a$-action  then
\beq
 t^{1-a} \| B(t) \|_2^2 + 2 \int_0^t s^{1-a}  \| A'(s) \|_2^2 ds
                                                                             = \rho_a(t).           \label{fa10a}                                                               
 \eeq
 \end{proposition}
              \begin{proof}   For $s > 0$  the identity  
   \beq 
   (d/ds) \| B(s)\|_2^2 = - 2 \| A'(s) \|_2^2                                          \label{5.21}
   \eeq
   holds, as shown in \cite[Equ. (5.7)]{CG1}. (It is also  special case of \eref{vs522} with $\phi=0$.)   In Lemma \ref{lem50} take $ f(s) = \|B(s)\|_2^2$, $g(s) = 2\|A'(s)\|_2^2$ and
   $h(s) = 0$. Then equality holds in \eref{vs90}.  Choose $b= a$ in Lemma \ref{lem50}.
  Then  \eref{vs91} (with equality) asserts that \eref{fa10a} holds.
 \end{proof}

\begin{notation}{\rm Recall from \eref{gaf70}, $\lambda(B(s)) = 1 + \gamma\|B(s)\|_2^4$.
 We take  from \cite[Equ. (6.1)]{CG1} the notation 
 \beq
 \psi_s^t = 2\int_s^t \lambda(B(s))ds\ \ \text{and}\ \  \psi(t) = \psi_0^t. \label{fa199a}
 \eeq
}
\end{notation}
       \begin{corollary} $($Order 1$)$ For $1/2 \le a < 1$ and $0 < t < \infty$ there holds
\begin{align}
 t^{2-2a} \|B(t)\|_2^4 &\le \rho_a(t)^2\ \ \text{and}\ \ \ 
                                                            t\|B(t)\|_2^4 \le t^{2a-1} \rho_a(t)^2,\label{fa200a}\\
 \int_0^t \|B(s)\|_2^4ds &\le(1-a)^{-1}  t^{2a -1}\rho_a(t)^2,                    \label{fa201a}\\
  \sup_{0 < s \le t} s\lambda(B(s))  &<\infty ,  \ \ \ \ \  \text{and}    \label{fa202a} \\ 
  \psi(t)   &< \infty.          \label{fa203a}                 
 \end{align}
\end{corollary} 
              \begin{proof} \eref{fa10a} shows that 
  \beq
  \|B(s)\|_2^2 \le s^{a-1} \rho_a(s).  \label{fa205a}
  \eeq	
  Square this to find \eref{fa200a}. Use it once more to find
\begin{align}
\int_0^t \|B(s)\|_2^4ds &\le \int_0^t (s^{a-1}\rho_a(s)) \|B(s)\|_2^2 ds  \notag\\
 &\le t^{2a -1}\rho_a(t)\int_0^t s^{-a}\|B(s)\|_2^2 ds, \notag
 \end{align}
 which, upon using the definition \eref{fa3}, gives \eref{fa201a}.
 Since $\lambda(B(s)) = 1 +\gamma \|B(s)\|_2^4$, 
 \eref{fa202a} and \eref{fa203a} follow immediately from
 \eref{fa200a} and \eref{fa201a} respectively.
\end{proof}

  \begin{corollary} \label{c1.Aord1} $($Order 1$)$
   If $A(\cdot)$ is  a strong solution with  finite  $a$-action  then
  \begin{align}
  t^{1-a} \|B(t)\|_2^2 + 2\kappa^{-2} \int_0^t s^{1-a} \|B(s)\|_6^2 ds 
  &\le  \rho_a(t)\Big(1 + 2 \int_0^t \lambda(B(s))ds\Big) \notag \\   
  &<\infty .    \label{fa12a}
  \end{align}   
\end{corollary}  
\begin{proof}
  Since $d_A B =0$  and $d_A^*B = -A'$ the Gaffney-Friedrichs-Sobolev 
   inequality \eref{gaf50}   gives  
 \beq
 \kappa^{-2} \|B(s)\|_6^2 \le \|A'(s)\|_2^2 + \lambda(B(s)) \|B(s)\|_2^2.   \label{fa13}
 \eeq
 Therefore 
 \begin{align*}
 2\kappa^{-2} s^{1-a} \|B(s)\|_6^2 \le 2 s^{1-a} \| A'(s)\|_2^2 + 2\lambda(B(s)) (s^{1-a}\|B(s)\|_2^2).
 \end{align*}
 But $s^{1-a} \|B(s)\|_2^2 \le \rho_a(s) \le \rho_a(t)$ by \eref{fa205a}.  Therefore
  \begin{align*}
  t^{1-a} \|B(t)\|_2^2 &+ 2\kappa^{-2} \int_0^t s^{1-a} \|B(s)\|_6^2 ds  \\
  &\le   t^{1-a} \|B(t)\|_2^2 + 2 \int_0^t s^{1-a} \| A'(s)\|_2^2 
  + 2\rho_a(t)  \int_0^t  \lambda(B(s))  ds \\
 & =  \rho_a(t)  +2\rho_a(t)  \int_0^t  \lambda(B(s))  ds,
  \end{align*}
which is finite by \eref{fa203a}. 
  \end{proof}

\begin{proposition}\label{p.Aord2}$($Order 2$)$ If $A(\cdot)$ is  a strong solution with  finite  
$a$-action  then
\begin{align}
t^{2-a} \| A'(t) \|_2^2 + \int_0^t s^{2-a} e^{\psi_s^t} \| B'(s) \|_2^2 ds
 \le   e^{\psi(t)}\rho_a(t).   \label{fa100a}
\end{align}
\end{proposition}
         \begin{proof} The inequality 
\beq
(d/ds)( e^{-\psi(s)} \|A'(s) \|_2^2 ) + e^{-\psi(s)} \|B'(s) \|_2^2  \le 0        \label{fa36}
\eeq
was proved in \cite[Equ. (6.13)]{CG1}. 
     In Lemma \ref{lem50} take $f(s) =e^{-\psi(s)} \|A'(s) \|_2^2$, $g(s) =  e^{-\psi(s)} \|B'(s) \|_2^2$
     and $h(s) =0$. Choose  $b = a -1$. Then \eref{vs90} holds and \eref{vs91} shows that
\begin{align*}
t^{2-a}  ( e^{-\psi(t)} \|A'(t) \|_2^2 )  +\int_0^t s^{2-a}  e^{-\psi(s)} \|B'(s) \|_2^2
\le (2-a) \int_0^t s^{1-a}e^{-\psi(s)} \|A'(s) \|_2^2 ds.
\end{align*}
But \eref{fa10a} shows that $\int_0^t s^{1-a}e^{-\psi(s)} \|A'(s) \|_2^2 ds \le (1/2)\rho_a(t)$.
Insert this bound into the last displayed inequality and multiply by $e^{\psi(t)}$ to find \eref{fa100a}.
\end{proof}

The bounds in the preceding inequalities depend on $t$ and on $\rho_a(t)$. It will be convenient
to emphasize this kind of dependence in the following, slightly more complicated inequalities
in terms of a standard kind of bounding function. We will call a continuous function  from
$[0,\infty)^2$ to $[0, \infty)$ a {\it standard dominating 
  function}  if it is zero at $(0,0)$ 
and non-decreasing in both arguments. In the following inequalities quantities arising from
previous estimates are bounded by standard dominating functions and consequently 
the new bounds are easily seen to be bounded by new standard dominating functions.

       \begin{corollary}\label{c1.Aord2} $($Order 2$)$ $($$L^6$ estimates.$)$  If $A(\cdot)$ is  a 
       strong solution with   finite  $a$-action  then
\begin{align}
 t^{2-a} \| B(t)\|_6^2  
               &+\int_0^t s^{2-a} e^{\psi_s^t} \| A'(s) \|_6^2 ds      \label{fa114}\\
&\le e^{\psi(t)}\rho_a(t) \Big(1+  t\lambda(B(t))e^{-\psi(t)} 
 +   \int_0^t \lambda(B(s))) ds \Big)                           \notag \\
& \le C_1(t, \rho_a(t))                                                 \notag
\end{align}
for some standard dominating function $C_1$.
\end{corollary} 
        \begin{proof} 
Since $d_A^* A' = 0 $ and $d_A A' = B'$, the Gaffney-Friedrichs-Sobolev
 inequality  \eref{gaf50}  gives
\begin{align}
\kappa^{-2}\|A'(s)\|_6^2 &\le \|B'(s)\|_2^2 + \lambda(B(s)) \|A'(s) \|_2^2.      \label{fa23}
\end{align}
Therefore,  in view of \eref{fa13} and \eref{fa23}, 
  we have  
\begin{align}
\kappa^{-2} &\Big\{  t^{2-a}  \|B(t)\|_6^2  
               + \int_0^t s^{2-a} e^{\psi_s^t} \| A'(s) \|_6^2 ds\Big\} \notag\\
  &\le  t^{2-a} (\|A'(t)\|_2^2 +  \lambda(B(t)) \| B(t)\|_2^2) \notag\\
       &\qquad \qquad  \ \ \        +\int_0^t s^{2-a} e^{\psi_s^t}\Big(\|B'(s)\|_2^2 
                    + \lambda(B(s)) \|A'(s) \|_2^2\Big)ds                        \label{fa115a}  \\
 & \le  e^{\psi(t)}  \rho_a(t)     +   t \lambda(B(t)) (t^{1-a} \| B(t)\|_2^2) 
+ \int_0^t e^{\psi_s^t} \lambda(B(s)) e^{\psi(s)}\rho_a(s) ds. \notag\\
&\le  e^{\psi(t)}\rho_a(t)  +  t \lambda(B(t)) \rho_a(t)  
      +  e^{\psi(t)} \rho_a(t) \int_0^t \lambda(B(s)) ds.     \notag
 \end{align} 
 We have applied \eref{fa100a} twice to  terms in line \eref{fa115a},
 once for dominating the sum of the first and third
  terms and once for dominating the   factor $s^{2-a} \|A'(s)\|_2^2$ in the integral. 
 In the transition to the last line 
  we have used   $e^{\psi_s^t}  e^{\psi(s)} = e^{\psi(t)}$ along with \eref{fa205a}.
  The last line is finite in virtue of \eref{fa202a} and \eref{fa203a}.
  \end{proof}

\begin{corollary}\label{c2.Aord2a} $($Order 2$)$
 $($Energy bounds$)$. Let $1/2 \le a <1$. Denoting again by $\#$ the pointwise product as in \eref{vs511}
 we have
\begin{align*}
\int_0^T s^{2-a} \| B(s)\# B(s)\|_2^2 ds < \infty.
\end{align*}
\end{corollary}
   \begin{proof} Just as in the proof of \eref{nd84} we have the bound
   \begin{align*}
   s^{2-a} &\| B(s)\# B(s)\|_2^2  \notag\\
   &\le c^2 s^{a-(1/2)}\Big(s^{(1-a)/2} \|B(s)\|_2\Big) 
                \Big(s^{(2-a)/2} \|B(s)\|_6\Big) \Big(s^{1-a} \|B(s)\|_6^2 \Big). 
   \end{align*}    
    where $c$ is the commutator bound in $\kf$.
The first two factors in parentheses are bounded by  \eref{fa12a} and \eref{fa114},
respectively.
The third factor is integrable by \eref{fa12a}.
   \end{proof}

        \begin{proposition} \label{p.Aord3} $($Order 3$)$ 
  For $1/2 \le a <1$ and $0 < t <\infty$ there holds
      \begin{align}
t^{3-a} \| B'(t)\|_2^2 
 &+ \int_0^t s^{3-a}e^{\psi_s^t} \Big( \|A''(s)\|_2^2 
                             +(1/2) \|d_{A(s)}^* B'(s) \|_2^2\Big) ds \le C_2(t, \rho_a(t)) \label{gf10}                                      
 \end{align}
 for some standard dominating function $C_2$.
\end{proposition}
The proof depends on the following lemmas.

    \begin{lemma} \label{lemId2} $($Integral Identity$)$ 
 \begin{align}
  (d/ds) &\|B'(s)\|_2^2 +   \|A''(s)\|_2^2 + \|d_{A(s)}^* B'(s)\|_2^2      \notag\\
      &=  \|\, [A'(s) \lrc B(s)]\, \|_2^2 
                         + 2( [A'(s)\wedge A'(s) ], B'(s)).   \label{5.23}
 \end{align}
\end{lemma}
\begin{proof}
The  first two of the identities
  \begin{align}
 A''  &=  -d_A^* B'   - [A'\lrc B]  \label{B32}\\
 B''  &=d_A A'' + [A' \wedge A']    \label{B33}\\
 d_A B' &= [B\wedge A']    \label{B35}
 \end{align} 
 follow by differentiating with respect to $s$,  first the Yang-Mills heat equation itself
 and then the identity  $B' = d_A A'$. The third follows from Bianchi's identity:
 $ d_A B' = (d_A)^2 A' = [B\wedge A']$. 
 From \eref{B33} we find that
   \begin{align*}
 (d/ds) \|B'(s)\|_{L^2}^2 & = 2(B'', B')\\
 & = 2(d_A A'' +[A'\wedge A'], B')\\
 &= 2(A'', d_A^* B') +  2([A'\wedge A'], B').
 \end{align*}
 We may evaluate the first term on the right in two different ways: 
 Replace $d_A^*B'$ using \eref{B32} or replace $A''$ using \eref{B32}. We find 
  $(A'', -A'' - [ A' \lrc B]) = ( A'', d_A^* B') = (-d_A^* B' - [ A' \lrc B], d_A^* B')$.
 Adding these two representations we find
  \begin{align*}
 2(A'', d_A^* B') &= -\|A''\|_2^2 - \|d_A^* B' \|_2^2  - ( A'' + d_A^*B', [A'\lrc B])\\
 &= -\|A''\|_2^2 - \|d_A^* B' \|_2^2 + \| A'\lrc B\|_2^2.
 \end{align*}
  This  proves \eref{5.23}. 
  \end{proof}

        \begin{lemma}\label{lemgf3} $($Differential inequality, order 3.$)$
\begin{align}
(d/ds) &\|B'(s)\|_2^2 +   \|A''(s)\|_2^2 + (1/2) \|d_{A(s)}^* B'(s)\|_2^2      \notag\\
&\le (3c^2/2) \|B(s)\|_6^2 \| A'(s)\|_3^2 +  2(c\kappa)^2 \|A'(s)\|_2^2 \| A'(s)\|_3^2 \notag\\
  &\ \ \ \ \ \ \ \ \ \ +(1/2) \lambda(B(s))  \| B'(s)\|_2^2     .            \label{gf3}
  \end{align}
\end{lemma}
        \begin{proof} We need only find appropriate bounds for the terms on the right side 
        of \eref{5.23}.      
        For the first term we have the simple H\"older bound 
        $\|\, [ A' \lrc B]\, \|_2^2 \le c^2 \| B\|_6^2 \| A'\|_3^2$. 
        
        Concerning the second term in \eref{5.23} we
        may apply H\"older and then the Gaffney-Friedrichs-Sobolev inequality \eref{gaf50}
        to find
        \begin{align}
        2 |( [A' &\wedge A'], B') | \le 2c \int_M |A'|\ |A'|\ |B'| dx   \notag\\
        &\le 2c \|A'\|_2 \| A'\|_3 \| B'\|_6   \notag\\
        &\le (1/2) \Big(2c\kappa \|A'\|_2 \| A'\|_3\Big)^2 + (1/2) \kappa^{-2} \| B'\|_6^2  \notag\\
        &\le 2(c\kappa)^2   \| A'\|_2^2 \| A'\|_3^2 +(1/2) \Big( \| d_A^* B'\|_2^2 + \| d_A B'\|_2^2 
         + \lambda(B) \| B'\|_2^2\Big)   \notag\\
       & = 2(c\kappa)^2   \| A'\|_2^2 \| A'\|_3^2 +(1/2) \| d_A^* B'\|_2^2     \notag\\
       &\ \ \ \ \ \ \ +  (1/2) \| \, [ B\wedge A']\, \|_2^2  + (1/2) \lambda(B) \| B'\|_2^2,  \notag
        \end{align}
        wherein we have used \eref{B35}.
        We can cancel $(1/2) \| d_A^* B'\|_2^2$ with a half of the corresponding term on the left
        side on \eref{5.23}. Using $\|\, [A' \lrc B]\, \|_2^2 +(1/2) \| \, [ B\wedge A']\, \|_2^2
        \le c^2(3/2) \| B\|_6^2 \| A'\|_3^2$ we arrive at \eref{gf3}.
        \end{proof}

        \begin{lemma}\label{lemgf4} There are constants $c_7, c_8$ depending only 
   on Sobolev constants and the commutator bound $c$ such that
        \begin{align}
        (d/ds)&\Big( e^{-\psi(s)} \|B'(s)\|_2^2\Big) 
                    + e^{-\psi(s)}\Big( \|A''(s)\|_2^2 +(1/2) \|d_{A(s)}^* B'(s) \|_2^2\Big) \notag\\
        &\le e^{-\psi(s)} \Big\{c_7 \|B(s)\|_6^2 \| A'(s)\|_3^2 
                    +  c_8 \|A'(s)\|_2^2 \| A'(s)\|_3^2\Big\}.        \label{gf7}
        \end{align}
        \end{lemma}
                 \begin{proof} The first term is 
 $e^{-\psi(s)}\Big((d/ds)\|B'(s)\|_2^2 - \lambda(B(s)) \|B'(s)\|_2^2\Big)$. Therefore
 multiplication of \eref{gf3} by  $e^{-\psi(s)}$ yields \eref{gf7} if one chooses $c_7 =3c^2/2$
 and $c_8 = 2c^2 \kappa^2 $.
\end{proof}

\bigskip
\noindent
               \begin{proof}[Proof of Propostition \ref{p.Aord3}]
   We will apply Lemma \ref{lem50} with 
   $f(s) = e^{-\psi(s)} \|B'(s)\|_2^2$, 
   $g(s) =  e^{-\psi(s)} \Big( \|A''(s)\|_2^2 +(1/2) \|d_{A(s)}^* B'(s) \|_2^2\Big)$ and $h(s)$
   equal to the entire right hand side of  \eref{gf7}. Then  \eref{vs90} holds in virtue of \eref{gf7}.
   Choose $ 1- b = 3-a$, i.e. $b = a-2$.  Then \eref{vs91} shows that
             \begin{align}
   t^{3-a} &e^{-\psi(t)} \|B'(t)\|_2^2 
   + \int_0^t s^{3-a} e^{-\psi(s)}  \Big( \|A''(s)\|_2^2 +(1/2) \|d_{A(s)}^* B'(s) \|_2^2\Big) ds \notag\\
   &\le (3-a) \int_0^t s^{2-a} e^{-\psi(s)} \| B'(s)\|_2^2 ds  \notag\\
   &+\int_0^t s^{3-a} e^{-\psi(s)} \Big\{c_7 \|B(s)\|_6^2 \| A'(s)\|_3^2 
                    +  c_8 \|A'(s)\|_2^2 \| A'(s)\|_3^2\Big\}  ds  .       \label{gf12}
    \end{align}
    The first integral on the right is finite by  \eref{fa100a}
     and justifies      use of Lemma \ref{lem50}. Upon multiplying
    \eref{gf12} by $e^{\psi(t)}$ we find an inequality whose left side is the left side of  \eref{gf10}.
      It remains to show that
    the last integral in line \eref{gf12} is finite. 
   From our bounds  \eref{fa100a} and \eref{fa114}  of order two
    we have
    \beq
    \eta \equiv \sup_{0 < s \le t}
                    s^{2-a}\Big(c_7 \| B(s)\|_6^2 + c_8\| A'(s)\|_2^2\Big) < \infty.  \notag
    \eeq
    Therefore the integral in line \eref{gf12} is at most
    \begin{align*}
    \eta \int_0^t s e^{-\psi(s)} \| A'(s)\|_3^2 ds 
         &\le \eta\int_0^t s^{1/4}\|A'(s)\|_2 s^{3/4}\|A'(s)\|_6 ds \\
    &\le \eta \Big(\int_0^t s^{1/2} \|A'(s)\|_2^2 ds \Big)^{1/2} 
     \Big( \int_0^t s^{3/2} \|A'(s)\|_6^2 ds \Big)^{1/2}\\& < \infty,
    \end{align*}
    wherein we have used \eref{fa10a} and \eref{fa114}, with $a = 1/2$, which is allowed because $\rho_{1/2}(t) < \infty$ if $\rho_a(t) < \infty$ for some $a \ge 1/2$.
   \end{proof}

   \begin{corollary}\label{c.Aord3} $($Order 3$)$
   For $ 1/2 \le a <1$ and $0 < t < \infty$ there holds
   \begin{align}
   t^{3-a} \| A'(t)\|_6^2 &\le C_{3}( t, \rho_a(t))\ \ \ \text{and}       \label{gf38}\\
   \int_0^t s^{3-a} \|B'(s)\|_6^2 ds &\le C_4(t, \rho_a(t))              \label{gf39}
   \end{align}
   for some standard dominating functions $C_3$ and $C_4$.
   \end{corollary}
                    \begin{proof}  Since $d_A^* A'= -d_A^*(d_A^* B) = 0$,  the 
        Gaffney-Friedrichs-Sobolev  inequality \eref{gaf50} gives
 \begin{align*}
 \kappa^{-2}\|A'(t)\|_6^2  &\le  \|d_A A'\|_2^2   +\lambda(B(t))\|A'(t)\|_2^2 \\
 &\le\| B'(t)\|_2^2 + (1 + \gamma \|B(t)\|_2^4) \|A'(t)\|_2^2.
 \end{align*}
 We see from \eref{gf10} that $ t^{3-a}\|B'(t)\|_2^2$ is bounded.
  Moreover \eref{fa100a}
 shows that $t^{2-a} \|A'(t)\|_2^2$ is also bounded. Since, by \eref{fa202a}, 
  $t\|B(t)\|_2^4$ is  also bounded, the inequality \eref{gf38} follows.     
  
  For the proof of \eref{gf39} observe that $d_A B' = d_A d_A A' = [B\wedge A']$. The Gaffney-Friedrichs-Sobolev inequality therefore gives
\begin{align*}
\kappa^{-2} &\| B'(s)\|_6^2 
  \le \|d_A^* B'(s)\|_2^2 + \| d_A B'(s)\|_2^2 + (1 + \gamma \|B(s)\|_2^4) \|B'(s)\|_2^2 \\
  &\le  \|d_A^* B'(s)\|_2^2 + \| \ [B(s)\wedge A'(s)]\, \|_2^2  + \Big(1 + \gamma \|B(s)\|_2^4\Big) \|B'(s)\|_2^2.
  \end{align*}
Hence
  \begin{align}
  \kappa^{-2}& s^{3-a}\| B'(s)\|_6^2 \le s^{3-a}\|d_A^* B'(s)\|_2^2          \label{gf41}\\
  &+\Big(c^2 s\|B(s)\|_3^2\Big) \Big(s^{2-a} \|A'(s)\|_6^2\Big)   
   +\Big(s + s\|B(s)\|_2^4\Big) \Big(s^{2-a} \|B'(s)\|_2^2\Big).          \notag     
\end{align}
The first term on the right hand side is integrable over $[0,t]$ by \eref{gf10}. 
   Since 
   \begin{align}
   s\|B(s)\|_3^2 \le s^{a-(1/2)} \Big(   s^{(1-a)/2}    \|B(s)\|_2\Big) \Big(  s^{(2-a)/2} \|B(s)\|_6\Big), \label{gf42}
  \end{align} 
  $s\|B(s)\|_3^2$  is bounded over $(0, t]$ by \eref{fa12a} and \eref{fa114}.
   The second term on the right side of \eref{gf41} is therefore a product
of a bounded function   and an integrable function, by \eref{fa114}.

The third term  is also a product of a bounded function, by \eref{fa202a} and an
integrable function, by \eref{fa100a}.
     \end{proof}

\subsubsection{Initial behavior by Neumann domination}

\begin{proposition}\label{p.Anda} $($Neumann Domination$)$  
    Let $1/2 \le a <1$.    Let $A(\cdot)$ be a strong solution to the Yang-Mills heat equation
     with finite a-action. Then, for $0 < T <\infty$, there holds
       \begin{align}
\int_0^T t^{(3/2)-a} \|B(t)\|_\infty^2 ds &< \infty.                \label{u28a}
\end{align}
          In particular,
    \begin{align}
    \int_0^T t \|B(t)\|_\infty^2 ds &< \infty\ \ \ \text{if}\ \ a = 1/2\ \ \     \text{and}    \label{u28} \\
    \int_0^T \|B(t)\|_\infty dt & < \infty \ \ \ \text{if}\ \ \ 1/2 < a < 1.               \label{u40}
    \end{align}
Further,
\begin{align}
\|B(t)\|_\infty  = o(t^{ \frac{a- (1/2)}{2} -1}) \ \ 
                                  \text{as}\ \ t\downarrow 0\ \ \ \text{if}\ 1/2 \le a <1.    \label{u41}
\end{align}
In  particular,
\begin{align}
\|B(t)\|_\infty &= o(t^{-1})\ \  \text{as} \ \  t\downarrow 0  \ \ 
                                                            \ \text{if}\ \ a = 1/2.\ \ \qquad \ \ \    \label{u29}                                                             
\end{align}
For $a =1/2$ we also have
\beq
 \|A'(t)\|_\infty = o(t^{-3/2})                                \label{nda60}
 \eeq
 and 
 \beq
 \int_0^T t^2 \|A'(t)\|_\infty^2 dt < \infty.    \label{nda61}
 \eeq 
\end{proposition}

       If $a >1/2$ then $\|A'(t)\|_\infty$ has more regular behavior near zero than that indicated in \eref{nda60}
       and \eref{nda61}. But I don't anticipate a need for these extensions. 
   
       \bigskip
       \noindent
              \begin{proof}
   The Yang-Mills heat equation is a little simpler than the augmented version. 
   The  equation \eref{vs511} for $B_C$ can be replaced by
\begin{align}
B'(s) =\sum_{j=1}^3 (\n_j^{A})^2 B + B\# B.                          \label{u30n} 
\end{align}
The derivation that led to \eref{nd20} now yields instead
\begin{align*}
|B(t, x)| \le \frac{1}{t} \int_0^t e^{(t-s)\Delta_N} |B(s, \cdot)| ds (x)
+\frac{1}{t} \int_0^t e^{(t-s)\Delta_N}  s |B(s)\# B(s)| ds  (x).
\end{align*}
Using $\|e^{(t-s)\Delta_N}\|_{2\rightarrow \infty} \le c_1 (t-s)^{-3/4}$ it follows that
\begin{align}
\|B(t)\|_\infty \le  \frac{1}{t}\int_0^t c_1(t-s)^{-3/4} \Big(\|B(s)\|_2 
                                          + s \|B(s)\# B(s)\|_2\Big) ds.                    \label{u33} 
\end{align}
 Lemma  \ref{lemu1}, with $c = 3/4$, $b = (3/2) -a$ and 
 $$\beta(s) = c_1 \Big(\|B(s)\|_2 
                                          + s \|B(s)\# B(s)\|_2\Big),$$
 shows that  
\begin{align*}
\int_0^T t^{(3/2) -a}\|B(t)\|_\infty^2 dt 
 \le 2c_1^2\gamma \int_0^T \Big(s^{-a} \|B(s)\|_2^2 + s^{2-a} \| B(s)\# B(s)\|_2^2 \Big) ds.
\end{align*}
It suffices to show therefore that the right hand side is finite.
But the integral of the first term is finite  by the assumption of finite a-action. 
The integral of the second term is finite  by Corollary \ref{c2.Aord2a}. This proves \eref{u28a}.
Put $a =1/2$ in \eref{u28a} to find \eref{u28}. If $a > 1/2$ then
\begin{align*}
\int_0^T\|B(t)\|_\infty dt \le \Big(\int_0^Tt^{a- (3/2)} dt \Big)^{1/2}
              \Big(\int_0^T t^{(3/2) - a} \|B(t)\|_\infty^2 dt \Big)^{1/2} < \infty,
\end{align*}
which proves \eref{u40}.

To prove \eref{u41} return to the inequality \eref{u32} and observe that by 
\eref{fa10a} and \eref{fa114} one has 
$\|B(s)\|_2 =o(s^{(a-1)/2})$ and $\|B(s)\|_6 = o(s^{(a-2)/2})$,  respectively.
Since 
\begin{align*}
\|B(s)\# B(s)\|_2 &\le c \|B(s)\|_4^2 
   \le c \|B(s)\|_2^{1/2} \|B(s)\|_6^{3/2} \\
   & =o(s^{(a-1)/4}) o(s^{3(a-2)/4}) = 
   o(s^{a -(7/4)}) 
   \end{align*}
   we find
   \begin{align*}
   s \|B(s)\# B(s)\|_2 = o(s^{a -(3/4)}).
   \end{align*}
   Hence
   \begin{align*}
  t\|B(t)\|_\infty &\le  c_1 \int_0^t (t-s)^{-3/4}\Big(o(s^{(a-1)/2}) + o(s^{a - (3/4)}) \Big) ds \\
  &=o(t^{(a - (1/2))/2}) + o(t^{a- (1/2)})
  \end{align*}
by \eref{rec519}. This proves \eref{u41}. Put $a=1/2$ in \eref{u41} to find \eref{u29}.

       For the proofs of \eref{nda60} and \eref{nda61}  we need to take from   \cite[Equ. (46)]{CG2}  
 the identity           
 \begin{align}
 (d/dt)A'(t) = \sum_{j=1}^3(\n_j^{A(t)} )^2 A'(t) +B(t)\# A'(t) - [A'(t) \lrc B(t)].     \label{nda62}
 \end{align}
 Take $\w(t) = A'(t)$ and $h(t) =  B(t)\# A'(t) - [A'(t) \lrc B(t)]$ in \eref{nda7} to find
 \begin{align}
 |A'(t,x)| \le t^{-1}\int_0^t e^{(t-s)\Delta_N} \Big(|A'(s, \cdot)| + s |h(s, \cdot)| \Big) ds.\    (x) \label{nda63}
 \end{align}
 Therefore
 \begin{align}
 \|A'(t)\|_\infty 
 &\le t^{-1} \int_0^t \|e^{(t-s)\Delta_N} \|_{2\rightarrow \infty} \Big(\|A'(s)\|_2 + s \|h(s)\|_2 \Big) ds \notag \\
 &\le (c_1/t)\int_0^t (t-s)^{-3/4}  \Big(\|A'(s)\|_2 + s \|h(s)\|_2 \Big) ds.              \label{nda65}
 \end{align}
  
 We will show that for each $t>0$ there is a constant $k_t$ such that $k_t \rightarrow 0$ as $t \downarrow 0$
 and
 \beq
 \Big(\|A'(s)\|_2 + s \|h(s)\|_2 \Big) \le k_t\ s^{-3/4},\ \ \ 0 < s \le t.\label{nda66}
 \eeq
 Using  this estimate then 
 in \eref{nda65} yields
 \begin{align}
 \|A'(t)\|_\infty &\le (c_1/t)\int_0^t (t-s)^{-3/4} k_t s^{-3/4} ds      \notag\\
 &= k_t  t^{-3/2} \cdot c_1C_{3/4,3/4} ,  \label{nda67}
 \end{align}
wherein we have used \eref{rec519}. 
This will prove \eref{nda60} once \eref{nda66} is shown.
 
 For the proof of \eref{nda66} observe that  from the second order initial behavior bound 
   \eref{fa100a} 
 with $a = 1/2$ we have $t^{3/2} \|A'(t)\|_2^2 = o(1)$ as $t\downarrow 0$. This proves the assertion in 
 \eref{nda66} in regard to the first term. Concerning the second term we have
 \begin{align}
 \|h(s)\|_2 &= \| B(s)\# A'(s) - [ A'(s)\lrc B(s)]\, \|_2   \notag\\
 &\le 2c \|\ |B(s)|\ |A'(s)|\ \|_2                         \notag\\
 & \le 2c \|B(s)\|_3 \| A'(s)\|_6                    \label{nda69}\\
 &=o(s^{-1/2}) o(s^{-5/4})                    \notag
 \end{align}
 by      \eref{gf42}   and \eref{gf38}, with $a =1/2$.  
 Hence $s \|h(s)\|_2 = o(s^{-3/4})$. This completes the proof of \eref{nda66} and \eref{nda60}.
 It will be useful to observe for later work that the bounds used above show that $k_t$ can be chosen
 to be dominated by a standard dominating function $C_5(t, \rho_{1/2}(t))$.

           In order to prove \eref{nda61} we will apply Lemma \ref{lemu1}. We need to take $b = 2$ and $c = 3/4$
in that lemma. In this case we have $b - 2c =1/2 < 1$. So we can apply the lemma, using \eref{nda65},
 to find
\begin{align}
\int_0^T t^2 \|A'(t)\|_\infty^2 dt \le \gamma \int_0^T s^{1/2} \Big( \| A'(s)\|_2 + s \| h(s)\|_2\Big)^2 ds.
\end{align} 
Now $\int_0^T s^{1/2} \|A'(s)\|_2^2 ds < \infty$ by \eref{fa10a}  with $a =1/2$.  
Moreover, using
 the bound in \eref{nda69} we find 
\begin{align}
\int_0^T s^{1/2} (s \|h(s)\|_2)^2 ds
&\le 4c^2\int_0^T s^{5/2} \|B(s)\|_3^2 \|A'(s)\|_6^2 ds                \notag\\
&= 4c^2\int_0^T \{s \|B(s)\|_3^2\} \{ s^{3/2} \|A'(s)\|_6^2 \} ds \\
&< \infty
\end{align}
because the first factor in braces is bounded, in accordance with  \eref{gf42}
 and the second factor
is integrable in accordance with  \eref{fa114} with $a=1/2$. 
 This completes the proof of the proposition.    
\end{proof}

\subsection{Uniqueness of $A$ } \label{secu}

         \begin{theorem}\label{thmufa}  $($Uniqueness for $a = 1/2$.$)$ 
Suppose that $A_1(\cdot)$ and $A_2(\cdot)$ are two
strong solutions with finite action and the same initial value. 
Assume that, if $M \ne \R^3$, then for all $t >0$, both  satisfy the boundary
 conditions \eref{u8} in case $($N$)$ 
or \eref{u9} in case $($D$)$. 
\begin{align}
 B_j(t)_{norm}&=0\ \ \ \text{in case} \ \ \ (N)  \label{u8}\\
A_j(t)_{tan} &= 0\ \ \ \text{in case} \ \ \ (D)     \label{u9}
\end{align}
Then $A_1(t) = A_2(t)$ for all $t \ge 0$.
\end{theorem}

The proof  will require  the following lemma.

          \begin{lemma}\label{lemu3} If $A_j(\cdot)$, $j = 1,2$, are two strong
  solutions of finite action  with the same initial value then
\beq
\| A_1(t) - A_2(t)\|_2^2 = o(t^{1/2})    \text{as} \ \ \ t\downarrow 0.         \label{u24}
\eeq
\end{lemma} 
         \begin{proof} Since 
 \begin{align*}
   \|A_1(t) - A_2(t)\|_2 &\le \| A_1(t) - A_0 +A_0 - A_2(t)\|_2  \\
   &\le  \| A_1(t) - A_0\|_2 + \|A_0 - A_2(t)\|_2 ,
\end{align*}
it suffices to show that each term is $o(t^{1/4})$. For any solution $A(\cdot)$ of
 finite action one has
 \begin{align*}
 \|A(t) - A_0\|_2 &\le \int_0^t \| A'(s)\|_2 ds \\
 &= \int_0^t s^{-1/4}\Big(s^{1/4} \| A'(s)\|_2 \Big) ds \\
 &\le \Big(\int_0^t s^{-1/2} ds\Big)^{1/2} \Big(\int_0^t s^{1/2} \| A'(s)\|_2^2 ds \Big)^{1/2}\\
 &=t^{1/4} \sqrt{2}  \Big(\int_0^t s^{1/2} \| A'(s)\|_2^2 ds \Big)^{1/2} .
 \end{align*}
 The integral is finite by the energy estimate \eref{fa10a} (with $a = 1/2$) and therefore the integral is $o(1)$ as $t\downarrow 0$.
\end{proof}

\bigskip
\noindent
\begin{proof}[Proof of Theorem \ref{thmufa}]  The identity
  \cite[Equ. (8.63)]{CG1} shows that 
  \begin{align}
  \frac{d}{dt}\| A_1(t) - A_2(t)\|_2^2 
  &\le c (\|B_1(t)\|_\infty + \| B_2(t)\|_\infty )\| A_1(t) - A_2(t)\|_2^2. \label{u11} 
    \end{align}
  This was derived in \cite{CG1} in case $M$ is compact. The proof in case $M =\R^3$ is 
  easier since one need not be concerned with boundary conditions. We omit the minor changes.
  
Let $f(t) = \| A_1(t) - A_2(t)\|_2^2$  and  
$u(t) = c (\|B_1(t)\|_\infty + \| B_2(t)\|_\infty)$. Then  
\begin{align}
f'(t) \le u(t)f(t),\ \ \ t >0     \label{u15}
\end{align}
for $t >0$ by \eref{u11}. 
$f$ is continuous on $[0,T]$ because each $A_j(t)$ converges to
  $A_0$ in $L^2(M)$ as $t\downarrow 0$.
Since $f(0) = 0$ it follows that
\begin{align} 
f(t) &= \int_0^t f'(s) ds \notag\\
& \le \int_0^t u(s) f(s) ds \notag \\
&\le \Big( \int_0^t s u(s)^2ds\Big)^{1/2} \Big(\int_0^t s^{-1} f(s)^2 ds \Big)^{1/2} .  \label{u56}
\end{align}
Let $g(t) = f(t)/\sqrt{t}$ for $t >0$. By Lemma \ref{lemu3} we see that $g$ is bounded on $(0, T]$ and
in fact goes to zero as $t \downarrow 0$. For convenience we may extend $g$ continuously
to $[0, T]$ by defining  $g(0) = 0$. Let 
\beq
w(t) =  \Big( \int_0^t s u(s)^2ds\Big)^{1/2}.
\eeq
Then $w(t) < \infty$  for $0\le t \le T$ by   \eref{u28}. 
Dividing \eref{u56} by $\sqrt{t}$ we find
        \begin{align}
g(t) \le w(t) \Big( \frac{1}{t} \int_0^t g(s)^2 ds \Big)^{1/2} .                          \label{u58}
\end{align}
There is a constant $C$ such that $g(t) \le C$ for $0 \le t \le T$. Insert this bound
in the integral in \eref{u58} to find that $g(t) \le w(t)C$. We can now proceed  by induction
using the fact that $w$ is non-decreasing:
Assuming that $g(s) \le w(s)^n C$  for $ 0\le s \le T$, \eref{u58} then implies that
\begin{align*}
g(t) &\le w(t)\Big(\frac{1}{t}\int_0^t w(s)^{2n} C^2 ds\Big)^{1/2} \\
&\le w(t)\{ w(t)^n C\}.
\end{align*}
  Consequently 
$g(t) \le w(t)^{n+1} C$. Thus if $t_0>0$ is such that $w(t) \le 1/2$ for $0 \le t \le t_0$ then
$g(t) = 0$ on $[0, t_0]$. Hence $A_1(t) = A_2(t)$ on this interval.  Since $A_j(t) \in H_1(M)$
for j = 1,2 and all $t >0$ we can now use the uniqueness theorem in \cite{CG1}
for $H_1$ initial data to conclude that $A_1(t) = A_2(t)$ for all $t > 0$.
\end{proof}

\begin{remark}{\rm  (Uniqueness for $a > 1/2$) 
If a solution to the Yang-Mills heat equation has finite a-action
for  some $a \ge 1/2$ then it has finite $(1/2)$-action, as is clear from the definition
\eref{fa3}.  Our uniqueness proof applies therefore to all $a \in [1/2, 1)$.
However if $a >1/2$ then the inequality \eref{u28}, on which our proof rests,
 can be replaced by \eref{u40}. Thus for $a > 1/2$ we have
 $\int_0^t \| B_j(s)\|_{\infty} ds < \infty$, $j =1,2$ by  the apriori bound \eref{u40}. 
 The function $u$
 that appears in \eref{u15} is therefore integrable over $[0, T]$. Consequently
 the standard Gronwall argument for uniqueness is applicable: the non-negative 
 function
 $h(t) \equiv e^{-\int_0^t u(s) ds} f(t)$ has a non-positive derivative on $(0, T]$ and is
 zero at $t=0$, hence is identically zero on $[0, T]$. This is the basis
  for the uniqueness proof used in  \cite{CG1} for the case of finite energy ($a = 1$).
  Here  we see  another instance of breakdown of standard techniques
 at criticality.
 }
\end{remark}

\bibliographystyle{amsplain}
\bibliography{ymh}

\end{document}